\documentclass[a4paper, 11pt]{article}

\usepackage{lmodern}
\usepackage[margin=2.5cm]{geometry}
\usepackage[utf8]{inputenc}
\usepackage{graphicx}
\usepackage[babel=true]{csquotes}
\usepackage[english]{babel}
\usepackage{amssymb}
\usepackage{amsmath}
\usepackage{bm}
\usepackage{mathtools}
\usepackage{faktor}
\usepackage{xfrac}
\usepackage[usenames,dvipsnames,svgnames,table]{xcolor}
\usepackage{enumerate}
\usepackage{enumitem}
\setlist[enumerate, 1]{label=(\roman*)}
\setlist[itemize]{leftmargin=1.5em}
\setlist[description]{leftmargin=1em}
\setlist[itemize, 1]{label=$\blacktriangleright$}
\setlist[itemize, 2]{label=$\bullet$}
\usepackage[normalem]{ulem} % for striking out text with \sout

\makeatletter 
\newcommand\nobreakpar{\par\nobreak\@afterheading} 
\makeatother

\usepackage{ifthen}

\usepackage{longtable}
\usepackage{makecell}
\usepackage{multirow}
\usepackage{array}
\usepackage{rotating}
\usepackage{capt-of}
\usepackage{arydshln}

\def\arraystretch{1.5}%  1 is the default, change whatever you need

\usepackage{overpic}
\usepackage{pgf,tikz}
\usepackage{mathrsfs}
\usetikzlibrary{arrows}
\usepackage[outline]{contour}
\contourlength{0.6pt}

\usepackage{graphicx}
\usepackage{float}
\graphicspath{ {pics_final/} }
\DeclareGraphicsExtensions{.pdf, .jpg, .jpeg, .png}

\usepackage[hidelinks]{hyperref}  % create links
\usepackage[format=plain,labelsep=period,justification=justified,font=small,labelfont=bf, margin=1em]{caption}

\usepackage{amsthm}

\theoremstyle{plain}
\newtheorem{theorem}{Theorem}[section]
\newtheorem{corollary}{Corollary}[section]
\newtheorem{lemma}{Lemma}[section]
\newtheorem{proposition}{Proposition}[section]

\theoremstyle{definition}
\newtheorem{definition}{Definition}[section]
\newtheorem{example}{Example}[section]

\theoremstyle{remark}
\newtheorem{remark}{Remark}[section]

\newcommand{\Z}{\mathbb{Z}}

\newcommand{\E}{\mathbb{E}}
\newcommand{\R}{\mathbb{R}}
\newcommand{\C}{\mathbb{C}}

\renewcommand{\L}{\mathbb{L}}
\newcommand{\HH}{\mathbb{H}}

\newcommand{\dS}{\mathrm{d}\S}
\renewcommand{\S}{\mathbb{S}}

\renewcommand{\P}{\mathrm{P}}
\newcommand{\RP}{\R\mathrm{P}}
\newcommand{\CP}{\C\mathrm{P}}

\newcommand{\arsinh}{\operatorname{arsinh}}
\newcommand{\arcosh}{\operatorname{arcosh}}

\newcommand{\PGL}{\operatorname{PGL}}
\newcommand{\PO}{\operatorname{PO}}

\newcommand{\set}[2]{\left\{#1\ \middle| \ #2 \right\}}

\newcommand{\abs}[1]{\left\vert #1 \right\vert}

\newcommand{\scalarprod}[2]{\left\langle #1, #2 \right\rangle}

\newcommand{\dotprod}[2]{#1 \cdot #2}

\newcommand\varpm{\mathbin{\vcenter{\hbox{%
  \oalign{\hfil$\scriptstyle+$\hfil\cr
          \noalign{\kern-.3ex}
          $\scriptscriptstyle({-})$\cr}%
      }}}}

\newcommand\restrict[1]{\raisebox{-.5ex}{$\big|$}_{#1}}

\newcommand\transpose[1]{{#1}^\intercal}
\newcommand\invtranspose[1]{{#1}^{-\intercal}}

\newcommand{\p}[1]{\boldsymbol{#1}}

\newcommand{\quadric}{\mathcal{Q}}
\newcommand{\secquadric}{\widetilde{\mathcal{Q}}}

\newcommand{\cone}[2]{C_{#1}(#2)}
\newcommand{\ck}[3]{K_{#1}\left(#2, #3\right)}
\newcommand{\cksphere}[2]{S_{#2}(#1)}

\newcommand{\inv}[4]{I_{#1, #2}\left(#3, #4\right)}

\newcommand{\mob}{\mathcal{S}}
\newcommand{\hyp}{\mathcal{H}}
\newcommand{\chyp}{\overline{\hyp}}
\newcommand{\ds}{\mathrm{dS}}
\newcommand{\cds}{\overline{\ds}}

\newcommand{\ellipb}{\mathcal{O}}
\newcommand{\ellip}{\mathcal{E}}

\newcommand{\euclb}{\mathcal{C}}
\newcommand{\eucl}{\mathbf{E}}

\newcommand{\pscalarprod}[3]{\scalarprod{\pi_{#1}(#2)}{\pi_{#1}(#3)}}

\newcommand{\spheres}{\mathfrak{S}}
\newcommand{\sprojection}[1]{\pi_{#1}^\spheres}
\newcommand{\spprojection}[1]{\pi_{#1}^{\spheres*}}

\newcommand{\lag}{\mathcal{B}}

\newcommand{\laghyp}{\lag_{\mathrm{hyp}}}
\newcommand{\laghyptrafos}{\mathbf{Lag}_{\mathrm{hyp}}}

\newcommand{\lagell}{\lag_{\mathrm{ell}}}
\newcommand{\lagelltrafos}{\mathbf{Lag}_{\mathrm{ell}}}

\newcommand{\lageucl}{\lag_{\mathrm{euc}}}
\newcommand{\lageucltrafos}{\mathbf{Lag}_{\mathrm{euc}}}

\newcommand{\lie}{\mathcal{L}}
\newcommand{\liesc}[2]{\scalarprod{#1}{#2}}
\newcommand{\ospheres}{\vec{\mathscr{S}}}
\newcommand{\nospheres}{\mathscr{S}}

\newcommand{\hyplag}{\lag_{\mathrm{hyp}}}

\newcommand{\baseplane}{\mathbf{B}}

\newcommand{\lieperp}{\perp}
\newcommand{\lietrafos}{\mathbf{Lie}}

\newcommand{\mobtrafos}{\mathbf{Mob}}

\newcommand{\lagperp}{\perp}

\newcommand{\cbichyp}{\mathcal{Q}}

\newcommand{\project}[1]{\widehat{#1}}

\newcommand{\id}{\mathrm{id}}

\newcommand{\embedS}{\sigma_{\S^N}}
\newcommand{\embedR}{\sigma_{\R^N}}

\newcommand{\Span}{\operatorname{span}}

\newcommand{\jac}[1]{\operatorname{#1}}
\newcommand{\KK}{\mathsf{K}}

\renewcommand{\eqref}[1]{(\refeq{#1})}

\mathtoolsset{showonlyrefs,showmanualtags}

\newenvironment{acknowledgements}{
\paragraph{Acknowledgement.}
}

\title{
   Non-Euclidean Laguerre geometry and incircular nets 
}
\author{
   Alexander I. Bobenko$^1$, Carl O. R. Lutz$^1$, Helmut Pottmann$^2$, Jan Techter$^1$ \bigskip\\
   $^1$Institut f\"ur Mathematik, TU Berlin, \\
   Str.\@ des 17.\@ Juni 136, 10623 Berlin, Germany \bigskip\\
   $^2$Visual Computing Center, KAUST, \\
   Thuwal 23955-6900, Saudi Arabia
}
\date{\today}

\begin{document}

\maketitle

\noindent
\textbf{Abstract.}
Classical (Euclidean) Laguerre geometry studies oriented hyperplanes, oriented hyperspheres,
and their oriented contact in Euclidean space.
We describe how this can be generalized to arbitrary Cayley-Klein spaces, in particular hyperbolic and elliptic space,
and study the corresponding groups of Laguerre transformations.
We give an introduction to Lie geometry
and describe how these Laguerre geometries can be obtained as subgeometries.
As an application of two-dimensional Lie and Laguerre geometry
we study the properties of %checkerboard 
incircular nets.
\vspace{0.1cm}

\begin{center}
  \includegraphics[width=0.6\textwidth]{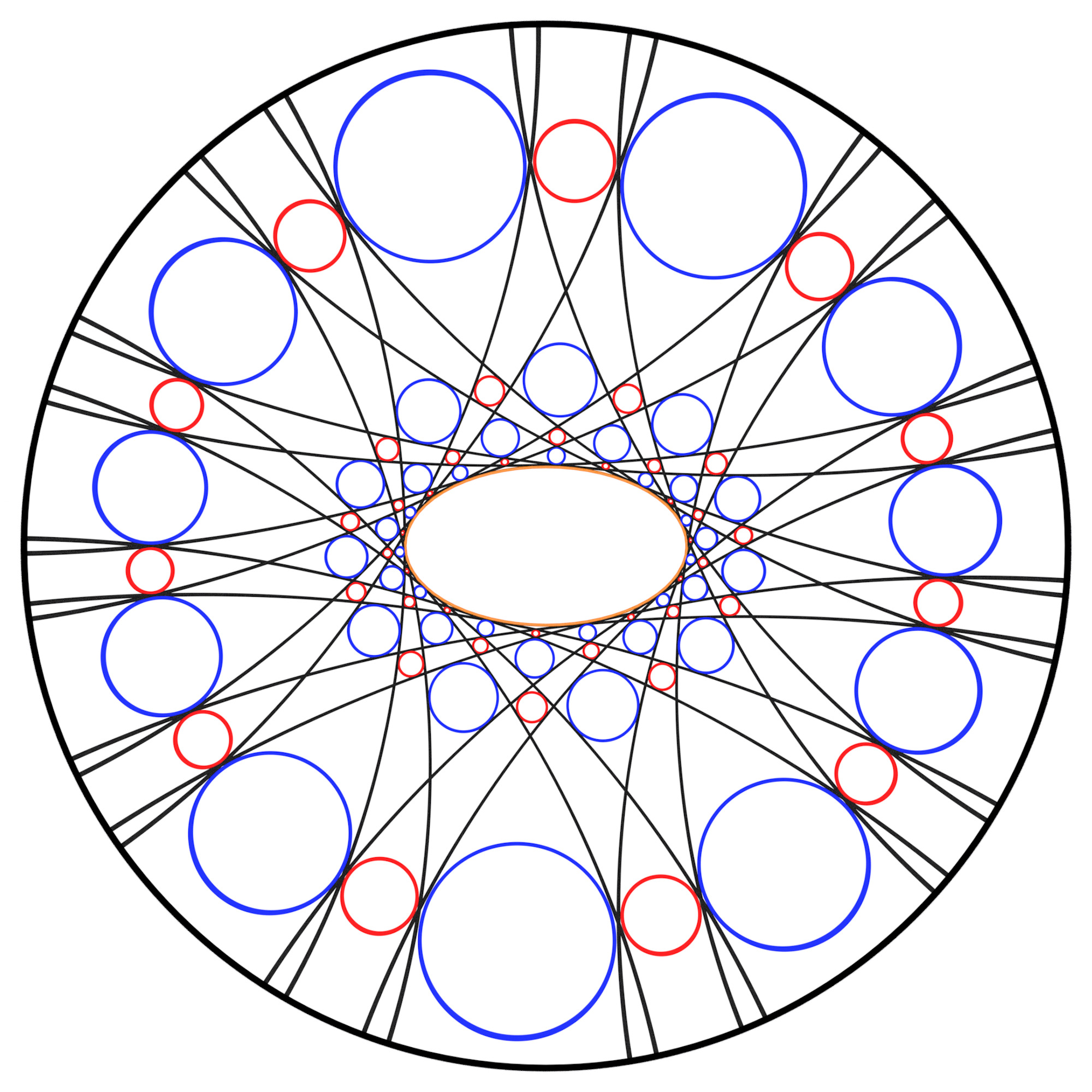}  
\end{center}

% \noindent
% \textbf{Key words:} space forms, incidence geometry, Laguerre geometry,
% Lie geometry, projective geometry, metric geometry

\vspace{\fill}

\begin{acknowledgements}
  This research was supported by the DFG Collaborative Research Center TRR 109 ``Discretization in Geometry and Dynamics''.
  H.\ Pottmann's participation in this program has been supported through grant I2978 of the Austrian Science Fund.
  We would like to thank Oliver Gross and Nina Smeenk for their assistance in creating the figures.
\end{acknowledgements}

\newpage

\tableofcontents

\newpage

\section{Introduction}
The discovery of \emph{non-Euclidean geometry} by Lobachevsky, Bolyai and Gauss was a revolution which might be compared with the discovery of the spherical form of the Earth. It turned out that there exist other geometric worlds with points, straight lines and circles, and they have natural geometric properties generalizing the ones of classical Euclidean geometry. The latter is recovered in the limit when the curvature of the space goes to zero. Almost immediately after the invention of hyperbolic geometry Lobachevsky and Gauss posed the question about the real geometry of our world and even tried to measure it experimentally. This played a crucial role in the further development of geometry and physics. Indeed, in the hyperbolic space conventional Euclidean translations and rotations are replaced by the group of hyperbolic isometric transformations. In the three-dimensional case this group coincides with the Lorentz group of our space-time, which is central in Einstein's special theory of relativity.

Felix Klein in his \emph{Erlangen program} of 1872 \cite{K2} revolutionized the point of view on geometry by declaring the transformation group as the conceptually central notion. The traditional view is that geometry studies the space around us. Due to Klein, geometry is the study of invariants under a group of transformations. This was the organizing principle which brought an order into various facts accumulated in geometry, or rather, into different geometries that had been discovered.

Various \emph{transformation groups} naturally lead to various geometries including projective, affine, spherical, hyperbolic, M\"obius, Lie, Pl\"ucker, and Laguerre geometries. Many beautiful results were obtained during the classical period of the theory. A good presentation can be found in the books by Wilhelm Blaschke \cite{Bl2a,Bl2}, which is probably the most comprehensive source of knowledge of the corresponding geometries. Unfortunately till now these books exist only in German. 

Modern revival of the interest in \emph{classical geometries} and their recent development is in much extent due to the possibility of their investigation by computational methods. Computers enable experimental and numerical investigations of geometries as well as their visualization. Classical geometries became visible! Also physics contributed with more and more involved transformation groups and problems.

Last but not least are \emph{applications} in computer graphics, geometry processing, architectural geometry and even computer simulation of dynamics and other physical processes. M\"obius geometry is probably the most popular geometry in this context. For numerous applications of classical geometries we refer in particular to \cite{BS,PW}. 
%{\bf the Arch. Geometry book is too elementary; more on applications, see below; }

This small book is on a rather ``exotic'' geometry called \emph{non-Euclidean Laguerre geometry}. Euclidean Laguerre geometry, M\"obius geometry and Lie geometry belong to its close environment and also appear in this book. Before we come to precise mathematical explanations let us give a rough idea of these geometries in the plane. The basic geometric objects in these geometries are points, straight lines and circles.  Whereas M\"obius geometry is dealing with points and circles and has no notion of a straight line, Laguerre geometry is the geometry of circles and straight lines and has no notion of a point. Incidences in M\"obius geometry, like points lie on circles, in Laguerre geometry correspond to the tangency condition between circles and straight lines (more precisely, oriented circles and lines which are in oriented contact). In the non-Euclidean case, straight lines are replaced by geodesics (see Figure~\ref{fig:intro}).
Generalizations of Laguerre geometry to non-Euclidean space have already been studied by Beck \cite{Be}, Graf \cite{G34, G37, G39} and Fladt \cite{F1, F2}, mainly in dimension 2.

\begin{figure}[h]
  \centering
  \raisebox{0.7cm}{
    \includegraphics[width=0.4\textwidth]{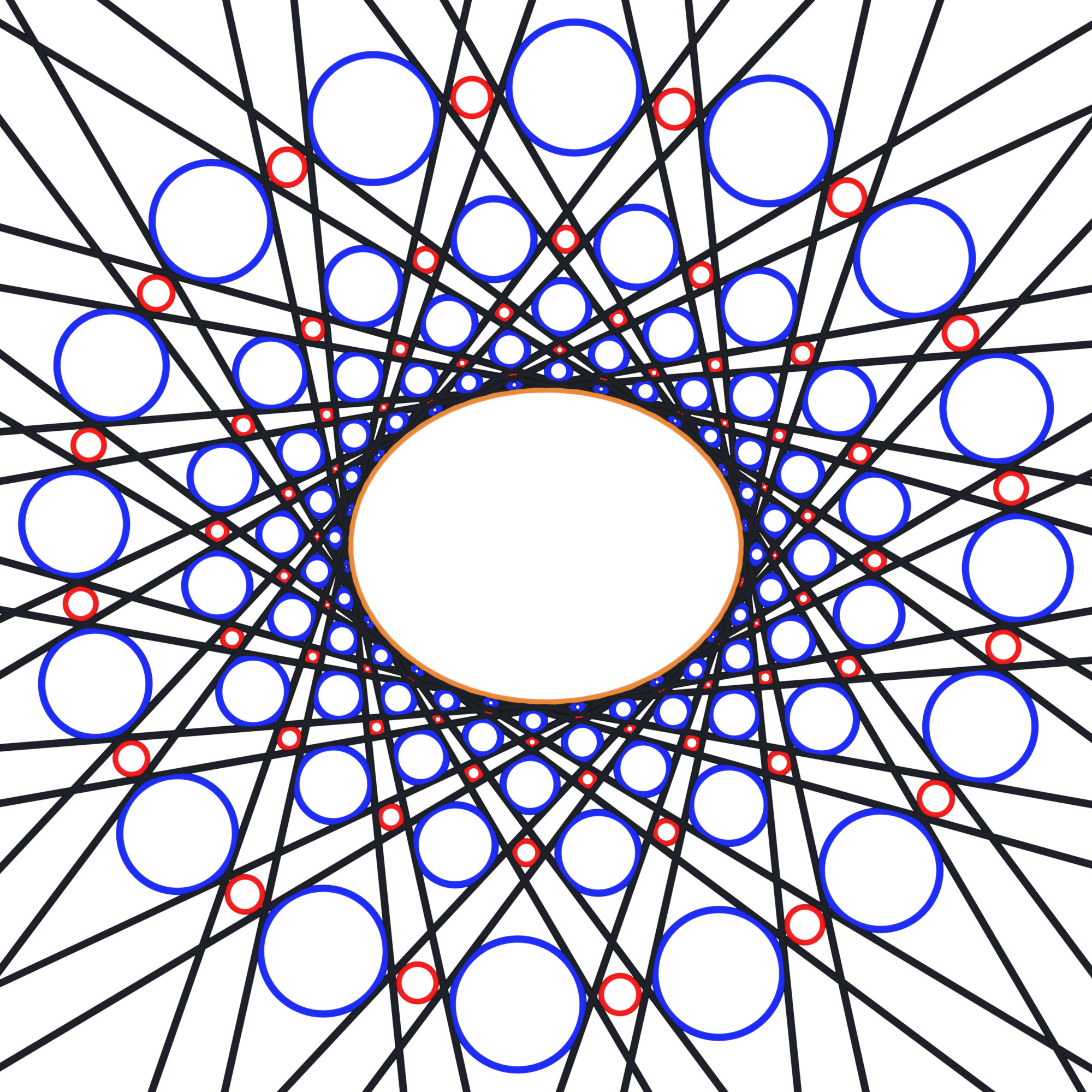}
  }
    \includegraphics[width=0.48\textwidth]{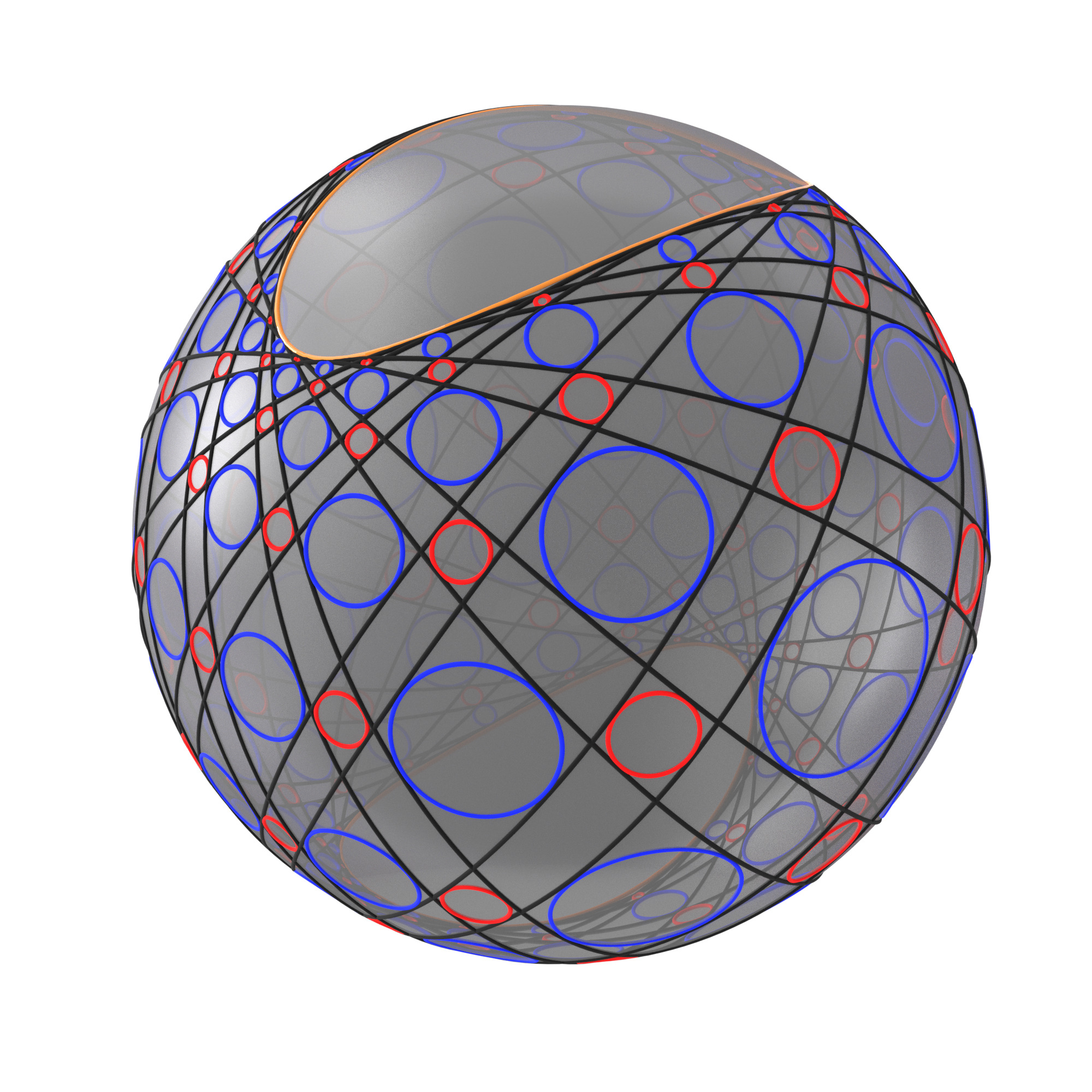}
  \caption{
    Euclidean and elliptic checkerboard incircular nets as instances of Euclidean and non-Euclidean (spherical) Laguerre geometries.
    Straight lines and circles are tangent and can be oriented so that their orientations coincide at the points of tangency (oriented contact).
    The ``straight lines'', or geodesics, on the sphere are great circles.
 }
\label{fig:intro}
\end{figure}

Classically, (Euclidean) \emph{Laguerre geometry} is the geometry of oriented hyperplanes, oriented hyperspheres, and their oriented contact in Euclidean space \cite{L}.
It is named after Laguerre \cite{HPR}, and was actively studied in dimensions 2 and 3 in the early twentieth century,
see, e.g., \cite{Bl1, Bl2}.
In \cite{Ben1} and \cite{Y} the relation between Laguerre geometry and projective planes over commutative rings, e.g.\ dual numbers, is investigated.

More recently, Laguerre geometry has been employed in specific applications, most notably in connection with offsets. 
These are curves or surfaces which lie at constant
normal distance to each other and have various applications in Computer-Aided Design and Manufacturing (see e.g. \cite{Far}). 
Viewing a curve or surface as a set of oriented tangents or tangent planes, respectively, the offsetting operation is a special 
Laguerre transformation and thus Laguerre geometry is a natural geometry for the study of offsets. Examples of its use include
the determination of all families of offsets that are rational algebraic and
therefore possess exact representations in NURBS-based 3D modeling systems \cite{Far,PPa,PPb}.  Discrete versions of offset surfaces
play an important role in \emph{discrete differential geometry} in connection with the definition of discrete curvatures \cite{BPW} and in \emph{architectural geometry} \cite{PLW}.

The knowledge of Laguerre geometry as a counterpart to the more familiar M\"obius geometry is a useful tool in
research. It allows one to study sphere geometric concepts within both of these two geometries, which may open up new applications. 
An example for that is furnished by \emph{circular meshes}, a M\"obius geometric concept, and \emph{conical meshes},
their Laguerre geometric counterparts \cite{BS07,BS,PW2}. Both of them are discrete versions of curvature line parameterizations of surfaces, but have different properties in view of applications. It turned out that conical meshes
are preferable for the realization of architectural freeform structures. The main reason is an offset property which 
facilitates the design and fabrication of supporting beam layouts \cite{PLW}. 
Even more remarkable is the fact that the supporting structures with the cleanest node geometry are based on so-called 
edge offset meshes and are also of a Laguerre geometric nature \cite{PGB}. Quadrilateral structures of this type impose a shape
restriction. They are discrete versions of Laguerre isothermic surfaces \cite{Bl2,BS06}, a special case of which are 
Laguerre minimal surfaces \cite{Bl2,PGM,PGB,SPG}. The ``dual'' viewpoints of M\"obius and Laguerre geometry also led
to different discretizations and applications of surface parameterizations which run symmetrically to the principal
directions \cite{PW20}.

The most comprehensive text on Laguerre geometry is the classical book by Blaschke \cite{Bl2}, where however only the Euclidean case is treated.
There exists no systematic presentation of non-Euclidean Laguerre geometry in the literature.
The goal of the present book is twofold.
On one hand, it is supposed to be a comprehensive presentation of non-Euclidean Laguerre geometry, and thus has the character of a textbook.
On the other hand, Section \ref{sec:icnets} presents new results.
We demonstrate the power of Laguerre geometry on the example of \emph{checkerboard incircular nets} introduced in \cite{AB},
give a unified treatment of these nets in all space forms, and describe them explicitly.
Checkerboard incircular nets are Laguerre geometric generalizations of incircular nets introduced by Böhm \cite{B}, which are defined as congruences of straight lines in the plane with the combinatorics of the square grid such that each elementary quadrilateral admits an incircle.
They are closely related to (discrete) confocal conics \cite{BSST16, BSST18}.
The construction and geometry of incircular nets and their Laguerre geometric generalization to checkerboard incircular nets have been discussed in great detail.
Explicit parametrizations for the Euclidean cases were derived in \cite{BST},
while different higher dimensional analogues of incircular nets were studied in \cite{ABST} and \cite{AB}.
In this book we further generalize planar checkerboard incircular nets to Lie geometry,
and show that these may be classified in terms of checkerboard incircular nets in hyperbolic/elliptic/Euclidean Laguerre geometry.
We prove incidence theorems of Miquel type and show that all lines of a checkerboard incircular net are tangent to a hypercycle.
This generalizes the results from \cite{BST} and leads to a unified treatment of checkerboard incircular nets in all space forms.
Visualizations and geometric data for checkerboard incircular nets can also be found at \cite{gallery}.

In Section \ref{sec:elementary} we begin our treatment of non-Euclidean Laguerre geometry by introducing elementary models for Laguerre geometry in the elliptic and hyperbolic plane.
The intention here is to enable the reader to quickly get a glimpse of this geometry without reference to the following more general discussions.

In Section \ref{sec:laguerre} we show how Laguerre geometry can be obtained in a unified way for an arbitrary \emph{Cayley-Klein space} of any dimension.
In the spirit of Klein's Erlangen program this is done in a purely projective setup for which we introduce the foundations on quadrics (Sections \ref{sec:preliminaries}), Cayley-Klein spaces (Section \ref{sec:Cayley-Klein-metric}), and central projections (Section \ref{sec:projection}) \cite{K, Blproj, Gie}.
For a Cayley-Klein space $\mathcal{K} \subset \RP^n$ the space of hyperplanes is lifted to a quadric $\lag \subset \RP^{n+1}$, which we call the Laguerre quadric.
Vice versa, the projection from the Laguerre quadric yields a double cover of the space of $\mathcal{K}$-hyperplanes
which can be interpreted as carrying their orientation.
In the projection hyperplanar sections of $\lag$ correspond to spheres of the Cayley-Klein space $\mathcal{K}$.
The corresponding group of quadric preserving transformations, which maps hyperplanar sections of $\lag$ to hyperplanar sections of $\lag$,
naturally induces the group of transformations of oriented $\mathcal{K}$-hyperplanes, which preserves the oriented contact to Cayley-Klein spheres.
We explicitly carry out this general construction in the cases of hyperbolic and elliptic geometry,
yielding hyperbolic Laguerre geometry and elliptic Laguerre geometry, respectively.
The (classical) Euclidean case constitutes a degenerate case of this construction, which we treat in Appendix \ref{sec:euclidean}.
In Appendix \ref{sec:invariant} we treat an invariant of two points on a quadric, which is closely related to the Cayley-Klein distance, and of which the classical \emph{inversive distance} introduced Coxeter \cite{Cox} turns out to be a special case.

In Section \ref{sec:lie} we show how the different Laguerre geometries appear as subgeometries of Lie geometry. 
\emph{Lie (sphere) geometry} is the geometry of oriented hyperspheres of the $n$-sphere $\S^n \subset \R^{n+1}$, and their oriented contact \cite{Bl2, C, BS}.
Laguerre geometry can be obtained by distinguishing the set of ``oriented hyperplanes''
as a sphere complex among the set of oriented hyperspheres, the so called plane complex.
Classically, the plane complex is taken to be parabolic, which leads to the notion of Euclidean Laguerre geometry,
where elements of the plane complex are interpreted as oriented hyperplanes of Euclidean space.
Choosing an elliptic or hyperbolic sphere complex, on the other hand, allows for the interpretation of the elements
of the plane complex as oriented hyperplanes in hyperbolic or elliptic space, respectively.
The group of Lie transformations that preserve the set of ``oriented hyperplanes'' covers the group of Laguerre transformations.

\newpage
\section{Two-dimensional non-Euclidean Laguerre geometry}
\label{sec:elementary}
In this section we present Laguerre geometry in the elliptic and hyperbolic plane in an elementary way, without reference to the following more general discussions.
The intention here is to enable the reader to quickly get a glimpse of this geometry without diving into the details.

In Laguerre geometry in these planes, we consider \emph{oriented lines} and \emph{oriented circles} as the basic objects and \emph{orientated contact} (tangency) as the basic relation between them. 
A point in the elliptic or hyperbolic plane is considered as an oriented circle, being in contact with all oriented lines passing through it.

A \emph{Laguerre transformation} is bijective in the sets of oriented lines and oriented circles, respectively, and preserves oriented contact.
It is important to note that in general Laguerre transformations do not preserve points.
Points are special oriented circles and are therefore mapped to oriented circles.

We will now present elementary \emph{quadric models} of Laguerre geometry in the elliptic and hyperbolic plane,
in which oriented lines are represented by points of a quadric in projective 3-space and oriented circles appear as the planar sections of that quadric.
In this quadric model, Laguerre transformations are seen as projective transformations which map the quadric onto itself.

We first discuss the simpler case of \emph{elliptic Laguerre geometry} and then proceed towards \emph{hyperbolic Laguerre geometry}.

\subsection{Two-dimensional elliptic Laguerre geometry}
We use the sphere model of the elliptic plane $\mathcal{E}$.
Points of $\mathcal{E}$ are seen as pairs of antipodal points of the unit sphere $\S^2 \subset \R^3$.
Oriented lines of $\mathcal{E}$ appear as oriented great circles in $\S^2$
and oriented circles in $\mathcal{E}$ correspond to oriented circles (different from great circles) in $\S^2$.

\begin{figure}[h]
  \centering
  \begin{overpic}[width=0.4\textwidth]{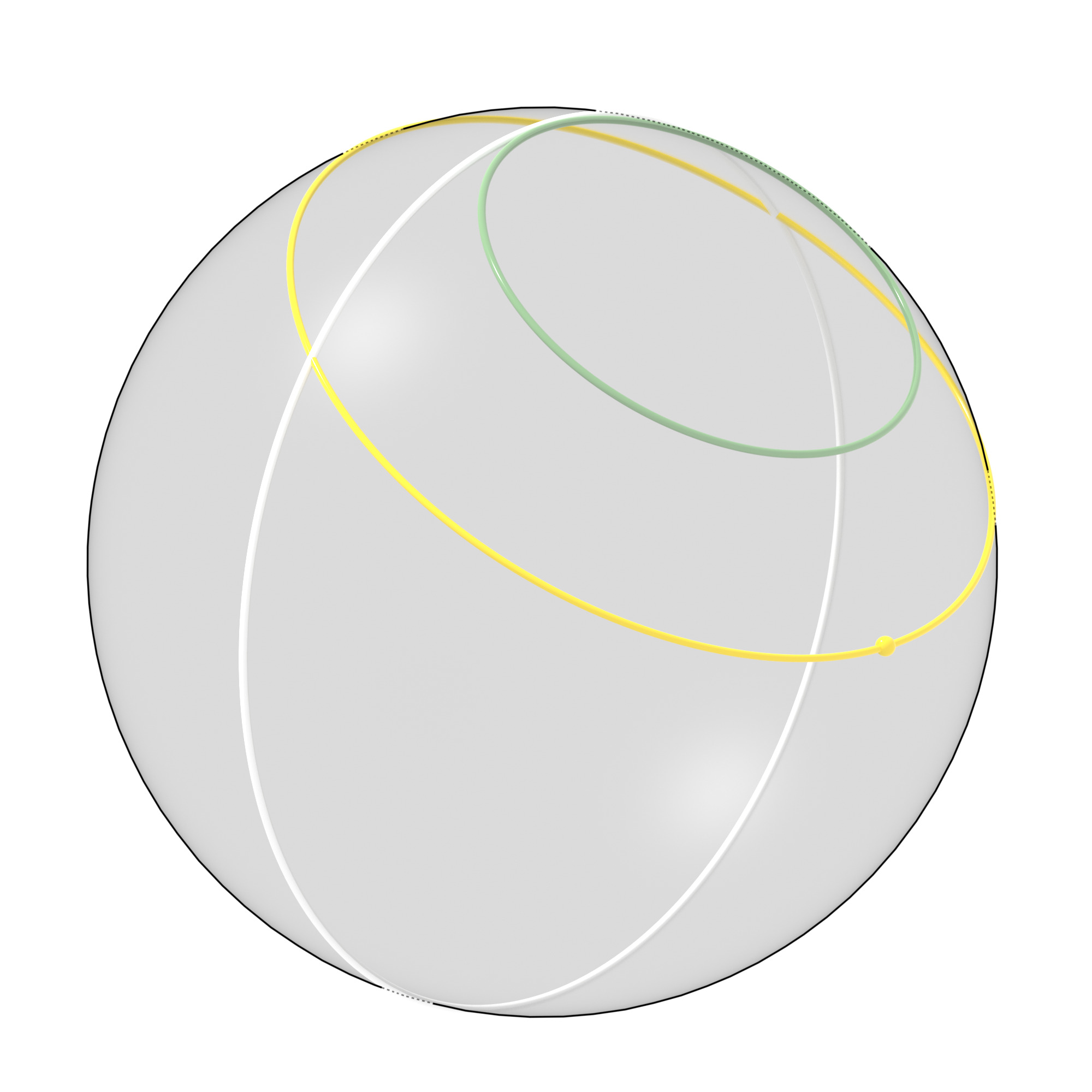}
    \put(80,15){$\S^2$}
    \put(24,30){$\ell$}
    \put(40,40){$G(c)$}
    \put(55,58){$c$}
    \put(81,36){$\p{x}$}
  \end{overpic}
  \caption{
    The unit sphere $\S^2$ as a model of elliptic Laguerre geometry.
    An oriented line $\ell$ is represented by the spherical center $\p{x}$ of a great circle.
    An oriented circle $c$ is represented by a planar section $G(c)$,
    which is composed of the spherical centers of the great circles in oriented contact with the circle.
  }
  \label{fig:elliptic-laguerre-geometry}
\end{figure}
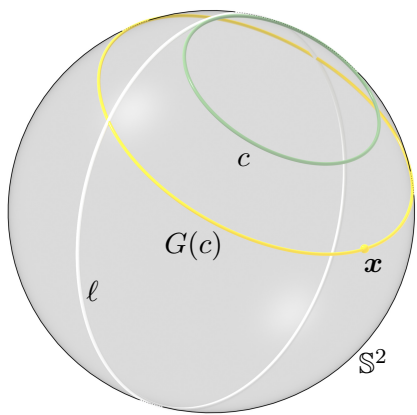

\paragraph{Oriented lines}
An oriented great circle $\ell \subset \S^2$ defines two half-spheres,
one of which lies on the positive side, which shall be the left side when running on $\ell$ in terms of the orientation (see Figure~\ref{fig:elliptic-laguerre-geometry}). 
We now represent an oriented great circle $\ell$ by its spherical center $\p{x}$.
This is the intersection point of the circle's rotational axis with $\S^2$ which lies on the positive side of $\ell$.

\paragraph{Oriented circles}
Let us now consider all oriented lines $\ell$ of $\mathcal{E}$ which are tangent to an oriented circle.
In the sphere model, this yields all oriented great circles of $\S^2$ which are tangent to an oriented circle $c \subset \S^2$.
By rotational symmetry, their centers $\p{x}$ form a circle $G(c)$.
The set of centers $G(c)$ does not degenerate to a point as $c$ is not a great circle, while it is a great circle if $c$ is a point.

\paragraph{Laguerre transformations}
It now becomes clear how to realize Laguerre transformations of the elliptic plane in the sphere model:
These are projective transformations that map the sphere $\S^2$ to itself.
Planar sections $G(c) \subset \S^2$ (oriented circles) are mapped to planar sections, but in general great circles are not mapped to great circles.
This expresses the fact that Laguerre transformations map oriented circles to oriented circles, but do in general not preserve points.

%Now the points of $S^2$ represent oriented great circles.
In Möbius geometry these projective automorphisms of $\S^2$ also appear as Möbius transformations, but the meaning of points is different.
Summarizing, we can say that the elliptic Laguerre group is isomorphic to the (Euclidean) M\"obius group.
Both appear as projective automorphisms of the sphere, but in elliptic Laguerre geometry,
points of the sphere have the meaning of centers of great circles, representing oriented straight lines of $\mathcal{E}$.

\subsection{Two-dimensional hyperbolic Laguerre geometry}
We employ the projective model of the hyperbolic plane $\mathcal{H}$ with an absolute circle $\mathcal{S}$.
Recall that points of $\mathcal{H}$ are points inside the circle $\mathcal{S}$,
with the points of $\mathcal{S}$ playing the role of points at infinity (ideal points). 
Straight lines $\ell \subset \mathcal{H}$ are seen as straight line segments bounded by two ideal points $\p{m}_-$ and $\p{m}_+$.
The line obtains an orientation by traversing it from $\p{m}_-$ to $\p{m}_+$ (see Figure~\ref{fig:hyperbolic-laguerre-geometry-lines}).

To obtain a quadric model for Laguerre geometry in $\mathcal{H}$, we view the absolute circle $\mathcal{S}$
\[
  x^2+y^2=1
\]
as the smallest circle on the rotational hyperboloid $\hyplag$
\[
  x^2+y^2-z^2 = 1,
\]
which lies in its symmetry plane $z=0$.
This hyperboloid carries two families of straight lines (rulings).
Two rulings are obtained by intersecting $\hyplag$ with the tangent plane $x=1$, yielding $z=\pm y$.
By rotation about the $z$-axis, the line $x=1, z=y$ generates the family of rulings $R_+$,
and likewise $x=1, z=-y$ generates the rulings $R_-$. 
Through each point of $\hyplag$ there passes exactly one line of $R_+$ and one line of $R_-$.

\begin{figure}[h]
  \centering
  \begin{overpic}[width=0.6\textwidth]{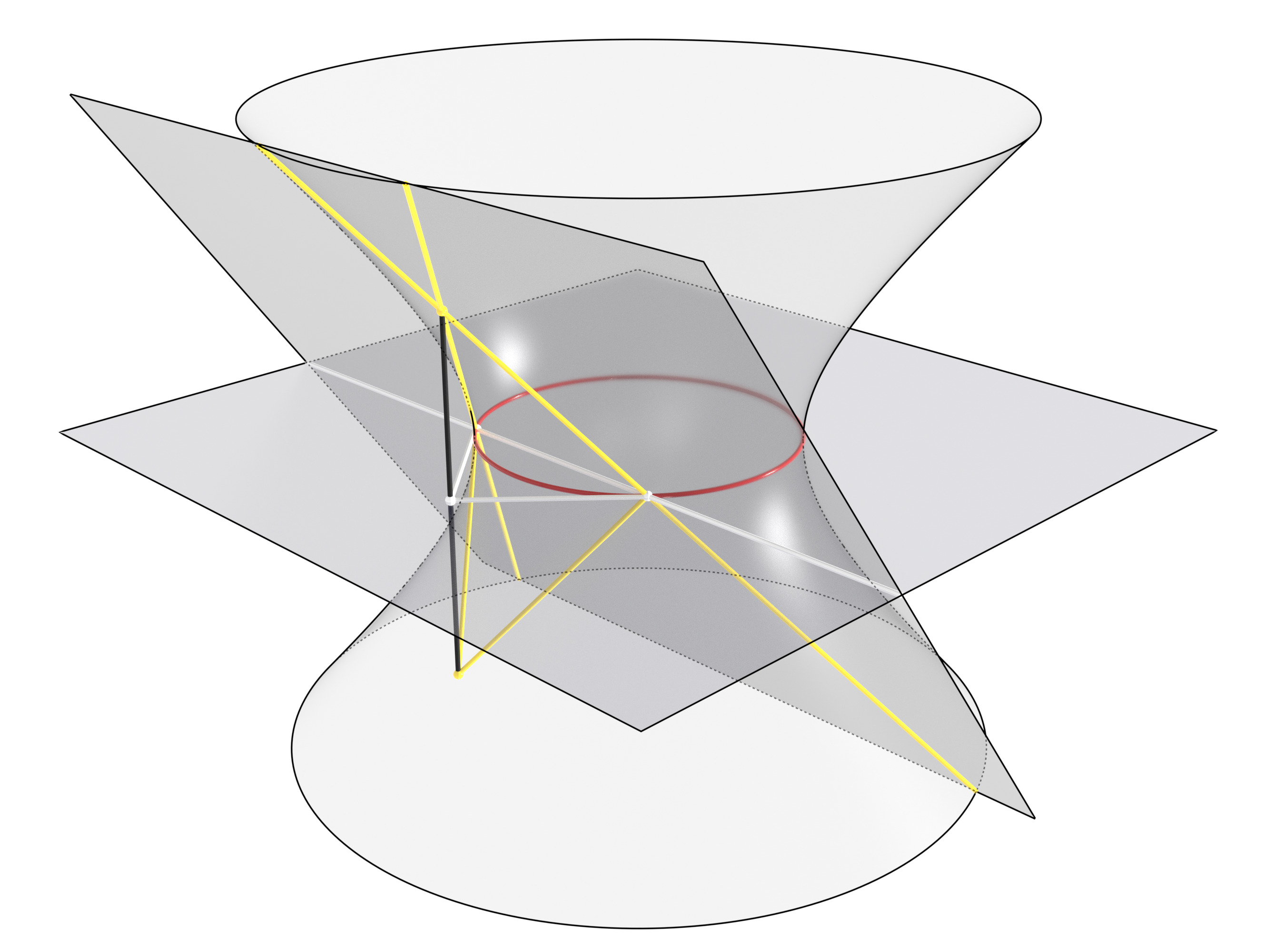}
    \put(82,63){$\mathcal{\hyplag}$}
    \put(30,48){$\p{x}$}
    \put(30,20){$\p{x}'$}
  \end{overpic}
  \hspace{\fill}
  \raisebox{0.7cm}{
    \begin{overpic}[width=0.35\textwidth]{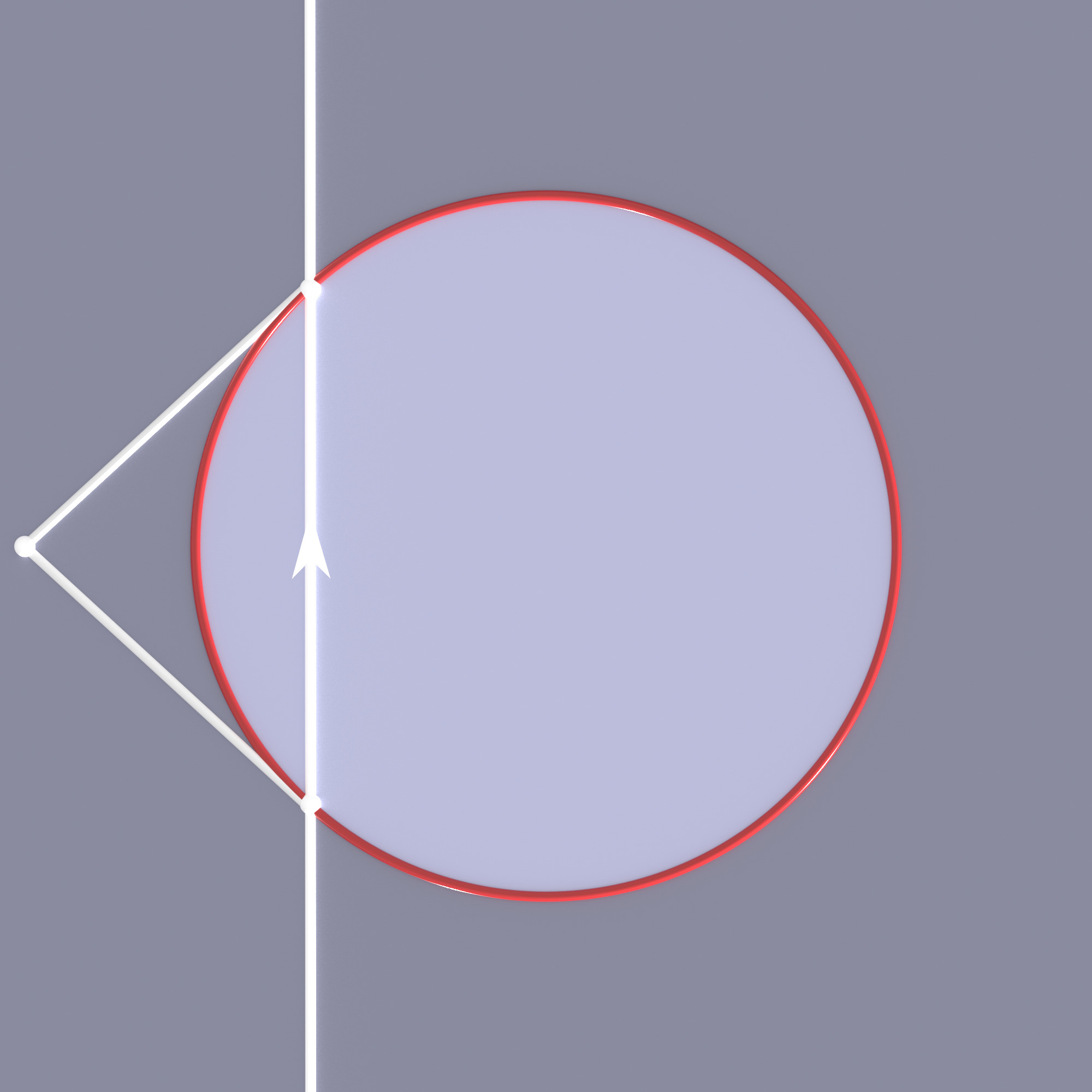}
      \put(32,47){$\ell$}
      \put(55,47){$\mathcal{H}$}
      \put(80,30){$\mathcal{S}$}
      \put(32,71){$\p{m}_+$}
      \put(32,26){$\p{m}_-$}
    \end{overpic}
  }
  \caption{
    A rotational hyperboloid $\hyplag$ as a model of hyperbolic Laguerre geometry.
    It contains the absolute circle $\mathcal{S}$ of the projective model of the hyperbolic plane $\mathcal{H}$ at $z=0$.
    An oriented line $\ell$ is represented by the intersection $\p{x}$ of two rulings of $\hyplag$ through the two ideal points $\p{m}_-$ and $\p{m}_+$.
    The tangent plane of $\hyplag$ at the point $\p{x}$ intersects the hyperbolic plane in the line $\ell$.
    A different choice of rulings yields the point $\p{x}'$ which represents the same line with opposite orientation.
  }
  \label{fig:hyperbolic-laguerre-geometry-lines}
\end{figure}
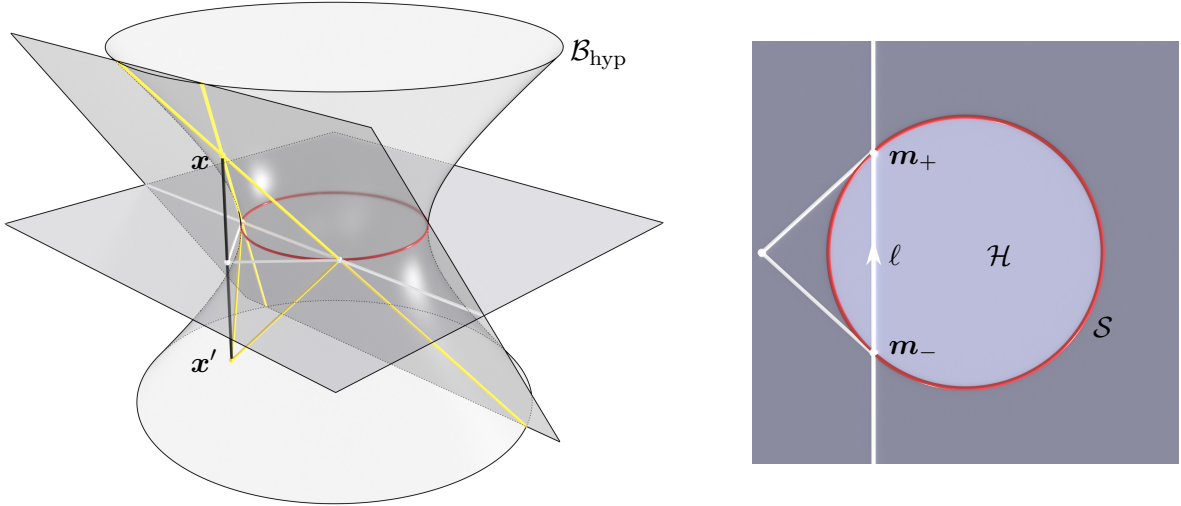

\paragraph{Oriented lines}
We now lift an oriented straight line $\ell$ of $\mathcal{H}$ to a point on the quadric $\hyplag$ as follows (see Figure~\ref{fig:hyperbolic-laguerre-geometry-lines}):
We intersect the ruling of $R_+$ through $\p{m}_-$ with the ruling of $R_-$ through $\p{m}_+$, yielding a point $\p{x} \in \hyplag$.
The plane spanned by $R_-$ and $R_+$ is the tangent plane of $\hyplag$ at $\p{x}$ and this tangent plane intersects the base plane $z=0$ in the line $\ell$.
Changing the orientation of $\ell$ to $\ell'$ yields a point $\p{x}'$
which is the reflection of $\p{x}$ in the base plane $z=0$.
The connecting line $\p{x} \wedge \p{x}'$ is the polar line of $\ell$ with respect to the quadric $\hyplag$,
and the orthogonal projection of $\p{x}$, respectively $\p{x}'$, onto the base plane $z=0$ is the pole of $\ell$ with respect to the absolute circle $\mathcal{S}$. 

Parallel oriented straight lines share an ideal point and thus correspond to points which lie on the same ruling of $\hyplag$.

\paragraph{Oriented circles}
For the oriented circles of $\mathcal{H}$ the situation is slightly more complicated since different types of circles come into play (see
Figures~\ref{fig:hyperbolic-laguerre-geometry-circles} and~\ref{fig:hyperbolic-laguerre-geometry-points}). 
The first three types (see Figure~\ref{fig:hyperbolic-laguerre-geometry-circles} (a,b,c)) arise from generalized hyperbolic circles.
Those appear as special conics $c$ in the projective model.
For visualization, it is probably easiest to employ the sphere model of hyperbolic geometry and view the conic $c$ as the orthogonal projection of a circle on the sphere $x^2+y^2+z^2=1$ onto the plane $z=0$.
Such a conic $c$ together with the absolute circle $\mathcal{S}$ spans a pencil which contains a doubly counted line $L$.
The conic $c$ and the absolute circle $\mathcal{S}$ touch in two points on the line $L$.
These may be two different real points as in Figure~\ref{fig:hyperbolic-laguerre-geometry-circles} (c),
in which case $c$ is a curve of constant distance to the hyperbolic line given by $L$.
But they can also be two complex conjugate points as in Figure~\ref{fig:hyperbolic-laguerre-geometry-circles} (a),
in which case $c$ is a hyperbolic circle with its center being the pole of the line $L$ with respect to the absolute circle $\mathcal{S}$.
Lastly, the two points may coincide as in Figure~\ref{fig:hyperbolic-laguerre-geometry-circles} (b),
which corresponds to the case of a horocycle, represented by the conic $c$ which has third order contact with $\mathcal{S}$.
Hyperbolic Laguerre geometry, however, also considers circles as in Figure~\ref{fig:hyperbolic-laguerre-geometry-circles} (d),
whose tangent lines correspond to real straight lines in $\mathcal{H}$, but the envelope lies in the deSitter space (outsite $\mathcal{S}$).
\begin{figure}[H]
  \centering
  \begin{overpic}[width=0.6\textwidth]{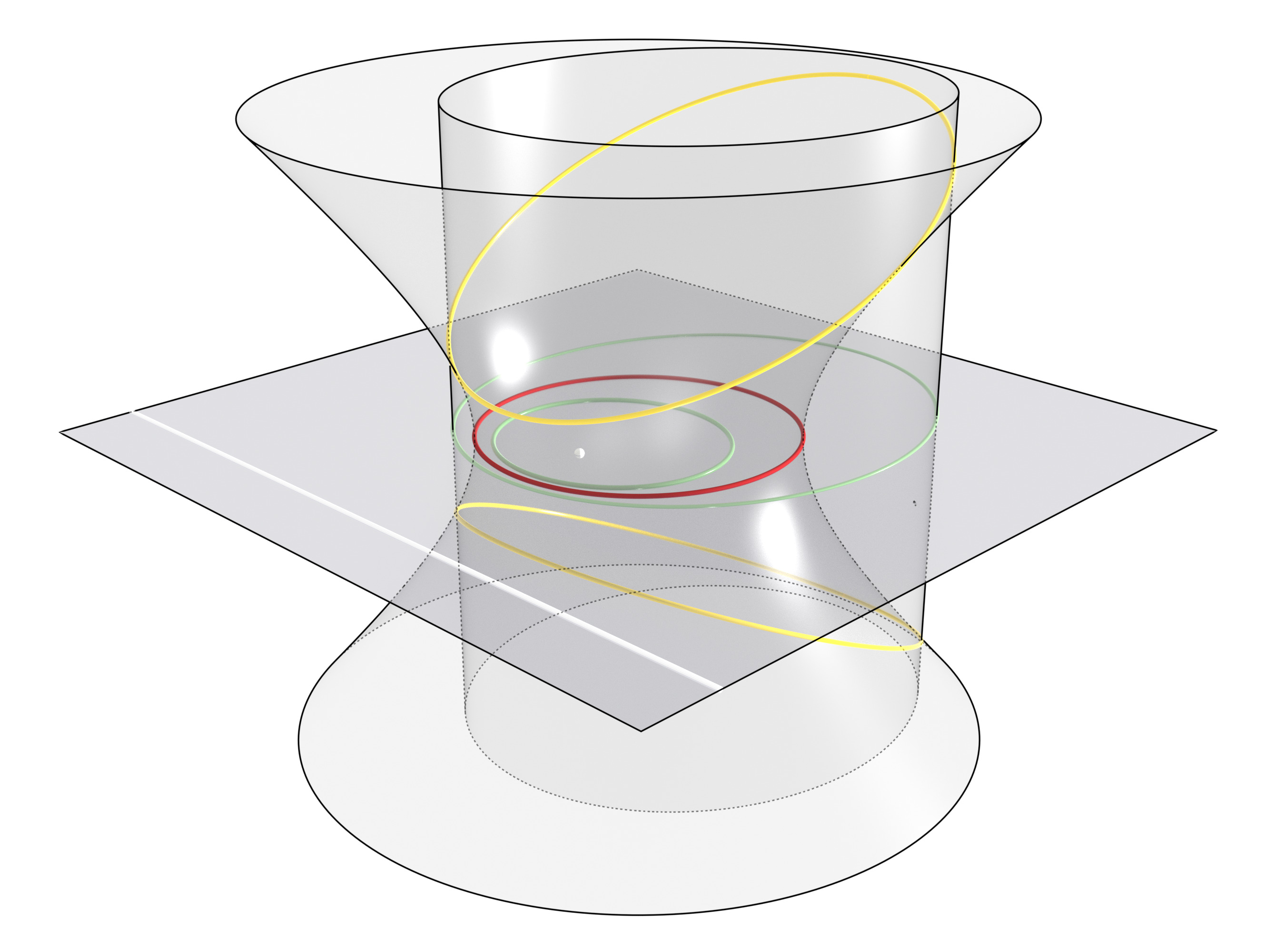}
    \put(82,63){$\mathcal{\hyplag}$}
    \put(33,14){$\Gamma$}
    \put(63,20){$G(c)$}
  \end{overpic}
  \hspace{\fill}
  \raisebox{0.7cm}{
    \begin{overpic}[width=0.35\textwidth]{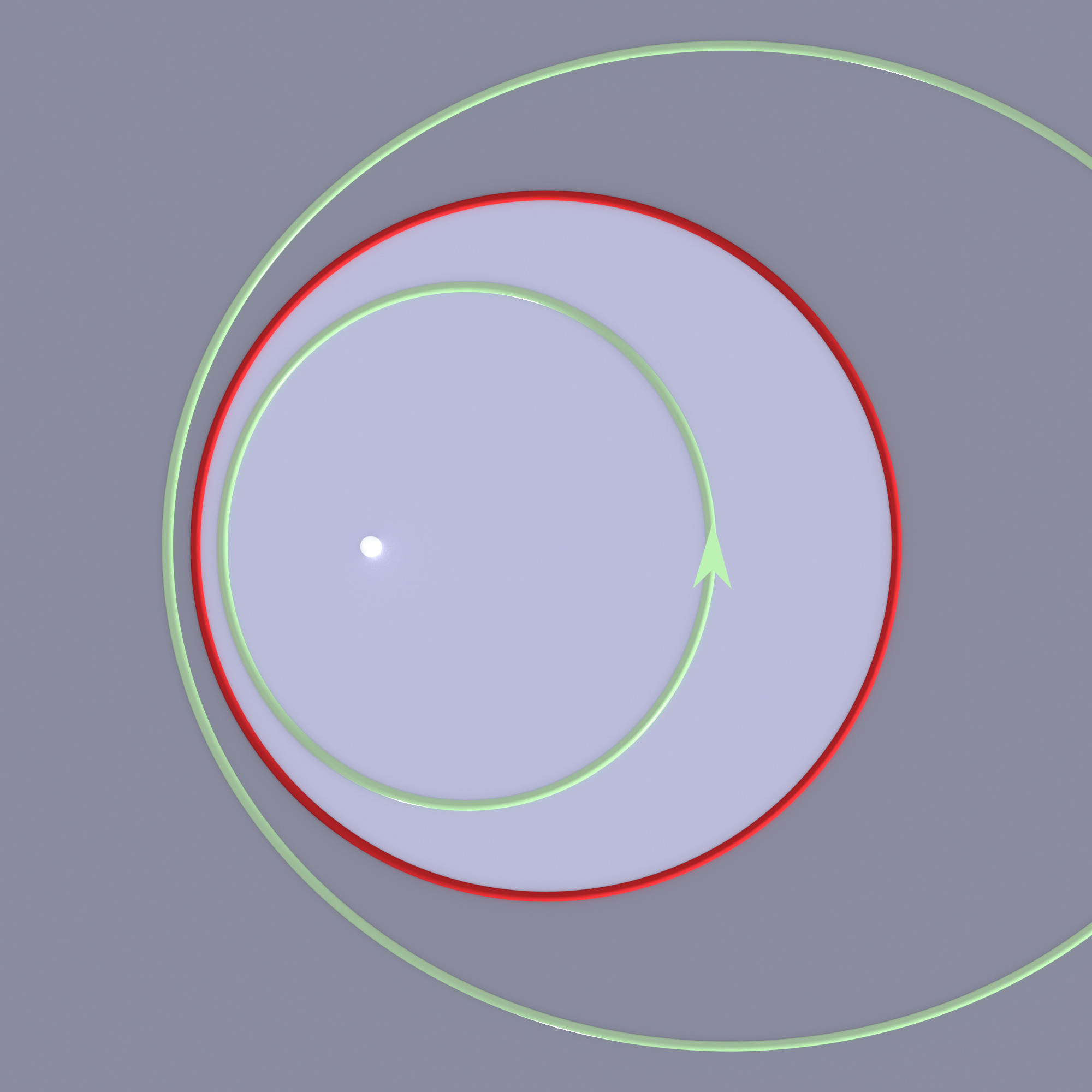}
      \put(80,30){$\mathcal{S}$}
      \put(50,47){$\mathcal{H}$}
      \put(33,30){$c$}
      \put(10,30){$c^\perp$}
    \end{overpic}
  }
  \caption{
    Lifting an oriented hyperbolic circle $c$ to the Laguerre quadric $\hyplag$.
    The polar conic $c^\perp$ of $c$ with respect to $\mathcal{S}$ is a deSitter circle.
    The quadratic cylinder $\Gamma$ over $c^\perp$ intersects the hyperboloid $\hyplag$ in two conics.
    The planar section $G(c)$ represents the circle $c$ in its given orientation.
  }
  \label{fig:hyperbolic-laguerre-geometry-circle-lift}
\end{figure}
\begin{figure}[H]
  \begin{center}
    \begin{overpic}[width=0.45\textwidth]{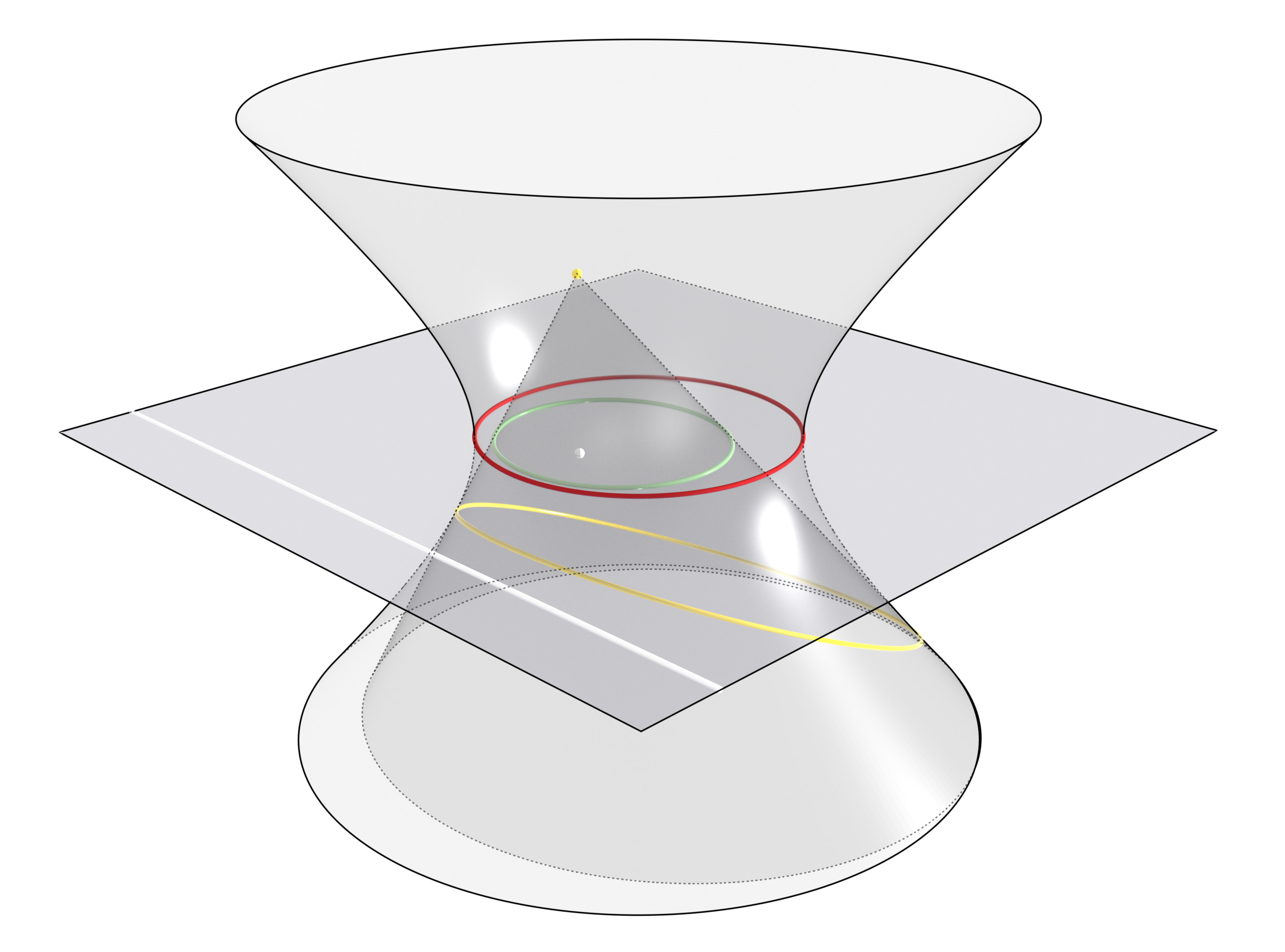}
      \put(-8,36){(a)}
      \put(82,63){$\mathcal{\hyplag}$}
      \put(63,18){$G(c)$}
    \end{overpic}
    \raisebox{0.35cm}{
      \begin{overpic}[width=0.28\textwidth]{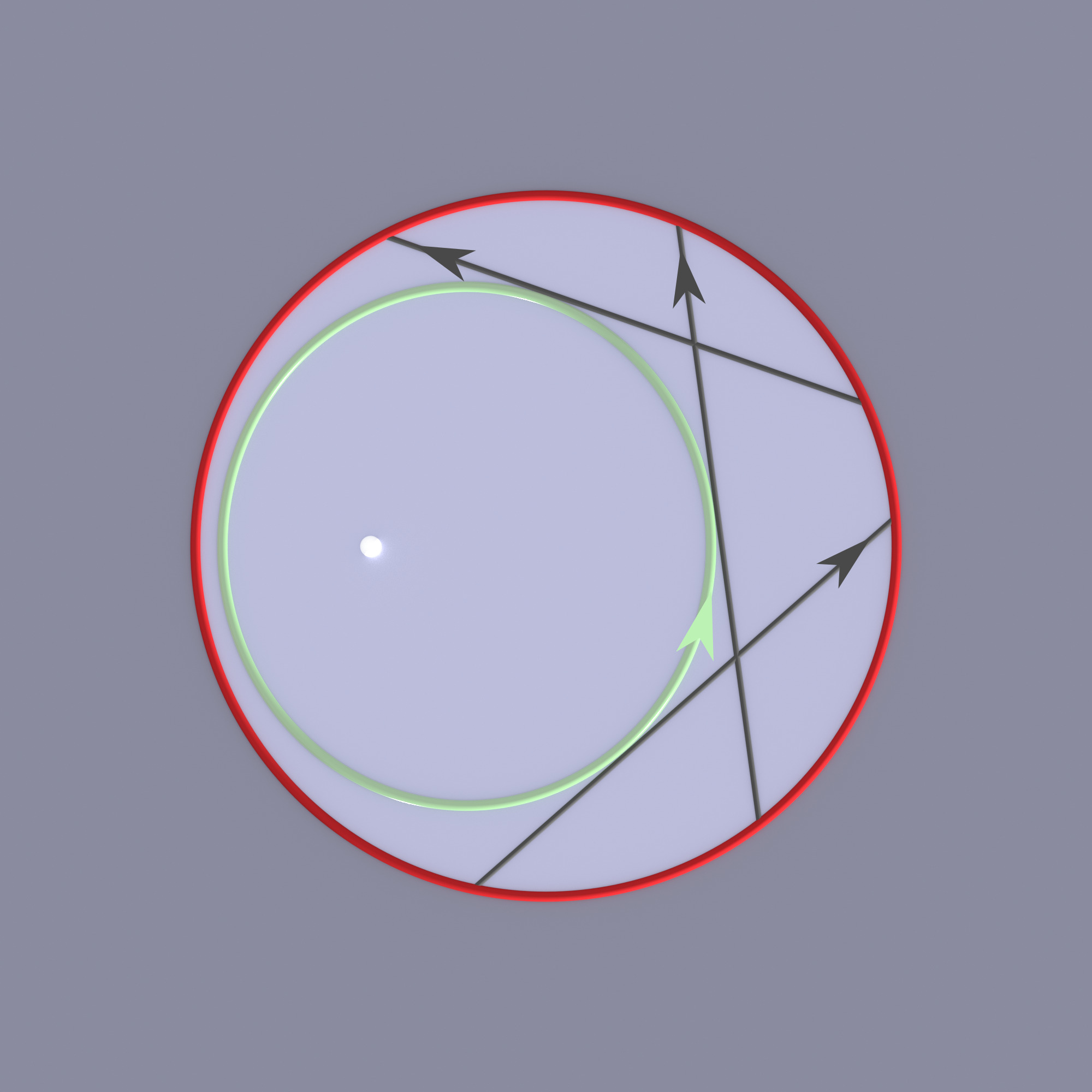}
        \put(50,47){$\mathcal{H}$}
        \put(80,30){$\mathcal{S}$}
        \put(33,30){$c$}
      \end{overpic}
    }\\
    \begin{overpic}[width=0.45\textwidth]{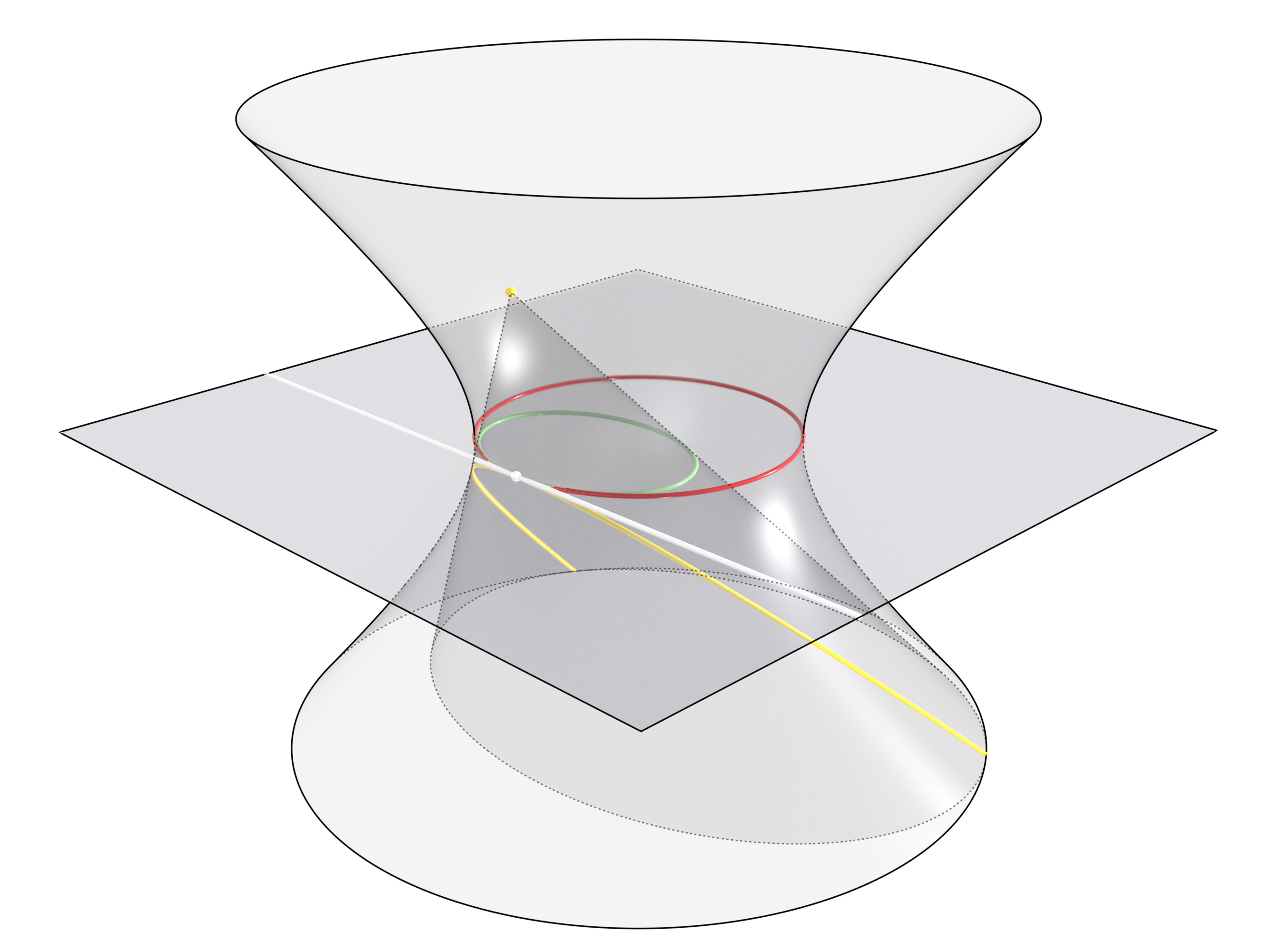}
      \put(-8,36){(b)}
    \end{overpic}
    \raisebox{0.35cm}{
      \begin{overpic}[width=0.28\textwidth]{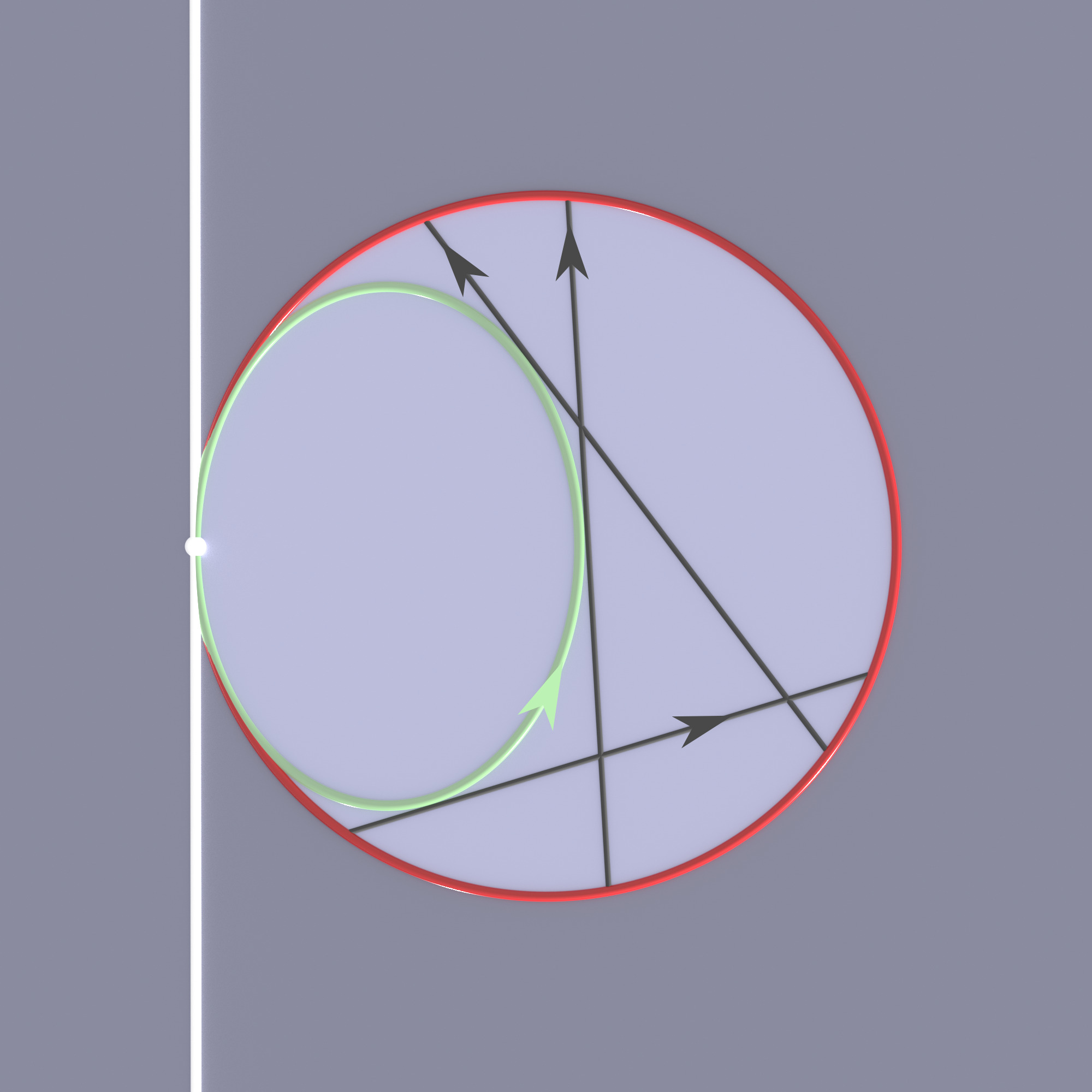}
        \put(20,5){$L$}
      \end{overpic}
    }\\
    \begin{overpic}[width=0.45\textwidth]{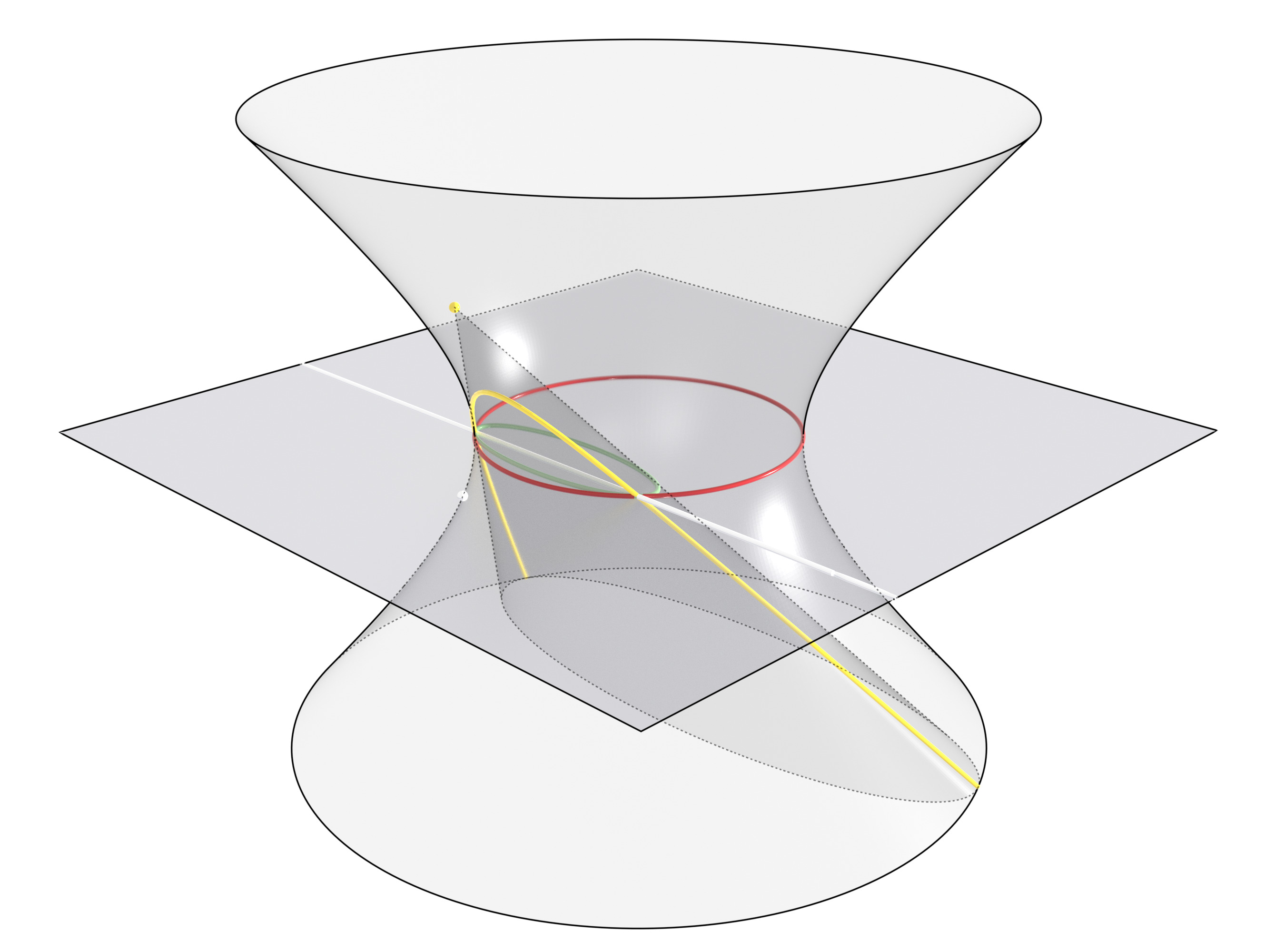}
      \put(-8,36){(c)}
    \end{overpic}
    \raisebox{0.35cm}{
      \includegraphics[width=0.28\textwidth]{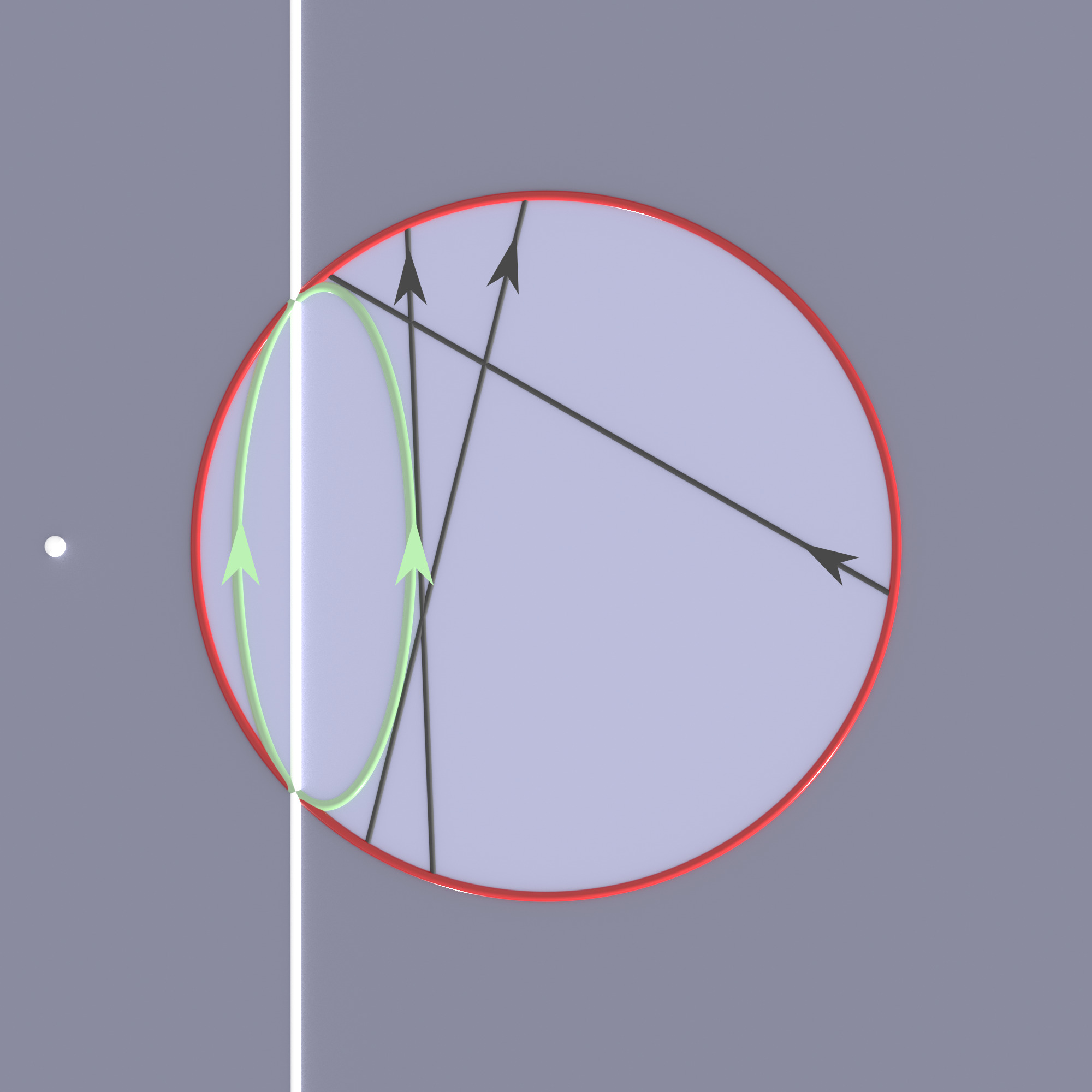}
    }\\
    \begin{overpic}[width=0.45\textwidth]{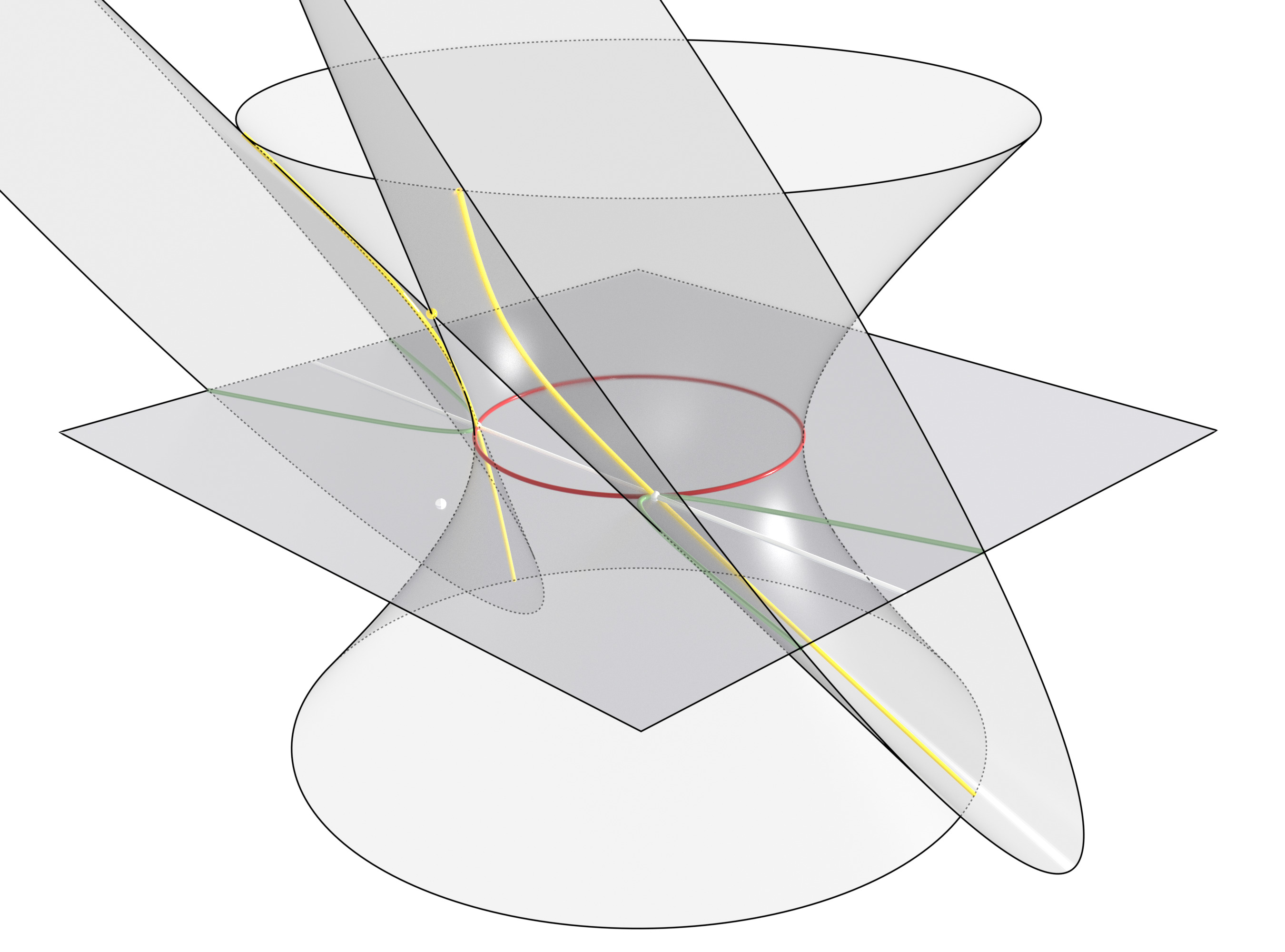}
      \put(-8,36){(d)}
    \end{overpic}
    \raisebox{0.35cm}{
      \includegraphics[width=0.28\textwidth]{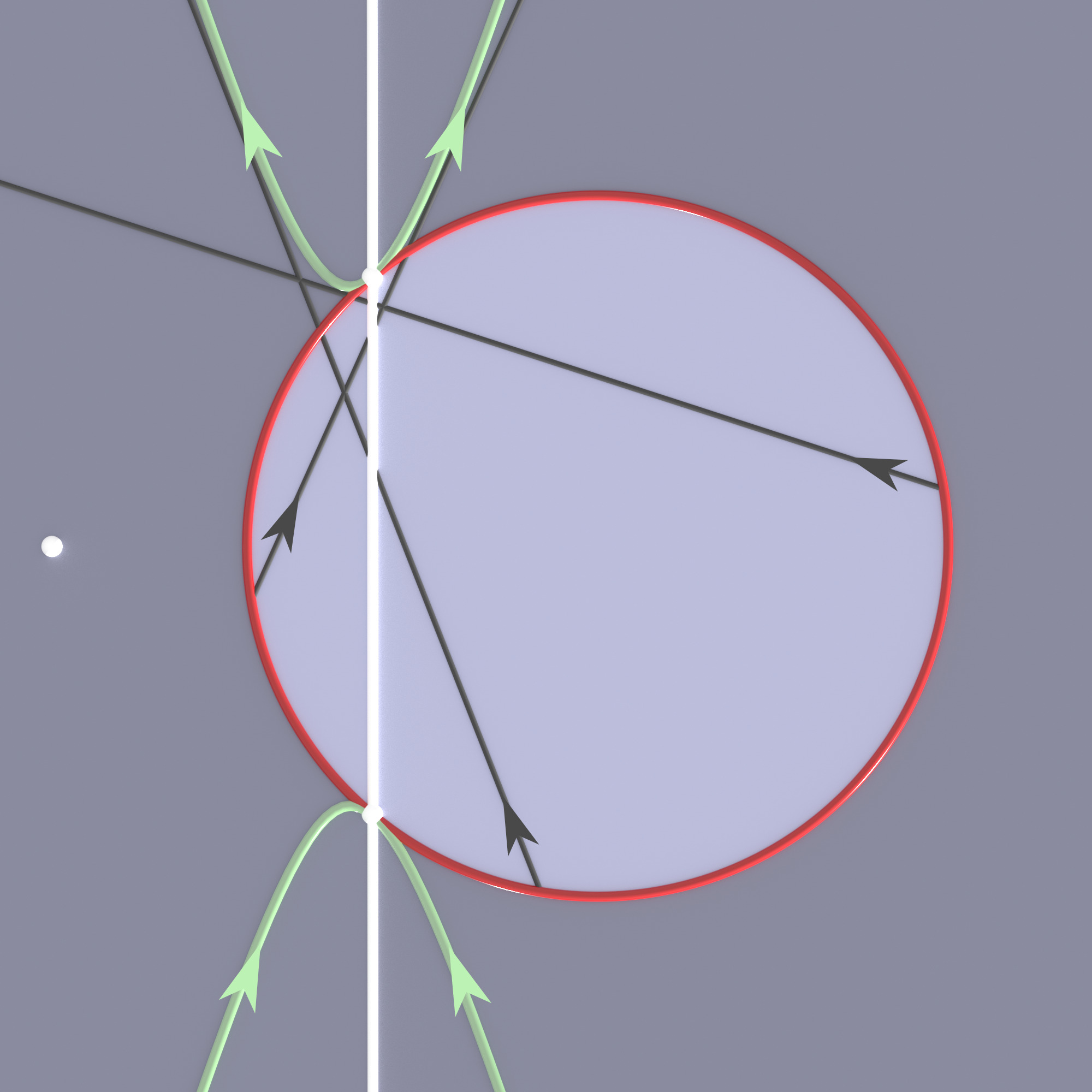}
    }
  \end{center}
  \caption{
    Oriented circles in hyperbolic Laguerre geometry.
    On the Laguerre quadric $\hyplag$ oriented circles are represented by planar sections,
    or, by polarity, the pole of the corresponding plane with respect to $\hyplag$.
    The first three types arise from generalized hyperbolic circles:
    (a) ordinary hyperbolic circle, (b) horocycle, (c) curve of constant distance to a hyperbolic line.
    The fourth type is a deSitter circle with hyperbolic tangent lines (d).
    Their envelope lies outside $\mathcal{H}$ (in the deSitter plane).
  }
  \label{fig:hyperbolic-laguerre-geometry-circles}
\end{figure}

\begin{figure}[H]
  \centering
  \begin{overpic}[width=0.45\textwidth]{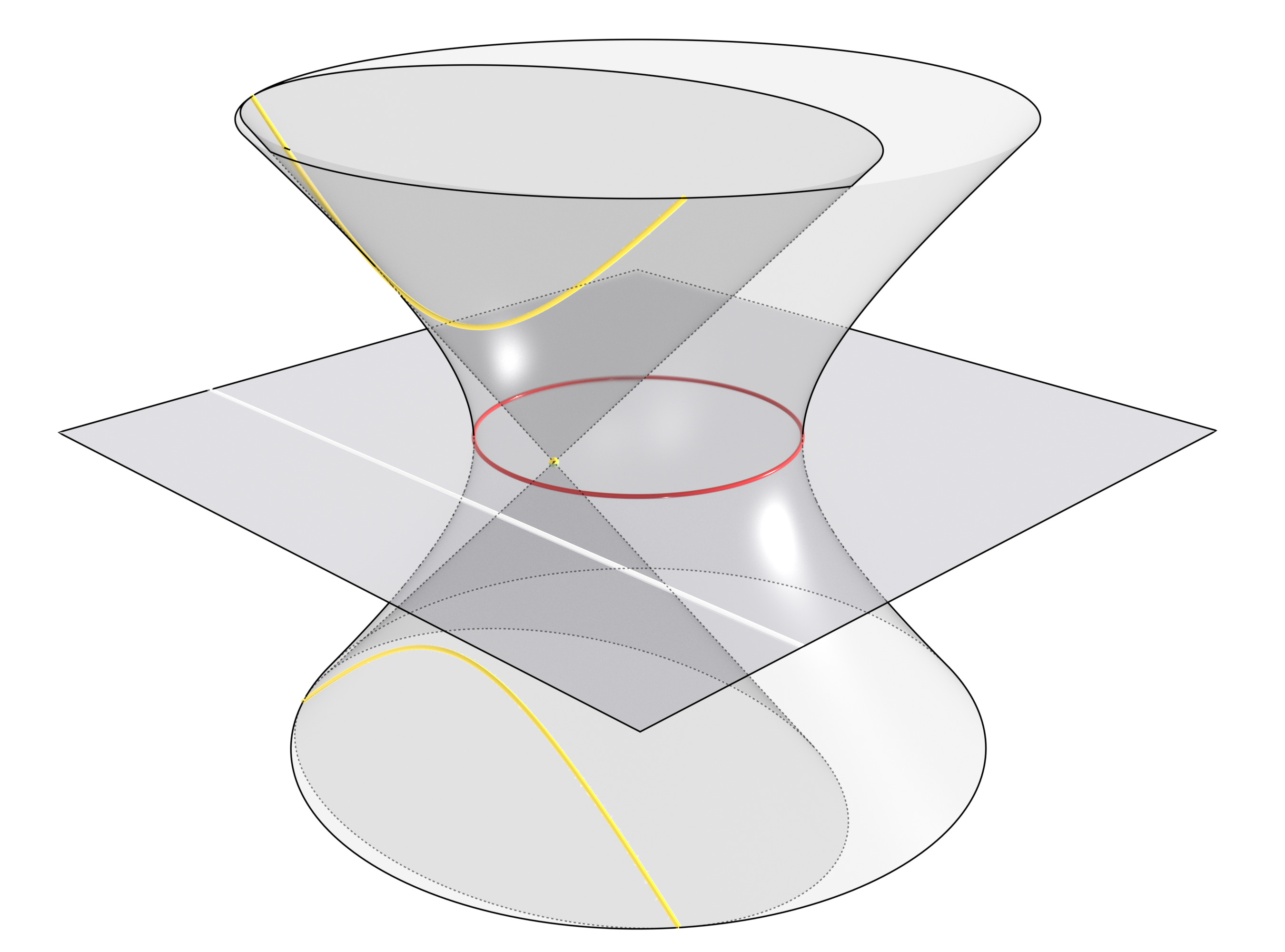}
    \put(-8,36){(a)}
    \put(82,63){$\mathcal{\hyplag}$}
    \put(33,10){$G(c)$}
  \end{overpic}
  \raisebox{0.35cm}{
    \begin{overpic}[width=0.28\textwidth]{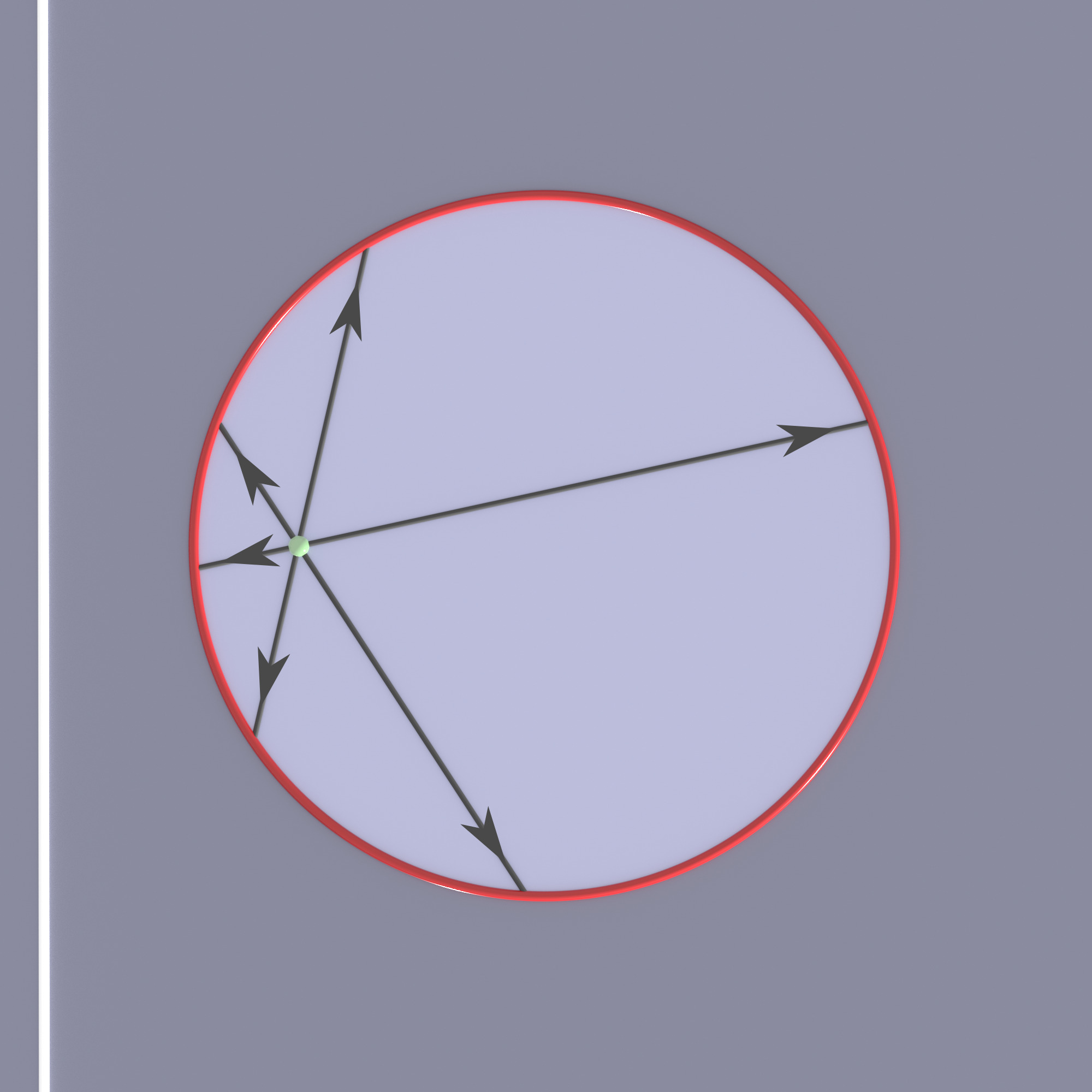}
      \put(80,30){$\mathcal{S}$}
      \put(55,42){$\mathcal{H}$}
    \end{overpic}
  }\\
  \begin{overpic}[width=0.45\textwidth]{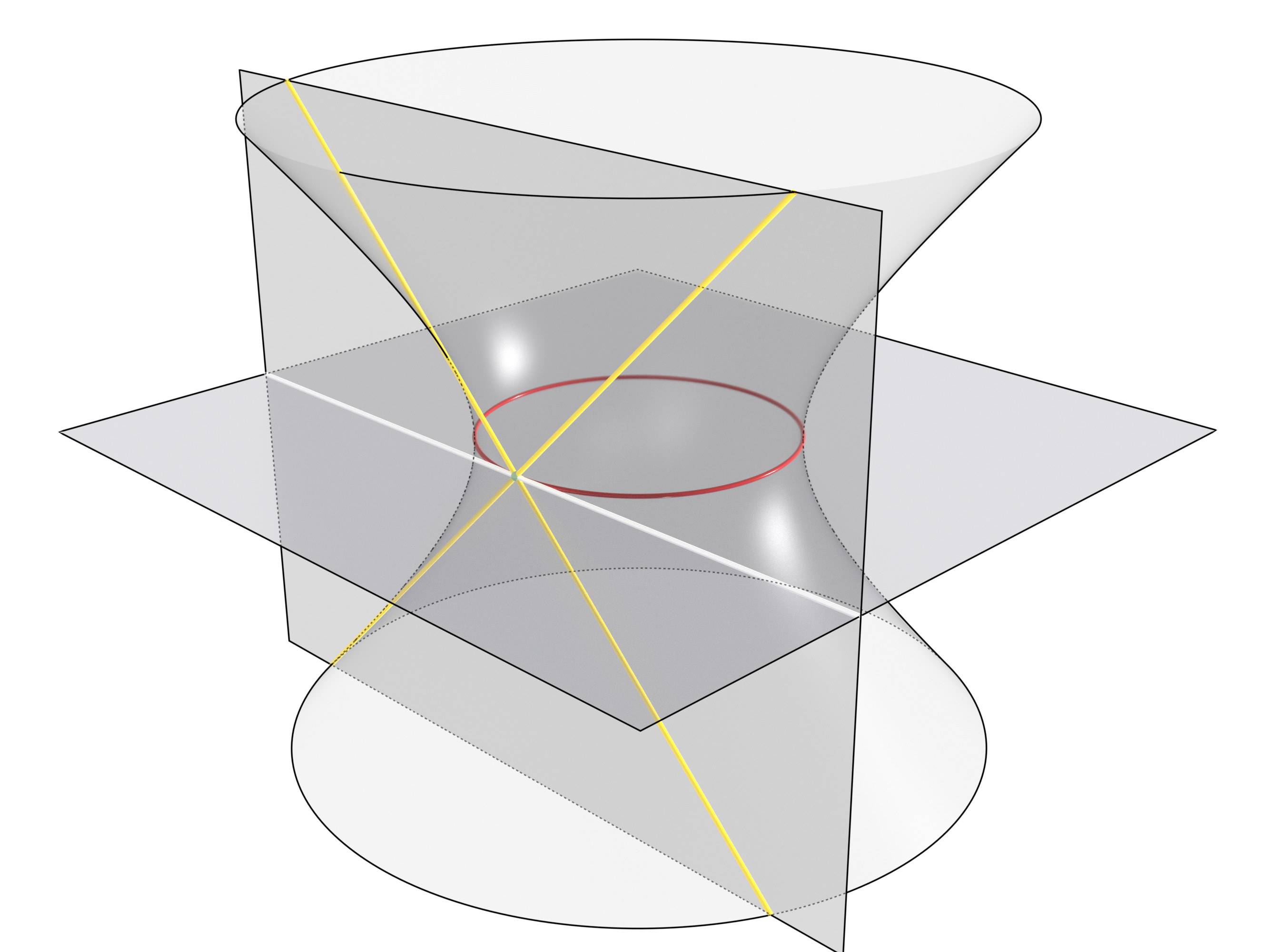}
    \put(-8,36){(b)}
  \end{overpic}
  \raisebox{0.35cm}{
    \includegraphics[width=0.28\textwidth]{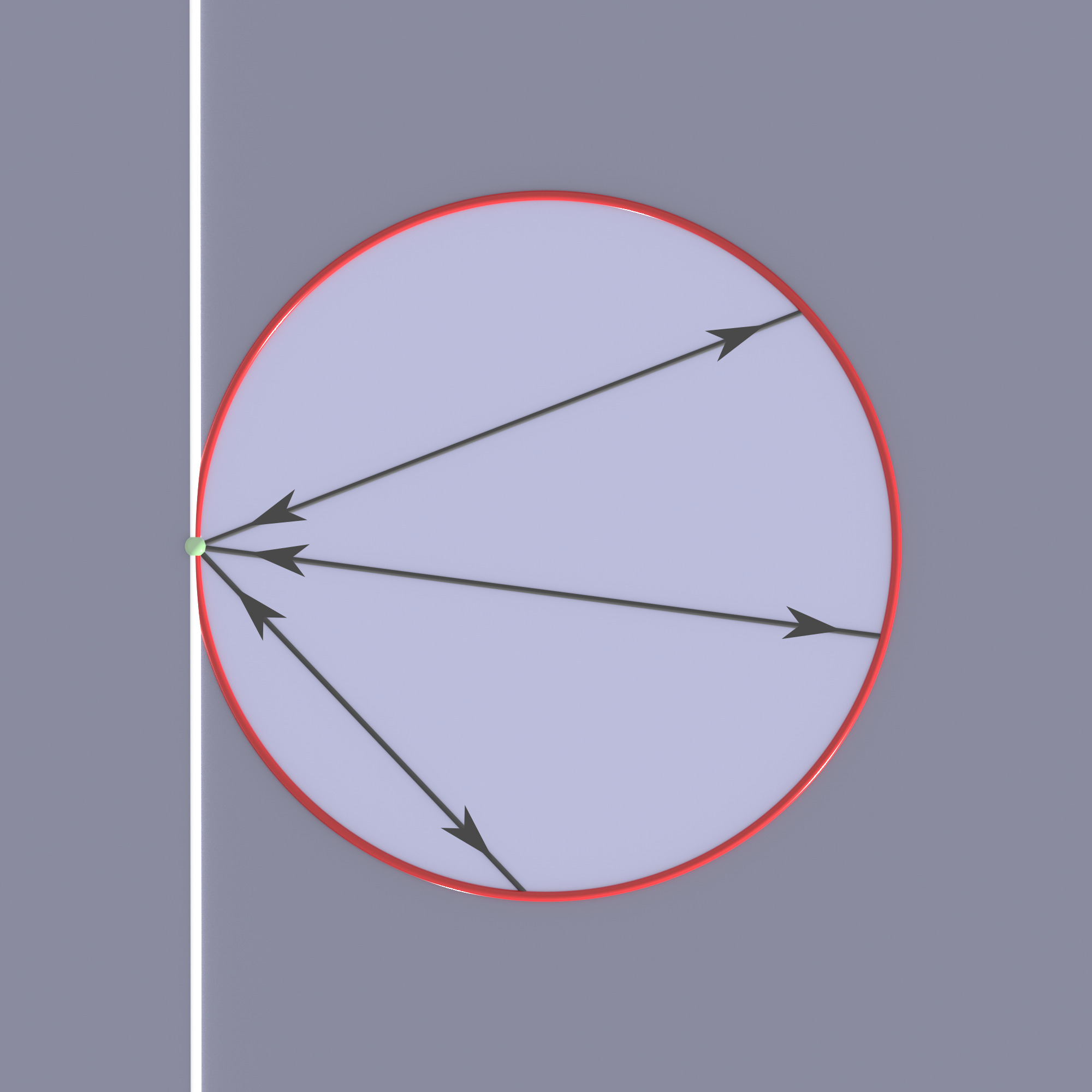}
  }\\
  \begin{overpic}[width=0.45\textwidth]{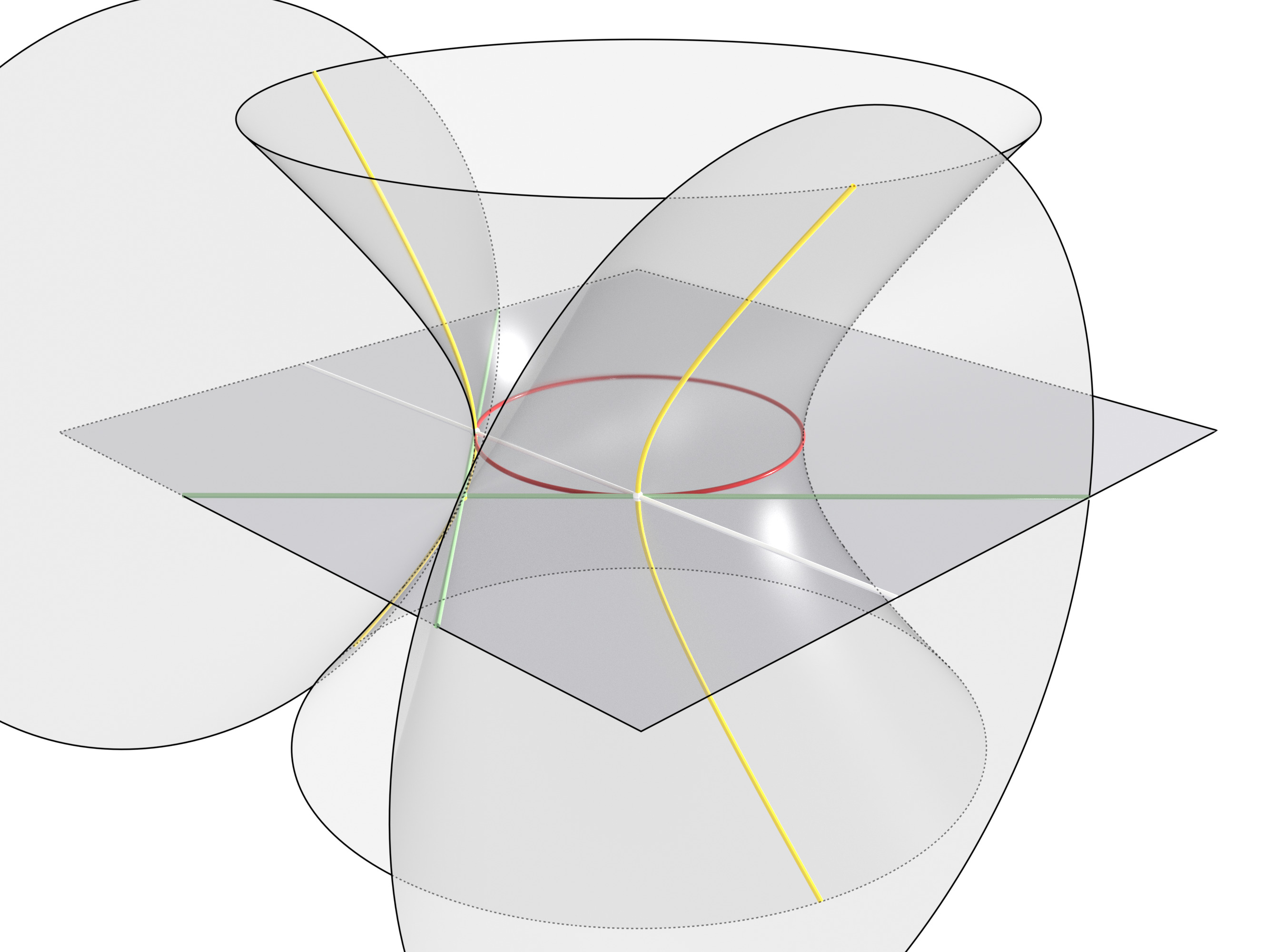}
    \put(-8,36){(c)}
  \end{overpic}
  \raisebox{0.35cm}{
    \includegraphics[width=0.28\textwidth]{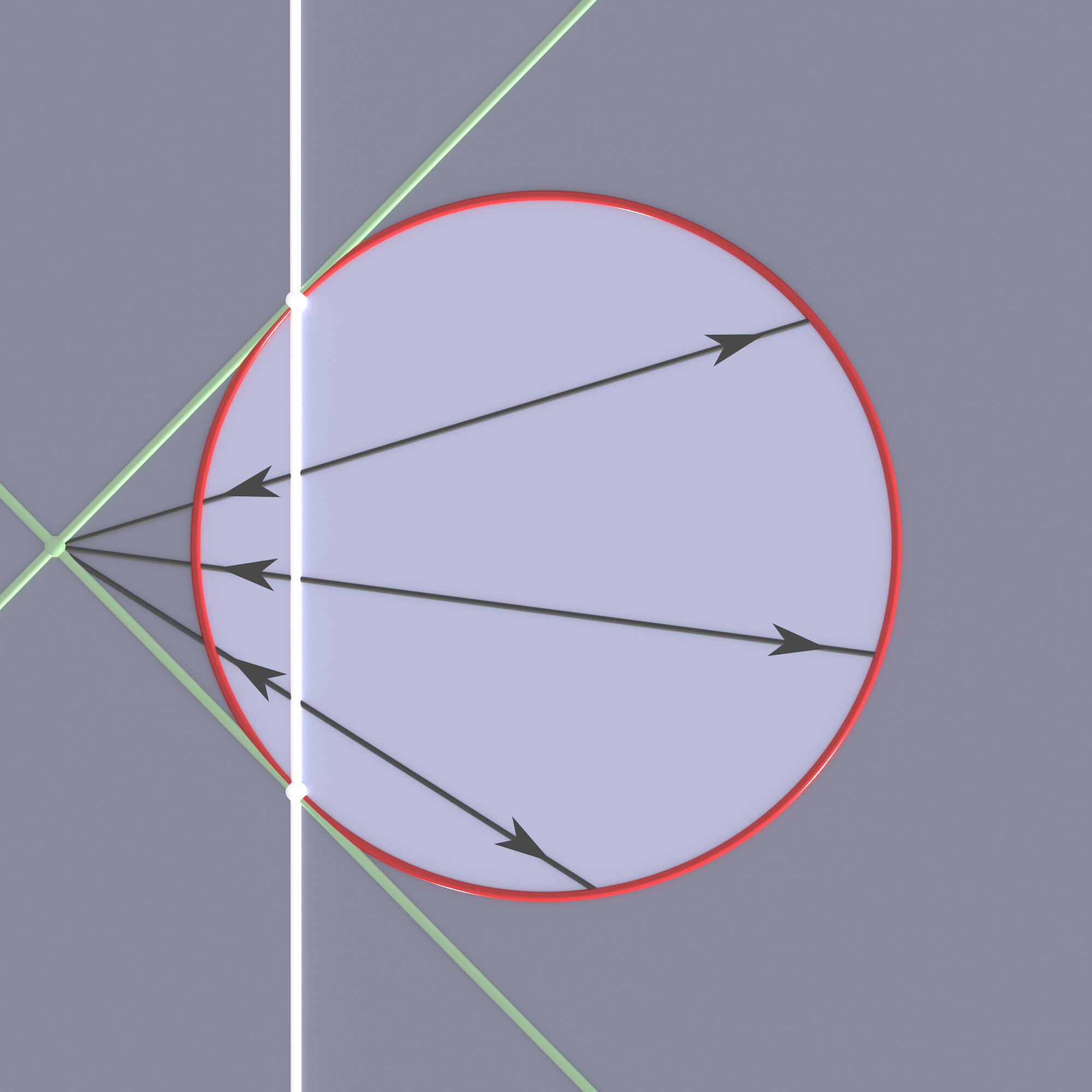}
  }
  \caption{
    Points in hyperbolic Laguerre geometry.
    In Laguerre geometry, points are oriented circles as well.
    Their images in the quadric model $\hyplag$ are sections with planes that are orthogonal to the base plane of $\mathcal{H}$.
    Every tangent line appears with both orientations.
    (a) An ordinary point in the hyperbolic plane $\mathcal{H}$.
    (b) An ideal point on the absolute circle $\mathcal{S}$.
    It defines two pencils of "parallel" oriented lines.
    In the quadric model these pencils correspond to the rulings of $\hyplag$.
    (c) A deSitter point.
    Only part of the lines through this point define hyperbolic lines.
  }
  \label{fig:hyperbolic-laguerre-geometry-points}
\end{figure}

To transform these oriented hyperbolic circles $c$ to the quadric model we apply polarity with respect to $\hyplag$. 
The tangents of $c$ are mapped to the lines of a quadratic cylinder $\Gamma$ (parallel to the $z$-axis, see Figure~\ref{fig:hyperbolic-laguerre-geometry-circle-lift}).
Its intersection with $\hyplag$ is composed of two conics (in planes symmetric to $z=0$).
This follows from the tangency of $\Gamma$ and $\hyplag$ in the ideal points $c \cap \mathcal{S} = L \cap \mathcal{S}$.
One of the two planar sections, $G(c)$, represents the oriented circle $c$ in its given orientation;
the other, symmetric section corresponds to the reverse orientation.
The orthogonal projection $c^\perp$ of such a planar section is polar to $c$ with respect to $\mathcal{S}$
and thus a conic which touches $\mathcal{S}$ in the same points as $c$ does.
However, this conic $c^\perp$ is not a hyperbolic circle, but a deSitter circle (outside~$\mathcal{S}$).

Let us now turn to points which also define Laguerre circles and are presented in Figure~\ref{fig:hyperbolic-laguerre-geometry-points}.
In the quadric model these circles correspond to sections with planes parallel to the $z$-axis
(passing through the polar line of the point $c$ with respect to $\mathcal{S}$). 
Note that beside hyperbolic points (Figure~\ref{fig:hyperbolic-laguerre-geometry-points} (a)) there is a further type of point-like circles, shown in Figure~\ref{fig:hyperbolic-laguerre-geometry-points} (c),
where the common point of tangents  lies in the deSitter space, outside $\mathcal{S}$.
Figure~\ref{fig:hyperbolic-laguerre-geometry-points} (b) illustrates the case of an ideal point $\p{x} \in \mathcal{S}$ viewed as set of lines.
This set can be oriented in two ways and such defines two pencils of parallel oriented lines.
They are not considered as oriented Laguerre circles and correspond to the intersection of $\hyplag$ with its tangent plane at $\p{x}$, which decomposes into two rulings.
Each of the two pencils corresponds to a ruling of $\hyplag$.

Having discussed all these cases we can state that oriented Laguerre circles of the hyperbolic plane $\mathcal{H}$ correspond precisely to the planar sections of $\hyplag$ different from rulings.

\paragraph{Laguerre transformations}
Finally, having the quadric model at our disposal, we turn to Laguerre transformations.
Laguerre transformations of the hyperbolic plane appear as projective transformations that map the hyperboloid $\hyplag$ to itself.
Those are exactly the maps that act bijectively on the set of points and planar sections of the projectively extended quadric $\hyplag$.

Again we see that Laguerre transformations do not preserve the special circles whose envelopes are points.
Those belong to planar sections of $\hyplag$ in $z$-parallel planes
and this special property of a plane is in general not preserved under a projective automorphism of $\hyplag$. 

A projective automorphism of $\hyplag$ maps rulings to rulings.
Thus, hyperbolic Laguerre transformations preserve parallelity of oriented straight lines.
The same is true in Euclidean Laguerre geometry but does not apply in the elliptic plane where there is no parallelism of straight lines.

%%% Local Variables:
%%% mode: latex
%%% TeX-master: "main"
%%% End:

\newpage
\section{Quadrics in projective space}
\label{sec:preliminaries}
We begin our general discussions with the introduction of \emph{quadrics} in projective space, see, e.g., \cite{K, Blproj, Gie}.
\subsection{Projective geometry}

Consider the $n$-dimensional \emph{real projective space}
\[
  \RP^n \coloneqq \P(\R^{n+1}) \coloneqq \faktor{ \left( \R^{n+1} \setminus \{0\} \right) }{ \sim }
\]
as it is generated via projectivization from its \emph{homogeneous coordinate space}
$\R^{n+1}$ using the equivalence relation
\[
x \sim y \quad\Leftrightarrow\quad x = \lambda y ~ \text{for some}~\lambda\in\R.
\]
We denote points in $\RP^{n}$ and its \emph{homogeneous coordinates} by
\[
  \p{x} = [x] = [x_1, \ldots, x_{n+1}].
\]
\emph{Affine coordinates} are given by normalizing one homogeneous coordinate to be equal to one
and then dropping this coordinate, e.g.,
\[
  \left( \frac{x_1}{x_{n+1}}, \ldots, \frac{x_n}{x_{n+1}} \right).
\]
Points with $x_{n+1}=0$, for which this normalization is not possible,
are said to lie on the \emph{hyperplane at infinity}.

The \emph{projectivization operator} $\P$ acts on any subset of the homogeneous coordinate space.
In particular, a \emph{projective subspace} $\p{U} \subset \RP^n$ is given by the projectivization
of a linear subspace $U \subset \R^{n+1}$,
\[
  \p{U} = \P(U), \quad \dim\p{U} = \dim U - 1.
\]
To denote projective subspaces spanned by a given set of points $\p{a}_1, \ldots \p{a}_k$
with linear independent homogeneous coordinate vectors we use the exterior product
\[
  \p{a}_1 \wedge \cdots \wedge \p{a}_k \coloneqq [a_1 \wedge \cdots \wedge a_k] = \P(\Span\{a_1,\ldots,a_k\}).
\]
 
The group of \emph{projective transformations} is induced by the group of linear
transformations of $\R^{n+1}$ and denoted by $\PGL(n+1)$.
A projective transformation maps projective subspaces to projective subspaces,
while preserving their dimension and incidences.
The \emph{fundamental theorem of real projective geometry}
states that this property characterizes projective transformations.
\begin{theorem}
  \label{thm:fundamental-theorem}
  Let $n \geq 2$, and $W \subset \RP^n$ be a non-empty open subset.
  Let $f : W \rightarrow \RP^n$ be an injective map that maps
  intersections of $k$-dimensional projective subspaces with $W$
  to intersections of $k$-dimensional projective subspaces with $f(W)$ for some $1 \leq k \leq n-1$.
  Then $f$ is the restriction of a unique projective transformation of $\RP^n$.
\end{theorem}

For a projective subgroup $G \subset \PGL(n+1)$ we denote the \emph{stabilizer} of a finite number of points $\p{v}_1, \ldots \p{v}_m \in \RP^n$ by
\begin{equation}
  \label{eq:stabilizer}
  G_{\p{v}_1, \ldots, \p{v}_m} \coloneqq \set{g \in G}{g(\p{v}_i) = \p{v}_i,~\text{for}~i=1,\ldots,m}.
\end{equation}

\subsection{Quadrics}
Let $\scalarprod{\cdot}{\cdot}$ be a non-zero symmetric bilinear form on $\R^{n+1}$.
A vector $x\in\R^{n+1}$ is called
\begin{itemize}
\item \emph{spacelike} if $\scalarprod{x}{x} > 0$,
\item \emph{timelike} if $\scalarprod{x}{x} < 0$,
\item \emph{lightlike}, or \emph{isotropic}, if $\scalarprod{x}{x} = 0$.
\end{itemize}
There always exists an \emph{orthogonal basis} with respect to $\scalarprod{\cdot}{\cdot}$,
i.e.\ a basis $(e_i)_{i = 1, \ldots, n+1}$ satisfying $\scalarprod{e_i}{e_j} = 0$ if $i\neq j$.
The triple $(r,s,t)$, consisting of the numbers of spacelike ($r$), timelike ($s$), and lightlike ($t$)
vectors in $(e_i)_{i = 1, \ldots, n+1}$ is called the \emph{signature} of $\scalarprod{\cdot}{\cdot}$.
It is invariant under linear transformations.
If $t=0$, the bilinear form $\scalarprod{\cdot}{\cdot}$ is called \emph{non-degenerate},
in which case we might omit its value in the signature.
We alternatively write the signature in the form
\[
  (\underbrace{+\cdots+}_r \underbrace{-\cdots-}_s \underbrace{0\cdots0}_t).
\]

The space $\R^{n+1}$ together with a bilinear form of signature $(r,s,t)$ is denoted by $\R^{r,s,t}$.
The zero set of the quadratic form corresponding to $\scalarprod{\cdot}{\cdot}$
\[
   \L^{r,s,t} \coloneqq \set{x\in\R^{n+1}}{\scalarprod{x}{x} = 0}
\]
is called the \emph{light cone}, or \emph{isotropic cone}.
Its projectivization
\[
  \quadric \coloneqq \P(\L^{r,s,t}) = \set{\p x \in\RP^n}{\scalarprod{x}{x} = 0} \subset \RP^n
\]
defines a \emph{quadric} in $\RP^{n}$
(quadrics in $\RP^2$ are called \emph{conics}).

A point $\p{x} \in \quadric$ contained in the kernel of the corresponding bilinear form $\scalarprod{\cdot}{\cdot}$, i.e.
\[
  \scalarprod{x}{y} = 0\qquad \text{for all}~ y \in \R^{n+1}
\]
is called a \emph{vertex} of $\quadric$.
A quadric is called \emph{non-degenerate} if it contains no vertices, or equivalently if $t = 0$.
If $\quadric$ is degenerate, i.e. $t > 0$, its set of vertices is a projective subspace of dimension $t-1$.

A non-zero scalar multiple of $\scalarprod{\cdot}{\cdot}$ defines the same quadric $\mathcal{Q}$.
Vice versa, if $\mathcal{Q}$ is non-empty and does not solely consist of vertices it uniquely determines its corresponding symmetric bilinear form up to a non-zero scalar multiple.
Upon considering the complexification of real quadrics
\[
  \quadric_\C \coloneqq \set{\p x \in\CP^n}{\scalarprod{x}{x} = 0} \subset \CP^n
\]
this correspondence holds in all cases,
and it is convenient to generally identify the term quadrics and symmetric bilinear forms up to non-zero scalar multiples.

The signature of a quadric is well-defined up to interchanging $r$ and $s$.
The signature of a projective subspace $\p{U} = \P(U)$ is defined by the signature of the bilinear form restricted to $U$.
After a choice of the sign for the bilinear form of $\quadric$ the signs for the signature of $\p{U}$ are fixed.

A quadric $\quadric$ naturally defines two regions in the projective space $\RP^n$,
\begin{equation}
  \label{eq:quadric-sides}
  \begin{aligned}
    \quadric^+ &\coloneqq \set{\p{x} \in \RP^n}{\scalarprod{x}{x} > 0},\\
    \quadric^- &\coloneqq \set{\p{x} \in \RP^n}{\scalarprod{x}{x} < 0},
  \end{aligned}
\end{equation}
which we call the two \emph{sides} of the quadric.
Which side is ``+'' and which side is ``-'' is only determined after choosing the sign for the bilinear form of $\quadric$.

A projective subspace entirely contained in the quadric $\quadric$ is called an \emph{isotropic subspace}.
A quadric with signature $(r,s,t)$ contains isotropic subspaces of dimension $\min\{r, s\} + t - 1$ through every point.

Consider the following examples of quadrics in $\RP^n$ with different signatures.
\begin{example}\
  \label{ex:quadrics}
  \nobreakpar
  \begin{enumerate}
  \item A quadric with signature $(n+1,0)$ is empty in $\RP^n$.
    By either identifying the quadric with its bilinear form up to non-zero scalar multiples
    or by complexification $\quadric_\C \subset \CP^n$,
    we consider this to be an admissible non-degenerate quadric,
    which only happens to have an empty real part.
    Note that one side of the quadric $\quadric^+ = \RP^n$ is the whole space,
    while the other side $\quadric^- = \varnothing$ is empty.
    \label{ex:quadrics-empty}
  \item A quadric with signature $(n,1)$ is an ``oval quadric''.
    It is projectively equivalent to the $(n-1)$-dimensional sphere $\S^{n-1}$.
  \item A quadric with signature $(n-1,2)$ is a higher dimensional analogue of a doubly ruled quadric in $\RP^3$.
    It contains lines as isotropic subspaces through every point, but no planes.
  \item A quadric with signature $(r,s,1)$ is a cone.
    It consists of all lines connecting its vertex to a non-degenerate quadric of signature $(r,s)$,
    given by its intersection with a hyperplane not containing the vertex.
    Note that if $r=0$ or $s=0$ (the real part of) the cone only consists of the vertex.
    The remaining part of the cone can be considered as imaginary (cf.\ Example~\ref{ex:quadrics-empty}).
    \label{ex:quadrics-cone}
  \item A quadric with signature $(1,0,n)$ is a ``doubly counted hyperplane''.
  \end{enumerate}
\end{example}

For non-neutral signature, i.e.\ $r \neq s$, and $rs \neq 0$, the subgroup of projective transformations
preserving the quadric $\quadric$ is exactly the \emph{projective orthogonal group} $\PO(r,s,t)$,
i.e.\ the projectivization of all linear transformations that preserve the bilinear form $\scalarprod{\cdot}{\cdot}$.
For simplicity, we call $\PO(r,s,t)$ the ``group of transformations that preserve the quadric $\quadric$'' for all signatures.
\begin{remark}
  In the case $r = s$ the statement remains true if we exclude projective transformations that interchange the two sides \eqref{eq:quadric-sides} of the quadric.
  In the case $rs=0$ the statement remains true upon complexification.
\end{remark}
The fundamental theorem of real projective geometry (see Theorem \ref{thm:fundamental-theorem})
may be specialized to quadrics.
\begin{theorem}
  \label{thm:fundamental-theorem-quadrics}
  Let $n \geq 3$, $\quadric \subset \RP^n$ be a non-degenerate non-empty quadric in $\RP^n$,
  and $W \subset \quadric$ be a non-empty open subset of the quadric.
  Let $f : W \rightarrow \quadric$ be an injective map that maps
  intersections of $k$-dimensional projective subspaces with $W$
  to intersections of $k$-dimensional projective subspaces with $f(W)$ for some $2 \leq k \leq n-1$.
  Then $f$ is the restriction of a unique projective transformation of $\RP^n$ that preserves the quadric $\quadric$.
\end{theorem}

For a non-degenerate quadric every such transformation
can be decomposed into a finite number of reflections in hyperplanes
by the theorem of Cartan and Dieudonn{\'e}.
\begin{theorem}
  \label{thm:cartan}
  Let $\quadric \subset \RP^n$ be a non-degenerate quadric of signature $(r,s)$.
  Then each element of the corresponding projective orthogonal group $\PO(r,s)$
  is the composition of at most $n+1$ reflections in hyperplanes,
  i.e. transformations of the form
  \[
    \sigma_{\p{q}} : \RP^n \rightarrow \RP^n, \qquad [x] \mapsto \left[x - 2 \frac{\scalarprod{x}{q}}{\scalarprod{q}{q}}q \right]
  \]
  for some $\p{q} \in \RP^n \setminus \quadric$.
\end{theorem}

\subsection{Polarity}
\label{sec:polarity}
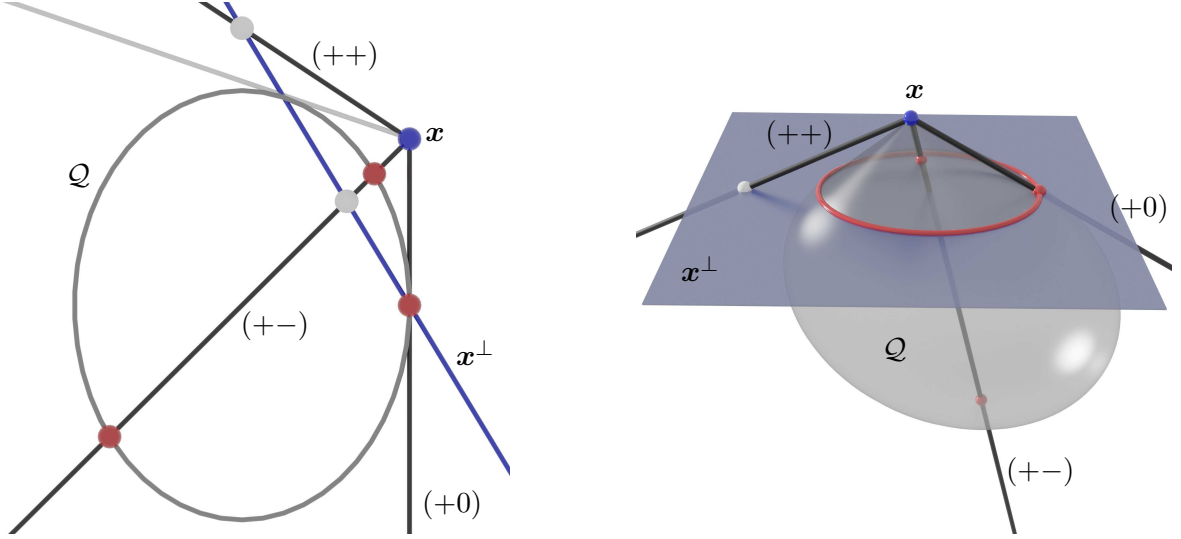
\begin{figure}
  \centering
  \def\svgwidth{0.44\textwidth}
%% Accompanies image file 'polarity_2d.pdf' (pdf, eps, ps)
\begingroup%
  \makeatletter%
  \providecommand\color[2][]{%
    \errmessage{(Inkscape) Color is used for the text in Inkscape, but the package 'color.sty' is not loaded}%
    \renewcommand\color[2][]{}%
  }%
  \providecommand\transparent[1]{%
    \errmessage{(Inkscape) Transparency is used (non-zero) for the text in Inkscape, but the package 'transparent.sty' is not loaded}%
    \renewcommand\transparent[1]{}%
  }%
  \providecommand\rotatebox[2]{#2}%
  \newcommand*\fsize{\dimexpr\f@size pt\relax}%
  \newcommand*\lineheight[1]{\fontsize{\fsize}{#1\fsize}\selectfont}%
  \ifx\svgwidth\undefined%
    \setlength{\unitlength}{2025.31640625bp}%
    \ifx\svgscale\undefined%
      \relax%
    \else%
      \setlength{\unitlength}{\unitlength * \real{\svgscale}}%
    \fi%
  \else%
    \setlength{\unitlength}{\svgwidth}%
  \fi%
  \global\let\svgwidth\undefined%
  \global\let\svgscale\undefined%
  \makeatother%
  \begin{picture}(1,1)%
    \lineheight{1}%
    \setlength\tabcolsep{0pt}%
    \put(0,0){\includegraphics[width=\unitlength,page=1]{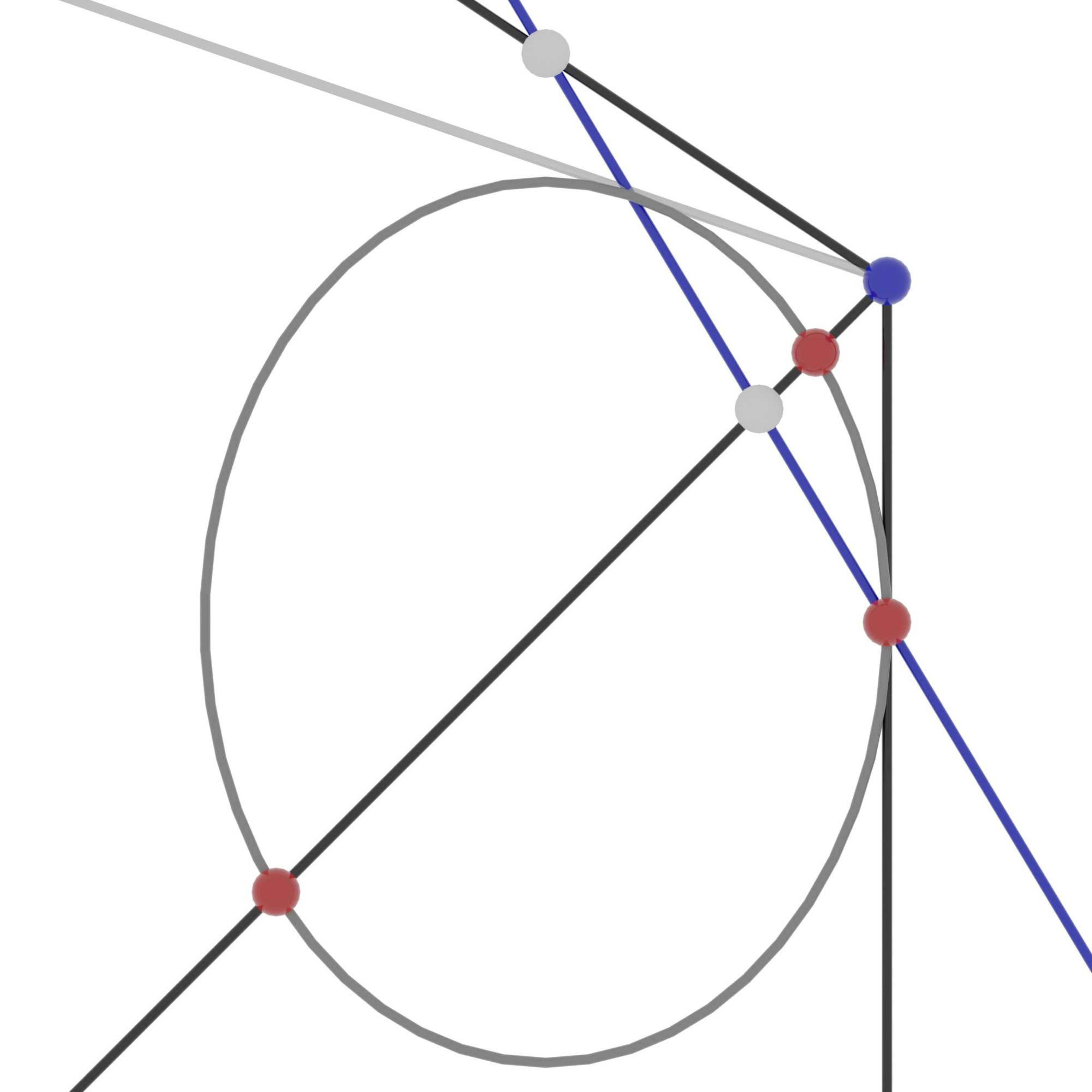}}%
    \put(0.84268291,0.73902438){\color[rgb]{0,0,0}\makebox(0,0)[lt]{\lineheight{1.25}\smash{\begin{tabular}[t]{l}$\p{x}$\end{tabular}}}}%
    \put(0.89999998,0.31707321){\color[rgb]{0,0,0}\makebox(0,0)[lt]{\lineheight{1.25}\smash{\begin{tabular}[t]{l}$\p{x}^{\perp}$\end{tabular}}}}%
    \put(0.83658532,0.04024391){\color[rgb]{0,0,0}\makebox(0,0)[lt]{\lineheight{1.25}\smash{\begin{tabular}[t]{l}$(+0)$\end{tabular}}}}%
    \put(0.49756096,0.37804879){\color[rgb]{0,0,0}\makebox(0,0)[lt]{\lineheight{1.25}\smash{\begin{tabular}[t]{l}$(+-)$\end{tabular}}}}%
    \put(0.62804879,0.88536586){\color[rgb]{0,0,0}\makebox(0,0)[lt]{\lineheight{1.25}\smash{\begin{tabular}[t]{l}$(++)$\end{tabular}}}}%
    \put(0.17073171,0.65853658){\color[rgb]{0,0,0}\makebox(0,0)[lt]{\lineheight{1.25}\smash{\begin{tabular}[t]{l}$\quadric$\end{tabular}}}}%
  \end{picture}%
\endgroup%

  \hspace{0.08\textwidth}
  \def\svgwidth{0.44\textwidth}
%% Accompanies image file 'polarity_3d.pdf' (pdf, eps, ps)
\begingroup%
  \makeatletter%
  \providecommand\color[2][]{%
    \errmessage{(Inkscape) Color is used for the text in Inkscape, but the package 'color.sty' is not loaded}%
    \renewcommand\color[2][]{}%
  }%
  \providecommand\transparent[1]{%
    \errmessage{(Inkscape) Transparency is used (non-zero) for the text in Inkscape, but the package 'transparent.sty' is not loaded}%
    \renewcommand\transparent[1]{}%
  }%
  \providecommand\rotatebox[2]{#2}%
  \newcommand*\fsize{\dimexpr\f@size pt\relax}%
  \newcommand*\lineheight[1]{\fontsize{\fsize}{#1\fsize}\selectfont}%
  \ifx\svgwidth\undefined%
    \setlength{\unitlength}{2025.31640625bp}%
    \ifx\svgscale\undefined%
      \relax%
    \else%
      \setlength{\unitlength}{\unitlength * \real{\svgscale}}%
    \fi%
  \else%
    \setlength{\unitlength}{\svgwidth}%
  \fi%
  \global\let\svgwidth\undefined%
  \global\let\svgscale\undefined%
  \makeatother%
  \begin{picture}(1,1)%
    \lineheight{1}%
    \setlength\tabcolsep{0pt}%
    \put(0,0){\includegraphics[width=\unitlength,page=1]{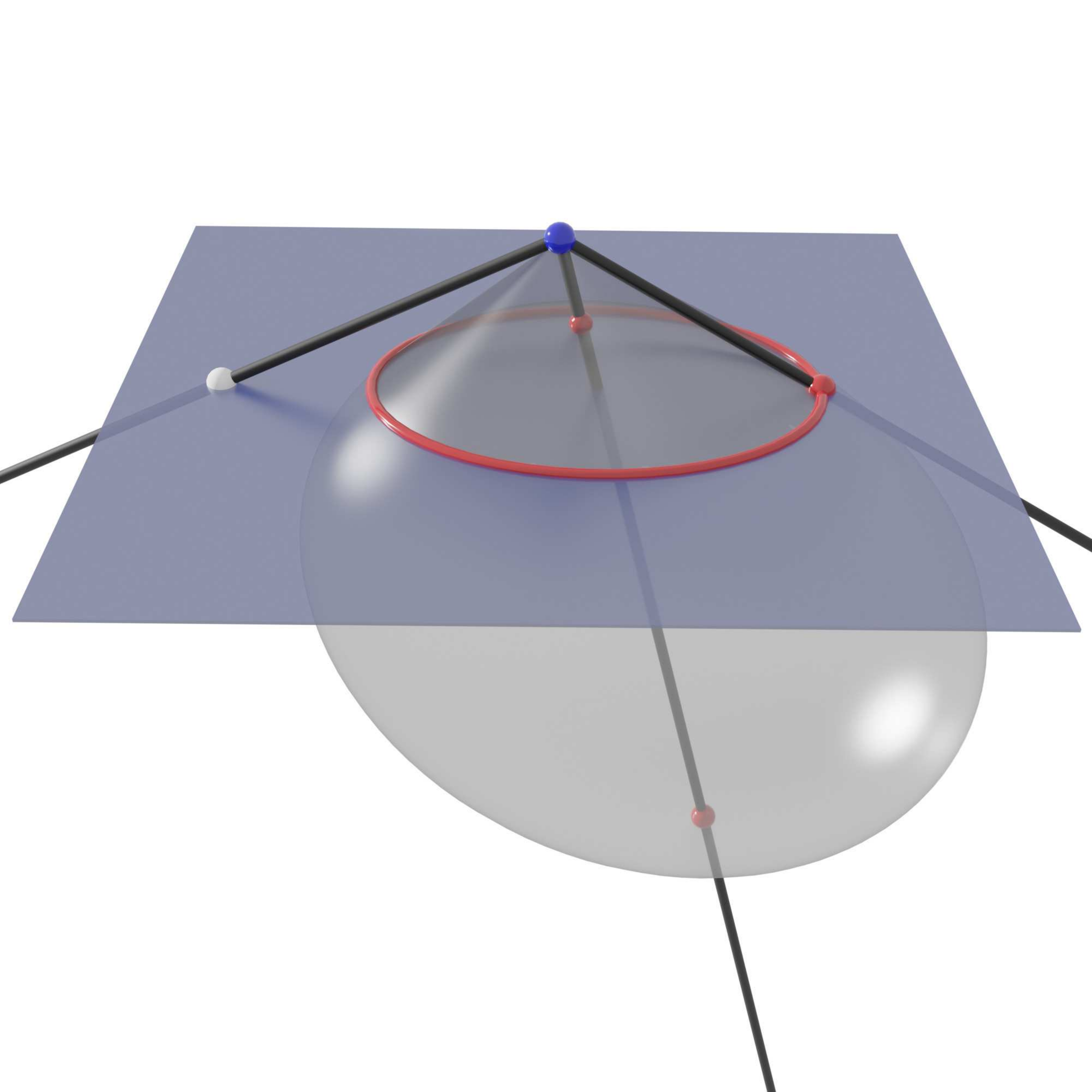}}%
    \put(0.50233642,0.82009346){\color[rgb]{0,0,0}\makebox(0,0)[lt]{\lineheight{1.25}\smash{\begin{tabular}[t]{l}$\p{x}$\end{tabular}}}}%
    \put(0.08411215,0.47546729){\color[rgb]{0,0,0}\makebox(0,0)[lt]{\lineheight{1.25}\smash{\begin{tabular}[t]{l}$\p{x}^{\perp}$\end{tabular}}}}%
    \put(0.69158881,0.10280376){\color[rgb]{0,0,0}\makebox(0,0)[lt]{\lineheight{1.25}\smash{\begin{tabular}[t]{l}$(+-)$\end{tabular}}}}%
    \put(0.24182243,0.73831775){\color[rgb]{0,0,0}\makebox(0,0)[lt]{\lineheight{1.25}\smash{\begin{tabular}[t]{l}$(++)$\end{tabular}}}}%
    \put(0.88317761,0.59813083){\color[rgb]{0,0,0}\makebox(0,0)[lt]{\lineheight{1.25}\smash{\begin{tabular}[t]{l}$(+0)$\end{tabular}}}}%
    \put(0.46028037,0.33060748){\color[rgb]{0,0,0}\makebox(0,0)[lt]{\lineheight{1.25}\smash{\begin{tabular}[t]{l}$\quadric$\end{tabular}}}}%
  \end{picture}%
\endgroup%

  \caption{
    Polarity with respect to a conic $\quadric$ with signature $(++-)$ in $\RP^2$ (\emph{left})
    and a quadric of signature $(+++-)$ in $\RP^3$ (\emph{right}).
    The point $\p{x}$ and its polar hyperplane $\p{x}^\perp$ are shown as well as the cone of contact from the point $\p{x}$.
    Lines through $\p{x}$ that are ``inside'' (signature $(+-)$), ``on'' (signature $(+0)$), and ``outside'' (signature $(++)$)
    the cone intersect the quadric in 2, 1, or 0 points respectively.
  }
  \label{fig:polarity}
\end{figure}

A quadric induces the notion of \emph{polarity} between projective subspaces (see Figure \ref{fig:polarity}).
For a projective subspace $\p{U} = \P(U)\subset\RP^{n}$,
where $U\subset\R^{n+1}$ is a linear subspace,
the \emph{polar subspace} of $\p{U}$ is defined as
\[
   \p{U}^{\perp}
   \coloneqq \set{\p x\in\RP^{n}}{\scalarprod{x}{y} = 0\;\text{for all}\;y\in U}.
\]
If $\quadric$ is non-degenerate, the dimensions of two polar subspaces satisfy the following relation:
\[
  \dim\p{U} + \dim\p{U}^\perp = n - 1.
\]
A refinement of this statement, which includes the signatures of the two polar subspaces, is captured in the following lemma.
\begin{lemma}
  \label{lemma:orthogonal-subspaces}
  Let $\quadric \subset \RP^n$ be a non-degenerate quadric of signature $(r, s)$.
  Then the signature $(\tilde r, \tilde s, \tilde t)$ of a subspace $\p{U} \subset \RP^n$
  and the signature $(\tilde r_\perp, \tilde s_\perp, \tilde t_\perp)$ of its polar subspace $\p{U}^\perp$
  with respect to $\quadric$ satisfy
  \[
    r = \tilde r + \tilde r_\perp + \tilde t, \quad
    s = \tilde s + \tilde s_\perp + \tilde t, \quad
    \tilde t = \tilde t_\perp.
  \]
  In particular, $\tilde t \leq \min\{r, s\}$.
\end{lemma}

For a point $\p{x} \in \quadric$ on a quadric, which is not a vertex,
the tangent hyperplane of $\quadric$ at $\p{x}$ is given by its polar hyperplane $\p{x}^\perp$.
If $\quadric$ has signature $(r,s,t)$ then the tangent plane has signature $(r-1,s-1,t+1)$.
Furthermore, for a non-degenerate quadric a projective subspace is tangent to $\quadric$ if and only if its signature is degenerate.

A projective line not contained in a quadric can intersect the quadric in either zero, one, or two points (see Figure \ref{fig:polarity}).
\begin{lemma}
  \label{lem:quadric-line-intersection}
  Let $\quadric \subset \RP^n$ be a quadric,
  $\p{x}, \p{y} \in \RP^{n}$, $\p{x} \neq \p{y}$ be two points,
  and define
  \[
    \Delta \coloneqq \scalarprod{x}{y}^2 - \scalarprod{x}{x}\scalarprod{y}{y}.
  \]
  \begin{itemize}
  \item If $\Delta > 0$, then the line $\p{x} \wedge \p{y}$ has signature $(+-)$ and intersects $\quadric$ in two points
    \[
      \p{x}_\pm = \left[ \scalarprod{y}{y}x + \left( -\scalarprod{x}{y} \pm \sqrt{\Delta} \right) y \right].
    \]
  \item If $\Delta < 0$, then the line $\p{x} \wedge \p{y}$ has signature $(++)$ or $(--)$
    and intersects $\quadric$ in no real points, but in two complex conjugate points
    \[
      \p{x}_{\pm} = \left[ \scalarprod{y}{y}x + \left( -\scalarprod{x}{y} \pm i\sqrt{-\Delta} \right) y \right].
    \]
  \item If $\Delta = 0$, then the line $\p{x} \wedge \p{y}$ has signature $(+0)$ or $(-0)$
    and it is tangent to $\quadric$ in the point
    \[
      \p{\tilde{x}} = \left[ \scalarprod{y}{y}x - \scalarprod{x}{y} y \right],
    \]
    or it has signature $(00)$ and is contained in $\quadric$ (isotropic line).
  \end{itemize}
\end{lemma}
The last point of the preceding lemma gives rise to the following definition of the cone of contact (see Figure \ref{fig:polarity}).
\begin{definition}
  \label{def:cone-of-contact}
  Let $\quadric \subset \RP^n$ be a quadric with corresponding bilinear form $\scalarprod{\cdot}{\cdot}$,
  and $\p{x} \in \RP^n \setminus \quadric$.
  Define the quadratic form
  \[
    \Delta_x(y) \coloneqq \scalarprod{x}{y}^2 - \scalarprod{x}{x}\scalarprod{y}{y}.
  \]
  Then the corresponding quadric
  \[
    \cone{\quadric}{\p{x}} \coloneqq \set{\p{y} \in \RP^n}{ \Delta_x(y) = 0}
  \]
  is called the \emph{cone of contact}, or \emph{tangent cone}, to $\quadric$ from the point $\p{x}$.
\end{definition}
The points of tangency of the cone of contact lie in the polar hyperplane of its vertex.
\begin{lemma}
  \label{lem:cone-of-contact}
  Let $\quadric \subset \RP^n$ be a quadric.
  For a point $\p{x} \in \RP^n \setminus \quadric$ the cone of contact to $\quadric$ from $\p{x}$ is given by
  \[
    \cone{\quadric}{\p{x}} = \bigcup\limits_{\p{y} \in \p{x}^\perp \cap \quadric} \p{x} \wedge \p{y}.
  \]
\end{lemma}
\begin{remark}
  For a non-degenerate quadric $\quadric$ the intersection $\p{x}^\perp \cap \quadric$ always results in a non-degenerate quadric in $\p{x}^\perp$.
  If the restriction of the corresponding bilinear form has signature $(n,0)$ or $(0,n)$ the intersection can be considered as imaginary.
  The real part of the cone only consists of the vertex in this case (cf.\ Example \ref{ex:quadrics} \ref{ex:quadrics-cone}).
\end{remark}

\subsection{Pencils of quadrics}
Let $\quadric_1, \quadric_2 \subset \RP^n$ be two distinct quadrics with corresponding bilinear forms $\scalarprod{\cdot}{\cdot}_1, \scalarprod{\cdot}{\cdot}_2$ respectively.
Every linear combination of these two bilinear forms yields a quadric.
The family of quadrics obtained by all linear combinations of the two bilinear forms is called a \emph{pencil of quadrics}
(see Figure \ref{fig:concentric-cayley-klein-circles}):
\[
  \quadric_1 \wedge \quadric_2
  \coloneqq \left( \quadric_{[\lambda_1,\lambda_2]} \right)_{[\lambda_1, \lambda_2] \in \RP^1}, \qquad
  \quadric_{[\lambda_1,\lambda_2]} \coloneqq \set{\p{x}\in\RP^n}{\lambda_1\scalarprod{x}{x}_1 + \lambda_2\scalarprod{x}{x}_2 = 0}.
\]
This is a line in the \emph{projective space of quadrics of $\RP^n$}.

A pencil of quadrics is called \emph{non-degenerate} if it does not consist exclusively of degenerate quadrics.
It contains at most $n+1$ degenerate quadrics.

A point contained in the intersection of two quadrics from a pencil of quadrics is called a \emph{base point}.
It is then contained in every quadric of the pencil.
The variety of base points has (at least) codimension~2.
\begin{example}
  \label{ex:maximal-contact-pencil}
  The pencil of quadrics $\quadric \wedge \cone{\quadric}{\p{x}}$ spanned by a non-degenerate quadric $\quadric$
  and the cone of contact $C_{\quadric}(\p{x})$ from a point $\p{x} \in \RP^n\setminus \quadric$ contains as degenerate quadrics only the
  cone $\cone{\quadric}{\p{x}}$ and the polar hyperplane $\p{x}^\perp$.
  It is comprised of exactly the quadrics that are tangent to $\quadric$ in $\quadric \cap \p{x}^\perp$.
\end{example}

%%% Local Variables:
%%% TeX-master: "main"
%%% End:

\newpage
\section{Cayley-Klein spaces}
\label{sec:Cayley-Klein-metric}

In Klein's \emph{Erlangen program} Euclidean and non-Euclidean geometries are considered as subgeometries of projective geometry.
Projective models for, e.g., hyperbolic, deSitter, and elliptic space can be obtained by using a quadric to induce the corresponding metric \cite{K}.
In this section we introduce the corresponding general notion of \emph{Cayley-Klein spaces} and their groups of \emph{isometries}, see, e.g., \cite{K, Blproj, Gie}.
We put a particular emphasis on the description of hyperplanes, hyperspheres, and their mutual relations.

\subsection{Cayley-Klein distance}
A quadric within a projective space induces an invariant for pairs of points.
\begin{definition}
  \label{def:Cayley-Klein-distance}
  Let $\quadric \subset \RP^n$ be a quadric with corresponding bilinear form $\scalarprod{\cdot}{\cdot}$.
  Then we denote by
  \[
    \ck{\quadric}{\p{x}}{\p{y}} \coloneqq \frac{\scalarprod{x}{y}^2}{\scalarprod{x}{x}\scalarprod{y}{y}}
  \]
  the \emph{Cayley-Klein distance} of any two points $\p{x}, \p{y} \in \RP^n \setminus \quadric$ that are not on the quadric.

  In the presence of a Cayley-Klein distance the quadric $\mathcal{Q}$ is called the \emph{absolute quadric}.
\end{definition}
\begin{remark}
  The name Cayley-Klein distance, or Cayley-Klein metric, is usually assigned to an actual metric derived from the above quantity
  as, for example, the hyperbolic metric (cf.\ Section \ref{sec:hyperbolic-space}).
  Nevertheless, we prefer to assign it to this basic quantity associated with an arbitrary quadric.
\end{remark}
The Cayley-Klein distance is projectively well-defined,
in the sense that it depends neither on the choice of the bilinear form corresponding to the quadric $\quadric$
nor on the choice of homogeneous coordinate vectors for the points $\p{x}$ and $\p{y}$.
Furthermore, it is invariant under the group of projective transformations that preserve the quadric $\quadric$,
which we call the corresponding group of \emph{isometries}.

The Cayley-Klein distance can be positive or negative depending on the relative location of the two points
with respect to the quadric, cf.\ \eqref{eq:quadric-sides}.
\begin{proposition}
  \label{prop:quadric-sides}
  For two points $\p{x}, \p{y} \in \RP^n \setminus \quadric$ with $\scalarprod{x}{y} \neq 0$:
  \begin{itemize}
  \item $\ck{\quadric}{\p{x}}{\p{y}} > 0$ if $\p{x}$ and $\p{y}$ are on the same side of $\quadric$,
  \item $\ck{\quadric}{\p{x}}{\p{y}} < 0$ if $\p{x}$ and $\p{y}$ are on opposite sides of $\quadric$.
  \end{itemize}
\end{proposition}
A \emph{Cayley-Klein space} is usually considered to be one side of the quadric, i.e. $\quadric^+$ or $\quadric^-$,
together with a (pseudo-)metric derived from the Cayley-Klein distance,
or equivalently, together with the transformation group of isometries.

\subsection{Cayley-Klein spheres}
\begin{figure}
  \centering
  \includegraphics[width=0.32\textwidth]{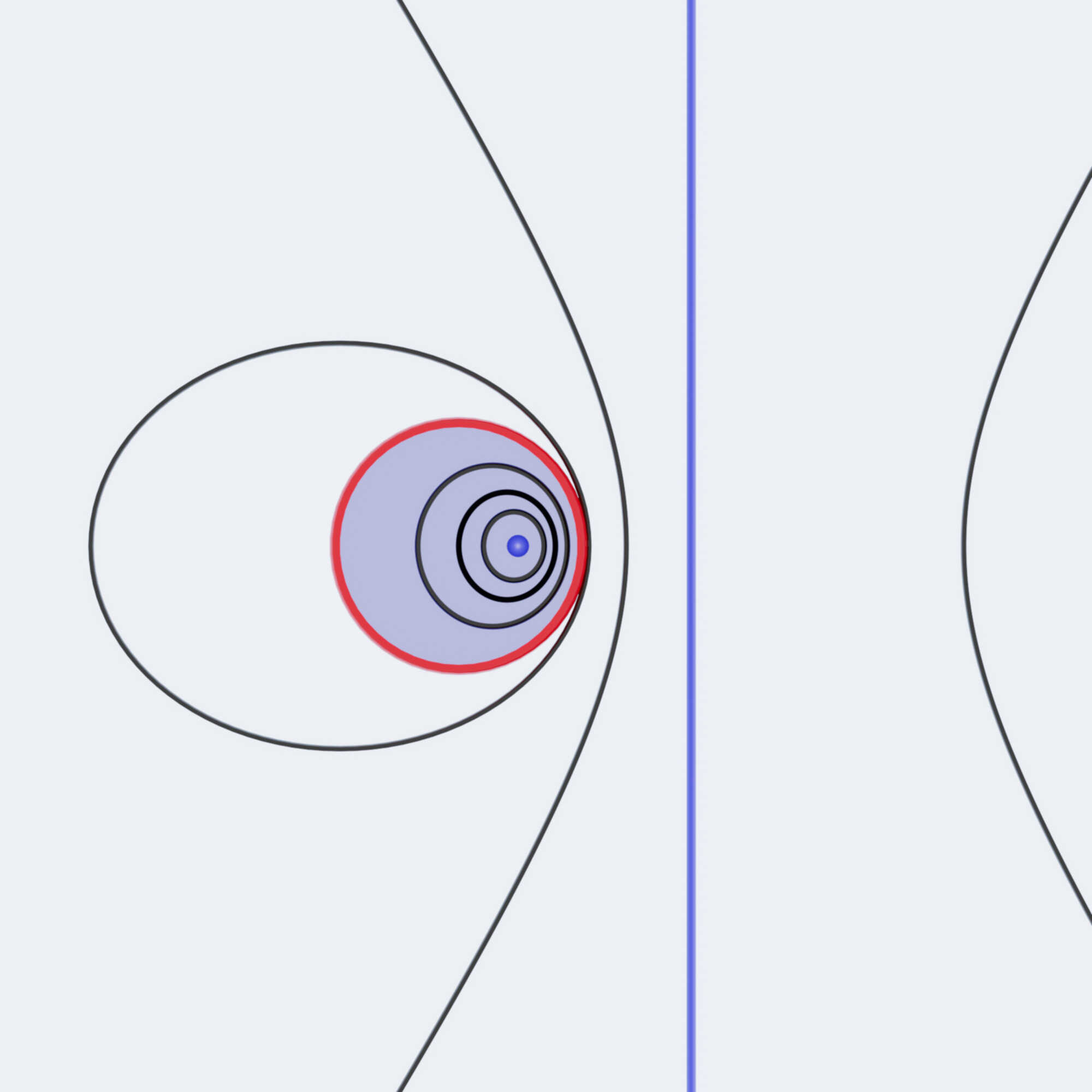}
  \includegraphics[width=0.32\textwidth]{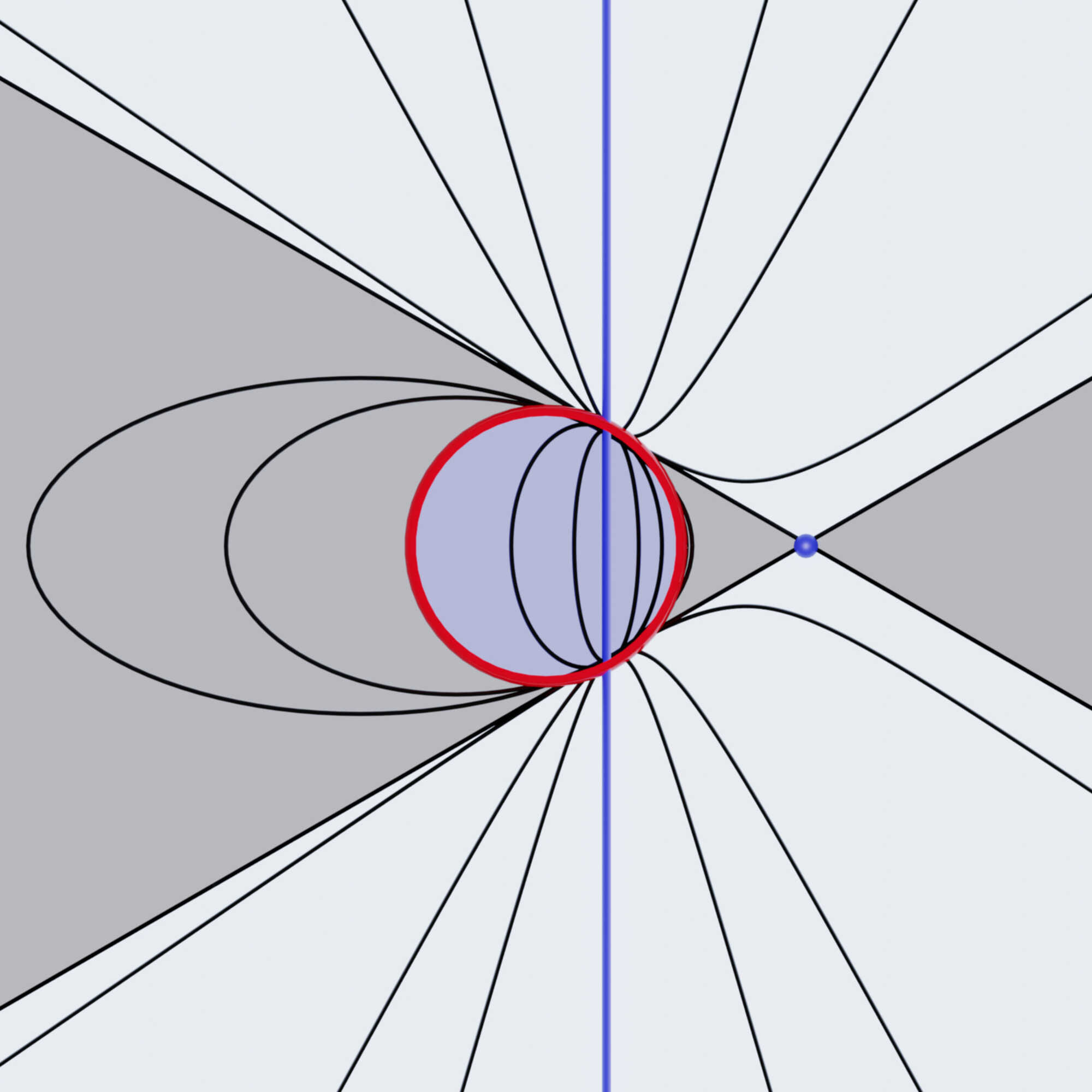}
  \includegraphics[width=0.32\textwidth]{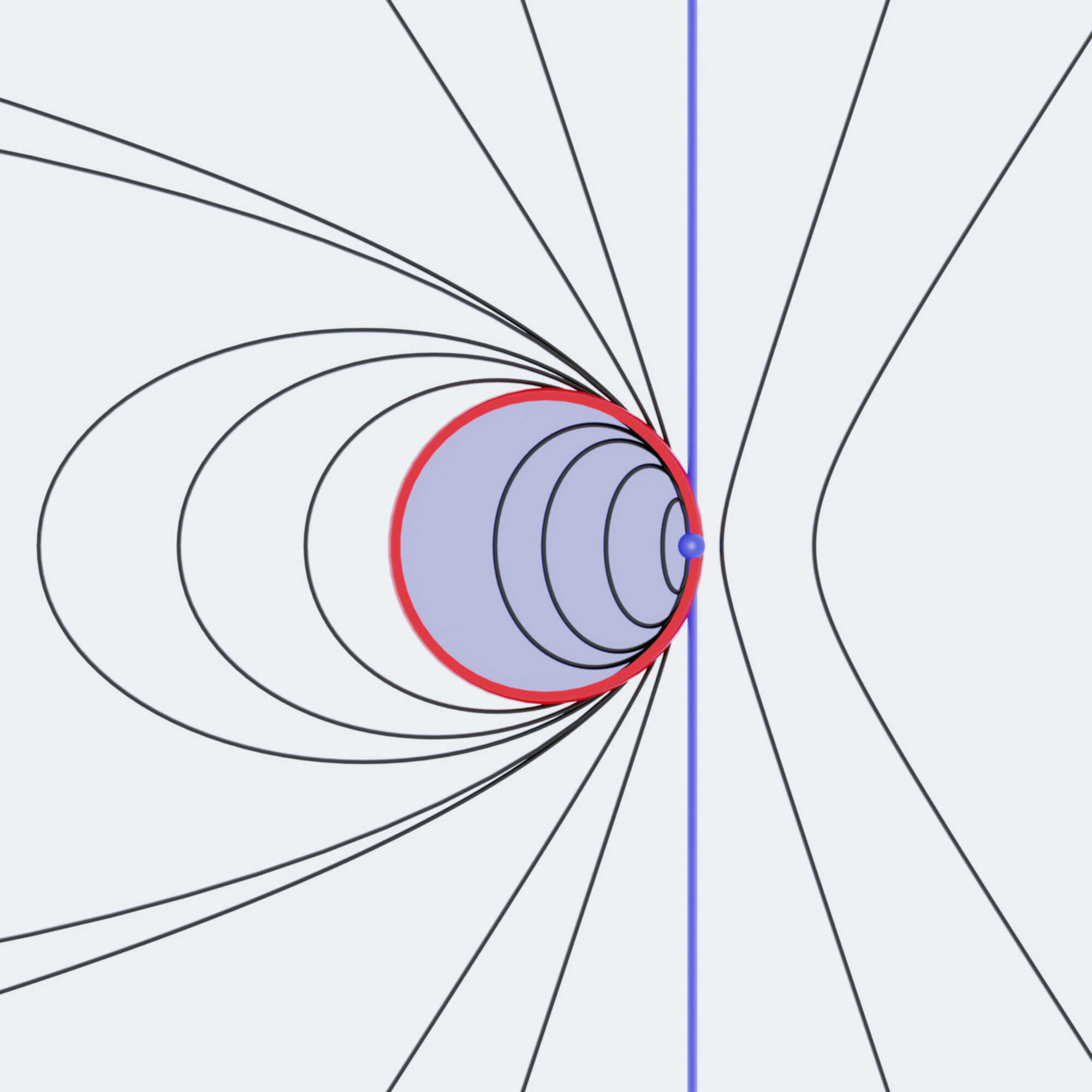}
  \caption{
    Concentric Cayley-Klein circles in the hyperbolic/deSitter plane.
    \emph{Left:} Concentric circles with hyperbolic center.
    \emph{Middle:} Concentric circles with deSitter center.
    \emph{Right:} Concentric horocycles with center on the absolute conic.
  }
\label{fig:concentric-cayley-klein-circles}
\end{figure}

Having a notion of ``distance'' allows for the definition of corresponding spheres (see Figure \ref{fig:concentric-cayley-klein-circles}).
\begin{definition}
  \label{def:Cayley-Klein-sphere}
  Let $\quadric \subset \RP^n$ be a quadric, $\p{x} \in \RP^n\setminus\quadric$, and $\mu \in \R \cup \{\infty\}$.
  Then we call the set
  \[
    \cksphere{\p{x}}{\mu} \coloneqq \set{\p{y} \in \RP^n}{ \ck{\quadric}{\p{x}}{\p{y}} = \mu}
  \]
  the \emph{Cayley-Klein hypersphere} with \emph{center} $\p{x}$ and \emph{Cayley-Klein radius} $\mu$.
\end{definition}
\begin{remark}\
  \label{rem:Cayley-Klein-spheres}
  \nobreakpar
  \begin{enumerate}
  \item
    Due to the fact that the Cayley-Klein sphere equation can be written as
    \begin{equation}
      \label{eq:Cayley-Klein-sphere}
      \scalarprod{x}{y}^2 - \mu \scalarprod{x}{x}\scalarprod{y}{y} = 0,
    \end{equation}
    we may include into the set $\cksphere{\p{x}}{\mu}$ points $\p{y} \in \quadric$ on the quadric,
    and also allow for $\mu = \infty$ (cf.\ Proposition~\ref{prop:concentric-sphere-pencil}).
  \item
    \label{rem:Cayley-Klein-horospheres}
    Given the center $\p{x}$ of a Cayley-Klein sphere one can further rewrite the Cayley-Klein sphere equation \eqref{eq:Cayley-Klein-sphere} as
    \[
      \scalarprod{x}{y}^2 - \tilde{\mu}\scalarprod{y}{y} = 0,
    \]
    where $\tilde{\mu} \coloneqq \mu\scalarprod{x}{x}$.
    While $\tilde{\mu}$ is not projectively invariant anymore,
    the solution set of this equation still invariantly describes a Cayley-Klein sphere.
    We may now allow for centers on the quadric $\p{x} \in \quadric$ which gives rise to \emph{Cayley-Klein horospheres}
    (see Figure \ref{fig:concentric-cayley-klein-circles}, right).
  \item
    We will occasionally denote all these cases including Cayley-Klein horospheres by the term \emph{Cayley-Klein spheres}.
  \end{enumerate}
\end{remark}
\begin{proposition}
  \label{prop:Cayley-Klein-spheres}
  For a Cayley-Klein sphere with center $\p{x} \in \RP^n \setminus \quadric$ and Cayley-Klein radius $\mu$ one has:
  \begin{itemize}
  \item If $\mu < 0$ the center and the points of a Cayley-Klein sphere are on \emph{opposite sides} of the quadric.
  \item If $\mu > 0$ the center and the points of a Cayley-Klein sphere are on the \emph{same side} of the quadric.
  \item If $\mu = 0$ the Cayley-Klein sphere is given by the (doubly counted) \emph{polar hyperplane} $\p{x}^\perp$.
  \item If $\mu = 1$ the Cayley-Klein sphere is the \emph{cone of contact} $\cone{\quadric}{\p{x}}$ touching $\quadric$,
    which is also called the \emph{null-sphere} with center $\p{x}$.
  \item If $\mu = \infty$ the Cayley-Klein sphere is the \emph{absolute quadric} $\quadric$.
  \end{itemize}
\end{proposition}
\begin{proof}
  Follows from Proposition \ref{prop:quadric-sides} and Lemma \ref{lem:cone-of-contact}.
\end{proof}
Fixing the center and varying the radius of a Cayley-Klein sphere results in a family of concentric spheres
(see Figure \ref{fig:concentric-cayley-klein-circles}).
\begin{definition}
  \label{def:cencentric-spheres}
  Given a quadric $\quadric \subset \RP^n$ and a point $\p{x} \in \RP^n \setminus \quadric$ we call the family
  \[
    \left( \cksphere{\p{x}}{\mu} \right)_{\mu \in \R\cup\{\infty\}}
  \]
  \emph{concentric Cayley-Klein spheres} with center $\p{x}$.
\end{definition}
\begin{proposition}
  \label{prop:concentric-sphere-pencil}
  Let $\quadric \subset \RP^n$ be a quadric.
  Then the family of concentric Cayley-Klein spheres with center $\p{x} \in \RP^n \setminus \quadric$
  is the pencil of quadrics $\quadric \wedge \cone{\quadric}{\p{x}}$
  spanned by the absolute quadric $\quadric$ and the cone of contact $\cone{\quadric}{\p{x}}$,
  or equivalently, by $\quadric$ and the (doubly counted) polar hyperplane $\p{x}^\perp$ (cf. Example \ref{ex:maximal-contact-pencil}).
\end{proposition}
\begin{proof}
  Writing the Cayley-Klein sphere equation as \eqref{eq:Cayley-Klein-sphere} 
  we find that it is a linear equation in $\mu$ describing a pencil of quadrics.
  As observed in Proposition \ref{prop:Cayley-Klein-spheres} it contains, in particular,
  the quadric $\quadric$, the cone $\cone{\quadric}{\p{x}}$, and the hyperplane $\p{x}^\perp$.
\end{proof}
This leads to a further characterization of Cayley-Klein spheres among all quadrics.
\begin{corollary}
  Let $\quadric \subset \RP^n$ be a non-degenerate quadric.
  Then another quadric is a Cayley-Klein sphere if and only if it is tangent to $\quadric$
  in the (possibly imaginary) intersection with a hyperplane.
\end{corollary}
\begin{proof}
  Follows from Proposition \ref{prop:concentric-sphere-pencil} and Example \ref{ex:maximal-contact-pencil}.
\end{proof}
\begin{remark}
  A pencil of concentric Cayley-Klein horospheres with center $\p{x}~\in~\quadric$
  is spanned by the absolute quadric $\quadric$ and the (doubly counted) tangent hyperplane $\p{x}^\perp$,
  which yields third order contact between each horosphere and the absolute quadric.
\end{remark}

\subsection{Polarity of Cayley-Klein spheres}
\begin{figure}
  \centering
  \def\svgwidth{0.44\textwidth}
%% Accompanies image file 'polar_spheres_point.pdf' (pdf, eps, ps)
\begingroup%
  \makeatletter%
  \providecommand\color[2][]{%
    \errmessage{(Inkscape) Color is used for the text in Inkscape, but the package 'color.sty' is not loaded}%
    \renewcommand\color[2][]{}%
  }%
  \providecommand\transparent[1]{%
    \errmessage{(Inkscape) Transparency is used (non-zero) for the text in Inkscape, but the package 'transparent.sty' is not loaded}%
    \renewcommand\transparent[1]{}%
  }%
  \providecommand\rotatebox[2]{#2}%
  \newcommand*\fsize{\dimexpr\f@size pt\relax}%
  \newcommand*\lineheight[1]{\fontsize{\fsize}{#1\fsize}\selectfont}%
  \ifx\svgwidth\undefined%
    \setlength{\unitlength}{2025.31640625bp}%
    \ifx\svgscale\undefined%
      \relax%
    \else%
      \setlength{\unitlength}{\unitlength * \real{\svgscale}}%
    \fi%
  \else%
    \setlength{\unitlength}{\svgwidth}%
  \fi%
  \global\let\svgwidth\undefined%
  \global\let\svgscale\undefined%
  \makeatother%
  \begin{picture}(1,1)%
    \lineheight{1}%
    \setlength\tabcolsep{0pt}%
    \put(0,0){\includegraphics[width=\unitlength,page=1]{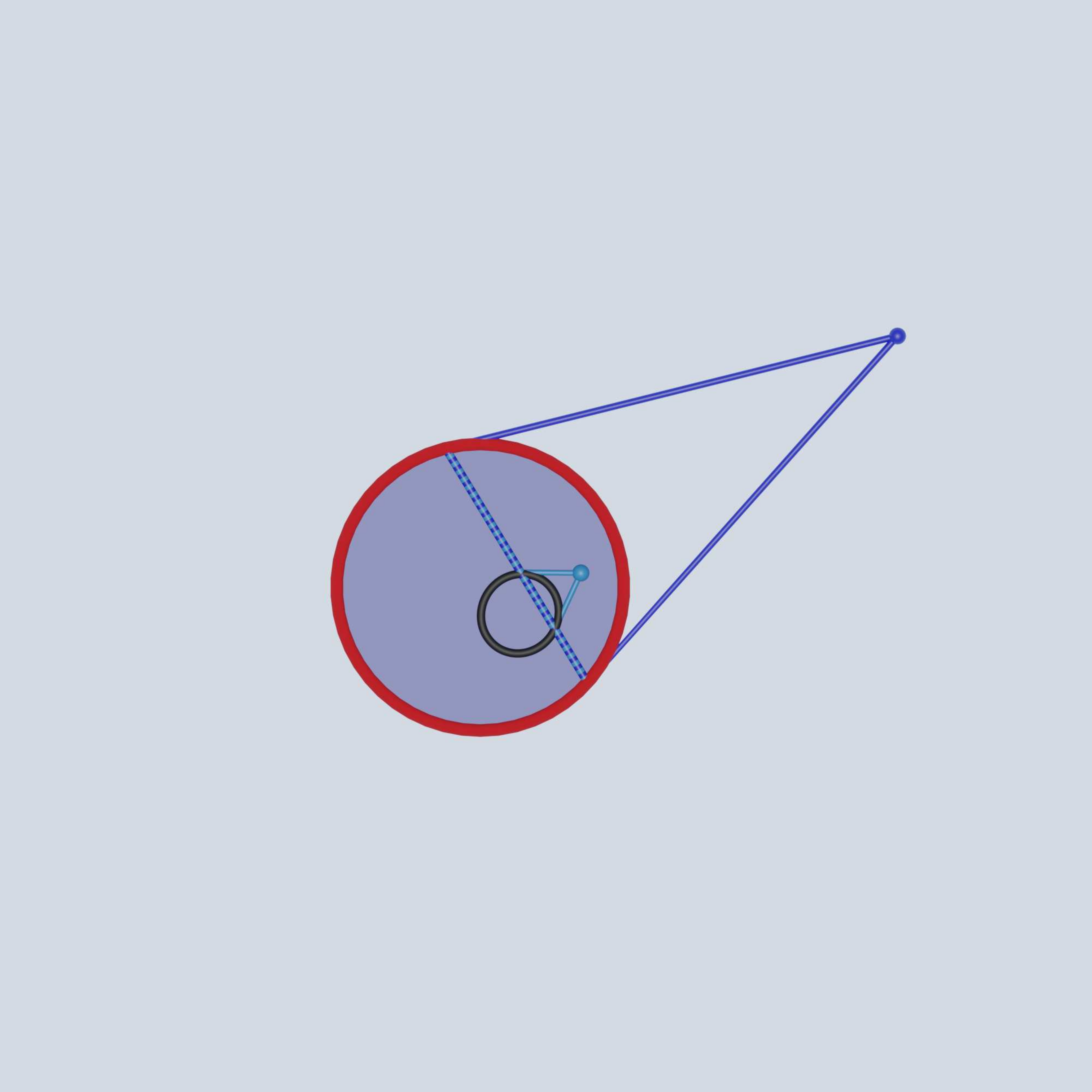}}%
    \put(0.83177569,0.69742991){\color[rgb]{0,0,0}\makebox(0,0)[lt]{\lineheight{1.25}\smash{\begin{tabular}[t]{l}{\footnotesize $\p{z}$}\end{tabular}}}}%
    \put(0.51546923,0.50039689){\color[rgb]{0,0,0}\makebox(0,0)[lt]{\lineheight{1.25}\smash{\begin{tabular}[t]{l}{\footnotesize $\p{y}$}\end{tabular}}}}%
    \put(0.34846635,0.29855397){\color[rgb]{0,0,0}\makebox(0,0)[lt]{\lineheight{1.25}\smash{\begin{tabular}[t]{l}{\footnotesize $\quadric$}\end{tabular}}}}%
    \put(0.32116137,0.45880734){\color[rgb]{0,0,0}\makebox(0,0)[lt]{\lineheight{1.25}\smash{\begin{tabular}[t]{l}{\footnotesize $\cksphere{\p{x}}{\mu}$}\end{tabular}}}}%
  \end{picture}%
\endgroup%

  \hspace{1cm}
  \def\svgwidth{0.44\textwidth}
%% Accompanies image file 'polar_spheres_line.pdf' (pdf, eps, ps)
\begingroup%
  \makeatletter%
  \providecommand\color[2][]{%
    \errmessage{(Inkscape) Color is used for the text in Inkscape, but the package 'color.sty' is not loaded}%
    \renewcommand\color[2][]{}%
  }%
  \providecommand\transparent[1]{%
    \errmessage{(Inkscape) Transparency is used (non-zero) for the text in Inkscape, but the package 'transparent.sty' is not loaded}%
    \renewcommand\transparent[1]{}%
  }%
  \providecommand\rotatebox[2]{#2}%
  \newcommand*\fsize{\dimexpr\f@size pt\relax}%
  \newcommand*\lineheight[1]{\fontsize{\fsize}{#1\fsize}\selectfont}%
  \ifx\svgwidth\undefined%
    \setlength{\unitlength}{2025.31640625bp}%
    \ifx\svgscale\undefined%
      \relax%
    \else%
      \setlength{\unitlength}{\unitlength * \real{\svgscale}}%
    \fi%
  \else%
    \setlength{\unitlength}{\svgwidth}%
  \fi%
  \global\let\svgwidth\undefined%
  \global\let\svgscale\undefined%
  \makeatother%
  \begin{picture}(1,1)%
    \lineheight{1}%
    \setlength\tabcolsep{0pt}%
    \put(0,0){\includegraphics[width=\unitlength,page=1]{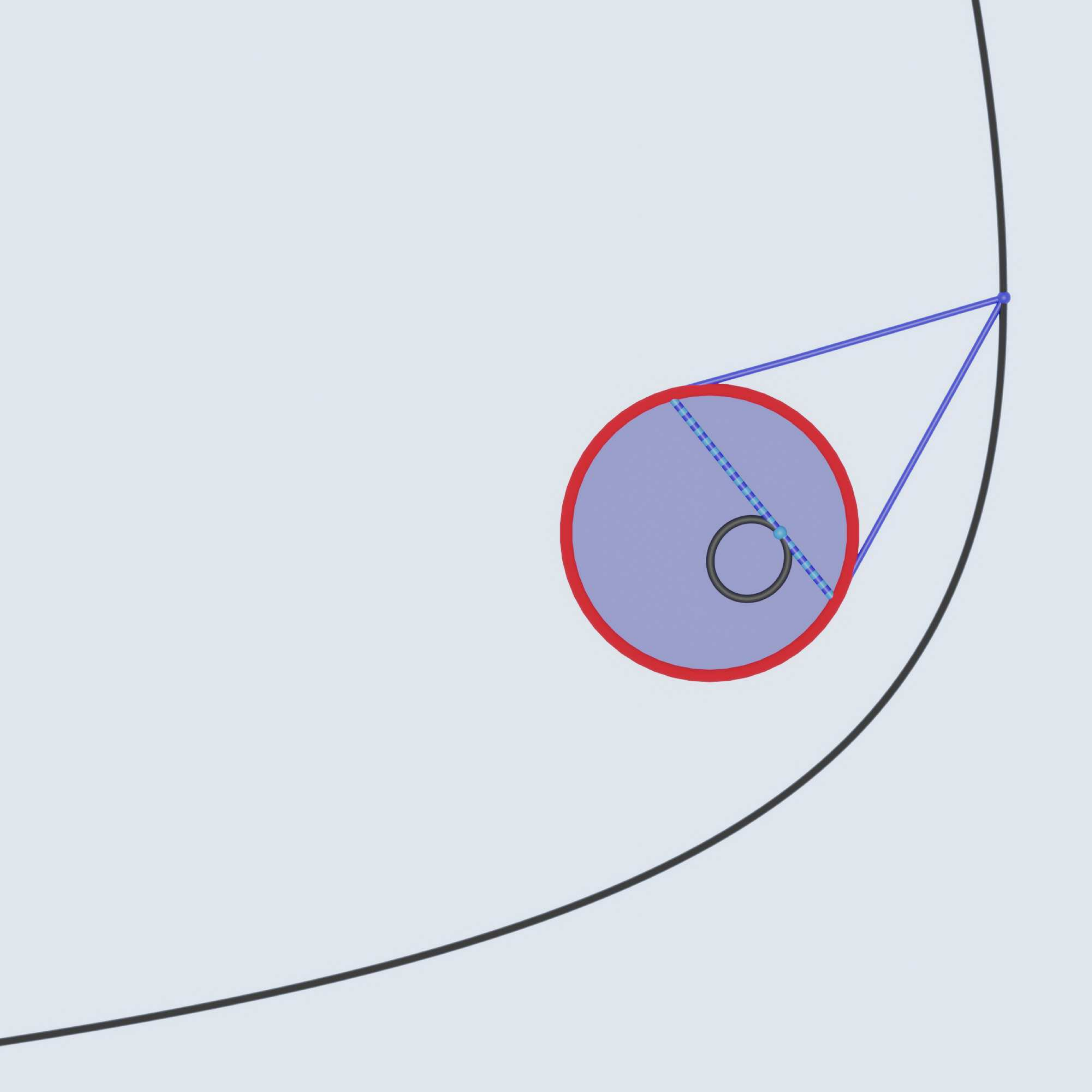}}%
    \put(0.52987227,0.48622236){\color[rgb]{0,0,0}\makebox(0,0)[lt]{\lineheight{1.25}\smash{\begin{tabular}[t]{l}{\footnotesize $\cksphere{\p{x}}{\mu}$}\end{tabular}}}}%
    \put(0.7095458,0.53872831){\color[rgb]{0,0,0}\makebox(0,0)[lt]{\lineheight{1.25}\smash{\begin{tabular}[t]{l}{\footnotesize $\p{y}$}\end{tabular}}}}%
    \put(0.94109093,0.72727245){\color[rgb]{0,0,0}\makebox(0,0)[lt]{\lineheight{1.25}\smash{\begin{tabular}[t]{l}{\footnotesize $\p{z}$}\end{tabular}}}}%
    \put(0.5786707,0.34183095){\color[rgb]{0,0,0}\makebox(0,0)[lt]{\lineheight{1.25}\smash{\begin{tabular}[t]{l}{\footnotesize $\quadric$}\end{tabular}}}}%
    \put(0.60163605,0.15550865){\color[rgb]{0,0,0}\makebox(0,0)[lt]{\lineheight{1.25}\smash{\begin{tabular}[t]{l}{\footnotesize $\cksphere{\p{x}}{\tilde{\mu}}$}\end{tabular}}}}%
  \end{picture}%
\endgroup%

  \caption{
    \emph{Left:} Polarity with respect to a Cayley-Klein sphere $\cksphere{\p{x}}{\mu}$ and the absolute quadric $\quadric$.
    \emph{Right:} A Cayley-Klein sphere $\cksphere{\p{x}}{\mu}$ and its (concentric) polar Cayley-Klein sphere $\cksphere{\p{x}}{\tilde{\mu}}$.
  }
\label{fig:polarity-cayley-klein-circles}
\end{figure}

To describe spheres in terms of their tangent planes we turn our attention towards polarity in Cayley-Klein spheres
(see Figure \ref{fig:polarity-cayley-klein-circles}).
\begin{lemma}
  \label{lem:Cayley-Klein-sphere-polarity}
  The bilinear form corresponding to a Cayley-Klein sphere $\cksphere{\p{x}}{\mu}$ with center ${\p{x}\in\RP^n\setminus \quadric}$
  and Cayley-Klein radius $\mu$ is given by
  \[
    b(y,\tilde{y}) = \scalarprod{x}{y}\scalarprod{x}{\tilde{y}} - \mu\scalarprod{x}{x}\scalarprod{y}{\tilde{y}},
    \qquad
    y, \tilde{y} \in \R^{n+1}.
  \]
  Thus, for a point $\p{y} \in \RP^n$ the pole $\p{z}$ with respect to the absolute quadric $\quadric$
  of the polar hyperplane of $\p{y}$ with respect to $\cksphere{\p{x}}{\mu}$ is given by
  \begin{equation}
    \label{eq:Cayley-Klein-sphere-pole}
    z = \scalarprod{x}{y}x - \mu\scalarprod{x}{x}y.
  \end{equation}
\end{lemma}
\begin{proof}
  The quadratic form of the Cayley-Klein sphere $\cksphere{\p{x}}{\mu}$ is given by \eqref{eq:Cayley-Klein-sphere}:
  \[
    \Delta(y) \coloneqq \scalarprod{x}{y}^2 - \mu\scalarprod{x}{x}\scalarprod{y}{y}.
  \]
  The corresponding bilinear form can be obtained by
  $b(y, \tilde{y}) = \frac{1}{2}(\Delta(y+\tilde{y}) - \Delta(y) - \Delta(\tilde{y}))$.
\end{proof}
For every point on a Cayley-Klein sphere the tangent hyperplane in that point is given by polarity in the Cayley-Klein sphere.
Now the tangent hyperplanes of a Cayley-Klein sphere, in turn, may equivalently be described
by their poles with respect to the absolute quadric $\quadric$ (see Figure~\ref{fig:polarity-cayley-klein-circles}).
\begin{proposition}
  Let $\p{x} \in \RP^n \setminus \quadric$ and $\mu \in \R\setminus\{0,1\}$.
  Then the poles (with respect to the absolute quadric $\quadric$) of the tangent hyperplanes of the Cayley-Klein sphere $\cksphere{\p{x}}{\mu}$
  are the points of a concentric Cayley-Klein sphere $\cksphere{\p{x}}{\tilde{\mu}}$ with
  \[
    \mu + \tilde{\mu} = 1,
  \]
  and vice versa.
\end{proposition}
\begin{proof}
  Let $\p{y} \in \cksphere{\p{x}}{\mu}$ be a point on the Cayley-Klein sphere.
  Then the tangent plane to $\cksphere{\p{x}}{\mu}$ at the point $\p{y}$ is the polar plane of $\p{y}$ with respect to $\cksphere{\p{x}}{\mu}$.
  According to Lemma \ref{lem:Cayley-Klein-sphere-polarity} the pole $\p{z}$
  of that tangent plane is given by \eqref{eq:Cayley-Klein-sphere-pole}.
  Computing the Cayley-Klein distance of this point to the center $\p{x}$ we obtain
  \[
    \ck{\quadric}{\p{x}}{\p{y}}
    = \frac{\scalarprod{x}{z}^2}{\scalarprod{x}{x}\scalarprod{z}{z}}
    = \frac{\scalarprod{x}{y}^2(1 - \mu)}{\scalarprod{x}{y}^2(1-2\mu) + \mu^2\scalarprod{x}{x}\scalarprod{y}{y}}
    = 1 - \mu,
  \]
  where we used $\scalarprod{x}{y}^2 = \mu\scalarprod{x}{x}\scalarprod{y}{y}$.
\end{proof}
\begin{definition}
  \label{def:polar-Cayley-Klein-sphere}
  For a Cayley-Klein sphere $\cksphere{\p{x}}{\mu}$ we call the Cayley-Klein sphere $\cksphere{\p{x}}{1-\mu}$,
  consisting of all poles (with respect to the absolute quadric $\quadric$) of tangent planes of $\cksphere{\p{x}}{\mu}$,
  its \emph{polar Cayley-Klein sphere}.
\end{definition}
\begin{remark}
  The two degenerate Cayley-Klein spheres $\p{x}^\perp$ and $\cone{\quadric}{\p{x}}$ corresponding
  to the values $\mu = 0$ and $\mu = 1$ respectively, may be treated as being mutually polar.
  Then polarity defines a projective involution on a pencil of concentric Cayley-Klein spheres
  with fixed points at $\mu = \frac{1}{2}$ and $\mu = \infty$.
\end{remark}

\subsection{Hyperbolic geometry}
\label{sec:hyperbolic-space}
Let $\scalarprod{\cdot}{\cdot}$ be the standard non-degenerate bilinear form of signature $(n,1)$, i.e.
\[
  \scalarprod{x}{y} \coloneqq x_1y_1 + \ldots + x_ny_n - x_{n+1}y_{n+1}
\]
for $x, y \in \R^{n+1}$, and denote by $\mob \subset \RP^n$ the corresponding quadric.
We identify the ``inside'' of $\mob$, cf.\ \eqref{eq:quadric-sides}, with the $n$-dimensional \emph{hyperbolic space}
\[
  \hyp \coloneqq \mob^-.
\]
For two points $\p{x}, \p{y} \in \hyp$ one has $\ck{\mob}{\p{x}}{\p{y}} \geq 1$,
and the quantity $d$ given by
\[
  \ck{\mob}{\p{x}}{\p{y}} = \cosh^2 d(\p{x}, \p{y})
\]
defines a metric on $\hyp$ of constant negative sectional curvature.
The corresponding group of isometries is given by $\PO(n,1)$ and called the group of \emph{hyperbolic motions}.
The absolute quadric $\mob$ consists of the points at (metric) infinity.
We call the union
\[
  \chyp \coloneqq \hyp \cup \mob
\]
the \emph{compactified hyperbolic space}.

In this \emph{projective model} of hyperbolic geometry \emph{geodesics} are given by intersections
of projective lines in $\RP^n$ with $\hyp$,
while, more generally, \emph{hyperbolic subspaces} (totally geodesic submanifolds) are given by intersections
of projective subspaces in $\RP^n$ with $\hyp$.
Thus, by polarity, every point $\p{m} \in \ds$ in the ``outside'' of hyperbolic space,
\[
  \ds \coloneqq \mob^+,
\]
which is called \emph{deSitter space}, corresponds to a hyperbolic hyperplane $\p{m}^\perp \cap \hyp$.

Consider two hyperbolic hyperplanes with poles $\p{k}, \p{m} \in \ds$.
\begin{itemize}
\item If $\ck{\mob}{\p{k}}{\p{m}} < 1$, the two hyperplanes intersect in $\hyp$,
  and their hyperbolic intersection angle~$\alpha$, or equivalently its conjugate angle $\pi - \alpha$ is given by
  \[
    \ck{\mob}{\p{k}}{\p{m}} = \cos^2 \alpha(\p{k}^\perp, \p{m}^\perp).
  \]
\item If $\ck{\mob}{\p{k}}{\p{m}} > 1$, the two hyperplanes do not intersect in $\hyp$,
  and their hyperbolic distance is given by
  \[
    \ck{\mob}{\p{k}}{\p{m}} = \cosh^2 d(\p{k}^\perp, \p{m}^\perp).
  \]
  The corresponding projective hyperplanes intersect in $(\p{k}\wedge\p{m})^\perp \subset \ds$.
\item If $\ck{\mob}{\p{k}}{\p{m}} = 0$, the two hyperplanes are parallel, i.e., they intersect on $\mob$.
\end{itemize}
Finally, the hyperbolic distance of a point $\p{x} \in \hyp$ and a hyperbolic hyperplane with pole $\p{m} \in \ds$ is given by
\[
  \ck{\mob}{\p{x}}{\p{m}} = -\sinh^2 d(\p{x}, \p{m}^\perp).
\]

It is occasionally useful to employ a certain normalization of the homogeneous coordinate vectors:
\[
  \begin{aligned}
    \HH^n &\coloneqq \set{x = (x_1, \ldots, x_{n+1})\in \R^{n,1}}{\scalarprod{x}{x} = -1,\, x_{n+1} \geq 0},\\
    \widetilde{\dS}^n &\coloneqq \set{m = (m_1, \ldots, m_{n+1})\in \R^{n,1}}{\scalarprod{m}{m} = 1}.
  \end{aligned}
\]
Then $\P(\HH^n) = \hyp$ is an embedding and $\P(\widetilde{\dS}^n) = \ds$ is a double cover.
For $x, y \in \HH^n$ and $k, m \in \widetilde{\dS}^n$ above distance formulas become
\[
  \begin{aligned}
    \scalarprod{x}{y} &= - \cosh d(\p{x}, \p{y}),\\
    \abs{\scalarprod{k}{m}} &= \cos \alpha(\p{k}^\perp, \p{m}^\perp), &&\text{if}\, \abs{\scalarprod{k}{m}} \leq 1  \\
    \abs{\scalarprod{k}{m}} &= \cosh d(\p{k}^\perp, \p{m}^\perp), &&\text{if}\, \abs{\scalarprod{k}{m}} \geq 1  \\
    \abs{\scalarprod{x}{m}} &= \sinh d(\p{x}, \p{m}^\perp).
  \end{aligned}
\]
\begin{remark}
  \label{rem:oriented-hyperbolic-planes}
  The double cover $\P(\widetilde{\dS}^n) = \ds$ of deSitter space can be used to encode the orientation of the corresponding polar hyperplanes,
  e.g., by endowing the hyperbolic hyperplane corresponding to $m \in \widetilde{\dS}^n$
  with a normal vector in the direction of the hyperbolic halfspace on which the bilinear form
  with points $x\in \HH^n$ is positive: $\scalarprod{x}{m} > 0$.
  Using the double cover to encode orientation
  one may omit the absolute value in $\scalarprod{x}{m} = \cos d$
  to obtain an oriented hyperbolic distance $d$ between a point and an hyperbolic hyperplane.
  Similarly, one may omit the absolute value in $\scalarprod{k}{m} = \cos\alpha$
  which allows to distinguish the intersection angle $\alpha$ and its conjugate angle $\pi - \alpha$.
\end{remark}

We now turn our attention to the Cayley-Klein spheres of hyperbolic/deSitter geometry.
First, consider a pencil of concentric Cayley-Klein spheres $\cksphere{\p{x}}{\mu}$
with center inside hyperbolic space $\p{x} \in \hyp$, $x \in \HH^n$.
Depending on the value of $\mu \in \R \cup \{ \infty \}$ we obtain the following types of hyperbolic/deSitter spheres
(see Figure \ref{fig:concentric-cayley-klein-circles}, left):
\begin{itemize}
\item $\mu < 0$: A \emph{deSitter sphere} with hyperbolic center.
\item $0<\mu<1$: $\cksphere{\p{x}}{\mu}$ is empty.
\item $1<\mu<\infty$: A \emph{hyperbolic sphere} with center $\p{x} \in \hyp$ and hyperbolic radius $r = \arcosh\sqrt{\mu} > 0$:
  \[
    \cksphere{\p{x}}{\mu} =
    \set{\p{y} \in \hyp}{ \ck{\mob}{\p{x}}{\p{y}} = \cosh^2 r } =
    \P\left( \set{y \in \HH^n}{ \scalarprod{x}{y} = - \cosh r } \right).
  \]
\end{itemize}
Second, consider a pencil of concentric Cayley-Klein spheres $\cksphere{\p{m}}{\mu}$
with center outside hyperbolic space $\p{m} \in \ds$, $m \in \widetilde{\dS}^n$
(see Figure \ref{fig:concentric-cayley-klein-circles}, middle):
\begin{itemize}
\item $\mu < 0$: A \emph{hypersurface of constant hyperbolic distance} $r = \arsinh\sqrt{\mu} > 0$ to the hyperbolic plane $\p{m}^\perp \cap \hyp$:
  \[
    \cksphere{\p{m}}{\mu}
    = \set{\p{y} \in \hyp}{ \ck{\mob}{\p{m}}{\p{y}} = -\sinh^2r}
    = \P\left( \set{y \in \HH^N}{ \abs{\scalarprod{m}{y}} = \sinh r} \right).
  \]
\item $0<\mu<1$: A \emph{deSitter sphere} tangent to $\mob$.
  All its tangent hyperplanes are hyperbolic hyperplanes.
\item $1<\mu<\infty$: A \emph{deSitter sphere} tangent to $\mob$
  with no hyperbolic tangent hyperplanes.
\end{itemize}
Third, a pencil of concentric Cayley-Klein horospheres with center on the absolute quadric $\p{x} \in \mob$, $x \in \L^{n,1}$
consists of \emph{hyperbolic horospheres} and \emph{deSitter horospheres} (see Figure \ref{fig:concentric-cayley-klein-circles}, right).

\subsection{Elliptic geometry}
\begin{figure}
  \centering
  \includegraphics[width=0.44\textwidth]{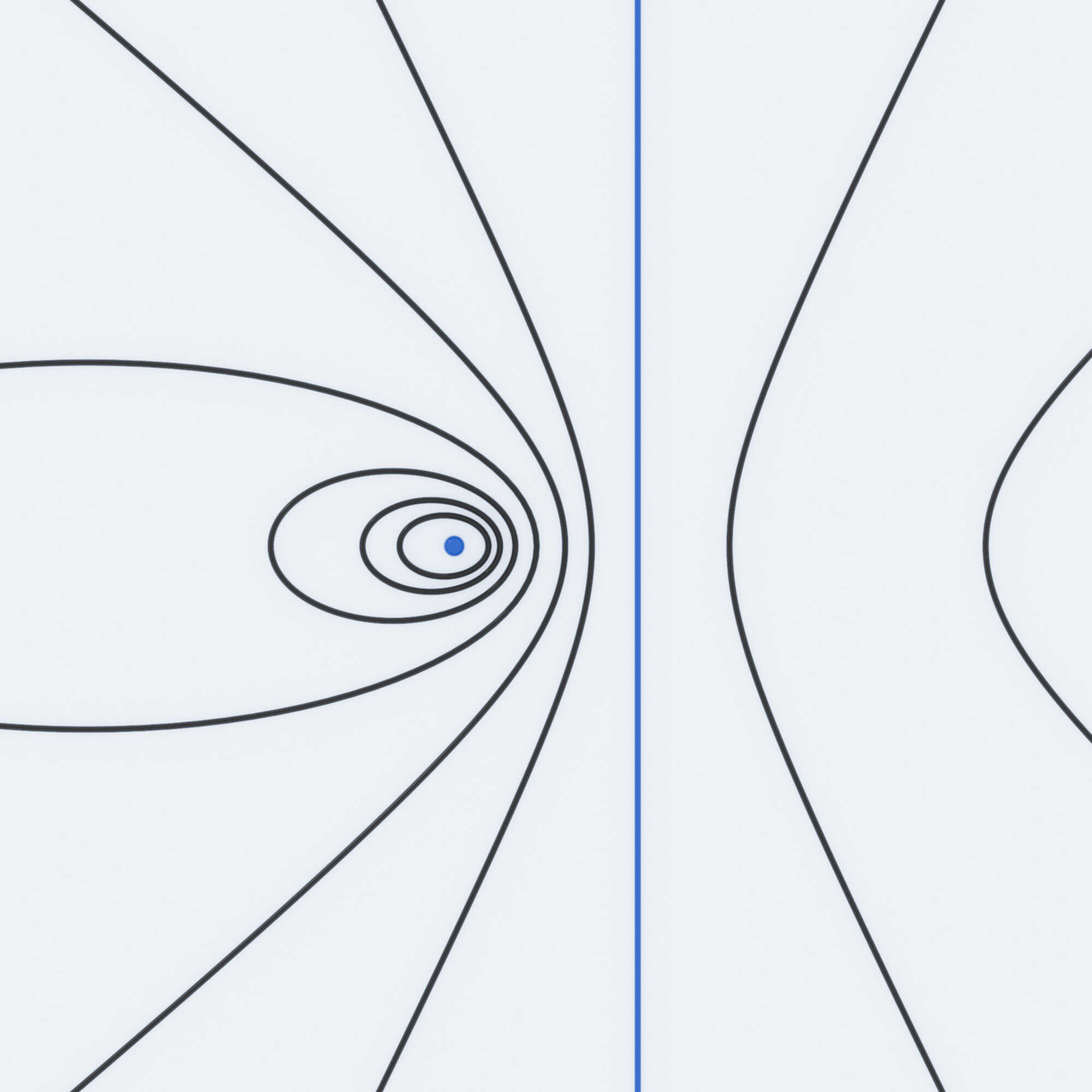}
  \caption{
    Concentric Cayley-Klein circles in the elliptic plane.
  }
\label{fig:concentric-elliptic-circles}
\end{figure}
\label{sec:elliptic-space}
For $x, y \in \R^{n+1}$ we denote by 
\[
  \dotprod{x}{y} \coloneqq x_1y_1 + \ldots x_ny_n + x_{n+1}y_{n+1}
\]
the standard (positive definite) scalar product on $\R^{n+1}$, i.e. the standard non-degenerate bilinear form of signature $(n+1,0)$.
The corresponding quadric $\ellipb \subset \RP^n$ is empty (or purely imaginary, cf.\ Example~\ref{ex:quadrics}~\ref{ex:quadrics-empty}),
as well as the set $\ellipb^- = \varnothing$, while
\[
  \ellip \coloneqq \ellipb^+ = \RP^n
\]
is the whole projective space, which we identify with the $n$-dimensional \emph{elliptic space}.
For two points $\p{x}, \p{y} \in \ellip$ one always has $0 \leq \ck{\ellipb}{\p{x}}{\p{y}} \leq 1$ and the quantity $d$ given by
\[
  \ck{\ellipb}{\p{x}}{\p{y}} = \cos^2 d(\p{x}, \p{y})
\]
defines a metric on $\ellip$ of constant positive sectional curvature.
The corresponding group of isometries is given by $\PO(n+1)$ and called the group of \emph{elliptic motions}.

In this \emph{projective model} of elliptic geometry \emph{geodesics} are given by projective lines,
while, more generally, \emph{elliptic subspaces} are given by projective subspaces.
By polarity, there is a one-to-one correspondence of points $\p{x} \in \ellip$ in elliptic space
and elliptic hyperplanes $\p{x}^\perp$.

Two hyperplanes in elliptic space always intersect.
If $\p{x}, \p{y} \in \ellip$ are the poles of two elliptic hyperplanes,
then their intersection angle $\alpha$, or equivalently its conjugate angle $\pi - \alpha$ is given by
\[
  \ck{\ellipb}{\p{x}}{\p{y}} = \cos^2 \alpha(\p{x}^\perp, \p{y}^\perp).
\]
The distance of a point $\p{x} \in \RP^n$ and an elliptic hyperplane with pole $\p{y} \in \RP^n$ is given by
\[
  \ck{\ellipb}{\p{x}}{\p{y}} = \sin^2 \alpha(\p{x}, \p{y}^\perp).
\]

One may normalize the homogeneous coordinate vectors of points in elliptic space to lie on a sphere:
\[
  \S^n \coloneqq \set{x \in \R^{n+1}}{\dotprod{x}{x} = 1}.
\]
Then $\P(\S^n) = \ellip$ is a double cover, where antipodal points of the sphere are identified.
In this normalization elliptic planes correspond to great spheres of $\S^n$,
and it turns out that elliptic geometry is a double cover of \emph{spherical geometry}.
For $x, y \in \S^n$ above distance formulas become
\[
  \begin{aligned}
    \abs{\dotprod{x}{y}} &= \cos d(\p{x}, \p{y}),\\
    \abs{\dotprod{x}{y}} &= \cos \alpha(\p{x}^\perp, \p{y}^\perp),\\
    \abs{\dotprod{x}{y}} &= \sin d(\p{x}, \p{y}^\perp),
  \end{aligned}
\]
\begin{remark}
  The pole $\p{x} \in \ellip$ of an elliptic hyperplane $\p{x}^\perp$ has two lifts to the sphere, $x, -x \in \S^n$,
  which may be used to encode the orientation of the hyperplane (cf.\ Remark \ref{rem:oriented-hyperbolic-planes}).
  This allows for omitting the absolute values in above distance formulas,
  while taking distances to be signed and distinguishing between intersection angles and their conjugate angles.
\end{remark}

A Cayley-Klein sphere in elliptic space $\cksphere{\p{x}}{\mu}$ with center $\p{x} \in \ellip$, $x \in \S^n$,
is not empty if and only if $0 \leq \mu \leq 1$ (see Figure \ref{fig:concentric-elliptic-circles}).
In this case it corresponds to an \emph{elliptic sphere} with center $\p{x} \in \ellip$ and elliptic radius $0 \leq r = \arccos\sqrt\mu \leq \frac{\pi}{2}$:
\[
  \cksphere{\p{x}}{\mu} =
  \set{\p{y} \in \ellip}{ \ck{\ellipb}{\p{x}}{\p{y}} = \cos^2 r } =
  \P\left( \set{y \in \S^n}{ \abs{\dotprod{x}{y}} = \cos r } \right).
\]

%%% Local Variables:
%%% mode: latex
%%% TeX-master: "main"
%%% End:

\newpage
\section{Central projection of quadrics and Möbius geometry}
\label{sec:projection}

In this section we study the general construction of \emph{central projection} of a quadric from a point onto its polar hyperplane, see, e.g., \cite{K, Blproj, Gie}.
This leads to a double cover of a Cayley-Klein space in the hyperplane such that the spheres in that Cayley-Klein space correspond to hyperplanar sections of the quadric.
Vice versa, a Cayley-Klein space can be lifted to a quadric in a projective space of one dimension higher,
such that Cayley-Klein spheres lift to hyperplanar sections of the quadric.
In this way, hyperbolic and elliptic geometry can be lifted to \emph{Möbius geometry},
and Möbius geometry may be seen as the geometry of points and spheres of the hyperbolic or elliptic space, respectively.
We demonstrate how the group of \emph{Möbius transformations} can be decomposed into the respective isometries
and scalings along concentric spheres.

\subsection{The involution and projection induced by a point}
\begin{figure}
  \centering
  \def\svgwidth{0.44\textwidth}
%% Accompanies image file 'involution_hyp.pdf' (pdf, eps, ps)
\begingroup%
  \makeatletter%
  \providecommand\color[2][]{%
    \errmessage{(Inkscape) Color is used for the text in Inkscape, but the package 'color.sty' is not loaded}%
    \renewcommand\color[2][]{}%
  }%
  \providecommand\transparent[1]{%
    \errmessage{(Inkscape) Transparency is used (non-zero) for the text in Inkscape, but the package 'transparent.sty' is not loaded}%
    \renewcommand\transparent[1]{}%
  }%
  \providecommand\rotatebox[2]{#2}%
  \newcommand*\fsize{\dimexpr\f@size pt\relax}%
  \newcommand*\lineheight[1]{\fontsize{\fsize}{#1\fsize}\selectfont}%
  \ifx\svgwidth\undefined%
    \setlength{\unitlength}{2025.31640625bp}%
    \ifx\svgscale\undefined%
      \relax%
    \else%
      \setlength{\unitlength}{\unitlength * \real{\svgscale}}%
    \fi%
  \else%
    \setlength{\unitlength}{\svgwidth}%
  \fi%
  \global\let\svgwidth\undefined%
  \global\let\svgscale\undefined%
  \makeatother%
  \begin{picture}(1,1)%
    \lineheight{1}%
    \setlength\tabcolsep{0pt}%
    \put(0,0){\includegraphics[width=\unitlength,page=1]{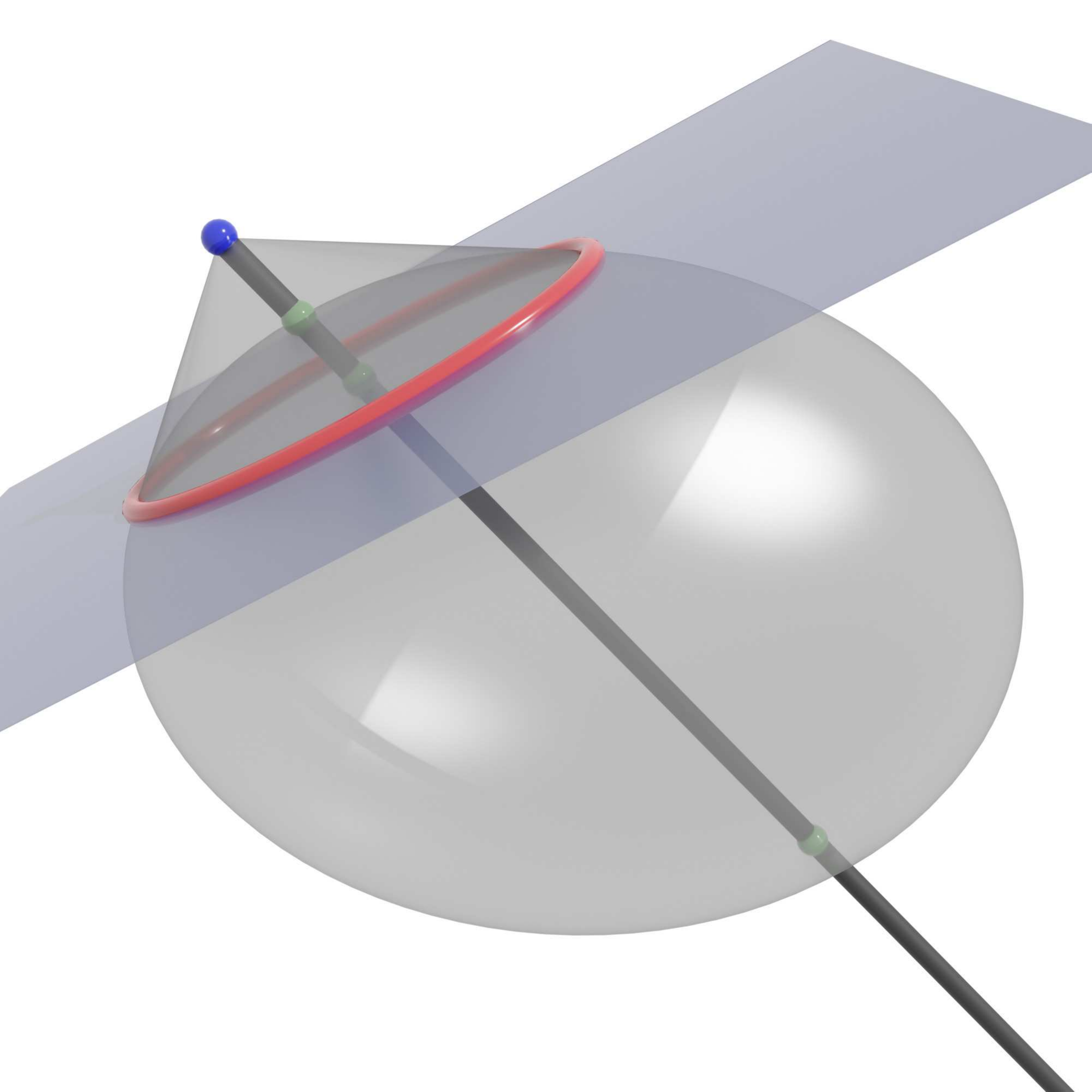}}%
    \put(0.81503577,0.84367542){\color[rgb]{0,0,0}\makebox(0,0)[lt]{\lineheight{1.25}\smash{\begin{tabular}[t]{l}{\footnotesize $\p{q}^{\perp}$}\end{tabular}}}}%
    \put(0.16467781,0.81742244){\color[rgb]{0,0,0}\makebox(0,0)[lt]{\lineheight{1.25}\smash{\begin{tabular}[t]{l}{\footnotesize $\p{q}$}\end{tabular}}}}%
    \put(0.29236274,0.74701671){\color[rgb]{0,0,0}\makebox(0,0)[lt]{\lineheight{1.25}\smash{\begin{tabular}[t]{l}{\footnotesize $\p{x}$}\end{tabular}}}}%
    \put(0.35560861,0.68257756){\color[rgb]{0,0,0}\makebox(0,0)[lt]{\lineheight{1.25}\smash{\begin{tabular}[t]{l}{\footnotesize $\pi_{\p{q}}(\p{x})$}\end{tabular}}}}%
    \put(0.76911539,0.26187866){\color[rgb]{0,0,0}\makebox(0,0)[lt]{\lineheight{1.25}\smash{\begin{tabular}[t]{l}{\footnotesize $\sigma_{\p{q}}(\p{x})$}\end{tabular}}}}%
    \put(0.37112171,0.29832935){\color[rgb]{0,0,0}\makebox(0,0)[lt]{\lineheight{1.25}\smash{\begin{tabular}[t]{l}{\footnotesize $\quadric$}\end{tabular}}}}%
  \end{picture}%
\endgroup%

  \hspace{0.08\textwidth}
  \def\svgwidth{0.44\textwidth}
%% Accompanies image file 'involution_ell.pdf' (pdf, eps, ps)
\begingroup%
  \makeatletter%
  \providecommand\color[2][]{%
    \errmessage{(Inkscape) Color is used for the text in Inkscape, but the package 'color.sty' is not loaded}%
    \renewcommand\color[2][]{}%
  }%
  \providecommand\transparent[1]{%
    \errmessage{(Inkscape) Transparency is used (non-zero) for the text in Inkscape, but the package 'transparent.sty' is not loaded}%
    \renewcommand\transparent[1]{}%
  }%
  \providecommand\rotatebox[2]{#2}%
  \newcommand*\fsize{\dimexpr\f@size pt\relax}%
  \newcommand*\lineheight[1]{\fontsize{\fsize}{#1\fsize}\selectfont}%
  \ifx\svgwidth\undefined%
    \setlength{\unitlength}{2025.31640625bp}%
    \ifx\svgscale\undefined%
      \relax%
    \else%
      \setlength{\unitlength}{\unitlength * \real{\svgscale}}%
    \fi%
  \else%
    \setlength{\unitlength}{\svgwidth}%
  \fi%
  \global\let\svgwidth\undefined%
  \global\let\svgscale\undefined%
  \makeatother%
  \begin{picture}(1,1)%
    \lineheight{1}%
    \setlength\tabcolsep{0pt}%
    \put(0,0){\includegraphics[width=\unitlength,page=1]{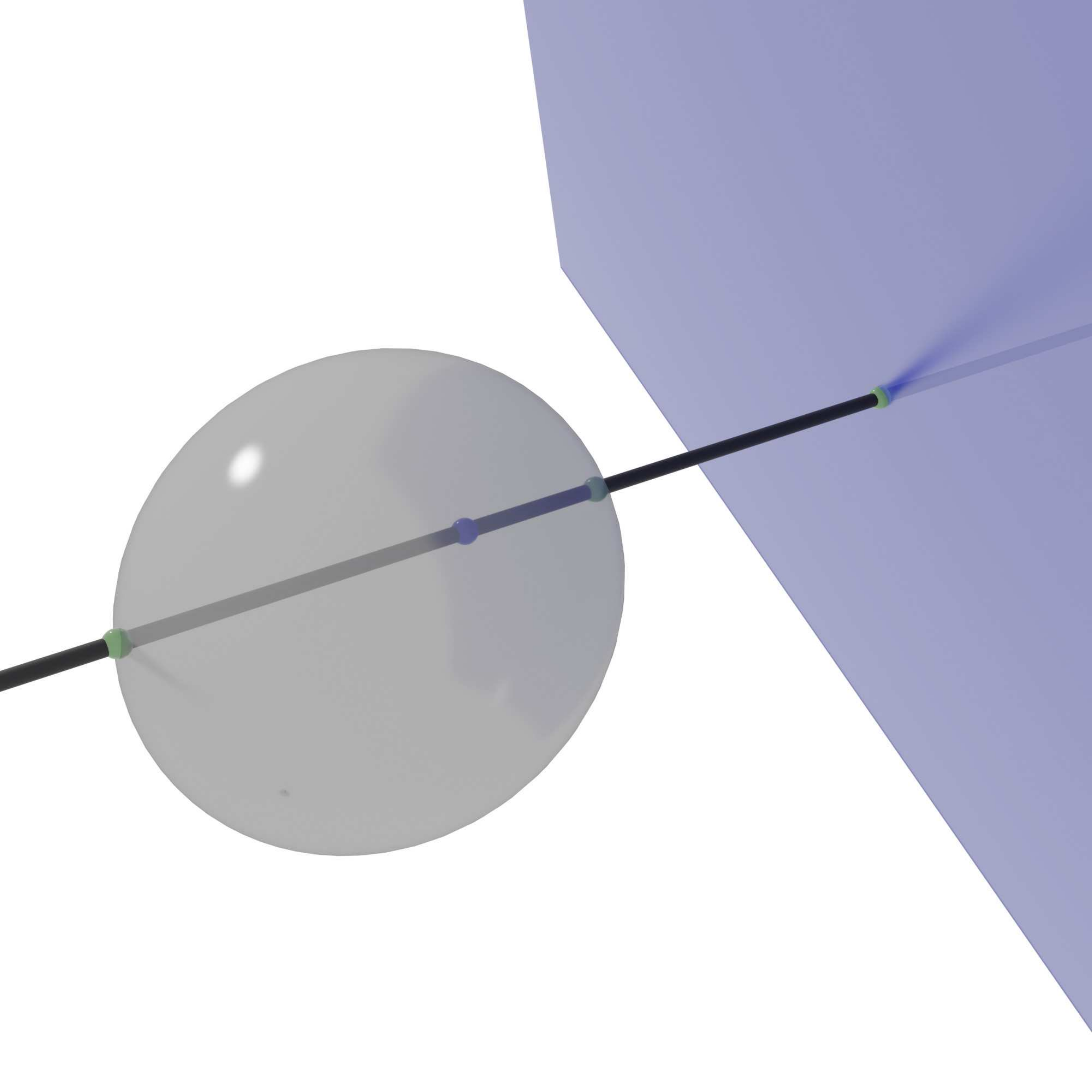}}%
    \put(0.4117023,0.44762505){\color[rgb]{0,0,0}\makebox(0,0)[lt]{\lineheight{1.25}\smash{\begin{tabular}[t]{l}{\footnotesize $\p{q}$}\end{tabular}}}}%
    \put(0.54367464,0.49576889){\color[rgb]{0,0,0}\makebox(0,0)[lt]{\lineheight{1.25}\smash{\begin{tabular}[t]{l}{\footnotesize $\p{x}$}\end{tabular}}}}%
    \put(0.77150199,0.58288107){\color[rgb]{0,0,0}\makebox(0,0)[lt]{\lineheight{1.25}\smash{\begin{tabular}[t]{l}{\footnotesize $\pi_{\p{q}}(\p{x})$}\end{tabular}}}}%
    \put(0.08208191,0.34660419){\color[rgb]{0,0,0}\makebox(0,0)[lt]{\lineheight{1.25}\smash{\begin{tabular}[t]{l}{\footnotesize $\sigma_{\p{q}}(\p{x})$}\end{tabular}}}}%
    \put(0.55970697,0.93490971){\color[rgb]{0,0,0}\makebox(0,0)[lt]{\lineheight{1.25}\smash{\begin{tabular}[t]{l}{\footnotesize $\p{q}^{\perp}$}\end{tabular}}}}%
    \put(0.30778765,0.23323902){\color[rgb]{0,0,0}\makebox(0,0)[lt]{\lineheight{1.25}\smash{\begin{tabular}[t]{l}{\footnotesize $\quadric$}\end{tabular}}}}%
  \end{picture}%
\endgroup%

  \caption{
    The involution and projection of an oval quadric $\quadric \subset \RP^3$ induced by a point $\p{q}$ not on the quadric.
    \emph{Left:} The point $\p{q}$ lies ``outside'' the quadric.
    \emph{Right:} The point $\p{q}$ lies ``inside'' the quadric.
  }
\label{fig:involution}
\end{figure}
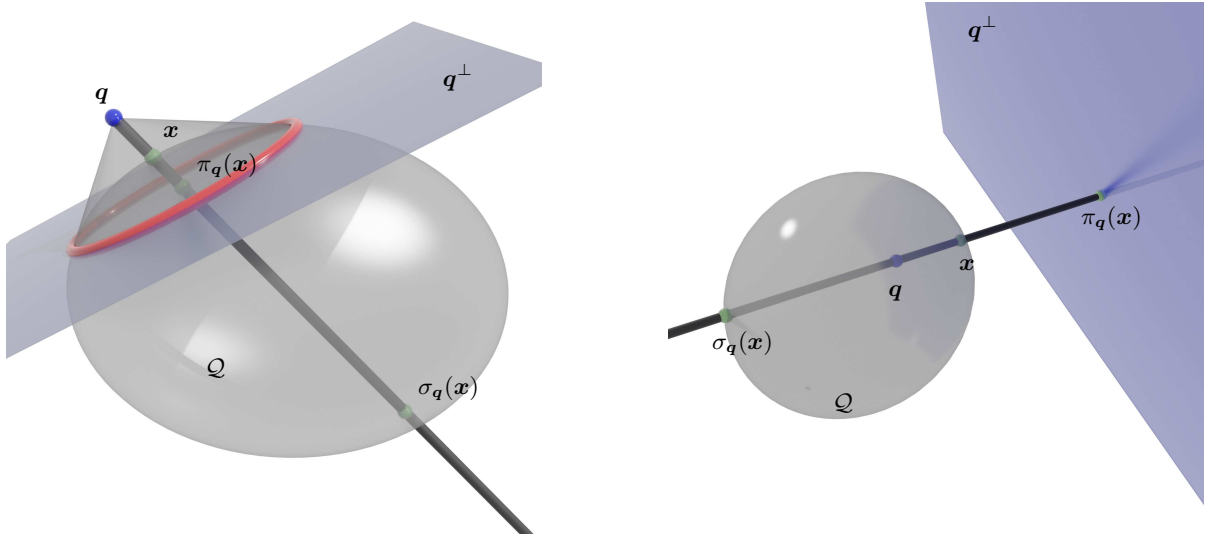
Let $\scalarprod{\cdot}{\cdot}$ be a bilinear form on $\R^{n+2}$ of signature $(r,s,t)$,
and denote by $\quadric \subset \RP^{n+1}$ the corresponding quadric.
We introduce the central projection of $\quadric$ from a point $\p{q}$ not on the quadric onto a hyperplane of $\RP^{n+1}$
which is canonically chosen to be the polar hyperplane of $\p{q}$.
\begin{definition}
  \label{def:involution-projection}
  A point $\p{q} \in \RP^{n+1} \setminus \quadric$ not on the quadric
  induces two maps
  \[
    \sigma_{\p{q}}, \pi_{\p{q}} \colon \RP^{n+1} \rightarrow \RP^{n+1}
  \]
  \begin{equation}
    \sigma_{\p{q}} : [x] \mapsto \left[ \sigma_{q}(x) \right] = \left[ x - 2 \frac{\scalarprod{x}{q}}{\scalarprod{q}{q}} q \right],\qquad
    \pi_{\p{q}} : [x] \mapsto \left[ \pi_q(x) \right] = \left[ x - \frac{\scalarprod{x}{q}}{\scalarprod{q}{q}} q \right],
  \end{equation}
  which we call the \emph{associated involution} and \emph{projection} respectively.
\end{definition}
\begin{remark}
  The involution $\sigma_{\p{q}}$ is also called \emph{reflection in the hyperplane $\p{q}^\perp$} (cf.\ Theorem~\ref{thm:cartan}).
\end{remark}
We summarize the main properties of this involution and projection in the following proposition.
\begin{proposition}\
  \label{prop:involution-projection}
  \nobreakpar
  \begin{enumerate}
  \item The map $\sigma_{\p{q}}$ is a projective involution that fixes $\p{q}$, i.e.,
    \[
      \sigma_{\p{q}} \in \PO(r,s,t)_{\p{q}},
      \qquad
      \sigma_{\p{q}} \circ \sigma_{\p{q}} = \id,
    \]
    It further fixes every point on the polar hyperplane $\p{q}^\perp$.
    
    For every line through $\p{q}$ that intersects the quadric $\quadric$ the involution $\sigma_{\p{q}}$ interchanges the two intersection points,
    while for a line through $\p{q}$ that touches the quadric $\quadric$ it fixes the touching point (cf.\ Lemma \ref{lem:quadric-line-intersection}).
  \item
    \label{prop:involution-projection-double-cover}
    The map $\pi_{\p{q}}$ is a projection onto $\p{q}^\perp \simeq \RP^n$.
    Its restriction onto the quadric
    \[
      \pi_{\p{q}}\restrict{\quadric} : \quadric \rightarrow \pi_{\p{q}}(\quadric)
    \]
    is a double cover with branch locus $\quadric \cap \p{q}^\perp$.
  \item The involution and projection together satisfy
    \[
      \pi_{\p{q}} \circ \sigma_{\p{q}} = \pi_{\p{q}}.
    \]
    Vice versa, if two distinct points $\p{x}, \p{y} \in \RP^{n+1}$ project to the same point $\pi_{\p{q}}(\p{x}) = \pi_{\p{q}}(\p{y})$,
    then $\p{x} = \sigma_{\p{q}}(\p{y})$.
    This gives rise to a one-to-one correspondence of the projection and the quotient
    \[
      \pi_{\p{q}}(\quadric) \simeq \faktor{\quadric}{\sigma_{\p{q}}}.
    \]
  \end{enumerate}
\end{proposition}
\begin{remark}
  The involution $\sigma_{\p{q}}$ and projection $\pi_{\p{q}}$ act in the same way as described in Proposition~\ref{prop:involution-projection}
  on every quadric from the pencil $\quadric \wedge \cone{\quadric}{\p{q}}$ spanned by $\quadric$
  and the cone of contact $\cone{\quadric}{\p{q}}$ with vertex $\p{q}$ (cf.\ Example \ref{ex:maximal-contact-pencil}).
\end{remark}

The intersection
\[
  \secquadric \coloneqq \quadric \cap \p{q}^\perp
\]
is a quadric of signature
\begin{itemize}
\item $(r-1,s,t)$ if $\scalarprod{q}{q} > 0$, or
\item $(r,s-1,t)$ if $\scalarprod{q}{q} < 0$.
\end{itemize}
The projection of a quadric $\quadric \subset \RP^{n+1}$ from a point $\p{q} \in \RP^{n+1} \setminus \quadric$
onto its polar hyperplane $\p{q}^\perp$ is a double cover of the ``inside'' or the ``outside'', cf.\ \eqref{eq:quadric-sides},
of $\secquadric = \p{q}^\perp \cap \quadric$ depending on the signature of $\p{q}$.
\begin{proposition}
  \label{prop:projection-side}
  Let $\p{q} \in \RP^{n+1} \setminus \quadric$.
  Then
  \begin{itemize}
  \item $\pi_{\p{q}}(\quadric) = \secquadric^- \cup \secquadric$, if $\scalarprod{q}{q} > 0$,
  \item $\pi_{\p{q}}(\quadric) = \secquadric^+ \cup \secquadric$, if $\scalarprod{q}{q} < 0$,
  \end{itemize}
\end{proposition}
\begin{proof}
  Decompose the homogeneous coordinate vector of a point $\p{x} \in \quadric$ into its projection onto $q$ and $q^\perp$
  \[
    x = \alpha q + \pi_q(x),
  \]
  with some $\alpha \in \R$.
  Then
  \[
    0 = \scalarprod{x}{x} = \alpha^2 \scalarprod{q}{q} + \pscalarprod{q}{x}{x}
  \]
  and thus
  \[
    \pscalarprod{q}{x}{x} = - \alpha^2 \scalarprod{q}{q}
    \left\{
    \begin{aligned}
      &< 0, &&\text{if}~\scalarprod{q}{q} \geq 0\\
      &> 0, &&\text{if}~\scalarprod{q}{q} \leq 0.
    \end{aligned}
    \right.
  \]
\end{proof}
The following proposition shows how the Cayley-Klein distance induced by $\secquadric$ for points in the projection $\pi_{\p{q}}(\quadric)$ 
can be lifted to the points on $\quadric$.
\begin{proposition}
  \label{prop:Cayley-Klein-distance-lift}
  Let $\p{q} \in \RP^{n+1} \setminus \quadric$ and $\p{x}, \p{y} \in \quadric$.
  Then the Cayley-Klein distance with respect to $\secquadric$ of their projections $\pi_{\p{q}}(\p{x}), \pi_{\p{q}}(\p{y})$ is given by
  \begin{equation}
    \label{eq:Cayley-Klein-distance-lift}
    \ck{\secquadric}{\pi_{\p{q}}(\p{x})}{\pi_{\p{q}}(\p{y})} = 
    \left( 1 - \frac{\scalarprod{x}{y}\scalarprod{q}{q}}{\scalarprod{x}{q}\scalarprod{y}{q}} \right)^2.
  \end{equation}
\end{proposition}
\begin{proof}
  We decompose the homogeneous coordinate vectors of $x, y$ into their projections onto $q$ and $q^\perp$
  \[
    x = \alpha q + \pi_q(x), \qquad y = \beta q + \pi_q(y)
  \]
  with some $\alpha, \beta \in \R$.
  Then,
  \[
    1 - \frac{\scalarprod{x}{y}\scalarprod{q}{q}}{\scalarprod{x}{q}\scalarprod{y}{q}}
    = 1 - \frac{\left(\alpha\beta\scalarprod{q}{q} + \scalarprod{x}{y}_q\right)\scalarprod{q}{q}}{\alpha\beta\scalarprod{q}{q}^2}
    = - \frac{\pscalarprod{q}{x}{y}}{\alpha\beta\scalarprod{q}{q}}.
  \]
  Now with
  \[
    0 = \scalarprod{x}{x} = \alpha^2 \scalarprod{q}{q} + \pscalarprod{q}{x}{x},
  \]
  and the analogous equation for $y$ we obtain
  \[
    \frac{\pscalarprod{q}{x}{y}^2}{\alpha^2\beta^2\scalarprod{q}{q}^2}
    = \frac{\pscalarprod{q}{x}{y}^2}{\pscalarprod{q}{x}{x}\pscalarprod{q}{y}{y}}.
  \]
\end{proof}
\begin{remark}
  Omitting the square for the quantity on the right hand side of equation \eqref{eq:Cayley-Klein-distance-lift}
  leads to a signed version of the lifted Cayley-Klein distance (see Appendix \ref{sec:invariant}).

  While the Cayley-Klein distance can, in general, be both positive or negative,
  the right hand side of equation \eqref{eq:Cayley-Klein-distance-lift} is always positive.
  This corresponds to the fact that the projection of $\quadric$ only always covers one side of $\secquadric$.
  Though having no real preimages the points on the other side of $\secquadric$ may be viewed as projections of certain imaginary points of $\quadric$
  (see Proposition \ref{prop:q-spheres-planar-sections2}).
\end{remark}

The transformation group induced by $\PO(r,s,t)_{\p{q}}$, cf.\ \eqref{eq:stabilizer}, onto $\p{q}^\perp$
is exactly the group of projective transformations $\PO(\tilde{r},\tilde{s},\tilde{t})$ that preserve the quadric $\secquadric$.
It is doubly covered by $\PO(r,s,t)_{\p{q}}$ and can be identified with the quotient
\begin{equation}
  \PO(\tilde{r},\tilde{s},\tilde{t})
  \simeq \faktor{\PO(r,s,t)_{\p q}}{\sigma_{\p q}}.
  \label{eq:lie-trafos-quotient}
\end{equation}
Note that $\PO(r,s,t)_{\p q}$ is the largest subgroup of $\PO(r,s,t)$ admitting this quotient,
i.e.\ the subgroup of transformations that commute with $\sigma_{\p q}$.

\subsection{Cayley-Klein spheres as planar sections}
\begin{figure}
  \centering
  \def\svgwidth{0.54\textwidth}    
  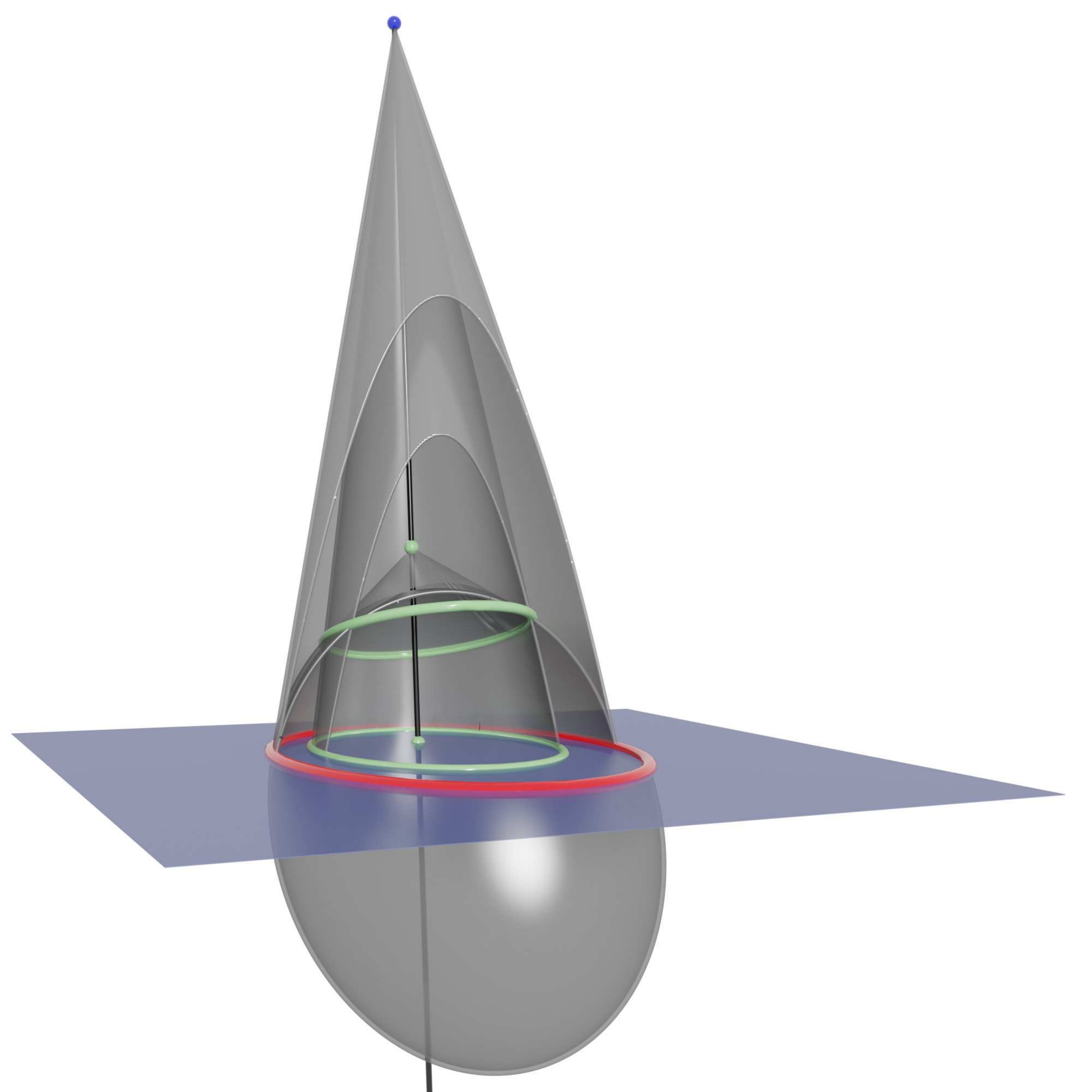
  %\hspace{0.08\textwidth}
  \def\svgwidth{0.44\textwidth}    
  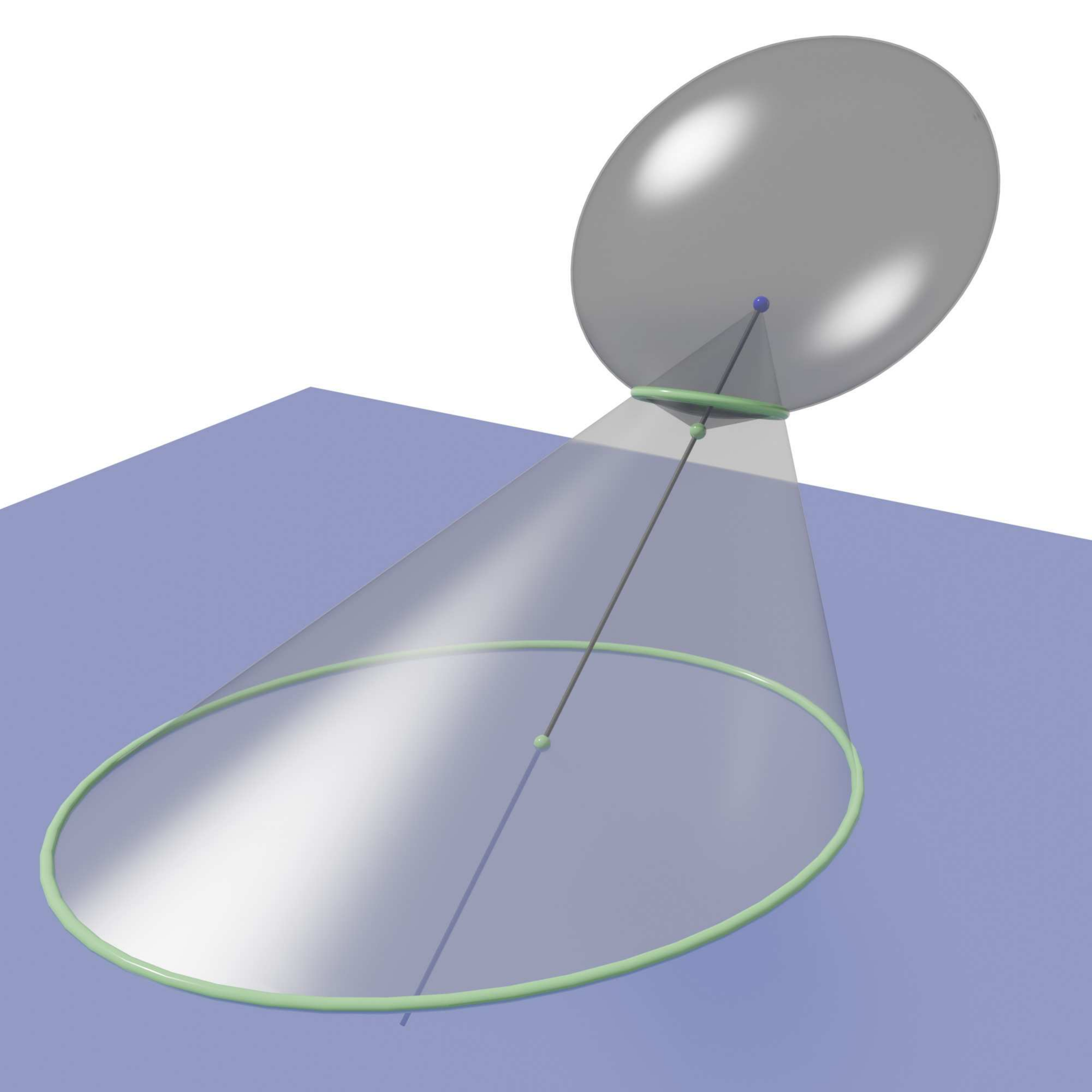
  \caption{
    The central projection of a hyperplanar section $\p{x}^\perp \cap \quadric$ of a quadric $\quadric \subset \RP^3$ from a point $\p{q}$.
    Its image is a Cayley-Klein sphere $\sprojection{\p{q}}(\p{x}) \subset \pi_{\p{q}}(\quadric)$ with respect to the absolute quadric $\secquadric$.
    Its center is given by $\pi_{\p{q}}(\p{x})$.
    The cone of contact can be used to distinguish the type of Cayley-Klein sphere that is obtained in the projection.
  }
  \label{fig:planar-section-projection}
\end{figure}
From now on, let $\quadric$ be a non-degenerate quadric of signature $(r,s)$.
Then each section of the quadric $\quadric$ with a hyperplane
can be identified with the pole of that hyperplane.
\begin{definition}
  \label{def:Q-sphere}
  We call a non-empty intersection of the quadric $\quadric$ with a hyperplane a \emph{$\quadric$-sphere},
  and identify it with the pole of the hyperplane.
  Thus, we call
  \[
    \spheres \coloneqq \set{\p{x} \in \RP^{n+1}}{\p{x}^\perp \cap \quadric \neq \varnothing}
  \]
  the \emph{space of $\quadric$-spheres}.
\end{definition}
\begin{remark}\
  \label{rem:Q-spheres}
  \nobreakpar
  \begin{enumerate}
  \item
    \label{rem:Q-spheres-degenerate}
    The intersection of $\quadric$ with a tangent hyperplane only consists of one point, or a cone (see Example~\ref{ex:quadrics}~\ref{ex:quadrics-cone}).
    To exclude these degenerate cases one might want to take $\spheres \setminus \quadric$ instead as the ``space of spheres''.
  \item
    Depending on the signature of of the quadric $\quadric$ only the following three cases can occur (w.l.o.g., $r \geq s$):
    \begin{itemize}
    \item $\spheres = \varnothing$ if $\quadric$ has signature $(n+2,0)$,
    \item $\spheres = \quadric^+ \cup \quadric$ if $\quadric$ has signature $(n+1,1)$,
    \item $\spheres = \RP^{n+1}$ else.
    \end{itemize}
  \end{enumerate}

\end{remark}
It turns out that every $\quadric$-sphere projects down to a Cayley-Klein sphere in $\pi_{\p{q}}(\quadric)$,
where the type of sphere can be distinguished by the two sides of the cone of contact $\cone{\quadric}{\p{q}}$.
Denote by
\begin{equation}
  \label{eq:quadratic-form-cone}
  \Delta_q(x) = \scalarprod{x}{q}^2 - \scalarprod{x}{x}\scalarprod{q}{q} = - \scalarprod{q}{q}\pscalarprod{q}{x}{x}
\end{equation}
the quadratic form of the cone of contact $\cone{\quadric}{\p{q}}$ (see Definition \ref{def:cone-of-contact}).
\begin{proposition}
  \label{prop:Q-sphere-projection}
  Consider the map
  \[
    \sprojection{\p{q}} : \p{x} \mapsto \pi_{\p{q}}(\p{x}^\perp \cap \quadric),
  \]
  for $\p{x} \in \spheres$.
  Then for every $\p{x} \in \spheres$ the image $\sprojection{\p{q}}(\p{x})$ is a Cayley-Klein sphere with points in $\pi_{\p{q}}(\quadric)$ (see Figure \ref{fig:planar-section-projection}).
  \begin{itemize}
  \item For $\Delta_q(x) \neq 0$ the image is a \emph{Cayley-Klein sphere} with center $\pi_{\p{q}}(\p{x})$ and Cayley-Klein radius
    \[
      \mu = \frac{\scalarprod{x}{q}^2}{\Delta_q(x)},
    \]
    i.e.\
    \[
      \pi_{\p{q}}(\p{x}^\perp \cap \quadric) = S_\mu\left(\pi_{\p{q}}(\p{x})\right).
    \]
    \begin{itemize}
    \item If $\Delta_q(x) > 0$, then $\pi_{\p{q}}(\p{x}) \in \pi_{\p{q}}(\quadric) \setminus \secquadric$.
    \item If $\Delta_q(x) < 0$, then $\pi_{\p{q}}(\p{x}) \in \p{q}^\perp \setminus \pi_{\p{q}}(\quadric)$.
    \end{itemize}
  \item For $\Delta_q(x) = 0$ the image is a \emph{Cayley-Klein horosphere} with center $\pi_{\p{q}}(x) \in \secquadric$.
  \item For $\p{x} \in \spheres \cap \p{q}^\perp$ the image is a \emph{hyperplane} in $\pi_{\p{q}}(\quadric)$ with pole $\p{x}$.
  \item For $\p{x} \in \quadric$ the image is the \emph{cone of contact} $\cone{\secquadric}{\pi_{\p{q}}(\p{x})} \subset \pi_{\p{q}}(\quadric)$.
  \item For $\p{x} = \p{q}$ the image is the \emph{absolute quadric} $\secquadric$.
  \end{itemize}
\end{proposition}
\begin{proof}
  We show the claim for points not on the cone of contact.
  Thus, let $\p{x} \in \spheres \setminus \cone{\quadric}{\p{q}}$, i.e., $\Delta_q(x) \neq 0$.
  Let $\p{y} \in \p{x}^\perp \cap \quadric$ be a point on the corresponding $\quadric$-sphere.
  Then we find for the projections of their homogeneous coordinate vectors
  \[
    \pscalarprod{q}{x}{y} = - \frac{\scalarprod{x}{q}\scalarprod{y}{q}}{\scalarprod{q}{q}},\quad
    \pscalarprod{q}{x}{x} = - \frac{\Delta_q(x)}{\scalarprod{q}{q}},\quad
    \pscalarprod{q}{y}{y} = - \frac{\scalarprod{y}{q}^2}{\scalarprod{q}{q}},
  \]
  and thus
  \[
    \ck{\secquadric}{\pi_{\p{q}}(\p{x})}{\pi_{\p{q}}(\p{y})}
    = \frac{\pscalarprod{q}{x}{y}^2}{\pscalarprod{q}{x}{x}\pscalarprod{q}{y}{y}}
    = \frac{\scalarprod{x}{q}^2}{\Delta_q(x)}
    = \mu.
  \]
  Therefore, $\sprojection{\p{q}}(\p{x})$ is a Cayley-Klein sphere with center $\pi_{\p{q}}(\p{x})$ and radius $\mu$.
  
  We know that $\pi_{\p{q}}(\p{y}) \in \pi_{\p{q}}(\quadric)$.
  Hence, according to Proposition \ref{prop:Cayley-Klein-spheres}, the sign of $\mu$, which is equal to the sign of $\Delta_q(x)$,
  determines which side of $\secquadric$ the center $\pi_{\p{q}}(\p{x})$ lies on.
  Further we find that,
  \[
    \mu = 0 ~\Leftrightarrow~ \p{x} \in \p{q}^\perp,
    ~\text{and}~
    \mu = 1 ~\Leftrightarrow~ \p{x} \in \quadric,
  \]
  which, again according to Proposition \ref{prop:Cayley-Klein-spheres},
  corresponds to a hyperplane and the cone of contact respectively.
\end{proof}
The map $\sprojection{\p{q}}$ covers the whole space of Cayley-Klein spheres with points in $\pi_{\p{q}}(\quadric)$.
\begin{proposition}
  \label{prop:sphere-double-cover}
  The map $\sprojection{\p{q}}$ constitutes a double cover of the set of Cayley-Klein spheres in $\pi_{\p{q}}(\quadric)$ with respect to $\secquadric$.
  Its ramification points are given by $(\p{q}^\perp \cup \{\p{q}\}) \cap \spheres$, and its covering involution is $\sigma_{\p{q}}$.
\end{proposition}
\begin{proof}
  We show that every Cayley-Klein sphere with points in $\pi_{\p{q}}(\quadric)$ possesses exactly two preimages,
  which are interchanged by $\sigma_{\p{q}}$, unless it is a hyperplane.
  The same is true for Cayley-Klein horospheres.

  Consider a Cayley-Klein sphere $S_{\mu}(\p{\widetilde{x}})$ with center $\p{\widetilde{x}} \in \p{q}^\perp \setminus \secquadric$,
  Cayley-Klein radius $\mu \in \R$ and points in $\pi_{\p{q}}(\quadric)$.
  Then, according to Proposition \ref{prop:Q-sphere-projection},
  a preimage $\p{x} \in \spheres$, $\sprojection{\p{q}}(\p{x}) = S_{\mu}(\p{\widetilde{x}})$ must satisfy $\pi_{\p{q}}(\p{x}) = \p{\widetilde{x}}$, i.e.
  \[
    x = \widetilde{x} + \lambda q
  \]
  for some $\lambda \in \R$, and
  \[
    \frac{\scalarprod{x}{q}^2}{\Delta_q(x)} = \mu,
  \]
  which is equivalent to
  \[
    \lambda^2 = - \mu \frac{\scalarprod{\widetilde{x}}{\widetilde{x}}}{\scalarprod{q}{q}}.
  \]
  According to Lemma \ref{lem:where-is-the-sphere} we have $- \mu \frac{\scalarprod{\widetilde{x}}{\widetilde{x}}}{\scalarprod{q}{q}} \geq 0$
  since $S_\mu(\p{\widetilde{x}}) \subset \pi_{\p{q}}(\quadric)$,
  and thus
  \[
    x_\pm \coloneqq \widetilde{x} \pm \sqrt{- \mu \frac{\scalarprod{\widetilde{x}}{\widetilde{x}}}{\scalarprod{q}{q}}} q
  \]
  defines one or two (real) points $\p{x}_\pm$ provided that $\p{x}_\pm \in \spheres$.
  
  The two points are interchanged by the involution, $\sigma_{\p{q}}(\p{x}_\pm) = \p{x}_\mp$,
  and we have
  \[
    \p{x}_+ = \p{x}_- ~\Leftrightarrow~ \mu = 0,
  \]
  in which case $\p{x}_\pm = \p{\widetilde{x}} \in \p{q}^\perp$.
  
  To see that $\p{x}_\pm^\perp \cap \quadric \neq \varnothing$, first assume $\mu \neq 0$.
  We show that any point $\p{\widetilde{y}} \in S_\mu(\p{\widetilde{x}})$ on the Cayley-Klein sphere,
  has (real) preimages $\p{y}_\pm \in \quadric$, i.e.\ $\pi_{\p{q}}(\p{y}_\pm) = \p{\widetilde{y}}$,
  that lie in the polar hyperplane of $\p{x}_\pm$ respectively.
  Indeed, the points
  \[
    y_\pm \coloneqq \pm \scalarprod{q}{q} \sqrt{- \mu \frac{\scalarprod{\widetilde{x}}{\widetilde{x}}}{\scalarprod{q}{q}}} \widetilde{y} - \scalarprod{\widetilde{x}}{\widetilde{y}} q
  \]
  satisfy
  \[
    \scalarprod{y_\pm}{y_\pm}
    = - \scalarprod{q}{q} \left(
      \mu \scalarprod{\widetilde{x}}{\widetilde{x}} \scalarprod{\widetilde{y}}{\widetilde{y}} - \scalarprod{\widetilde{x}}{\widetilde{y}}^2
    \right)
    = 0,
  \]
  and
  \[
    \scalarprod{x_\pm}{y_\pm}
    = \pm \scalarprod{q}{q} \scalarprod{\widetilde{x}}{\widetilde{y}} \sqrt{- \mu \frac{\scalarprod{\widetilde{x}}{\widetilde{x}}}{\scalarprod{q}{q}}}
    \mp \scalarprod{q}{q} \scalarprod{\widetilde{x}}{\widetilde{y}} \sqrt{- \mu \frac{\scalarprod{\widetilde{x}}{\widetilde{x}}}{\scalarprod{q}{q}}}
    = 0.
  \]

  If $\mu = 0$, then $\p{\widetilde{x}} = \p{x}_+ = \p{x}_-$,
  and the whole line $\p{\widetilde{y}} \wedge \p{q}$ lies in the polar hyperplane of $\p{\widetilde{x}}$.
  Since $\p{y} \in \pi_{\p{q}}(\quadric)$ the line $\p{\widetilde{y}} \wedge \p{q}$ has two real intersection points with $\quadric$,
  which serve as preimages for $\p{\widetilde{y}}$.
\end{proof}
\begin{lemma}
  \label{lem:where-is-the-sphere}
  A Cayley-Klein sphere with center $\p{\widetilde{x}} \in \p{q}^\perp \setminus \secquadric$ and Cayley-Klein radius $\mu \in \R$
  has points in $\pi_{\p{q}}(\quadric)$ if and only if
  \[
    - \mu \frac{\scalarprod{\widetilde{x}}{\widetilde{x}}}{\scalarprod{q}{q}} \geq 0.
  \]
\end{lemma}
\begin{proof}
  Follows from Proposition \ref{prop:Cayley-Klein-spheres} and Proposition \ref{prop:projection-side}.
\end{proof}

\begin{figure}
  \centering
  \def\svgwidth{0.5\textwidth}
  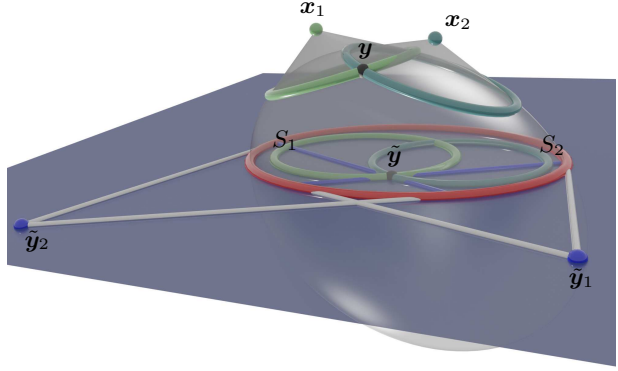
  \caption{
    The Cayley-Klein distance with respect to $\quadric$ corresponds to the Cayley-Klein intersection angle
    in the central projection to $\pi_{\p{q}}(\quadric)$ (see Proposition \ref{prop:sphere-angle}).
  }
  \label{fig:intersection-angle}
\end{figure}
Thus, we have found that the lift of the Cayley-Klein space $\pi_{\p{q}}(\quadric)$ to the quadric $\quadric$
leads to a linearization of the corresponding Cayley-Klein spheres, in the sense that they become planar sections of $\quadric$,
which we represent by their polar points.

For two intersecting Cayley-Klein spheres we call the Cayley-Klein distance of the poles
of the two tangent hyperplanes (with respect to the absolute quadric) their \emph{Cayley-Klein intersection angle}.
It is independent of the chosen intersection point.
The Cayley-Klein distance of two points in $\spheres$ describes exactly this Cayley-Klein intersection angle
in the projection to $\pi_{\p{q}}(\quadric)$ (see Figure \ref{fig:intersection-angle}).
\begin{proposition}
  \label{prop:sphere-angle}
  Let $\p{x}_1, \p{x}_2 \in \spheres$ such that the corresponding $\quadric$-spheres intersect.
  Let
  \[
    \p{y} \in \quadric \cap \p{x}_1^\perp \cap \p{x}_2^\perp
  \]
  be a point in that intersection, and $\p{\tilde{y}} \coloneqq \pi_{\p{q}}(\p{y})$ its projection.
  Let $S_1, S_2$ be the two projected Cayley-Klein spheres corresponding to $\p{x}_1$, $\p{x}_2$ respectively
  \[
    S_1 \coloneqq \sprojection{\p{q}}(\p{x}_1), \qquad S_2 \coloneqq \sprojection{\p{q}}(\p{x}_2).
  \]
  Let $\p{\tilde{y}}_1$, $\p{\tilde{y}}_2$ be the two poles of the tangent hyperplanes of $S_1$, $S_2$ at $\p{\tilde{y}}$ respectively.
  Then
  \[
    \ck{\quadric}{\p{x}_1}{\p{x}_2} = \ck{\widetilde{\quadric}}{\p{\tilde{y}}_1}{\p{\tilde{y}}_2}.
  \]
\end{proposition}
\begin{proof}
  First, we express the Cayley-Klein distance $\ck{\quadric}{\p{x}_1}{\p{x}_2}$
  in terms of the projected centers $\p{\tilde{x}}_1 \coloneqq \pi_{\p{q}}(\p{x}_1)$, $\p{\tilde{x}}_2 \coloneqq \pi_{\p{q}}(\p{x}_2)$
  and the projected intersection point $\p{\tilde{y}}$.
  To this end, we write
  \[
    x_1 = \tilde{x}_1 + \alpha_1 q,\qquad
    x_2 = \tilde{x}_2 + \alpha_2 q,\qquad
    y = \tilde{y} + \lambda q,
  \]
  for some $\alpha_1, \alpha_2, \lambda \in \R$.
  From $\scalarprod{y}{y} = \scalarprod{x_1}{y} = \scalarprod{x_2}{y} = 0$ we obtain
  \[
    \lambda^2 = - \frac{\scalarprod{\tilde{y}}{\tilde{y}}}{\scalarprod{q}{q}}, \qquad
    \alpha_1\lambda = - \frac{\scalarprod{\tilde{y}}{\tilde{x}_1}}{\scalarprod{q}{q}}, \qquad
    \alpha_2\lambda = - \frac{\scalarprod{\tilde{y}}{\tilde{x}_2}}{\scalarprod{q}{q}},
  \]
  and therefore
  \[
    \alpha_1\alpha_2 = - \frac{\scalarprod{\tilde{y}}{\tilde{x}_1}\scalarprod{\tilde{y}}{\tilde{x}_2}}{\scalarprod{\tilde{y}}{\tilde{y}}\scalarprod{q}{q}},\qquad
    \left(\alpha_1\right)^2 = - \frac{\scalarprod{\tilde{y}}{\tilde{x}_1}^2}{\scalarprod{\tilde{y}}{\tilde{y}}\scalarprod{q}{q}},\qquad
    \left(\alpha_2\right)^2 = - \frac{\scalarprod{\tilde{y}}{\tilde{x}_2}^2}{\scalarprod{\tilde{y}}{\tilde{y}}\scalarprod{q}{q}}.
  \]
  Using this we find
  \[
    \scalarprod{x_1}{x_2} = \scalarprod{\tilde{x}_1}{\tilde{x}_2} - \frac{\scalarprod{\tilde{y}}{\tilde{x}_1}\scalarprod{\tilde{y}}{\tilde{x}_2}}{\scalarprod{\tilde{y}}{\tilde{y}}},\quad
    \scalarprod{x_1}{x_1} = \scalarprod{\tilde{x}_1}{\tilde{x}_1} - \frac{\scalarprod{\tilde{y}}{\tilde{x}_1}^2}{\scalarprod{\tilde{y}}{\tilde{y}}},\quad
    \scalarprod{x_1}{x_1} = \scalarprod{\tilde{x}_1}{\tilde{x}_1} - \frac{\scalarprod{\tilde{y}}{\tilde{x}_1}^2}{\scalarprod{\tilde{y}}{\tilde{y}}},
  \]
  and thus
  \begin{equation}
    \label{eq:angle-proof-lhs}
    \ck{\quadric}{\p{x}_1}{\p{x}_2} =
    \frac{\scalarprod{x_1}{x_2}^2}{\scalarprod{x_1}{x_1}\scalarprod{x_2}{x_2}} =
    \frac{\left(\scalarprod{\tilde{x}_1}{\tilde{x}_2}\scalarprod{\tilde{y}}{\tilde{y}} - \scalarprod{\tilde{y}}{\tilde{x}_1}\scalarprod{\tilde{y}}{\tilde{x}_2}\right)^2}{\left(\scalarprod{\tilde{x}_1}{\tilde{x}_1}\scalarprod{\tilde{y}}{\tilde{y}}-\scalarprod{\tilde{y}}{\tilde{x}_1}^2\right)\left(\scalarprod{\tilde{x}_2}{\tilde{x}_2}\scalarprod{\tilde{y}}{\tilde{y}}-\scalarprod{\tilde{y}}{\tilde{x}_2}^2\right)}.
  \end{equation}
  
  Secondly, we express the right hand side $\ck{\widetilde{\quadric}}{\p{\tilde{y}}_1}{\p{\tilde{y}}_2}$ in terms of the same quantities.
  From \eqref{eq:Cayley-Klein-sphere-pole} we know that the poles $\p{\tilde{y}}_1$, $\p{\tilde{y}}_2$ of the tangent planes
  (with respect to $\widetilde{\quadric}$) are given by
  \[
    \tilde{y}_1 = \scalarprod{\tilde{x}_1}{\tilde{y}}\tilde{x}_1 - \mu_1 \scalarprod{\tilde{x}_1}{\tilde{x}_1}\tilde{y},\qquad
    \tilde{y}_2 = \scalarprod{\tilde{x}_2}{\tilde{y}}\tilde{x}_2 - \mu_2 \scalarprod{\tilde{x}_2}{\tilde{x}_2}\tilde{y},\qquad
  \]
  where
  \[
    \mu_1 = \frac{\scalarprod{\tilde{x}_1}{\tilde{y}}^2}{\scalarprod{\tilde{x}_1}{\tilde{x}_1}\scalarprod{\tilde{y}}{\tilde{y}}},\qquad
    \mu_2 = \frac{\scalarprod{\tilde{x}_2}{\tilde{y}}^2}{\scalarprod{\tilde{x}_2}{\tilde{x}_2}\scalarprod{\tilde{y}}{\tilde{y}}}
  \]
  are the Cayley-Klein radii of $S_1$ and $S_2$.
  From this we obtain
  \[
    \begin{aligned}
      &\scalarprod{\tilde{y}_1}{\tilde{y}_2} = \scalarprod{\tilde{x}_1}{\tilde{y}}\scalarprod{\tilde{x}_2}{\tilde{y}}\left(
        \scalarprod{\tilde{x}_1}{\tilde{x}_2} - \frac{\scalarprod{\tilde{x}_1}{\tilde{y}}\scalarprod{\tilde{x}_2}{\tilde{y}}}{\scalarprod{\tilde{y}}{\tilde{y}}}
      \right),\\
      \scalarprod{\tilde{y}_1}{\tilde{y}_1} = \scalarprod{\tilde{x}_1}{\tilde{y}}^2&\left(
        \scalarprod{\tilde{x}_1}{\tilde{x}_1} - \frac{\scalarprod{\tilde{x}_1}{\tilde{y}}^2}{\scalarprod{\tilde{y}}{\tilde{y}}}
      \right),\qquad
      \scalarprod{\tilde{y}_2}{\tilde{y}_2} = \scalarprod{\tilde{x}_2}{\tilde{y}}^2\left(
        \scalarprod{\tilde{x}_2}{\tilde{x}_2} - \frac{\scalarprod{\tilde{x}_2}{\tilde{y}}^2}{\scalarprod{\tilde{y}}{\tilde{y}}}
      \right)
    \end{aligned}
  \]
  Substituting into
  \[
    \ck{\widetilde{\quadric}}{\p{\tilde{y}}_1}{\p{\tilde{y}}_2} =
    \frac{\scalarprod{\tilde{y}_1}{\tilde{y}_2}^2}{\scalarprod{\tilde{y}_1}{\tilde{y}_1}\scalarprod{\tilde{y}_2}{\tilde{y}_2}}
  \]
  leads to the same as in \eqref{eq:angle-proof-lhs}.
\end{proof}
\begin{remark}\
  \label{rem:intersection-angle}
  \nobreakpar
  \begin{enumerate}
  \item
    \label{rem:intersection-angle-lift}
    Starting with two intersecting Cayley-Klein spheres in $\pi_{\p{q}}(\quadric)$
    the lifted $\quadric$-spheres must be chosen such that they intersect as well.
    Only then will the Cayley-Klein distance of the poles of the lifted spheres recover the Cayley-Klein intersection angle.
  \item
    Every quadric comes with a naturally induced (pseudo-)conformal structure, see e.g.\ \cite{P}.
    The Cayley-Klein distance between the two points $\p{x}_1, \p{x}_2 \in \spheres$ also coincides with the angle
    measured in this conformal structure.
  \end{enumerate}
\end{remark}
As a corollary of Theorem \ref{thm:fundamental-theorem-quadrics} we can now characterize the (local) transformations
of a Cayley-Klein space $\pi_{\p{q}}(\quadric)$ that map hyperspheres to hyperspheres as the projective orthogonal transformations in the lift to the quadric $\quadric$.
\begin{theorem}
  \label{thm:mobius-transformation-lift}
  Let $n \geq 2$, $\quadric \subset \RP^{n+1}$ be a non-degenerate quadric, and $\p{q} \in \RP^{n+1} \setminus \quadric$.
  Consider the Cayley-Klein space $\pi_{\p{q}}(\quadric)$ endowed with the Cayley-Klein metric induced by $\widetilde{\quadric} = \quadric \cap \p{q}^\perp$.
  Let $W \subset \pi_{\p{q}}(\quadric)$ be a non-empty open subset,
  and $f : W \rightarrow \pi_{\p{q}}(\quadric)$ be an injective map that maps intersections of Cayley-Klein hyperspheres with $W$
  to intersections of Cayley-Klein hyperspheres with $f(W)$.
  Then $f$ is the restriction of a projective transformation ${\RP^{n+1} \rightarrow \RP^{n+1}}$ that preserves the quadric $\quadric$.
\end{theorem}
\begin{proof}
  After lifting the open sets $W$ and $f(W)$ to $\quadric$ the statement follows from Theorem \ref{thm:fundamental-theorem-quadrics}.
\end{proof}
\begin{remark}\
  \label{rem:moebius-transformation-lift}
  \nobreakpar
  \begin{enumerate}
  \item
    If the transformation $f$ is defined on the whole space $\pi_{\p{q}}(\quadric)$ its lift must fix the point $\p{q}$.
    Thus, in this case $f$ must be an isometry of $\pi_{\p{q}}(\quadric)$.
  \item
    If $n \geq 3$ the condition on $f$ of mapping hyperspheres to hyperspheres may be weakened to $f$ being a conformal transformation,
    i.e. preserving Cayley-Klein angles between arbitrary hypersurfaces (generalized Liouville's theorem, see \cite{P, Ben2}).
  \item
    \label{rem:conformal-geometry-lift}
    The group of projective transformations $\PO(r,s)$ that preserve the quadric $\quadric$ maps $\quadric$-spheres to $\quadric$-spheres.
    In the projection to $\pi_{\p{q}}(\quadric)$ it may be interpreted as the group of transformations
    that map ``oriented points'' of $\pi_{\p{q}}(\quadric)$ to ``oriented points'' of $\pi_{\p{q}}(\quadric)$,
    while preserving Cayley-Klein spheres.
    It contains the subgroup $\PO(r,s)_{\p{q}}$ of isometries of $\pi_{\p{q}}(\quadric)$.
    The involution $\sigma_{\p{q}}$ plays the role of ``orientation reversion''.
    For ``hyperbolic Möbius geometry'' see, e.g.\ \cite{Som}, and for ``oriented points'' of the hyperbolic plane \cite{Y}.
  \end{enumerate}
\end{remark}

\subsection{Scaling along concentric spheres}
\begin{figure}
  \centering
  \def\svgwidth{0.5\textwidth}
%% Accompanies image file 'pencil_scaling.pdf' (pdf, eps, ps)
\begingroup%
  \makeatletter%
  \providecommand\color[2][]{%
    \errmessage{(Inkscape) Color is used for the text in Inkscape, but the package 'color.sty' is not loaded}%
    \renewcommand\color[2][]{}%
  }%
  \providecommand\transparent[1]{%
    \errmessage{(Inkscape) Transparency is used (non-zero) for the text in Inkscape, but the package 'transparent.sty' is not loaded}%
    \renewcommand\transparent[1]{}%
  }%
  \providecommand\rotatebox[2]{#2}%
  \newcommand*\fsize{\dimexpr\f@size pt\relax}%
  \newcommand*\lineheight[1]{\fontsize{\fsize}{#1\fsize}\selectfont}%
  \ifx\svgwidth\undefined%
    \setlength{\unitlength}{2025.31640625bp}%
    \ifx\svgscale\undefined%
      \relax%
    \else%
      \setlength{\unitlength}{\unitlength * \real{\svgscale}}%
    \fi%
  \else%
    \setlength{\unitlength}{\svgwidth}%
  \fi%
  \global\let\svgwidth\undefined%
  \global\let\svgscale\undefined%
  \makeatother%
  \begin{picture}(1,1)%
    \lineheight{1}%
    \setlength\tabcolsep{0pt}%
    \put(0,0){\includegraphics[width=\unitlength,page=1]{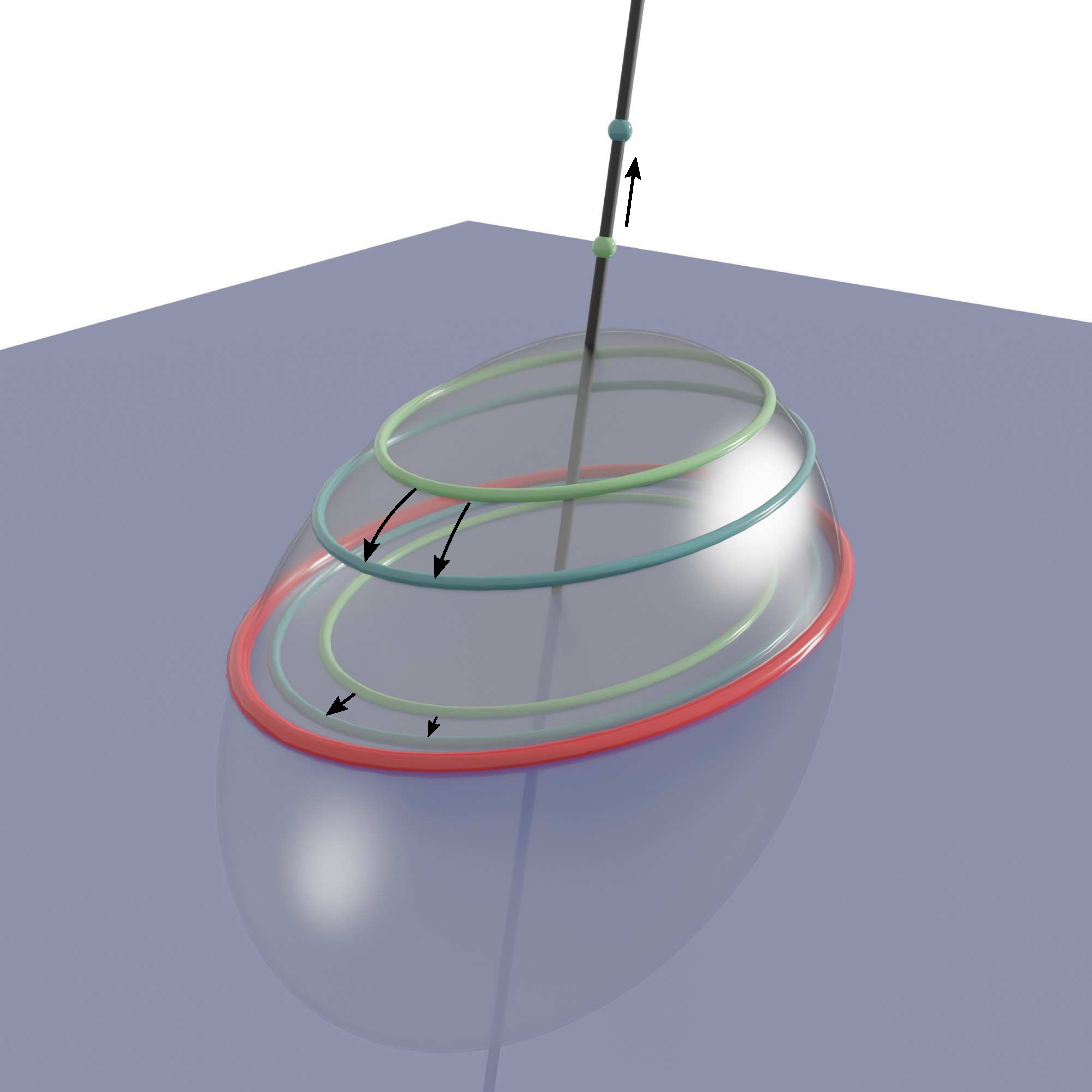}}%
    \put(0.57109557,0.75407927){\color[rgb]{0,0,0}\makebox(0,0)[lt]{\lineheight{1.25}\smash{\begin{tabular}[t]{l}{\small $\p{x}_1$}\end{tabular}}}}%
    \put(0.5885781,0.88228439){\color[rgb]{0,0,0}\makebox(0,0)[lt]{\lineheight{1.25}\smash{\begin{tabular}[t]{l}{\small $\p{x}_2$}\end{tabular}}}}%
    \put(0.60722608,0.81118882){\color[rgb]{0,0,0}\makebox(0,0)[lt]{\lineheight{1.25}\smash{\begin{tabular}[t]{l}{\small $T_{\p{x}_1,\p{x}_2}$}\end{tabular}}}}%
  \end{picture}%
\endgroup%

  \caption{
    Scaling along a pencil of concentric Cayley-Klein spheres in the lift and in the projection.
  }
  \label{fig:scalings}
\end{figure}
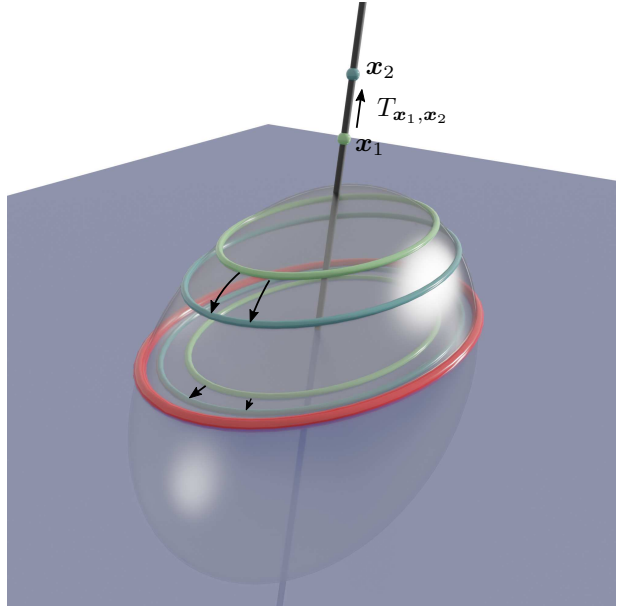
The transformation group $\PO(r,s)$ contains the isometries of $\pi_{\p{q}}(\quadric)$,
given by $\PO(r,s)_{\p{q}}$.
It turns out that the only transformations additionally needed to generate the whole group $\PO(r,s)$
are ``scalings'' along concentric spheres.

In the lift to $\spheres$ pencils of concentric Cayley-Klein spheres in $\pi_{\p{q}}(\quadric)$
correspond to lines in $\spheres$ through $\p{q}$ (cf.\ Proposition \ref{prop:Q-sphere-projection}).
\begin{proposition}
  The preimage under the map $\sprojection{\p{q}}$ of a family of concentric Cayley-Klein spheres
  in $\pi_{\p{q}}(\quadric)$ with center $\p{\widetilde{x}} \in \p{q}^\perp$
  is given by the line $\ell \coloneqq (\p{\widetilde{x}} \wedge \p{q}) \cap \spheres$.

  For every $\p{x} \in \ell$ the hyperplane $\p{x}^\perp$ that defines the $\quadric$-sphere by intersection with $\quadric$
  contains the polar subspace $(\p{\widetilde{x}} \wedge \p{q})^\perp$.
\end{proposition}
\begin{definition}
  We call a line in $\spheres$ a \emph{pencil of $\quadric$-spheres},
  and a line in $\spheres$ containing the point $\p{q}$ a \emph{pencil of concentric $\quadric$-spheres (with respect to $\p{q}$)}.
\end{definition}
For every pencil of $\quadric$-spheres there is a distinguished one-parameter family of projective orthogonal transformations
that preserve the pencil and each hyperplane through the corresponding line (see Figure \ref{fig:scalings}).
\begin{proposition}
  Let $\p{x}_1, \p{x}_2 \in \spheres$ with $\ck{\quadric}{\p{x}_1}{\p{x}_2} > 0$.
  Then there is a unique transformation $T_{\p{x}_1, \p{x}_2} \in \PO(r,s)$ that maps $\p{x}_1$ to $\p{x}_2$
  and preserves every hyperplane through the line $\p{x}_1 \wedge \p{x}_2$.

  It satisfies $\left(T_{\p{x}_1, \p{x}_2}\right)^{-1} = T_{\p{x}_2, \p{x}_1}$.
\end{proposition}
\begin{definition}
  For two points $\p{x}_1, \p{x}_2 \in \spheres$ on a pencil of concentric $\quadric$-spheres with respect to $\p{q}$,
  i.e. $\p{q} \in \p{x}_1 \wedge \p{x}_2$, we call the transformation $T_{\p{x}_1, \p{x}_2} \in \PO(r,s)$
  a \emph{scaling along the pencil of concentric spheres $\p{x}_1 \wedge \p{x}_2$}.
\end{definition}
\begin{remark}
  \label{rem:scaling-types}
  There are three types of scalings along concentric spheres depending on the signature of the line $\p{x}_1 \wedge \p{x}_2$.
\end{remark}
In the projection $\sprojection{\p{q}}$ to $\pi_{\p{q}}(\quadric)$ the line $\p{x}_1 \wedge \p{x}_2$ through $\p{q}$
corresponds to a pencil of concentric Cayley-Klein spheres.
The transformation $T_{\p{x}_1, \p{x}_2}$ maps spheres of this pencil to spheres of this pencil.

Every transformation from $\PO(r,s)$ may be decomposed into a (lift of an) isometry of $\pi_{\p{q}}(\quadric)$
and a scaling along a pencil of concentric spheres.
\begin{proposition}
  Let $f \in \PO(r,s)$. Then $f$ can be written as
  \[
    f = T_{\p{q}, \p{x}} \circ \Phi = \Psi \circ T_{\p{y}, \p{q}}
  \]
  with $\p{x} \coloneqq f(\p{q})$, $\p{y} \coloneqq f^{-1}(\p{q})$ and some $\Phi, \Psi \in \PO(r,s)_{\p{q}}$.
\end{proposition}
\begin{proof}
  We observe that $T_{\p{x}, \p{q}} \circ f, f \circ T_{\p{q}, \p{y}} \in \PO(r,s)_{\p{q}}$.
\end{proof}
\begin{remark}
  \label{rem:mobius-decomposition}
  To generate all transformations of $\PO(r,s)$ one may further restrict to (at most) three arbitrarily chosen
  one-parameter families of scalings (one of each type, cf. Remark \ref{rem:scaling-types}).
  Then a transformation $f \in \PO(r,s)$ can be written as
  \[
    f = \Phi \circ T \circ \Psi
  \]
  where $\Phi, \Psi \in \PO(r,s)_{\p{q}}$ and $T$ is exactly one of the three chosen scalings.
\end{remark}

\subsection{Möbius geometry}
\label{sec:mobius-geometry}
Let $\scalarprod{\cdot}{\cdot}$ be the standard non-degenerate bilinear form of signature $(n+1,1)$, i.e.
\[
  \scalarprod{x}{y} \coloneqq x_1y_1 + \ldots + x_{n+1}y_{n+1} - x_{n+2}y_{n+2}
\]
for $x, y \in \R^{n+2}$, and denote by $\mob \subset \RP^{n+1}$ the corresponding quadric,
which we call the \emph{Möbius quadric}.

The Möbius quadric is projectively equivalent to the standard round sphere $\mob \simeq \S^n \subset \R^{n+1}$.
In this correspondence intersections of $\mob$ with hyperplanes of $\RP^{n+1}$,
i.e. the $\mob$-spheres, are identified with hyperspheres of $\S^n$ (cf.\ Proposition~\ref{prop:non-oriented-sphere-polarity}).
The corresponding transformation group
\[
  \mobtrafos \coloneqq \PO(n+1,1)
\]
of \emph{Möbius transformations} leaves the quadric $\mob$ invariant and maps hyperplanes to hyperplanes.
Thus Möbius geometry may be seen as the geometry of points on $\S^n$ in which hyperspheres are mapped to hyperspheres.

The set of poles of hyperplanes that have non-empty intersections with the Möbius quadric $\mob$, i.e., the space of $\mob$-spheres, is given by
\[
  \spheres = \mob^+ \cup \mob.
\]
where $\mob^+$ is the ``outside'' of $\mob$.

\begin{remark}
  \label{rem:inversive-distance}
  The Cayley-Klein metric on $\spheres$ that is induced by the Möbius quadric $\mob$ is called the \emph{inversive distance}, see \cite{Cox}.
  For two intersecting hyperspheres of $\S^n$ it is equal to the cosine of their intersection angle.
  For a signed version of this quantity see Section \ref{sec:signed-inversive-distance}.
  Comparing with Section \ref{sec:hyperbolic-space} this same Cayley-Klein metric
  also induces $(n+1)$-dimensional hyperbolic geometry on the ``inside'' $\mob^-$ of the Möbius quadric,
  and $(n+1)$-dimensional deSitter geometry on the ``outside'' $\mob^+$ of the Möbius quadric.
\end{remark}

Central projection of the $(n+1)$-dimensional Möbius quadric from a point leads to a double cover of $n$-dimensional hyperbolic/elliptic space.

\subsection{Hyperbolic geometry and Möbius geometry}
\label{sec:hyperbolic-geometry-from-mobius}
\begin{figure}
  \centering
  \includegraphics[width=0.32\textwidth]{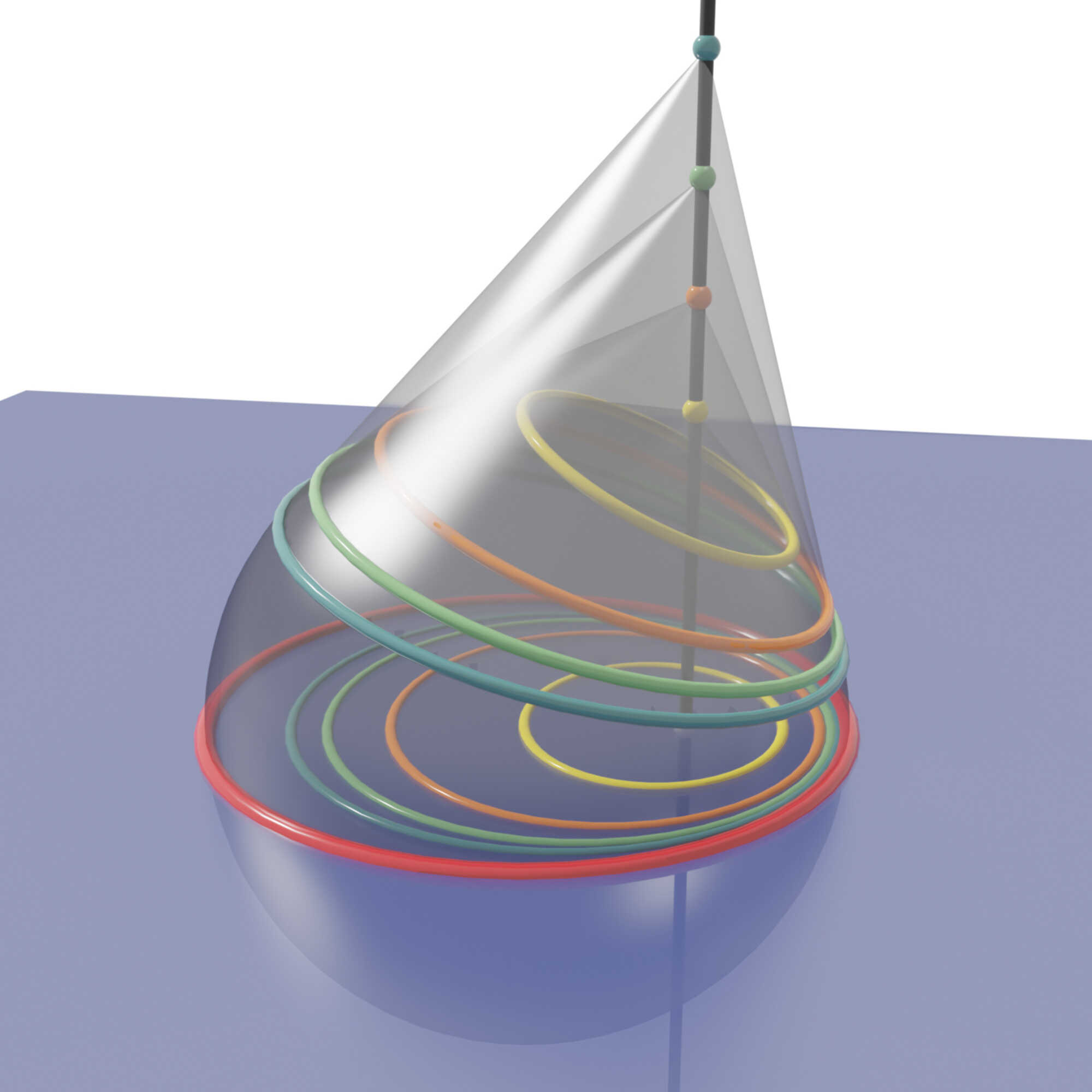}
  \includegraphics[width=0.32\textwidth]{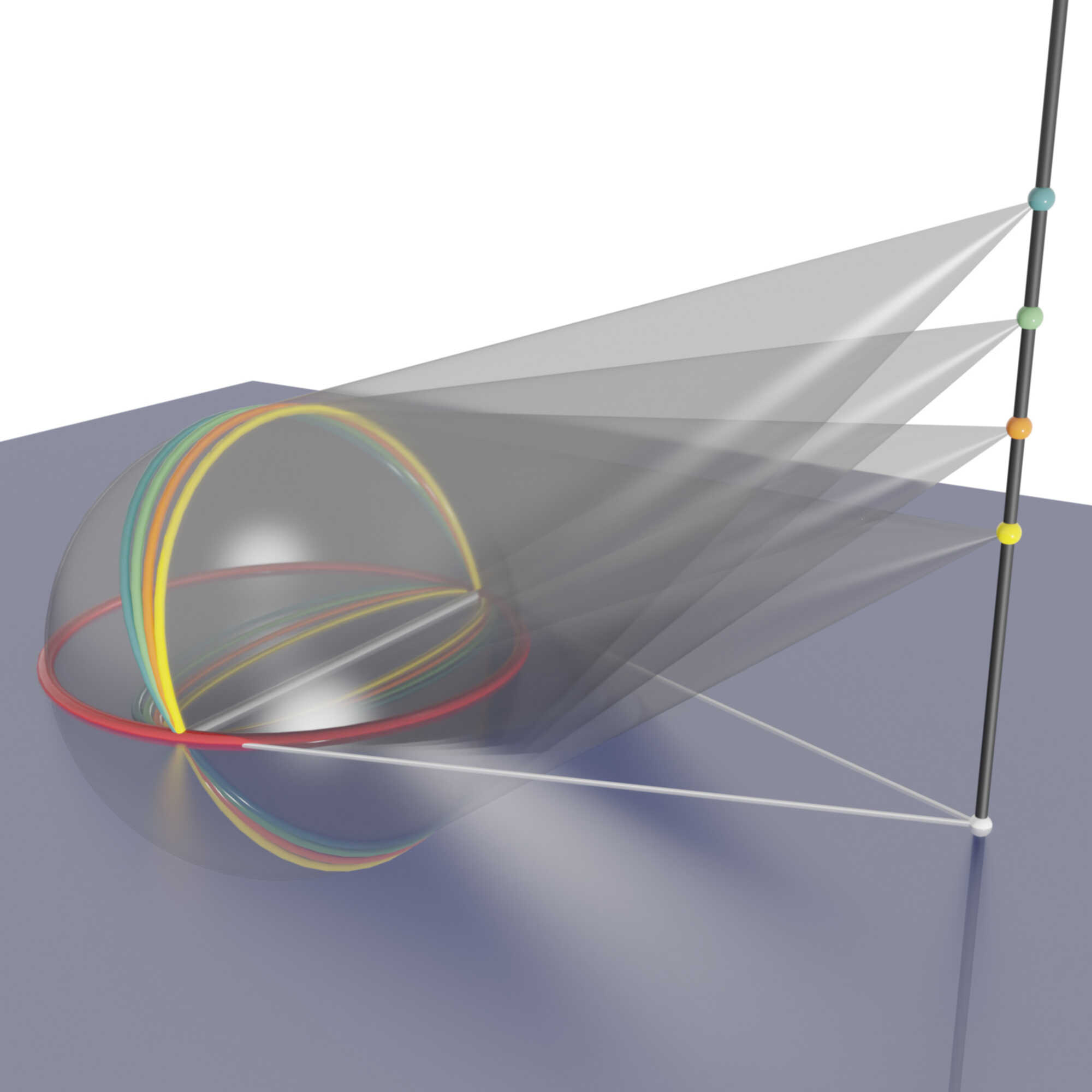}
  \includegraphics[width=0.32\textwidth]{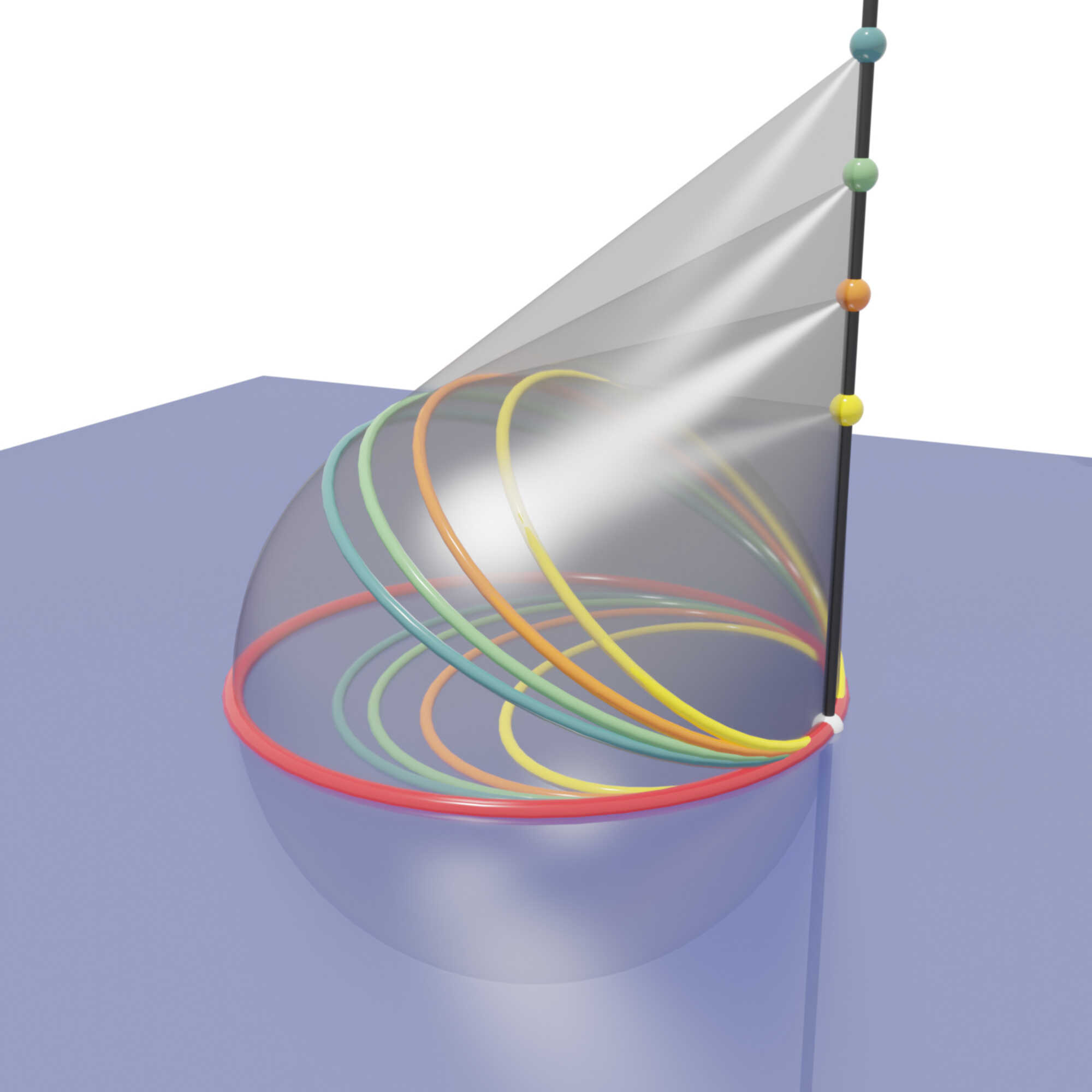}
  \caption{
    Hyperbolic geometry and its lift to M\"obius geometry.
    \emph{Left:} Concentric hyperbolic circles.
    \emph{Middle:} Constant distance curves to a common line.
    \emph{Right:} Concentric horocycles with center on the absolute conic.
  }
  \label{fig:hyperbolic-mobius-geometry}
\end{figure}

Given the Möbius quadric $\mob \subset \RP^n$ choose a point $\p{q} \in \RP^{n+1}$, $\scalarprod{q}{q} > 0$, w.l.o.g.
\[
  \p{q} \coloneqq [e_{n+1}] = [0, \ldots 0, 1, 0].
\]
The corresponding involution and projection take the form
\[
  \begin{aligned}
    &\sigma_{\p{q}} : [x_1, \ldots, x_n, x_{n+1}, x_{n+2}] \mapsto [x_1, \ldots, x_n, -x_{n+1}, x_{n+2}],\\
    &\pi_{\p{q}} : [x_1, \ldots, x_n, x_{n+1}, x_{n+2}] \mapsto [x_1, \ldots, x_n, 0, x_{n+2}].
  \end{aligned}
\]
The quadric in the polar hyperplane of $\p{q}$
\[
  \widetilde{\mob} = \mob \cap \p{q}^\perp
\]
has signature $(n,1)$.
Its ``inside'' $\hyp = \widetilde{\mob}^-$ can be identified with $n$-dimensional hyperbolic space (cf.\ Section \ref{sec:hyperbolic-space}),
and the Möbius quadric projects down to the compactified hyperbolic space
\[
  \chyp = \pi_{\p{q}}(\mob).
\]

According to Proposition \ref{prop:Q-sphere-projection}, an $\mob$-sphere,
which we identify with a point in $\spheres = \mob^+ \cup \mob$ projects to the different types of \emph{generalized hyperbolic spheres} in $\chyp$
(see Figure~\ref{fig:hyperbolic-mobius-geometry} and Table~\ref{tab:hyperbolic-mobius-geometry}).
\begin{proposition}
  Under the map
  \[
    \sprojection{\p{q}} : \p{x} \mapsto \pi_{\p{q}}(x^\perp \cap \mob)
  \]
  a point $\p{x} \in \spheres = \mob^+ \cup \mob$
  \begin{itemize}
  \item with $\p{x} \in \mob$ is mapped to a \emph{point} $\pi_{\p{q}}(\p{x}) \in \chyp$,
  \item with $\p{x} \in \p{q}^\perp$, i.e.\ $x_{n+1} = 0$, is mapped to a \emph{hyperbolic hyperplane} in $\chyp$ with pole $\p{x}$,
  \item with $\pscalarprod{q}{x}{x} < 0$ is mapped to a \emph{hyperbolic sphere} in $\chyp$ with center $\pi_{\p{q}}(\p{x})$. In the normalization $\pscalarprod{q}{x}{x} = -1$ its hyperbolic radius is given by $r \geq 0$, where $\cosh^2 r = x_{n+1}^2$,
  \item with $\pscalarprod{q}{x}{x} > 0$ is mapped to a \emph{hyperbolic surface of constant distance} in $\chyp$
    to a hyperbolic hyperplane with pole $\pi_{\p{q}}(x)$,
    In the normalization $\pscalarprod{q}{x}{x} = 1$ its hyperbolic distance is given by $r \geq 0$, where $\sinh^2 r = x_{n+1}^2$.
  \item with $\pscalarprod{q}{x}{x} = 0$ is mapped to a \emph{hyperbolic horosphere}.
  \end{itemize}
\end{proposition}
\begin{proof}
  Compare Section \ref{sec:hyperbolic-space} for the different types of possible Cayley-Klein spheres in hyperbolic space.
  Following Proposition \ref{prop:Q-sphere-projection} they can be distinguished by the sign
  of the quadratic form $\Delta_q(x)$, or, comparing with equation \eqref{eq:quadratic-form-cone}, by the sign of $\pscalarprod{q}{x}{x}$.
  Furthermore, the center of the Cayley-Klein sphere corresponding to $\p{x}$
  is given by $\pi_{\p{q}}(\p{x})$, while its Cayley-Klein radius is given by
  \[
    \mu
    = \frac{\scalarprod{x}{q}^2}{\Delta_q(x)}
    = - \frac{\scalarprod{x}{q}^2}{\scalarprod{q}{q}\pscalarprod{q}{x}{x}}
    = - \frac{x_{n+1}^2}{\pscalarprod{q}{x}{x}}.
  \]
\end{proof}
\begin{remark}\
  \label{rem:hyperbolic-mobius}
  \nobreakpar
  \begin{enumerate}
  \item The map $\sprojection{\p{q}}$ is a double cover of the set of generalized hyperbolic spheres,
    branching on the subset of hyperbolic planes (see Proposition \ref{prop:sphere-double-cover}).
  \item The Cayley-Klein distance induced on $\spheres$ by $\mob$ measures the Cayley-Klein angle
    between the corresponding generalized hyperbolic spheres (see Proposition \ref{prop:sphere-angle})
    if their lifts intersect (see Remark~\ref{rem:intersection-angle}~\ref{rem:intersection-angle-lift}),
    and more generally their inversive distance (see Remark \ref{rem:inversive-distance}).
  \item
    \label{rem:hyperbolic-mobius-transformations}
    In the projection to $\chyp$ Möbius transformations map generalized hyperbolic spheres to generalized hyperbolic spheres
    (see Remark ~\ref{rem:moebius-transformation-lift}~\ref{rem:conformal-geometry-lift}).
    Vice versa, every (local) transformation of the hyperbolic space that maps generalized hyperbolic spheres to generalized hyperbolic spheres
    is the restriction of a Möbius transformation
    (see Theorem \ref{thm:mobius-transformation-lift})
  \item Every Möbius transformation can be decomposed into two hyperbolic isometries
    and a scaling along either a fixed pencil of concentric hyperbolic spheres, distance surfaces, or horospheres
    (see Remark \ref{rem:mobius-decomposition}).
  \end{enumerate}
\end{remark}
\begin{table}[H]
  \centering
  \def\arraystretch{2.0}
  \begin{tabular}{c|c}
    \textbf{hyperbolic geometry} & \textbf{Möbius geometry}\\
    \hline\hline
    \makecell{\emph{point}\\ $\p{y} \in \hyp$, $y = (\hat{y},y_{n+1}) \in \HH^n$} & $[\hat{y}, \pm 1, y_{n+1}] \in \mob$ \\
    \hline
    \makecell{\emph{hyperplane}\\ with pole $\p{y} \in \ds$, $y = (\hat{y},y_{n+1}) \in \dS^n$} & $[\hat{y}, 0, y_{n+1}] \in \mob^+ \cap \p{q}^\perp$ \\
    \hline
    \makecell{\emph{sphere}\\ with center $\p{y} \in \hyp$, $y = (\hat{y},y_{n+1}) \in \HH^n$\\ and radius $r > 0$} & $[\hat{y}, \pm \cosh r, y_{n+1}] \in \mob^+ \cap \cone{\mob}{\p{q}}^+$ \\
    \hline
    \makecell{\emph{surface of constant distance}\\ $r > 0$ to a hyperplane\\ with pole $\p{y} \in \ds$, $y = (\hat{y},y_{n+1}) \in \dS^n$} & $[\hat{y}, \pm \sinh r, y_{n+1}] \in \mob^+ \cap \cone{\mob}{\p{q}}^-$ \\
    \hline
    \makecell{\emph{horosphere}\\ with center $\p{y} \in \widetilde{\mob}$, $y = (\hat{y},y_{n+1}) \in \L^{n,1}$} & $[\hat{y}, \pm e^r, y_{n+1}] \in \mob^+ \cap \cone{\mob}{\p{q}}$
  \end{tabular}
  \caption{
    The lifts of generalized hyperbolic spheres to Möbius geometry.
  }
  \label{tab:hyperbolic-mobius-geometry}
\end{table}

\subsection{Elliptic geometry and Möbius geometry}
\label{sec:elliptic-geometry-from-mobius}
\begin{figure}
  \centering
  \includegraphics[width=0.45\textwidth]{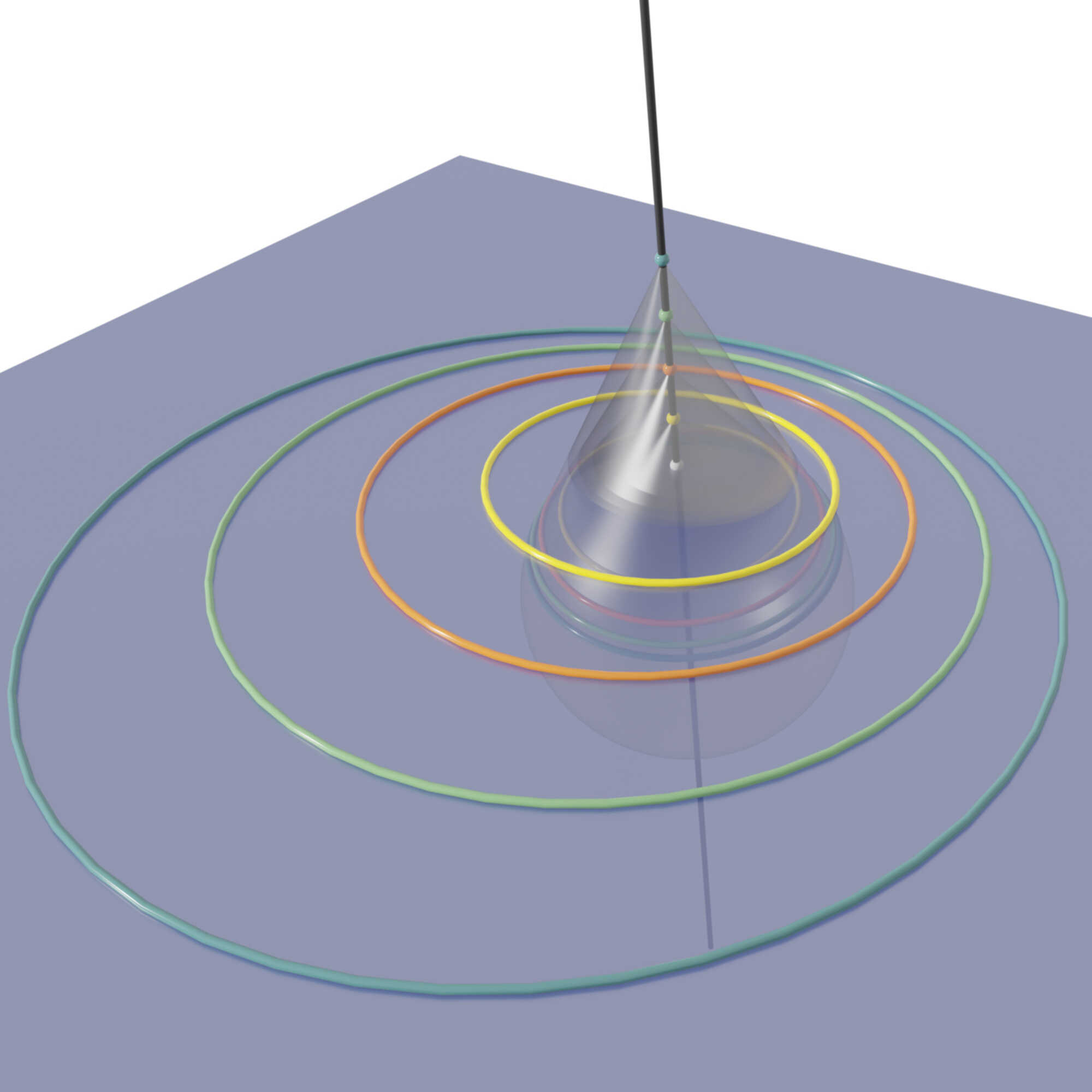}
  \caption{
    Concentric circles in elliptic geometry and its lift to Möbius geometry.
  }
  \label{fig:elliptic-mobius-geometry}
\end{figure}

Given the Möbius quadric $\mob \subset \RP^n$ choose a point $\p{q} \in \RP^{n+1}$, $\scalarprod{q}{q} < 0$, w.l.o.g.
\[
  \p{q} \coloneqq [e_{n+1}] = [0, \ldots 0, 0, 1].
\]
The corresponding involution and projection take the form
\[
  \begin{aligned}
    &\sigma_{\p{q}} : [x_1, \ldots, x_{n+1}, x_{n+2}] \mapsto [x_1, \ldots, x_{n+1}, -x_{n+2}]\\
    &\pi_{\p{q}} : [x_1, \ldots, x_{n+1}, x_{n+2}] \mapsto [x_1, \ldots, x_{n+1}, 0]
  \end{aligned}
\]
The quadric in the polar hyperplane of $\p{q}$
\[
  \widetilde{\mob} = \mob \cap \p{q}^\perp
\]
is imaginary, and has signature $(n+1,0)$.
The Möbius quadric projections down to its ``outside''
\[
  \ellip = \widetilde{\mob}^+,
\]
which can be identified with $n$-dimensional elliptic space (cf.\ Section \ref{sec:elliptic-space}).

According to Proposition \ref{prop:Q-sphere-projection} an $\mob$-sphere projects to an elliptic sphere in $\ellip$
(see Figure~\ref{fig:elliptic-mobius-geometry} and Table \ref{tab:elliptic-mobius-geometry}).
\begin{proposition}
  Under the map
  \[
    \sprojection{\p{q}} : x \mapsto \pi_{\p{q}}(x^\perp \cap \mob)
  \]
  a point $\p{x} \in \spheres = \mob^+ \cup \mob$
  \begin{itemize}
  \item with $\p{x} \in \mob$ is mapped to a \emph{point} $\pi_{\p{q}}(\p{x}) \in \ellip$,
  \item with $\p{x} \in \p{q}^\perp$, i.e.\ $x_{n+2} = 0$, is mapped to an \emph{elliptic plane} in $\ellip$ with pole $\p{x}$,
  \item with $\pscalarprod{q}{x}{x} > 0$ is mapped to an \emph{elliptic sphere} in $\ellip$ with center $\pi_{\p{q}}(\p{x})$,
    In the normalization $\pscalarprod{q}{x}{x} = 1$ its elliptic radius is given by $r \geq 0$, where $\cos^2 r = x_{n+2}^2$.
    It also has constant elliptic distance $R \geq 0$, where $\sin^2 R = x_{n+1}^2$, to the polar hyperplane of $\pi_{\p{q}}(\p{x})$.
  \end{itemize}
\end{proposition}
\begin{remark}\
  \label{rem:elliptic-mobius}
  \nobreakpar
  \begin{enumerate}
  \item The map $\sprojection{\p{q}}$ is a double cover of the set of elliptic spheres,
    branching on the subset of elliptic planes (see Proposition \ref{prop:sphere-double-cover}).
  \item The Cayley-Klein distance induced on $\spheres$ by $\mob$ measures the Cayley-Klein angle
    between the corresponding elliptic spheres (see Proposition \ref{prop:sphere-angle})
    if their lifts intersect (see Remark~\ref{rem:intersection-angle}~\ref{rem:intersection-angle-lift}),
    and more generally their inversive distance (see Remark \ref{rem:inversive-distance}).
  \item
    \label{rem:elliptic-mobius-transformations}
    In the projection to $\ellip$, Möbius transformations map elliptic spheres to elliptic spheres
    (see Remark ~\ref{rem:moebius-transformation-lift}~\ref{rem:conformal-geometry-lift}).
    Vice versa, every (local) transformation of elliptic space that maps elliptic spheres to elliptic spheres
    is the restriction of a Möbius transformation (see Theorem \ref{thm:mobius-transformation-lift}).
  \item Every Möbius transformation can be decomposed into two elliptic isometries
    and a scaling along one fixed pencil of concentric elliptic spheres (see Remark \ref{rem:mobius-decomposition}).
  \end{enumerate}
\end{remark}
\begin{remark}
  Upon the identification of the Möbius quadric with the sphere $\mob \simeq \S^n$ the group of Möbius transformations
  fixing the point $\p{q}$ is the group of spherical transformations, yielding spherical geometry,
  which is a double cover of elliptic geometry.
\end{remark}
\begin{table}[H]
  \centering
  \def\arraystretch{2.0}
  \begin{tabular}{c|c}
    \textbf{elliptic geometry} & \textbf{Möbius geometry}\\
    \hline\hline
    \makecell{\emph{point}\\ $\p{y} \in \ellip$, $y \in \S^n$} & $[y, \pm 1] \in \mob$ \\
    \hline
    \makecell{\emph{hyperplane}\\ with pole $\p{y} \in \ellip$, $y \in \S^n$} & $[y, 0] \in \p{q}^\perp$ \\
    \hline
    \makecell{\emph{sphere}\\ with center $\p{y} \in \ellip$, $y \in \S^n$\\ and radius $r > 0$} & $[y, \pm \cos r] \in \mob^+$
  \end{tabular}
  \caption{
    The lifts of elliptic spheres to Möbius geometry.
  }
  \label{tab:elliptic-mobius-geometry}
\end{table}

%%% Local Variables:
%%% mode: latex
%%% TeX-master: "main"
%%% End:

\newpage
\section{Non-Euclidean Laguerre geometry}
\label{sec:laguerre}

The primary objects in \emph{Möbius geometry} are points on $\mob$, which yield a double cover of the points in hyperbolic/elliptic space,
and spheres, which yield a double cover of the spheres in hyperbolic/elliptic space.
The primary incidence between these objects is \emph{a point lying on a sphere}.

\emph{Laguerre geometry} is dual to Möbius geometry in the sense that the primary objects are hyperplanes, and spheres,
both being a double cover of the corresponding objects in hyperbolic/elliptic space,
while the primary incidence between these objects is \emph{a plane being tangent to a sphere}.

In this section we introduce the concept of \emph{polar projection} of a quadric.
Similar to the central projection of a quadric it yields a double cover of certain hyperplanes of a Cayley-Klein space.
While the double cover of points in a space form in the case of Möbius geometry may be interpreted
as ``oriented points'' (cf.\ Remark ~\ref{rem:moebius-transformation-lift}~\ref{rem:conformal-geometry-lift}), in the case of Laguerre geometry this leads to the perhaps more intuitive notion of ``oriented hyperplanes''.

A decomposition of the corresponding groups of \emph{Laguerre transformations} into isometries and scalings can be obtained in an analogous way to the decomposition of the Möbius group.

We discuss the cases of hyperbolic and elliptic Laguerre geometry in detail, including coordinate representations for the different geometric objects appearing in each case.
A treatment of the Euclidean case can be found in Appendix \ref{sec:euclidean}.

\subsection{Polar projection}
\begin{figure}
  \centering
  \def\svgwidth{0.47\textwidth}
  %% Creator: Inkscape inkscape 0.92.5, www.inkscape.org
%% Accompanies image file 'dsitter_laguerre_tangent.pdf' (pdf, eps, ps)
\begingroup%
  \makeatletter%
  \providecommand\color[2][]{%
    \errmessage{(Inkscape) Color is used for the text in Inkscape, but the package 'color.sty' is not loaded}%
    \renewcommand\color[2][]{}%
  }%
  \providecommand\transparent[1]{%
    \errmessage{(Inkscape) Transparency is used (non-zero) for the text in Inkscape, but the package 'transparent.sty' is not loaded}%
    \renewcommand\transparent[1]{}%
  }%
  \providecommand\rotatebox[2]{#2}%
  \newcommand*\fsize{\dimexpr\f@size pt\relax}%
  \newcommand*\lineheight[1]{\fontsize{\fsize}{#1\fsize}\selectfont}%
  \ifx\svgwidth\undefined%
    \setlength{\unitlength}{2025.31640625bp}%
    \ifx\svgscale\undefined%
      \relax%
    \else%
      \setlength{\unitlength}{\unitlength * \real{\svgscale}}%
    \fi%
  \else%
    \setlength{\unitlength}{\svgwidth}%
  \fi%
  \global\let\svgwidth\undefined%
  \global\let\svgscale\undefined%
  \makeatother%
  \begin{picture}(1,1)%
    \lineheight{1}%
    \setlength\tabcolsep{0pt}%
    \put(0,0){\includegraphics[width=\unitlength,page=1]{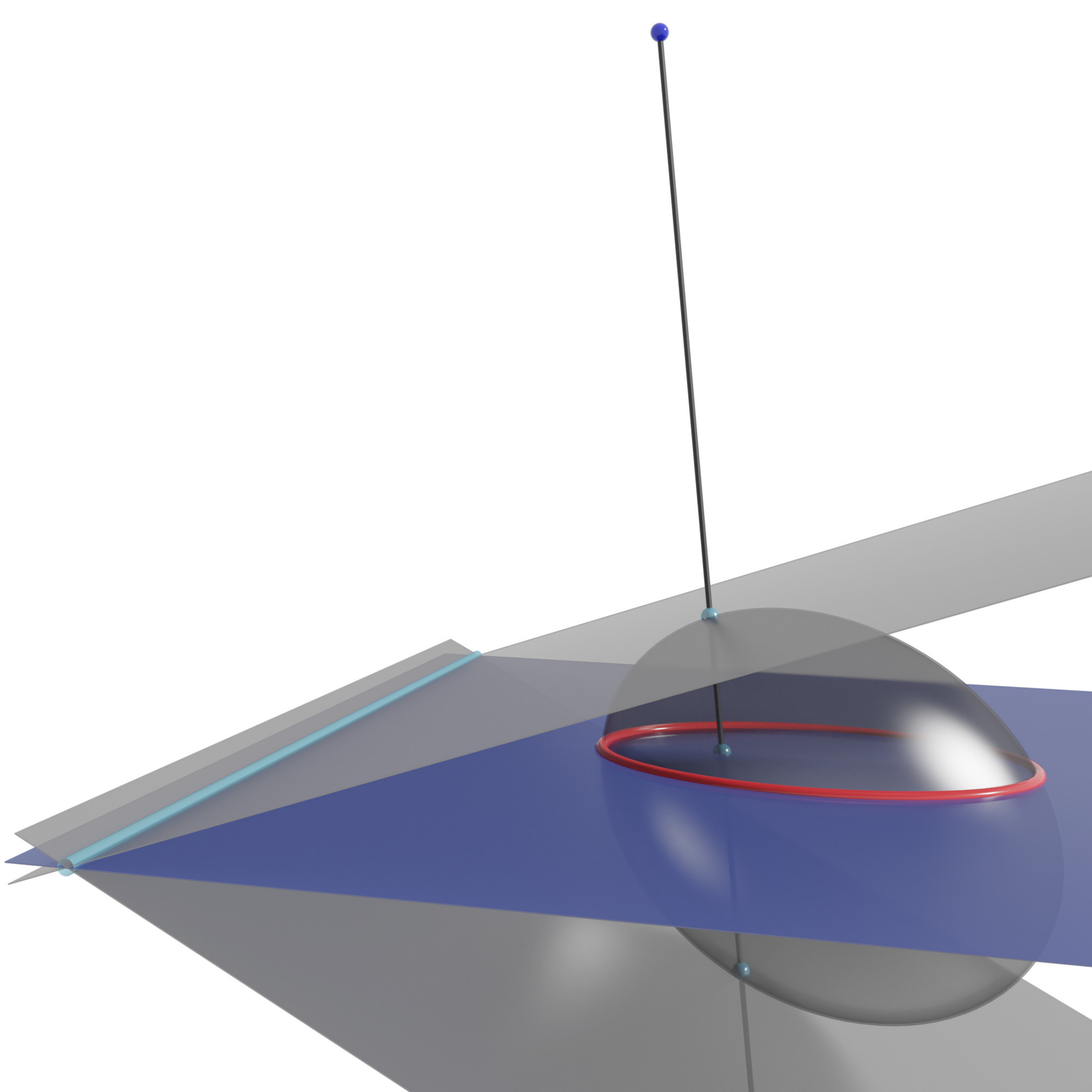}}%
    \put(0.6630064,0.44803619){\color[rgb]{0,0,0}\makebox(0,0)[lt]{\lineheight{1.25}\smash{\begin{tabular}[t]{l}{\footnotesize $\p{x}$}\end{tabular}}}}%
    \put(0.61724084,0.98144908){\color[rgb]{0,0,0}\makebox(0,0)[lt]{\lineheight{1.25}\smash{\begin{tabular}[t]{l}{\footnotesize $\p{p}$}\end{tabular}}}}%
    \put(0.7867122,0.21653265){\color[rgb]{0,0,0}\makebox(0,0)[lt]{\lineheight{1.25}\smash{\begin{tabular}[t]{l}{\footnotesize $\secquadric$}\end{tabular}}}}%
    \put(0.67722993,0.3167713){\color[rgb]{0,0,0}\makebox(0,0)[lt]{\lineheight{1.25}\smash{\begin{tabular}[t]{l}{\footnotesize $\pi_{\p{p}}(\p{x})$}\end{tabular}}}}%
    \put(0.09369771,0.3167713){\color[rgb]{0,0,0}\makebox(0,0)[lt]{\lineheight{1.25}\smash{\begin{tabular}[t]{l}{\footnotesize $\pi^*_{\p{p}}(\p{x})$}\end{tabular}}}}%
  \end{picture}%
\endgroup%

  \hspace{0.02\textwidth}
  \def\svgwidth{0.47\textwidth}
  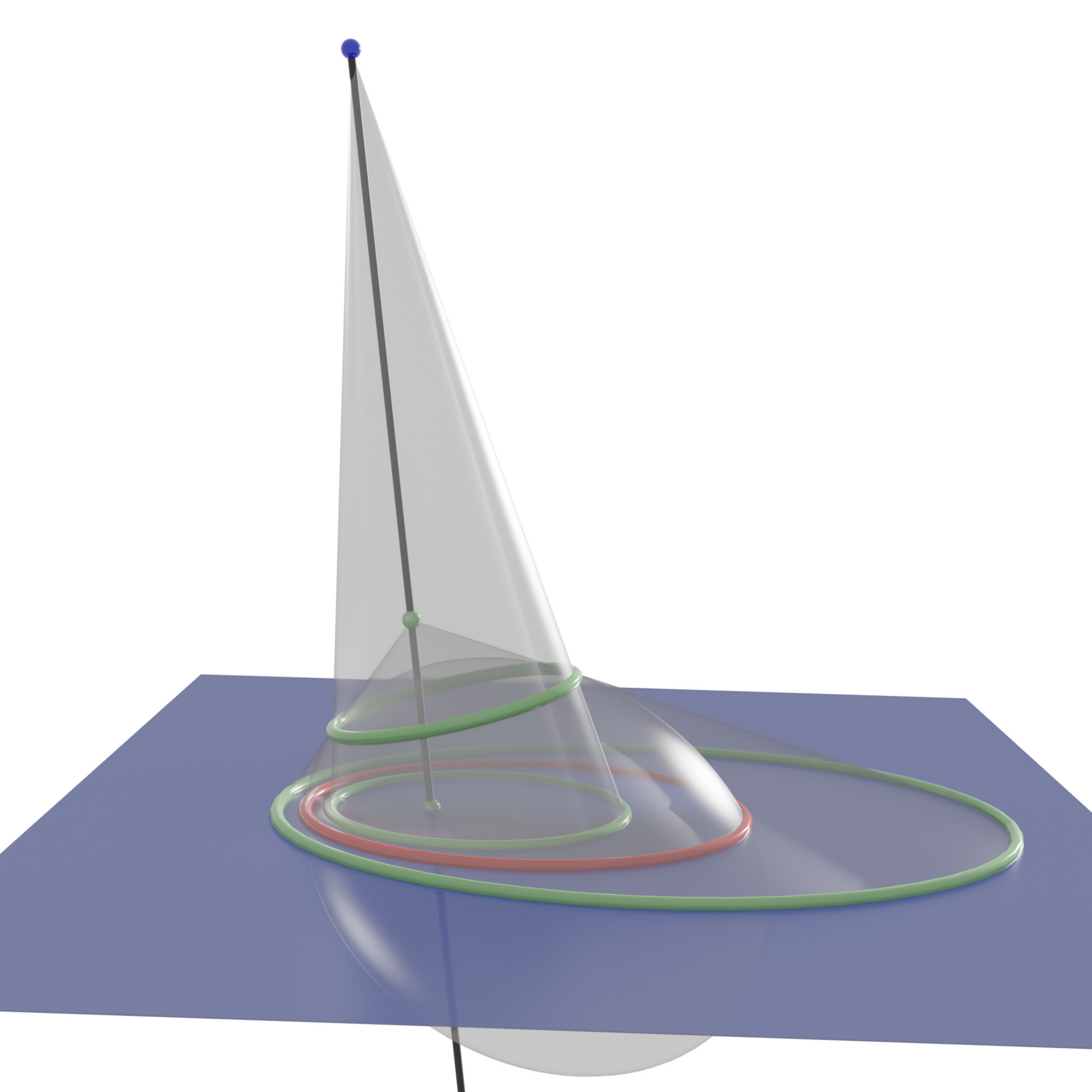
  \caption{
    \emph{Left:} Polar projection of points on the quadric $\quadric$.
    \emph{Right:} Polar projection of $\quadric$-spheres.
  }
\label{fig:polar-projection}
\end{figure}
Let $\quadric \subset \RP^{n+1}$ be a quadric.
We have seen that the projection $\pi_{\p{p}}$ of the quadric $\mathcal{Q}$ leads to a double cover of the points of $\pi_{\p{p}}(\quadric) \subset \p{p}^\perp$ 
(cf.\ Proposition \ref{prop:involution-projection}), i.e.\ the points ``inside'' or ``outside'' the quadric
\[
  \widetilde{\quadric} = \quadric \cap \p{p}^\perp.
\]
Correspondingly, the $\quadric$-spheres yield a double cover of the Cayley-Klein spheres in $\pi_{\p{p}}(\quadric)$
(cf.\ Proposition \ref{prop:sphere-double-cover}).
We now investigate the corresponding properties for polar hyperplanes and polar Cayley-Klein spheres.
\begin{definition}
  Let $\mathcal{Q} \subset \RP^{n+1}$ be a quadric and $\p{p} \in \RP^{n+1} \setminus \mathcal{Q}$.
  Then we call the map
  \[
    \pi^*_{\p{p}} : \p{x} \mapsto \p{x}^\perp \cap \p{p}^\perp = (\p{x} \wedge \p{p})^\perp,
  \]
  that maps a point $\p{x} \in \RP^{n+1}$ to the intersection of its polar hyperplane $\p{x}^\perp$ with $\p{p}^\perp$,
  the \emph{polar projection (associated with the point $\p{p}$)}.
\end{definition}
The projection $\pi_{\p{p}}$ and the polar projection $\pi^*_{\p{p}}$ map the same point to
a point in $\p{p}^\perp$ and its polar hyperplane respectively.
\begin{proposition}
  For a point $\p{x} \in \RP^{n+1}$ its projection $\pi_{\p{p}}(\p{x}) \in \p{p}^\perp$
  is the pole of its polar projection $\pi_{\p{p}}^*(\p{x}) \subset \p{p}^\perp$,
  where polarity in $\p{p}^\perp \simeq \RP^n$ is taken with respect to $\widetilde{\quadric}$.
\end{proposition}
If we restrict the polar projection $\pi_{\p{p}}^*$ to the quadric $\quadric$ we obtain a map to the hyperplanes of $\p{p}^\perp$,
which are poles of image points of the projection $\pi_{\p{p}}$ (see Figure \ref{fig:polar-projection}).
This map leads to a double cover of the polar hyperplanes (cf.\ Proposition \ref{prop:involution-projection}).
\begin{proposition}
  The restriction of the polar projection onto the quadric $\pi_{\p{p}}^*\restrict{\mathcal{Q}}$
  is a double cover of the set of all hyperplanes that are polar to the points in $\pi_{\p{p}}(\mathcal{Q})$
  with branch locus $\mathcal{Q} \cap \p{p}^\perp$.
\end{proposition}
\begin{remark}
  \label{rem:oriented-planes}
  The double cover can be interpreted as carrying the additional information of the orientation of these hyperplanes,
  where the involution $\sigma_{\p{p}}$ plays the role of \emph{orientation reversion}.
\end{remark}
By polarity every point $\p{x} \in \spheres$ corresponds to a $\quadric$-sphere $\p{x}^\perp \cap \quadric$ (see Definition \ref{def:Q-sphere}).
In the projection to $\pi_{\p{p}}(\quadric)$ it becomes a Cayley-Klein sphere (see Proposition \ref{prop:Q-sphere-projection}),
which is obtained from the point $\p{x}$ by the map
\[
  \sprojection{\p{p}} : \p{x} \mapsto \pi_{\p{p}}(\p{x}^\perp \cap \quadric)
\]
The polar projection $\pi_{\p{p}}^*$ of each point of a $\quadric$-sphere
yields a tangent plane of the polar Cayley-Klein sphere of $\sprojection{\p{p}}(\p{x})$ (see Definition \ref{def:polar-Cayley-Klein-sphere}).
The points of the polar Cayley-Klein sphere are therefore obtained by the map
\[
  \spprojection{\p{p}} : \p{x} \mapsto \cone{\quadric}{\p{x}} \cap \p{p}^\perp,
\]
where $\cone{\quadric}{\p{x}}$ is the cone of contact (see Definition \ref{def:cone-of-contact})
to $\quadric$ with vertex $\p{x}$ (see Figure \ref{fig:polar-projection}).
\begin{proposition}
  For $\p{x} \in \spheres$ the two Cayley-Klein spheres $\sprojection{\p{p}}(\p{x})$ and $\spprojection{\p{p}}(\p{x})$ are mutually polar
  Cayley-Klein spheres in $\p{p}^\perp$ with respect to $\widetilde{\quadric}$.
\end{proposition}
This leads to a polar version of Proposition \ref{prop:sphere-double-cover}.
\begin{proposition}
  \label{prop:polar-sphere-double-cover}
  The map $\spprojection{\p{p}}$ constitutes a double cover of the set of Cayley-Klein spheres
  which are polar to Cayley-Klein spheres in $\pi_{\p{p}}(\quadric)$ with respect to $\secquadric$.
  Its ramification points are given by $(\p{p}^\perp \cup \{\p{p}\}) \cap \spheres$,
  and its covering involution is $\sigma_{\p{p}}$.
\end{proposition}
\begin{remark}
  \label{rem:Laguerre-spheres}
  Following Remark \ref{rem:oriented-planes} we may endow the Cayley-Klein spheres that are polar to Cayley-Klein spheres
  in $\pi_{\p{p}}(\quadric)$ with an orientation by lifting them to planar sections of $\quadric$, i.e.\ $\quadric$-spheres.
  We call the planar section, or equivalently their oriented projections, \emph{Laguerre spheres (of $\pi_{\p{p}}(\quadric)$)}.
  The involution $\sigma_{\p{p}}$ acts on Laguerre spheres as orientation reversion.
\end{remark}
The Cayley-Klein distance of two points in $\spheres$ describes the Cayley-Klein tangent distance
between the two corresponding Cayley-Klein spheres in the projection to $\pi_{\p{p}}(\quadric)$.
This is the polar version of Proposition \ref{prop:sphere-angle}.
\begin{proposition}
  \label{prop:tangent-distance}
  Let $\p{x}_1, \p{x}_2 \in \spheres$ such that the corresponding $\quadric$-spheres intersect.
  Let
  \[
    \p{y} \in \quadric \cap \p{x}_1^\perp \cap \p{x}_2^\perp
  \]
  be a point in that intersection, and $\p{\tilde{y}} \coloneqq \pi_{\p{p}}^*(\p{y})$ its polar projection.
  Let $S_1, S_2$ be the two polar projected Cayley-Klein spheres corresponding to $\p{x}_1$, $\p{x}_2$ respectively
  \[
    S_1 \coloneqq \spprojection{\p{p}}(\p{x}_1), \qquad S_2 \coloneqq \spprojection{\p{p}}(\p{x}_2).
  \]
  Let $\p{\tilde{y}}_1$, $\p{\tilde{y}}_2$ be the two tangent points of $\p{\tilde{y}}$ to $S_1$, $S_2$ respectively.
  Then
  \[
    \ck{\quadric}{\p{x}_1}{\p{x}_2} = \ck{\widetilde{\quadric}}{\p{\tilde{y}}_1}{\p{\tilde{y}}_2}.
  \]
\end{proposition}
\begin{proof}
  Consider Proposition \ref{prop:sphere-angle}.
  By polarity, the intersection point of the spheres becomes a common tangent hyperplane,
  and the intersection angle becomes the distance of the two tangent points.
\end{proof}
\begin{remark}
  \label{rem:tangent-distance}
  Following Remark \ref{rem:oriented-planes} and Remark \ref{rem:Laguerre-spheres}, a common point in the lift
  of two (oriented) Laguerre spheres corresponds to a common oriented tangent hyperplane.
  Thus the Cayley-Klein distance on $\spheres$ is the Cayley-Klein tangent distance between two (oriented) Laguerre spheres
  (cf.\ Remark~\ref{rem:intersection-angle}~\ref{rem:intersection-angle-lift}).
\end{remark}

\subsection{Hyperbolic Laguerre geometry}
\label{sec:hyperbolic-laguerre-geometry}
When projecting down from Möbius geometry to hyperbolic geometry (cf.\ Section \ref{sec:hyperbolic-geometry-from-mobius})
we obtain a double cover of the points in hyperbolic space.
Hyperbolic planes, on the other hand, are represented by points in deSitter space, or ``outside'' hyperbolic space, by polarity.
Thus, to obtain hyperbolic Laguerre geometry, instead of the Möbius quadric,
we choose a quadric that projects to deSitter space.
\begin{definition}\
  \label{def:laguerre-hyperbolic}
  \nobreakpar
  \begin{enumerate}
  \item We call the quadric
    \[
      \laghyp \subset \RP^{n+1}
    \]
    corresponding to the standard bilinear form of signature $(n,2)$ in $\R^{n+2}$, i.e.,
    \[
      \scalarprod{x}{y} \coloneqq x_1y_1 + \ldots + x_ny_n - x_{n+1}y_{n+1} - x_{n+2}y_{n+2}
    \]
    for $x, y \in \R^{n+2}$, the \emph{hyperbolic Laguerre quadric}.
  \item The corresponding transformation group
    \[
      \laghyptrafos \coloneqq \PO(n,2)
    \]
    is called the group of \emph{hyperbolic Laguerre transformations}.
  \end{enumerate}
\end{definition}
To recover hyperbolic space in the projection, choose a point $\p{p} \in \RP^{n+1}$ with $\scalarprod{p}{p} < 1$, w.l.o.g.,
\[
  \p{p} \coloneqq [e_{n+2}] = [0, \ldots, 0, 1].
\]
The corresponding involution and projection take the form
\[
  \begin{aligned}
    &\sigma_{\p{p}} : [x_1, \ldots, x_{n+1}, x_{n+2}] \mapsto [x_1, \ldots, , x_{n+1}, -x_{n+2}],\\
    &\pi_{\p{p}} : [x_1, \ldots, x_{n+1}, x_{n+2}] \mapsto [x_1, \ldots, , x_{n+1}, 0].
  \end{aligned}
\]
The quadric in the polar hyperplane of $\p{p}$
\[
  \widetilde{\mob} \coloneqq \laghyp \cap \p{p}^\perp
\]
has signature $(n,1)$, and its ``inside'' $\hyp = \widetilde{\mob}^-$ can be identified with $n$-dimensional hyperbolic space,
while its ``outside'' $\ds = \widetilde{\mob}^+$ can be identified with $n$-dimensional deSitter space (cf.~Section \ref{sec:hyperbolic-space}).

Under the projection $\pi_{\p{p}}$ the hyperbolic Laguerre quadric projects down to the compactified deSitter space
\[
  \pi_{\p{p}}(\laghyp) = \cds = \ds \cup \widetilde{\mob}.
\]
Thus the polar projection $\pi_{\p{p}}^*$ of a point on $\laghyp$ yields a hyperbolic hyperplane,
where the double cover can be interpreted as encoding the orientation of that hyperplane (see Figure~\ref{fig:hyperbolic-laguerre-geometry-lines}).
\begin{remark}
  The quadric $\laghyp$ is the projective version of the hyperboloid $\widetilde{\dS}^n$ introduced in Section \ref{sec:hyperbolic-space}
  as a double cover of deSitter space.
\end{remark}

We call the hyperplanar sections of $\laghyp$, i.e.\ the $\laghyp$-spheres, \emph{hyperbolic Laguerre spheres}.
By polarity, we identify the space of hyperbolic Laguerre spheres with the whole space
\begin{equation}
  \label{eq:laguerre-spheres-hyperbolic}
  \spheres = \RP^{n+1}.
\end{equation}
\begin{remark}
  As we have done in Section \ref{sec:elementary}, one might want to exclude oriented hyperplanes from the space of hyperbolic Laguerre spheres and thus take $\spheres = \RP^{n+1} \setminus \laghyp$ (cf.\ Remark~\ref{rem:Q-spheres}~\ref{rem:Q-spheres-degenerate}).
\end{remark}
Under the polar projection $\spprojection{\p{p}}$ points in $\spheres$ are mapped to the spheres that are polar to deSitter spheres
(see Table \ref{tab:hyperbolic-laguerre-geometry} and Figure~\ref{fig:hyperbolic-laguerre-geometry-circles}).
\begin{theorem}
  \label{thm:hyperbolic-Laguerre-spheres}
  Under the map
  \[
    \spprojection{\p{p}} : \p{x} \mapsto \cone{\laghyp}{\p{x}} \cap \p{p}^\perp
  \]
  a point $\p{x} \in \spheres = \RP^{n+1}$
  \begin{itemize}
  \item with $\p{x} \in \laghyp$ is mapped to a \emph{hyperbolic hyperplane} with pole $\pi_{\p{p}}(\p{x}) \in \cds$,
  \item with $\pscalarprod{p}{x}{x} < 0$ is mapped to a \emph{hyperbolic sphere} in $\chyp$ with center $\pi_{\p{p}}(\p{x})$. In the normalization $\pscalarprod{p}{x}{x} = -1$ its hyperbolic radius is given by $r \geq 0$, where $\sinh^2 r = x_{n+2}^2$,
  \item with $\pscalarprod{p}{x}{x} = 0$ is mapped to a \emph{hyperbolic horosphere}.
  \item with $\pscalarprod{p}{x}{x} > 0$ and $\scalarprod{x}{x} < 0$ is mapped to a \emph{hyperbolic surface of constant distance} in $\chyp$
    to a hyperbolic hyperplane with pole $\pi_{\p{p}}(x)$,
    In the normalization $\pscalarprod{p}{x}{x} = 1$ its hyperbolic distance is given by $r \geq 0$, where $\cosh^2 r = x_{n+2}^2$.
  \item with $\pscalarprod{p}{x}{x} > 0$ and $\scalarprod{x}{x} > 0$ is mapped to a \emph{deSitter sphere} in $\cds$ with center $\pi_{\p{p}}(\p{x})$.
    In the normalization $\pscalarprod{p}{x}{x} = 1$ its deSitter radius is given by $r \geq 0$, where $\cos^2 r = x_{n+2}^2$.

  %\item with $\p{x} \in \p{p}^\perp$, i.e.\ $x_{n+2} = 0$, is mapped to a deSitter \emph{null-sphere} in $\cds$ with center $\p{x}$,
  \end{itemize}
\end{theorem}
\begin{remark}\
  \label{rem:hyperbolic-Laguerre-spheres}
  \nobreakpar
  \begin{enumerate}
  \item
    \label{rem:hyperbolic-Laguerre-spheres-distinction}
    The points $\p{x}$ representing hyperbolic spheres/distance hypersurfaces/horospheres can be distinguished
    from the points representing deSitter spheres by the first satisfying $\scalarprod{x}{x} < 0$,
    i.e.\ lying ``inside'' the hyperbolic Laguerre quadric,
    and the latter satisfying $\scalarprod{x}{x} > 0$,
    i.e.\ lying ``outside'' the hyperbolic Laguerre quadric.
  \item The map $\spprojection{\p{p}}$ is a double cover of the spheres described in
    Theorem \ref{thm:hyperbolic-Laguerre-spheres}, branching on the subset of hyperbolic points, and deSitter null-spheres
    (see Proposition \ref{prop:polar-sphere-double-cover}).
    We interpret the lift to carry the orientation of the hyperbolic Laguerre spheres.
    Upon the normalization given in Theorem \ref{thm:hyperbolic-Laguerre-spheres} the orientation is encoded in the sign of the $x_{n+2}$-component.
    The involution $\sigma_{\p{p}}$ acts on the set of hyperbolic Laguerre spheres as orientation reversion.
  \item The Cayley-Klein distance induced on $\spheres$ by the hyperbolic Laguerre quadric $\laghyp$ measures the Cayley-Klein tangent distance
    between the corresponding hyperbolic Laguerre spheres (see Proposition \ref{prop:tangent-distance})
    if they possess a common oriented tangent hyperplane (see Remark~\ref{rem:tangent-distance}).
  %\item The hyperbolic Laguerre quadric $\hyplag$ contains one-dimensional isotropic subspaces (cf.\ Lemma \ref{lem:isotropic-subspaces}). Those are mapped under the polar projection $\pi_{\p{p}}^*$ to parallel hyperbolic subspaces..
  \item Using the projection $\pi_{\p{p}}$ instead of the polar projection $\pi_{\p{p}}^*$ hyperbolic Laguerre geometry
    may be interpreted as the ``Möbius geometry'' of deSitter space (cf.\ Section \ref{sec:classification}).
  \end{enumerate}
\end{remark}
\begin{table}[H]
  \centering
  \def\arraystretch{2.0}
  \begin{tabular}{c|c}
    \textbf{hyperbolic geometry} & \textbf{Laguerre geometry}\\
    \hline\hline
    \makecell{\emph{hyperplane}\\ with pole $\p{y} \in \ds$, $y \in \dS^n$} & $[y, \pm 1] \in \laghyp$ \\
    \hline
    \makecell{\emph{sphere}\\ with center $\p{y} \in \hyp$, $y \in \HH^n$\\ and radius $r > 0$} & $[y, \pm \sinh r] \in \laghyp^- \cap \cone{\laghyp}{\p{p}}^-$ \\
    \hline
    \makecell{\emph{horosphere}\\
    with center $\p{y} \in \widetilde{\mob}$} & $[y, \pm e^r] \in \hyplag^- \cap \cone{\hyplag}{\p{p}}$ \\
    \hline
    \makecell{\emph{surface of constant distance}\\ $r > 0$ to a hyperplane\\ with pole $\p{y} \in \ds$, $y \in \dS^n$} & $[y, \pm \cosh r] \in \laghyp^- \cap \cone{\laghyp}{\p{p}}^+$ \\
    \hline
    %\makecell{\emph{deSitter null-sphere}\\ with center $\p{y} \in \ds$, $y \in \dS^n$} & $[y, 0] \in \laghyp^+ \cap \p{p}^\perp$ \\
    %\hline
    \makecell{\emph{deSitter sphere}\\ with center $\p{y} \in \ds$, $y \in \dS^n$\\ and deSitter radius $r > 0$} & $[y, \pm \cos r] \in \laghyp^+$
  \end{tabular}
  \caption{
    Laguerre spheres in hyperbolic Laguerre geometry.
  }
  \label{tab:hyperbolic-laguerre-geometry}
\end{table}

\subsubsection{Hyperbolic Laguerre transformations}
\label{sec:hyperbolic-laguerre-transformations}
\begin{figure}
  \centering
  \includegraphics[width=0.32\textwidth]{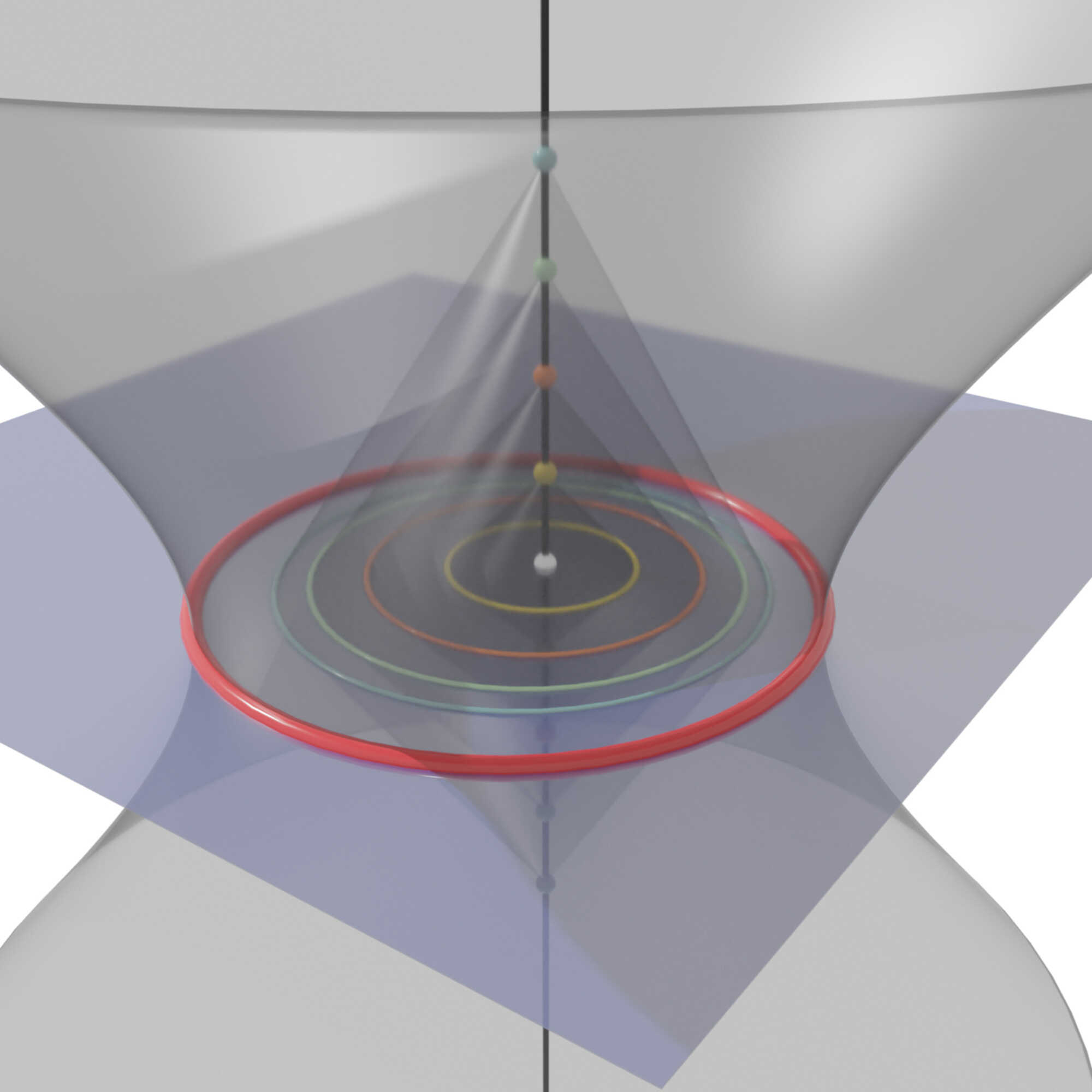}
  \includegraphics[width=0.32\textwidth]{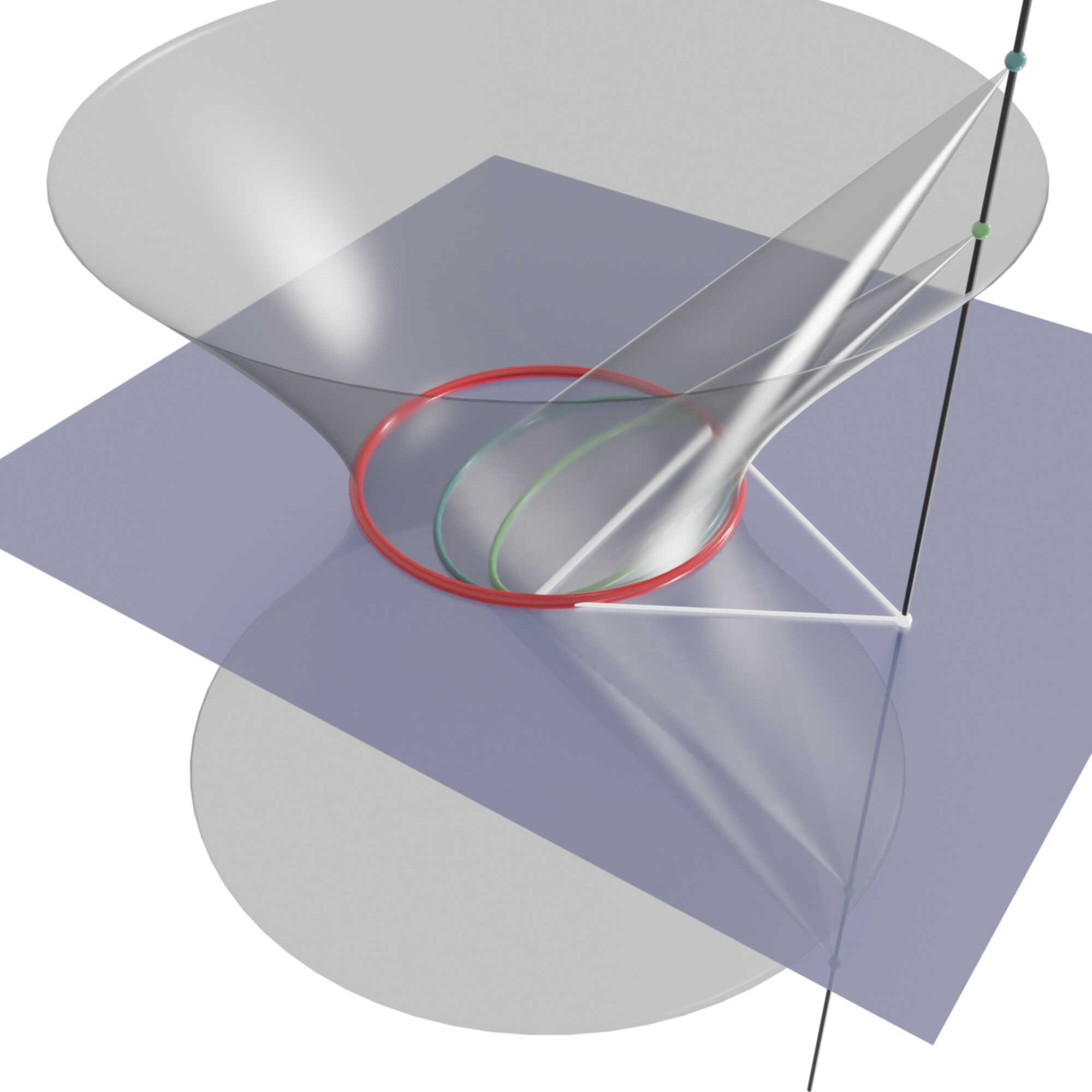}
  \includegraphics[width=0.32\textwidth]{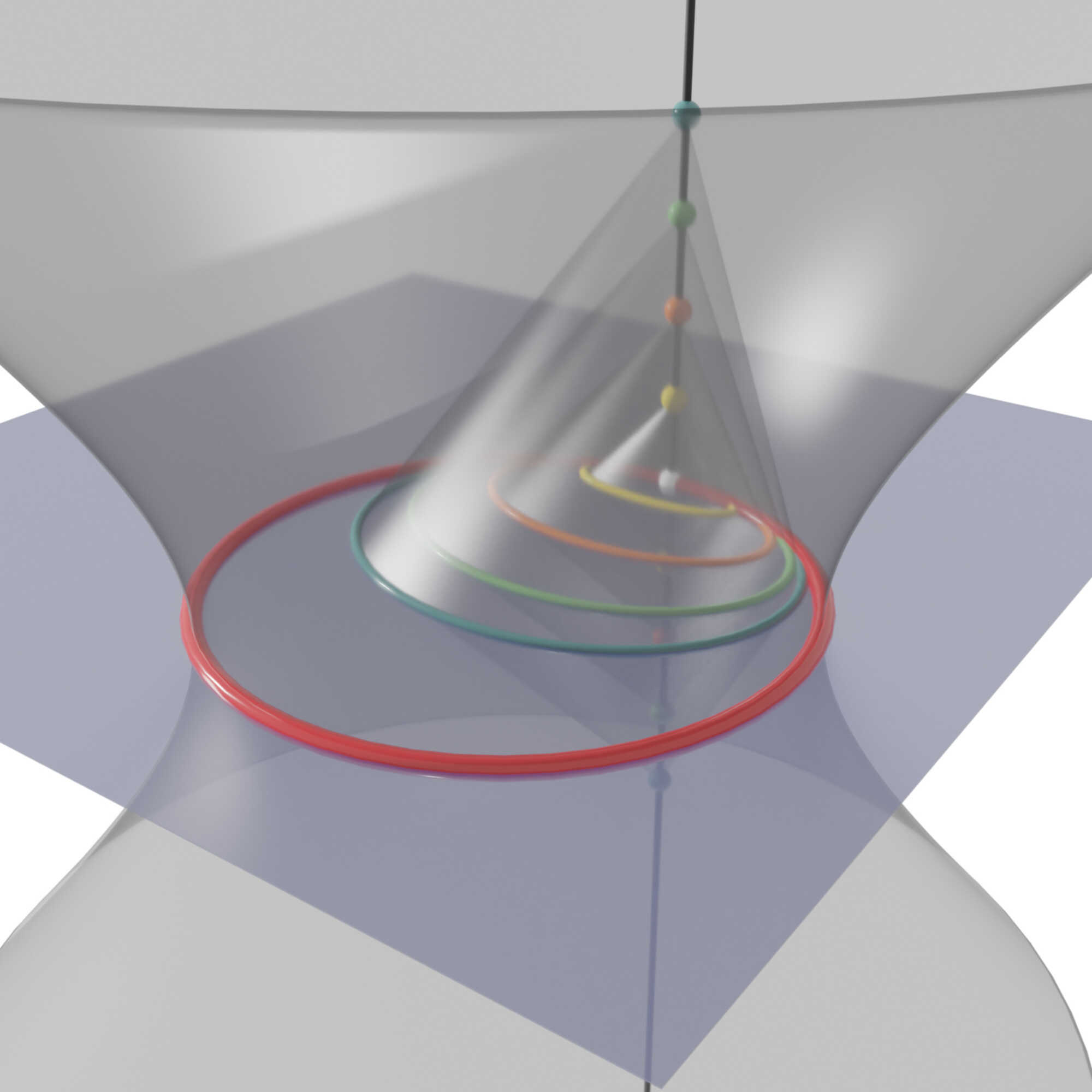}
  \caption{
    Concentric circles in the hyperbolic plane.
    \emph{Left:} Concentric hyperbolic circles.
    \emph{Middle:} Curves of constant distance to a common line.
    \emph{Right:} Concentric horocycles with center on the absolute conic.
  }
  \label{fig:hyperbolic-concentric-circles}
\end{figure}
Every (local) transformation mapping (non-oriented) hyperbolic hyperplanes to hyperbolic hyperplanes (not necessarily points to points)
while preserving (tangency to) hyperbolic spheres can be lifted and extended to a hyperbolic Laguerre transformation
(see Theorem \ref{thm:mobius-transformation-lift}).

The hyperbolic Laguerre group
\[
  \laghyptrafos = \PO(n,2)
\]
preserves the hyperbolic Laguerre quadric $\laghyp$
and maps planar sections of $\laghyp$ to planar sections of $\laghyp$.
Under the polar projection this means it maps oriented hyperplanes to oriented hyperplanes,
or hyperbolic Laguerre spheres to hyperbolic Laguerre spheres,
while preserving the tangent distance and in particular the oriented contact (see Remark~\ref{rem:tangent-distance}).

The hyperbolic Laguerre group contains (doubly covers) the group of hyperbolic isometries as $\PO(n,2)_{\p{p}}$.
To generate the whole Laguerre group we only need to add three specific one-parameter families of scalings along concentric Laguerre spheres (see Remark~\ref{rem:mobius-decomposition} and Figure~\ref{fig:hyperbolic-concentric-circles}).
\begin{itemize}
\item Consider the family of transformations
  \[
    T^{\text{(s)}}_t \coloneqq
    \left[
      \def\arraystretch{1.}
      \begin{array}{c|cc}
        I_n & \begin{matrix}0\\\vdots\\0\end{matrix} & \begin{matrix}0\\\vdots\\0\end{matrix}\\
        \hline
        \begin{matrix}0&\cdots&0\end{matrix} & \cos t & \sin t\\
        \begin{matrix}0&\cdots&0\end{matrix} & -\sin t & \cos t
      \end{array}
    \right]
    \qquad
    \text{for}\;t\in[-\pi/2,\pi/2].
  \]
  It maps the absolute $\p{p} = [0,\ldots,0,1]$ to
  \[
    T^{\text{(s)}}_t(\p{p}) = [0, \ldots, 0, \sin t, \cos t],
  \]
  which is a hyperbolic sphere with center $[0,\ldots,1,0]$.
  It turns from the absolute for $t=0$ into a point for $t=\pm\pi/2$,
  while changing orientation when it passes through the center or through the absolute,
  i.e.\ when $\cos t\cdot\sin t$ changes sign.
\item Consider the family of transformations
  \[
    T^{\text{(c)}}_t \coloneqq
    \left[
      \def\arraystretch{1.}
      \begin{array}{c|ccc}
            I_{n-1} & \begin{matrix}0\\\vdots\\0\end{matrix} & \begin{matrix}0\\\vdots\\0\end{matrix} & \begin{matrix}0\\\vdots\\0\end{matrix}\\  
            \hline
            \begin{matrix}0&\cdots&0\end{matrix}& \cosh t & 0 & \sinh t\\
            \begin{matrix}0&\cdots&0\end{matrix}& 0 & 1 & 0\\
            \begin{matrix}0&\cdots&0\end{matrix}& \sinh t & 0 & \cosh t
      \end{array}
    \right]
    \qquad
    \text{for}\;t\in\R\,
  \]
  It maps the absolute $\p{p} = [0,\ldots,0,1]$ to
  \[
    T^{\text{(c)}}_t(\p{p}) = [0, \ldots, 0, \sinh t, 0, \cosh t],
  \]
  which is an oriented hypersurface of constant distance to the hyperbolic hyperplane $[0,\ldots,0,1,0,1]$.
  It turns from the absolute for $t=0$ into the hyperplane for $t=\infty$,
  while changing orientation when it passes through the absolute, i.e. when $t$ changes sign.
\item Consider the family of transformations
  \[
    T^{\text{(h)}}_t \coloneqq
    \left[
      \def\arraystretch{1.1}
      \begin{array}{c|cc|c}
        I_{n-1} & \begin{matrix}0\\\vdots\\0\end{matrix}& \begin{matrix}0\\\vdots\\0\end{matrix}& \begin{matrix}0\\\vdots\\0\end{matrix}\\
        \hline
        \begin{matrix}0&\cdots&0\end{matrix}& 1+\tfrac{t^2}{2} & \tfrac{t^2}{2} &t\\
        \begin{matrix}0&\cdots&0\end{matrix}& -\tfrac{t^2}{2} & 1-\tfrac{t^2}{2} &-t\\
        \hline
        \begin{matrix}0&\cdots&0\end{matrix}& t & t & 1
      \end{array}
    \right]
    \qquad
    \text{for}\;t\in\R.
  \]
  It maps the absolute $\p{p} = [0, \ldots, 0, 1]$ to
  \[
    T^{\text{(h)}}_t(\p{s}) = [0, \ldots, 0, t, -t, 1],
  \]
  which is a horosphere with center $[0, \ldots, 0, 1, -1, 0]$ on the absolute.
  It turns from the absolute for $t=0$ into the center for $t=\infty$,
  while changing orientation when $t$ changes its sign.
\end{itemize}
\begin{remark}
  \label{remark:character_not_preserved}
  While Laguerre transformations preserve oriented contact
  they do not preserve the notion of sphere, horosphere and
  constant distance surface. For example the transformation $T^{\text{(s)}}_{\pi/2}$
  transforms the origin into the absolute and thus turns all
  spheres which contain the origin into horospheres.
\end{remark}
Now the hyperbolic Laguerre group can be generated by hyperbolic motions
and the three introduced one-parameter families of scalings (see Remark \ref{rem:mobius-decomposition}).
\begin{theorem}
  Every hyperbolic Laguerre transformation $f \in \PO(n,2)$ can be written as
  \[
    f = \Phi T_t\Psi,
  \]
  where $\Phi, \Psi \in \PO(n,2)_{\p{p}}$ are hyperbolic motions
  and $T_t\in\{T^{\mathrm{(s)}}_t,T^{\mathrm{(c)}}_t,T^{\mathrm{(h)}}_t\}$ a scaling for some $t\in\R$.
\end{theorem}

\subsection{Elliptic Laguerre geometry}
\label{sec:elliptic-laguerre-geometry}
When projecting down from Möbius geometry to elliptic geometry (cf.\ Section \ref{sec:elliptic-geometry-from-mobius})
we obtain a double cover of the points in elliptic space.
Since every elliptic hyperplane has a pole in the elliptic space, this equivalently leads to a double cover of the elliptic hyperplanes.
\begin{definition}\
  \label{def:laguerre-elliptic}
  \nobreakpar
  \begin{enumerate}
  \item We call the quadric
    \[
      \lagell \subset \RP^{n+1}
    \]
    corresponding to the standard bilinear form of signature $(n+1,1)$ in $\R^{n+2}$, i.e.,
    \[
      \scalarprod{x}{y} \coloneqq x_1y_1 + \ldots + x_{n+1}y_{n+1} - x_{n+2}y_{n+2}
    \]
    for $x, y \in \R^{n+2}$ the \emph{elliptic Laguerre quadric}.
  \item The corresponding transformation group
    \[
      \lagelltrafos \coloneqq \PO(n+1,1) \simeq \mobtrafos
    \]
    is called the group of \emph{elliptic Laguerre transformations}.
  \end{enumerate}
\end{definition}
\begin{remark}
  Thus, $n$-dimensional elliptic Laguerre geometry is isomorphic to $n$-dimensional Möbius geometry.
\end{remark}
To recover elliptic space in the projection, choose a point $\p{p} \in \RP^{n+1}$ with $\scalarprod{p}{p} < 1$, w.l.o.g.,
\[
  \p{p} \coloneqq [e_{n+2}] = [0, \ldots, 0, 1].
\]
The corresponding involution and projection take the form
\[
  \begin{aligned}
    &\sigma_{\p{p}} : [x_1, \ldots, x_{n+1}, x_{n+2}] \mapsto [x_1, \ldots, , x_{n+1}, -x_{n+2}],\\
    &\pi_{\p{p}} : [x_1, \ldots, x_{n+1}, x_{n+2}] \mapsto [x_1, \ldots, , x_{n_1}, 0].
  \end{aligned}
\]
The quadric in the polar hyperplane of $\p{p}$
\[
  \ellipb = \lagell \cap \p{p}^\perp
\]
has signature $(n+1,0)$, and its non-empty side $\ellip = \ellipb^+$ can be identified with $n$-dimensional elliptic space
(see Section \ref{sec:elliptic-space})

Under the projection $\pi_{\p{p}}$ the elliptic Laguerre quadric projects down to the elliptic space
\[
  \pi_{\p{p}}(\lagell) = \ellip
\]
Thus, the polar projection $\pi_{\p{p}}^*$ of every point on $\lagell$ yields an elliptic hyperplane,
where the double cover can be interpreted as carrying the orientation of that hyperplane.
\begin{remark}
  The quadric $\lagell$ is the projective version of the sphere $\S^n$ introduced in Section~\ref{sec:elliptic-space}
  as a double cover of elliptic space.
\end{remark}

We call the hyperplanar sections of $\lagell$, i.e.\ the $\lagell$-spheres, \emph{elliptic Laguerre spheres}.
By polarity, we identify the space of hyperbolic Laguerre spheres with the the outside of $\lagell$
\begin{equation}
  \label{eq:laguerre-spheres-elliptic}
  \spheres = \lagell^+ \cup \lagell.
\end{equation}
\begin{remark}
  As we have done in Section \ref{sec:elementary}, one might want to exclude oriented hyperplanes from the space of elliptic Laguerre spheres and thus take $\spheres = \lagell^+$ (cf.\ Remark~\ref{rem:Q-spheres}~\ref{rem:Q-spheres-degenerate}).
\end{remark}
Under the polar projection $\spprojection{\p{p}}$ points in $\spheres$ are mapped to spheres that are polar to elliptic spheres,
i.e.\ they are mapped to all elliptic spheres (Table~\ref{tab:elliptic-laguerre-geometry} and Figure~\ref{fig:elliptic-laguerre-geometry}).
\begin{theorem}
  \label{thm:elliptic-Laguerre-spheres}
  Under the map
  \[
    \spprojection{\p{p}} : \p{x} \mapsto \cone{\laghyp}{\p{x}} \cap \p{p}^\perp
  \]
  a point $\p{x} \in \spheres = \RP^{n+1}$ is mapped to an \emph{elliptic sphere} in $\ellip$ with center $\pi_{\p{p}}(\p{x})$.
  In the normalization $\pscalarprod{p}{x}{x} = 1$ its elliptic radius is given by $r \geq 0$, where $x_{n+2}^2 = \sin^2r$.
  In particular, $\p{x} \in \lagell$ is mapped to an \emph{elliptic hyperplane} with pole $\pi_{\p{p}}(\p{x}) \in \ellip$.
  % \begin{itemize}
  % \item $\p{x} \in \lagell$ is mapped to an \emph{elliptic hyperplane} with pole $\pi_{\p{p}}(\p{x}) \in \ellip$,
  % \item $\p{x} \in \p{p}^\perp = \ellip$, i.e.\ $x_{n+2} = 0$, is mapped to an \emph{elliptic point} $\p{x} \in \ellip$.
  % \end{itemize}
\end{theorem}
\begin{remark}\
  \nobreakpar
  \begin{enumerate}
  \item The map $\spprojection{\p{p}}$ is a double cover of elliptic spheres,
    branching on the subset of elliptic points
    (see Proposition \ref{prop:polar-sphere-double-cover}).
    We interpret the lift to carry the orientation of the elliptic Laguerre spheres.
    Upon the normalization given in Theorem \ref{thm:elliptic-Laguerre-spheres} the orientation is encoded in the sign of the $x_{n+2}$-component.
    The involution $\sigma_{\p{p}}$ acts on the set of elliptic Laguerre spheres as orientation reversion.
  \item The Cayley-Klein distance induced on $\spheres$ by the elliptic Laguerre quadric $\lagell$ measures the Cayley-Klein tangent distance
    between the corresponding elliptic Laguerre spheres (see Proposition \ref{prop:tangent-distance})
    if they possess a common oriented tangent hyperplane (see Remark \ref{rem:tangent-distance}).
  \item Using the projection $\pi_{\p{q}}$ instead of the polar projection $\pi_{\p{q}}^*$ elliptic Laguerre geometry
    coincides with Möbius geometry (cf.\ Section \ref{sec:mobius-geometry} and Section \ref{sec:classification}).
  \end{enumerate}
\end{remark}
\begin{table}[H]
  \centering
  \def\arraystretch{2.0}
  \begin{tabular}{c|c}
    \textbf{elliptic geometry} & \textbf{Laguerre geometry}\\
    \hline\hline
    \makecell{\emph{hyperplane}\\ with pole $\p{y} \in \ellip$, $y \in \S^n$} & $[y, \pm 1] \in \lagell$ \\
    \hline
    \makecell{\emph{sphere}\\ with center $\p{y} \in \ellip$, $y \in \S^n$\\ and radius $r > 0$} & $[y, \pm \sin r] \in \lagell^+$
    % \hline
    % \makecell{\emph{point}\\ $\p{y} \in \ellip$, $y \in \S^n$} & $[y, 0] \in \lagell^+ \cap \p{p}^\perp$
  \end{tabular}
  \caption{
    The lifts of elliptic spheres to elliptic Laguerre geometry.
  }
  \label{tab:elliptic-laguerre-geometry}
\end{table}

\subsubsection{Elliptic Laguerre transformations}
\label{sec:elliptic-laguerre-transformations}
\begin{figure}
  \centering
  \includegraphics[width=0.45\textwidth]{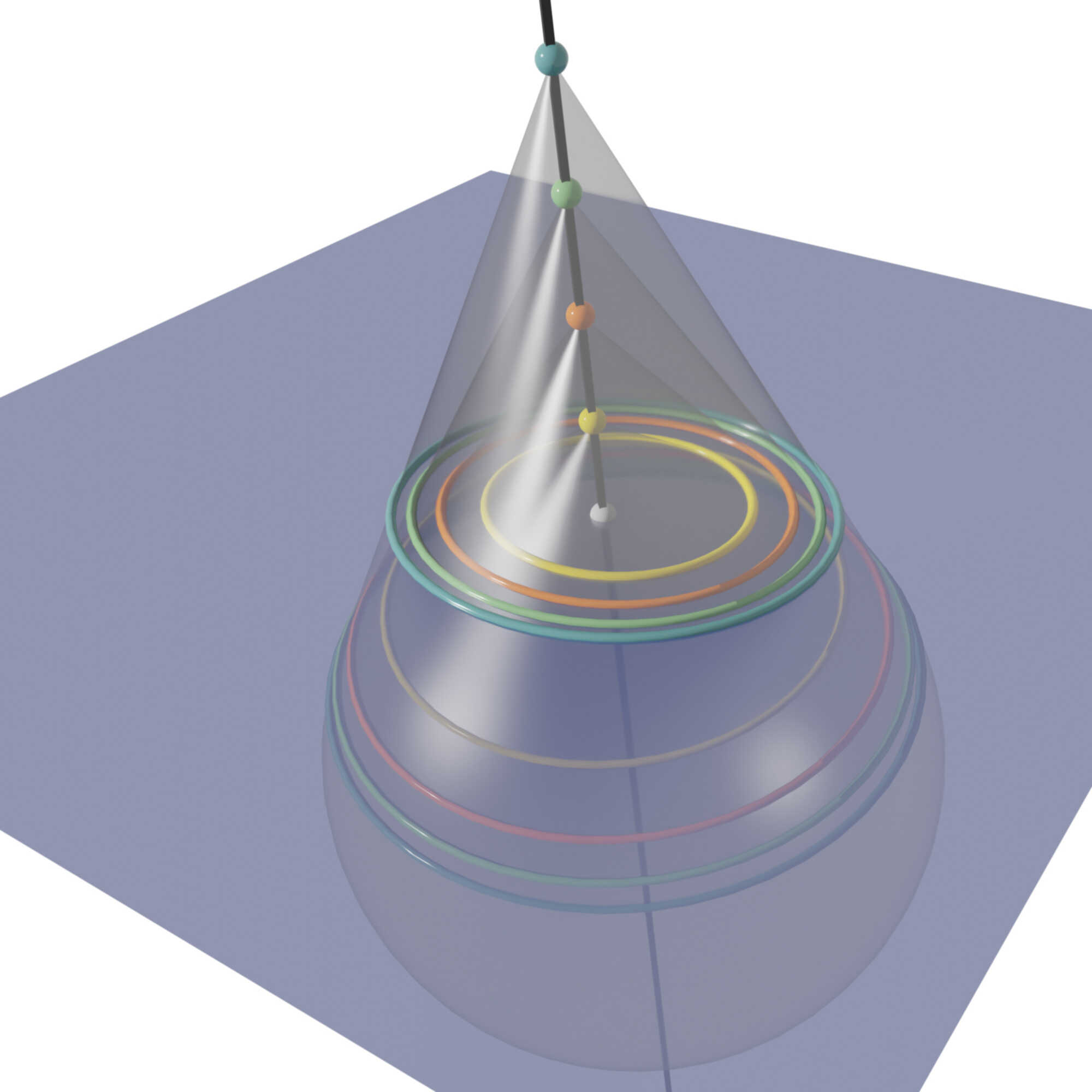}
  \caption{
    Concentric circles in the elliptic plane.
  }
  \label{fig:elliptic-concentric-circles}
\end{figure}
Every (local) transformation mapping (non-oriented) elliptic hyperplanes to elliptic hyperplanes (not necessarily points to points)
while preserving (tangency to) elliptic spheres can be lifted and extended to an elliptic Laguerre transformation
(see Theorem \ref{thm:mobius-transformation-lift}).

The elliptic Laguerre group
\[
  \lagelltrafos = \PO(n+1,1)
\]
preserves the elliptic Laguerre quadric $\lagell$
and maps planar sections of $\lagell$ to planar sections of $\lagell$.
Under the polar projection this means it maps oriented hyperplanes to oriented hyperplanes,
or elliptic Laguerre spheres to elliptic Laguerre spheres,
while preserving the tangent distance and in particular the oriented contact (see Remark \ref{rem:tangent-distance}).

The elliptic Laguerre group contains (doubly covers) the group of elliptic isometries as ${\PO(n+1,1)_{\p{p}}}$.
To generate the whole Laguerre group we only need to add one specific one-parameter family of scalings along concentric Laguerre spheres
(see Remark~\ref{rem:mobius-decomposition} and Figure~\ref{fig:elliptic-concentric-circles}).
\begin{itemize}
\item Consider the family of transformations
  \[
    S_t \coloneqq
    \left[
      \def\arraystretch{1.}
      \begin{array}{c|cc}
        I_n & \begin{matrix}0\\\vdots\\0\end{matrix} & \begin{matrix}0\\\vdots\\0\end{matrix}\\
        \hline
        \begin{matrix}0&\cdots&0\end{matrix} & \cosh t & \sinh t\\
        \begin{matrix}0&\cdots&0\end{matrix} & \sinh t & \cosh t
      \end{array}
    \right]
    \qquad
    \text{for}\;t\in\R.
  \]
  It maps the absolute $\p{p} = [0,\ldots,0,1]$ to
  \[
    T^{\text{(s)}}_t(\p{p}) = [0, \ldots, 0, \sinh t, \cosh t],
  \]
  which is an elliptic sphere with center $[0,\ldots,1,0]$.
  It turns from the absolute for $t=0$ into a point for $t=\infty$,
  while changing orientation when it passes through the center or through the absolute,
  i.e.\ when $t$ changes sign.
\end{itemize}
Now the elliptic Laguerre group can be generated by elliptic motions
and this one-parameter family of scalings (see Remark \ref{rem:mobius-decomposition}).
\begin{theorem}
  Any elliptic Laguerre transformation $f \in \PO(n+1,1)$ can be written as
  \[
    f = \Phi S_t\Psi,
  \]
  where $\Phi, \Psi \in \PO(n+1,1)_{\p{p}}$ are elliptic motions
  and $t \in \R$.
\end{theorem}

\newpage
\section{Lie geometry} 
\label{sec:lie}
\begin{figure}
  \centering
  \def\svgwidth{0.47\textwidth}
%% Accompanies image file 'lie_quadric.pdf' (pdf, eps, ps)
\begingroup%
  \makeatletter%
  \providecommand\color[2][]{%
    \errmessage{(Inkscape) Color is used for the text in Inkscape, but the package 'color.sty' is not loaded}%
    \renewcommand\color[2][]{}%
  }%
  \providecommand\transparent[1]{%
    \errmessage{(Inkscape) Transparency is used (non-zero) for the text in Inkscape, but the package 'transparent.sty' is not loaded}%
    \renewcommand\transparent[1]{}%
  }%
  \providecommand\rotatebox[2]{#2}%
  \newcommand*\fsize{\dimexpr\f@size pt\relax}%
  \newcommand*\lineheight[1]{\fontsize{\fsize}{#1\fsize}\selectfont}%
  \ifx\svgwidth\undefined%
    \setlength{\unitlength}{2025.31640625bp}%
    \ifx\svgscale\undefined%
      \relax%
    \else%
      \setlength{\unitlength}{\unitlength * \real{\svgscale}}%
    \fi%
  \else%
    \setlength{\unitlength}{\svgwidth}%
  \fi%
  \global\let\svgwidth\undefined%
  \global\let\svgscale\undefined%
  \makeatother%
  \begin{picture}(1,1)%
    \lineheight{1}%
    \setlength\tabcolsep{0pt}%
    \put(0,0){\includegraphics[width=\unitlength,page=1]{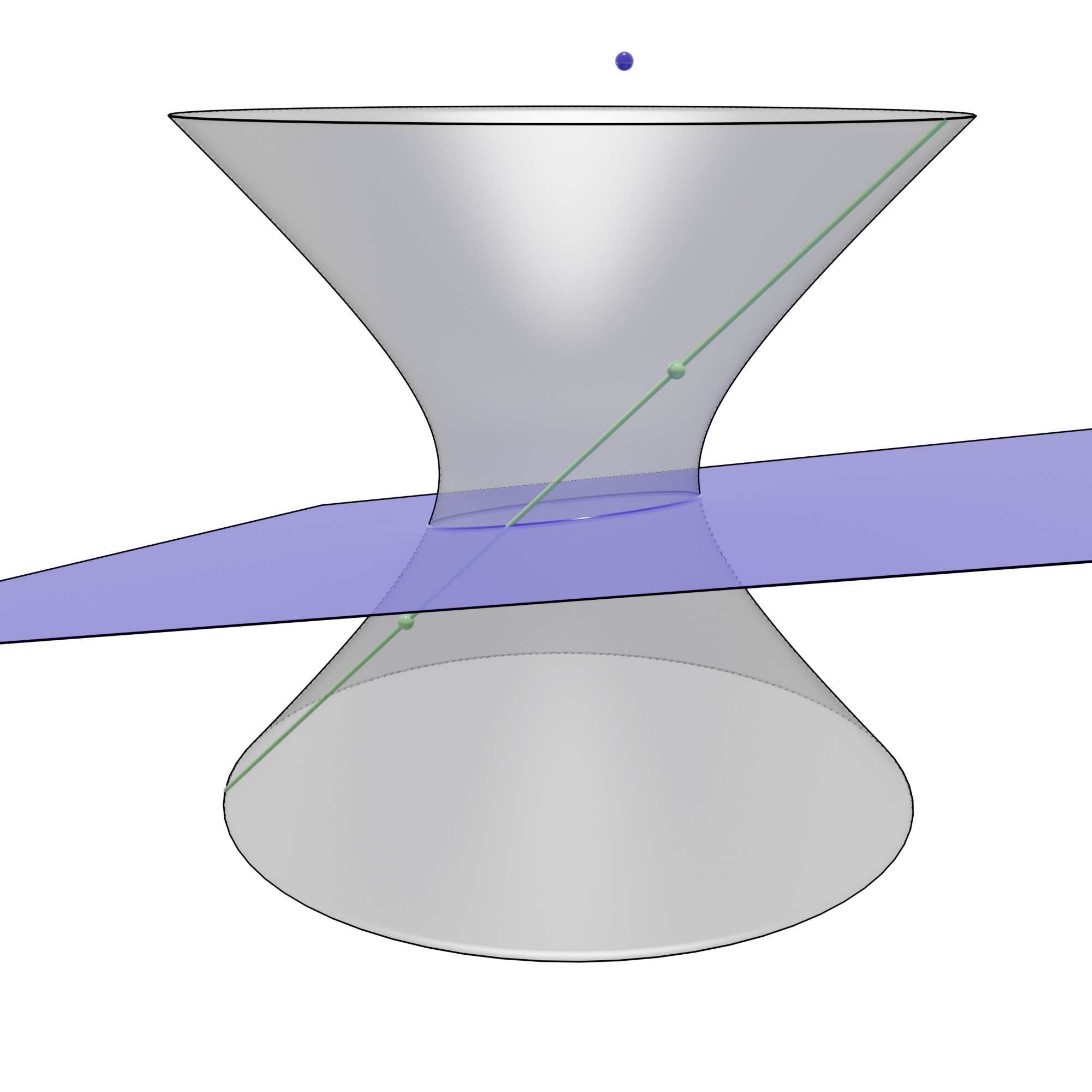}}%
    \put(0.59153319,0.94965675){\color[rgb]{0,0,0}\makebox(0,0)[lt]{\lineheight{1.25}\smash{\begin{tabular}[t]{l}{\small $\p{p}$}\end{tabular}}}}%
    \put(0.0617849,0.43020595){\color[rgb]{0,0,0}\makebox(0,0)[lt]{\lineheight{1.25}\smash{\begin{tabular}[t]{l}{\small $\p{p}^{\perp}$}\end{tabular}}}}%
    \put(0.68077805,0.20137298){\color[rgb]{0,0,0}\makebox(0,0)[lt]{\lineheight{1.25}\smash{\begin{tabular}[t]{l}{\small $\lie$}\end{tabular}}}}%
    \put(0.65102978,0.53089248){\color[rgb]{0,0,0}\makebox(0,0)[lt]{\lineheight{1.25}\smash{\begin{tabular}[t]{l}{\small $\mob$}\end{tabular}}}}%
    \put(0.56178492,0.68077803){\color[rgb]{0,0,0}\makebox(0,0)[lt]{\lineheight{1.25}\smash{\begin{tabular}[t]{l}{\small $\p{s}_1$}\end{tabular}}}}%
    \put(0.3798627,0.40045767){\color[rgb]{0,0,0}\makebox(0,0)[lt]{\lineheight{1.25}\smash{\begin{tabular}[t]{l}{\small $\p{s}_2$}\end{tabular}}}}%
    \put(0.02059497,0.94393593){\color[rgb]{0,0,0}\makebox(0,0)[lt]{\lineheight{1.25}\smash{\begin{tabular}[t]{l}{\small $\RP^{n+2}$}\end{tabular}}}}%
  \end{picture}%
\endgroup%

  \hspace{0.02\textwidth}
  \def\svgwidth{0.47\textwidth}
%% Accompanies image file 'touching_circles.pdf' (pdf, eps, ps)
\begingroup%
  \makeatletter%
  \providecommand\color[2][]{%
    \errmessage{(Inkscape) Color is used for the text in Inkscape, but the package 'color.sty' is not loaded}%
    \renewcommand\color[2][]{}%
  }%
  \providecommand\transparent[1]{%
    \errmessage{(Inkscape) Transparency is used (non-zero) for the text in Inkscape, but the package 'transparent.sty' is not loaded}%
    \renewcommand\transparent[1]{}%
  }%
  \providecommand\rotatebox[2]{#2}%
  \newcommand*\fsize{\dimexpr\f@size pt\relax}%
  \newcommand*\lineheight[1]{\fontsize{\fsize}{#1\fsize}\selectfont}%
  \ifx\svgwidth\undefined%
    \setlength{\unitlength}{2025.31640625bp}%
    \ifx\svgscale\undefined%
      \relax%
    \else%
      \setlength{\unitlength}{\unitlength * \real{\svgscale}}%
    \fi%
  \else%
    \setlength{\unitlength}{\svgwidth}%
  \fi%
  \global\let\svgwidth\undefined%
  \global\let\svgscale\undefined%
  \makeatother%
  \begin{picture}(1,1)%
    \lineheight{1}%
    \setlength\tabcolsep{0pt}%
    \put(0,0){\includegraphics[width=\unitlength,page=1]{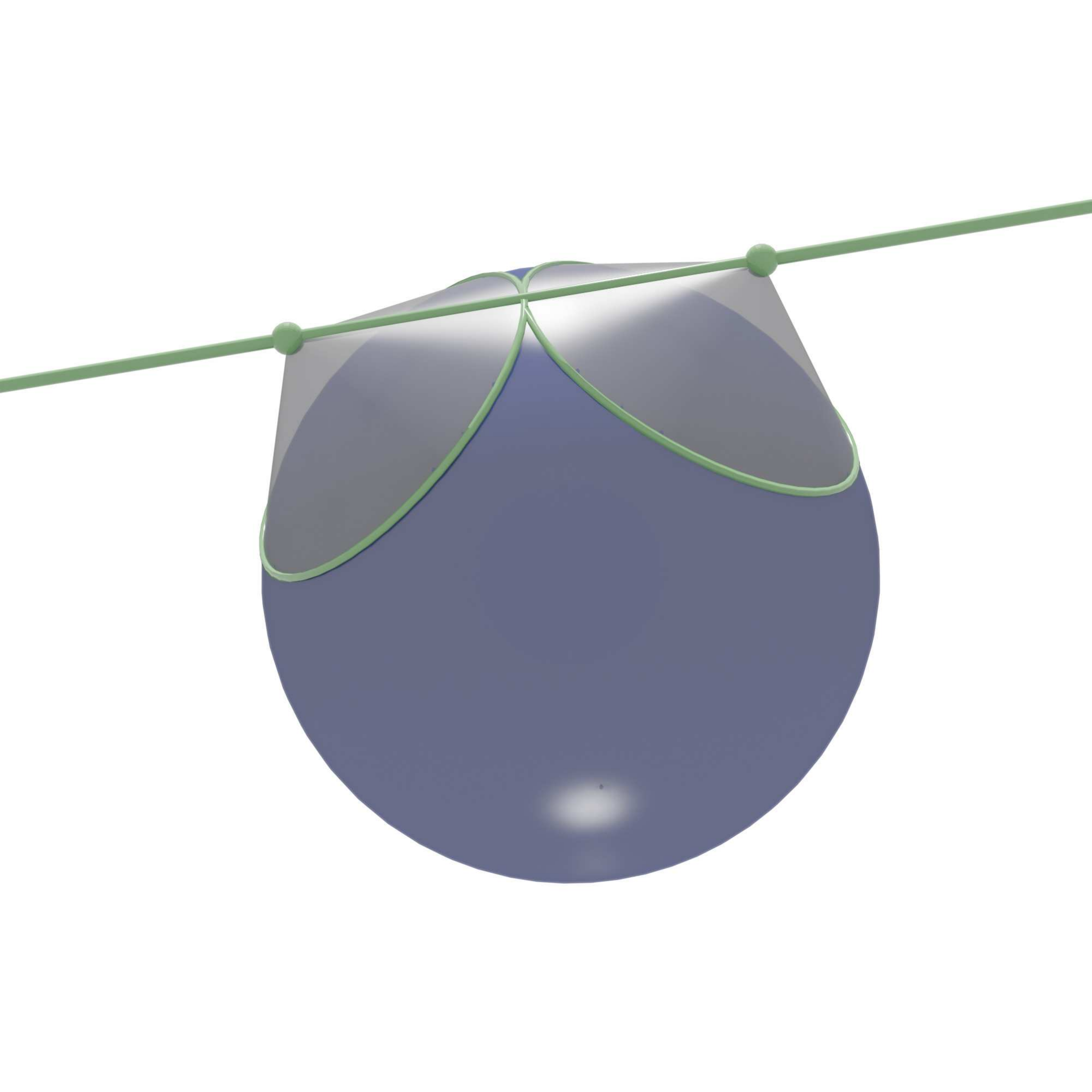}}%
    \put(0.60780501,0.31118949){\color[rgb]{0,0,0}\makebox(0,0)[lt]{\lineheight{1.25}\smash{\begin{tabular}[t]{l}{\small $\mob$}\end{tabular}}}}%
    \put(0.01362679,0.94436703){\color[rgb]{0,0,0}\makebox(0,0)[lt]{\lineheight{1.25}\smash{\begin{tabular}[t]{l}{\small $\p{p}^{\perp} \simeq \RP^{n+1}$}\end{tabular}}}}%
    \put(0.24723687,0.7317502){\color[rgb]{0,0,0}\makebox(0,0)[lt]{\lineheight{1.25}\smash{\begin{tabular}[t]{l}{\small $\pi_{\p{p}}(\p{s}_1)$}\end{tabular}}}}%
    \put(0.67714341,0.80184366){\color[rgb]{0,0,0}\makebox(0,0)[lt]{\lineheight{1.25}\smash{\begin{tabular}[t]{l}{\small $\pi_{\p{p}}(\p{s}_2)$}\end{tabular}}}}%
    \put(0,0){\includegraphics[width=\unitlength,page=2]{touching_circles.pdf}}%
  \end{picture}%
\endgroup%

  \caption{
    \emph{Left:} The Lie quadric $\lie \subset \RP^{n+2}$
    (depicted in the case $n=1$).  The choice of $\p{p}$ with
    $\liesc{p}{p} < 0$ determines the point complex, or Möbius quadric, $\mob \subset \lie$.
    The points $\p{s}_1, \p{s}_2 \in \lie$ contained in a common isotropic subspace of $\lie$
    correspond to two oriented hyperspheres in oriented contact.
    \emph{Right:} The Möbius quadric $\mob \subset \RP^{n+1}$
    (depicted in the case $n=2$).
    The two points $\pi_{\p{p}}(\p{s}_1), \pi_{\p{p}}(\p{s}_2) \in \mob^{+}$
    correspond to two hyperspheres in $\mob$ via polarity.
    In the chosen normalization, their orientation is encoded in the last component of $s_1$, $s_2$ respectively.
  }
  \label{fig:lie}
\end{figure}

Möbius geometry (signature $(n+1,1)$, see Section \ref{sec:mobius-geometry}),
hyperbolic Laguerre geometry (signature $(n,2)$, see Section \ref{sec:hyperbolic-laguerre-geometry}),
elliptic Laguerre geometry (signature $(n+1,1)$, see Section \ref{sec:elliptic-laguerre-geometry}),
as well as Euclidean Laguerre geometry (signature $(n,1,1)$, see Appendix \ref{sec:euclidean-laguerre-geometry})
can all be lifted to \emph{Lie geometry} (signature $(n+2,2)$)
using the methods from Sections \ref{sec:projection} and \ref{sec:laguerre}.

In this section we first give an elementary introduction to Lie geometry, which leads to its projective model, see, e.g. \cite{Bl2, C}.
We then show how to unify Möbius and Laguerre geometry of Cayley-Klein spaces in the framework of Lie geometry by considering certain (compatible) \emph{sphere complexes}.
The groups of Möbius transformations, Laguerre transformations, and isometries appear as quotients of the group of \emph{Lie transformations}.

\subsection{Oriented hyperspheres of $\S^n$}
We first give an intuitive description of Lie (sphere) geometry
as the geometry of oriented hyperspheres of the
$n$-dimensional sphere $\S^n$ and their oriented contact.

Thus, let
\[
  \S^n = \set{y \in \R^{n+1}}{ \dotprod{y}{y} = 1} \subset \R^{n+1},
\]
where $\dotprod{y}{y}$ denotes the standard scalar product on $\R^{n+1}$.
An oriented hypersphere of $\S^n$ can be represented by its center $c \in \S^n$
and its signed spherical radius $r \in \R$ (see Figure \ref{fig:lie}).
Tuples $(c, r) \in \S^n \times \R$ represent the same oriented hypersphere
if they are related by a sequence of the transformations
\begin{equation}
  \label{eq:oriented-spheres-equivalence}
  \rho_1 : (c, r) \mapsto (c, r + 2\pi), \quad
  \rho_2 : (c, r) \mapsto (-c, r - \pi).
\end{equation}
The corresponding hypersphere as a set of points is given by
\begin{equation}
  \label{eq:non-oriented-sphere}
  \set{ y \in \S^n }{ \dotprod{c}{y} = \cos r },
\end{equation}
while its orientation is obtained in the following way:
The hypersphere separates the sphere $\S^n$ into two regions.
For $r \in [0,\pi)$ consider the region which contains the center $c$ to be the ``inside'' of the hypersphere,
and endow the hypersphere with an orientation by assigning normal vectors pointing towards the other region, the ``outside'' of the hypersphere.
The orientation of the hypersphere for other values of $r$ is then obtained by \eqref{eq:oriented-spheres-equivalence}.
\begin{definition}
  We call 
  \[
    \ospheres \coloneqq \faktor{\left( \S^n \times \R \right)}{\{\rho_1, \rho_2\}}.
  \]
  the \emph{space of oriented hyperspheres} of $\S^n$.
\end{definition}
\begin{remark}
  \label{rem:non-oriented-spheres}
  Orientation reversion defines an involution on $\ospheres$, which is given by
  \[
    \rho : (c, r) \mapsto (c, -r).
  \]
  Thus, the \emph{space of (non-oriented) hyperspheres} of $\S^n$ may be represented by
  \[
    \nospheres \coloneqq \faktor{\ospheres}{\rho} = \faktor{\left( \S^n \times \R \right)}{\{\rho, \rho_1, \rho_2\}}.
  \]
\end{remark}
Two oriented hyperspheres $(c_1, r_1)$ and $(c_2, r_2)$ are in \emph{oriented contact} if (see Figure \ref{fig:lie})
\begin{equation}
  \label{eq:oriented-spheres-contact}
  \dotprod{c_1}{c_2} = \cos(r_1 - r_2),
\end{equation}
which is a well-defined relation on $\ospheres$.
Upon introducing coordinates $(c, \cos r, \sin r)$ on $\ospheres$
the transformations \eqref{eq:oriented-spheres-equivalence} may be replaced by
\begin{equation}
  \label{eq:oriented-spheres-equivalence-projective}
  (c, \cos r, \sin r) \mapsto (-c, \cos(r-\pi), \sin(r-\pi)) = - (c, \cos r, \sin r),
\end{equation}
while \eqref{eq:oriented-spheres-contact} becomes a bilinear relation, i.e.
\begin{equation}
  \label{eq:oriented-spheres-contact-projective}
  \dotprod{c_1}{c_2} - \cos r_1 \cos r_2 - \sin r_1 \sin r_2 = 0.
\end{equation}
This gives rise to a projective model of Lie geometry as described in the following.
\begin{definition}\
  \nobreakpar
  \begin{enumerate}
  \item The quadric
    \[
      \lie \subset \RP^{n+2}
    \]
    corresponding to the standard bilinear form of signature $(n+1, 2)$
    \[
      \liesc{x}{y} \coloneqq \sum_{i=1}^{n+1}x_i y_i - x_{n+2} y_{n+2} - x_{n+3} y_{n+3}
    \]
    for $x, y \in \R^{n+3}$,
    is called the \emph{Lie quadric}.
  \item Two points $\p{s}_1, \p{s}_2 \in \lie$ on the Lie quadric are called
    \emph{Lie orthogonal} if $\liesc{s_1}{s_2} = 0$,
    or equivalently if the line $\p{s}_1 \wedge \p{s}_2$ is isotropic,
    i.e. is contained in $\lie$.
    An isotropic line is called a \emph{contact element}.
  \item The projective transformations of $\RP^{n+2}$ that preserve the Lie quadric $\lie$
    \[
      \lietrafos \coloneqq \PO(n+1,2).
    \]
    are called \emph{Lie transformations}.
  \end{enumerate}
  \label{def:lie-geometry}
  \label{def:contact-element}
\end{definition}
\begin{proposition}
  \label{prop:lie-sphere-correspondence}
  The set of oriented hyperspheres $\ospheres$ of $\S^n$
  is in one-to-one correspondence with the Lie quadric $\lie$ by the map
  \[
    \vec{S} : \ospheres \rightarrow \lie, \quad (c, r) \mapsto (c, \cos r, \sin r)
  \]
  such that two oriented hyperspheres are in oriented contact if and only if
  their corresponding points on the Lie quadric are Lie orthogonal.
\end{proposition}
\begin{proof}
  A point $\p{s} \in \lie$ can always be represented by $\p{s} = [c, \cos r, \sin r]$ with $c \in \S^n$, $r \in \R$.
  Now the statement follows from \eqref{eq:oriented-spheres-equivalence-projective} and \eqref{eq:oriented-spheres-contact-projective}.
\end{proof}
\begin{table}[H]
  \centering
  \def\arraystretch{1.7}
  \begin{tabular}{c|c}
    \textbf{spherical geometry} & \textbf{Lie geometry}\\
    \hline\hline
    \emph{point} $\hat{x}\in\S^n$ & $\left[\hat{x}, 1, 0\right] \in \lie$\\
    \hline
    \makecell{\emph{oriented hypersphere}\\ with center $\hat{s}\in\R^n$ and signed radius $r \in \R$} & $\left[\hat{s}, \cos r, \sin r\right] \in \lie$
  \end{tabular}
  \caption{Correspondence of hyperspheres of the $n$-sphere $\S^n$ and points on the Lie quadric $\lie$.}
  \label{tab:lie-sphere-correspondence}
\end{table}
This correspondence leads to an embedding of $\S^n$ into the Lie quadric in the following way.
Among all oriented hyperspheres the map $\vec{S}$ distinguishes the set of ``points'', or \emph{null-spheres},
as the set of oriented hyperspheres with radius $r=0$.
It turns out that
\[
  \set{\vec{S}(c,0)}{c \in \S^n} = \set{\p{x} \in \lie}{x_{n+3} = 0} = \lie \cap \p{p}^\lieperp,
\]
where
\[
  \p{p} \coloneqq [e_{n+3}] = [0, \ldots, 0, 1] \in \RP^{n+2}.
\]
\begin{definition}
  \label{def:point-complex}
  The quadric
  \[
    \mob \coloneqq \lie \cap \p{p}^\lieperp
  \]
  is called the \emph{point complex}.
\end{definition}
\begin{remark}
  Every choice of a timelike point $\p{p} \in \RP^{n+2}$, i.e.\ $\liesc{p}{p} < 0$,
  leads to the definition of a point complex $\mob = \lie \cap \p{p}^\lieperp$,
  all of which are equivalent up to a Lie transformation.
  The chosen point complex $\mob$ then leads to a correspondence
  of points on the Lie quadric $\lie$ and oriented hyperspheres on $\mob \simeq \S^n$.
\end{remark}
The point complex is a quadric of signature $(n+1, 1)$
which we identify with the Möbius quadric (see Section \ref{sec:mobius-geometry}).
The corresponding involution and projection associated with the point $\p{p}$
(see Definition \ref{def:involution-projection}) take the form
\[
  \begin{aligned}
    &\sigma_{\p{p}} : [x_1, \ldots, x_{n+2}, x_{n+3}] \mapsto [x_1, \ldots, , x_{n+2}, -x_{n+3}],\\
    &\pi_{\p{p}} : [x_1, \ldots, x_{n+2}, x_{n+3}] \mapsto [x_1, \ldots, , x_{n+2}, 0].
  \end{aligned}
\]
The image of the Lie quadric $\lie$ under the projection $\pi_{\p p}$ is given by
\[
  \pi_{\p p}(\lie) = \mob^+ \cup \mob = \set{\p{s} \in \p{p}^\lieperp}{ \liesc{s}{s} \geq 0}.
\]
By polarity, each point $\mob^+ \cup \mob$ corresponds to a hyperplanar section of $\mob$.
Thus, the Lie quadric can be seen as a double cover of the set of spheres of Möbius geometry, encoding their orientation,
while, vice versa, the orientation of hyperspheres in Lie geometry vanishes in the projection to Möbius geometry.
\begin{proposition}\
  \label{prop:non-oriented-sphere-polarity}
  \nobreakpar
  \begin{enumerate}
  \item
    \label{prop:non-oriented-sphere-polarity-involution}
    The involution
    $\sigma_{\p p} : \lie \rightarrow \lie$
    corresponds to the orientation reversion on $\ospheres$.
  \item
    \label{prop:non-oriented-sphere-polarity-projection}
    The projection
    $\pi_{\p p} : \lie \rightarrow \mob^+ \cup \mob$
    defines a double cover with branch locus $\mob$.
  \item The set $\nospheres$ of non-oriented hyperspheres of $\S^n$ (see Remark \ref{rem:non-oriented-spheres})
    is in one-to-one correspondence with $\mob^+ \cup \mob$ by the map
    \[
      S = \pi_{\p{p}}\circ \vec{S} : \nospheres \rightarrow \mob^+ \cup \mob, \quad
      (c, r) \mapsto (c, \cos r, 0).
    \]
  \item
    \label{prop:non-oriented-sphere-polarity-points-on-sphere}
    The set of ``points'' on $\mob \subset \lie$ lying on an oriented hypersphere $\p{s} \in \lie$,
    or equivalently lying on the non-oriented hypersphere $\pi_{\p{p}}(\p{s}) \in \mob^+ \cup \mob$ is given by
    \begin{equation}
      \p{s}^\lieperp \cap \mob = \pi_{\p{p}}(\p{s})^\perp \cap \mob.
    \end{equation}
  \item 
    The non-oriented hyperspheres corresponding to two points $\p{s}_1, \p{s}_2 \in \mob^+ \cup \mob$
    touch if and only if the line $\p{s}_1 \wedge \p{s}_2$ connecting them is tangent to $\mob$.

    Thus, the points on the cone of contact $\cone{\mob}{\p{s}}$ (see Definition \ref{def:cone-of-contact})
    correspond to all spheres touching the sphere corresponding to $\p{s} \in \mob^+ \cup \mob$.    
  \end{enumerate}
\end{proposition}
\begin{proof}\
  \nobreakpar
  \begin{enumerate}
  \item Note that $\sigma_{\p{p}}(\vec{S}(c,r)) = \vec{S}(c, -r)$ and compare with Remark \ref{rem:non-oriented-spheres}.
  \item See Proposition \ref{prop:involution-projection} \ref{prop:involution-projection-double-cover}.
  \item Follows from \ref{prop:non-oriented-sphere-polarity-involution} and \ref{prop:non-oriented-sphere-polarity-projection}.
  \item
    The set $\p{s}^\lieperp \cap \mob \subset \lie$ describes all hyperspheres in oriented contact with $\p{s}$
    that simultaneously correspond to ``points'', i.e.\ ``points'' that lie on the hypersphere.
    Indeed, with $\p{s} = [\project{s}, \cos r, \sin r] \in \lie$ we find for a ``point'' $\p{x} = [\project{x}, 1, 0] \in \mob$ that
    \[
      \liesc{s}{x} = 0
      ~\Leftrightarrow~
      \scalarprod{\project{s}}{\project{x}} = \cos r.
    \]
  \item
    This generalizes the statement in \ref{prop:non-oriented-sphere-polarity-points-on-sphere}
    and follows from the fact that the isotropic subspaces of $\lie$ (contact elements, cf.\ Definition \ref{def:contact-element})
    project to tangent lines of $\mob$.
  \end{enumerate}
\end{proof}
The subgroup $\lietrafos_{\p{p}}$ of Lie transformations that preserve the point complex $\mob$,
i.e.\ map ``points'' to ``points'',
becomes the group of Möbius transformations in the projection to $\p{p}^\perp$
\[
  \mobtrafos
  = \faktor{\lietrafos_{\p p}}{\sigma_{\p p}}
  \simeq \PO(n+1,1).
\]

\subsection{Laguerre geometry from Lie geometry}
%
% \begin{figure}[H]
%   \centering
%   \includegraphics[width=0.46\textwidth]{pics/general-laguerre-lie-quadric}
%   \hspace{1cm}
%   \includegraphics[width=0.46\textwidth]{pics/general-laguerre-moebius-quadric}
%   \caption{
%     Laguerre geometry from choosing two points $\p{p}$ and $\p{q}$ with $\liesc{p}{q} = 0$ in Lie geometry
%     {\emph Left:} Lie quadric ($n=1$).
%     {\emph Right:} Projection to $\p{p}^\lieperp$ ($n=2$).
%   }
% \label{fig:general-laguerre}
% \end{figure}
%
A sphere complex in Lie geometry is given by the intersection of the Lie quadric with a hyperplane of $\RP^{n+2}$.
It may equivalently be described by the polar point of this hyperplane.
Two points in $\RP^{n+2}$ can be mapped to each other by a Lie transformation
if and only if they have the same signature.
Thus, any two sphere complexes of the same signature are Lie equivalent.
\begin{definition}
  \label{def:sphere-complexes}
  For a point $\p{q} \in \RP^{n+2}$ the set of points
  \[
    \lie \cap \p{q}^\perp
  \]
  on the Lie quadric 
  as well as the $n$-parameter family of oriented hyperspheres
  corresponding to these points is called a \emph{sphere complex}.
  A sphere complex is further called
  \begin{itemize}
  \item
    \emph{elliptic} if $\liesc{q}{q} > 0$,
  \item
    \emph{hyperbolic} if $\liesc{q}{q} < 0$,
  \item
    \emph{parabolic} if $\liesc{q}{q} = 0$.
  \end{itemize}
\end{definition}
\begin{remark}\
  \nobreakpar
  \begin{enumerate}
  \item We adopted the classical naming convention for sphere complexes here, see e.g.~\cite{Bl2}.
  \item The point complex (see Definition \ref{def:point-complex}) is a hyperbolic sphere complex.
  \item
    A non-parabolic sphere complex induces an invariant for pairs of oriented spheres (see Appendix \ref{sec:invariant}).
    In particular, the invariant induced by the point complex, i.e., the point $\p{p}$, is the \emph{signed inversive distance} (see Appendix \ref{sec:signed-inversive-distance}), which generalizes the intersection angle of spheres.
    It further allows for a geometric description of sphere complexes (see Appendix \ref{sec:sphere-complexes-geometric}).
  \end{enumerate}
\end{remark}
Laguerre geometry is the geometry of oriented hyperplanes and oriented hyperspheres in a certain space form,
and their oriented contact (cf.\ Section \ref{sec:laguerre}).
It appears as a subgeometry of Lie geometry by distinguishing the set of ``oriented hyperplanes''
as a sphere complex among the set of oriented hyperspheres.

The point complex $\mob = \lie \cap \p{p}^\lieperp$,
where $\p{p} \in \RP^{n+2}$ is a timelike point,
induces the notion of orientation reversion given by the involution $\sigma_{\p p}$.
For another sphere complex $\lie \cap \p{q}^\lieperp$, where $\p{q} \in \RP^{n+2}$,
to play the distinguished role of the set of ``oriented hyperplanes'' on $\mob$
it must be invariant under orientation reversion, i.e.,\ $\sigma_{\p p}(\lie \cap \p{q}^\lieperp) = \lie \cap \p{q}^\lieperp$,
which is equivalent to $\liesc{p}{q} = 0$.
\begin{definition}
  \label{def:plane-complex}
  For a point $\p{q} \in \RP^{n+2}$ with
  \[
    \liesc{p}{q} = 0
  \]
  we call the sphere complex
  \[
    \lag \coloneqq \lie \cap \p{q}^{\lieperp},
  \]
  a \emph{plane complex}.
\end{definition}

Up to a Lie transformation that fixes $\p{p}$, i.e.\ a Möbius transformation (cf.\ Section \ref{sec:mobius-geometry}), 
we can set, w.l.o.g.,
\[
  \p{q} = \left\{
    \begin{aligned}
      [e_{n+1}] &= [0, \ldots, 1, 0, 0] &&\text{if}~\liesc{q}{q} > 0\\
      [e_{n+2}] &= [0, \ldots, 0, 1, 0] &&\text{if}~\liesc{q}{q} < 0\\
      [e_\infty] &= [0, \ldots, \tfrac{1}{2}, \tfrac{1}{2}, 0] &&\text{if}~\liesc{q}{q} = 0.
    \end{aligned}
  \right.
\]
Consider the restriction of the Lie quadric to $\p{q}^\lieperp \simeq \RP^{n+1}$.
Then for the non-parabolic cases we identify each of the plane complexes with one of the \emph{Laguerre quadrics}
which we have introduced in Section \ref{sec:laguerre}.
The parabolic plane complex corresponds to the classical case of Euclidean Laguerre geometry.
Thus, we recover (see Figure \ref{fig:hyperbolic_lag})
\begin{itemize}
\item
  \emph{hyperbolic Laguerre geometry} if $\liesc{q}{q} > 0$ (see Section \ref{sec:hyperbolic-laguerre-geometry}),
\item
  \emph{elliptic (``spherical'') Laguerre geometry} if $\liesc{q}{q} < 0$ (see Section \ref{sec:elliptic-laguerre-geometry}),
\item
  \emph{Euclidean Laguerre geometry} if $\liesc{q}{q} = 0$ (see Section \ref{sec:euclidean-laguerre-geometry}).
\end{itemize}
\begin{remark}
  Note that according to the classical naming convention of sphere complexes,
  which we adopted in Definition \ref{def:sphere-complexes},
  an elliptic sphere complex is associated with hyperbolic Laguerre geometry,
  while a hyperbolic sphere complex is associated with elliptic Laguerre geometry.
\end{remark}
The corresponding groups of \emph{Laguerre transformations} are induced by the groups of Lie transformations
that preserve the corresponding Laguerre quadric $\lag$, or equivalently the point $\p{q}$,
\[
  \faktor{\lietrafos_{\p q}}{\sigma_{\p q}} \simeq \left\{
    \begin{aligned}
      &\PO(n,2)   &&\text{if}~\liesc{q}{q} > 0\\
      &\PO(n+1,1) &&\text{if}~\liesc{q}{q} < 0\\
      &\PO(n,1,1) &&\text{if}~\liesc{q}{q} = 0,
    \end{aligned}
  \right.
\]
where $\sigma_{\p q}$ is the involution associated with the plane complex (cf.\ Definition \ref{def:involution-projection}),
and we set $\sigma_{\p q} = \id$ if $\liesc{q}{q} = 0$.
\begin{remark}
  \label{rem:involutions-commute}
  In the non-parabolic cases,
  the condition $\liesc{p}{q}=0$ is equivalent to the condition that the two involutions $\sigma_{\p p}$ and $\sigma_{\p q}$ commute, i.e.
  \[
    \sigma_{\p p} \circ \sigma_{\p q} = \sigma_{\p q} \circ \sigma_{\p p}.
  \]
\end{remark}

%\begin{figure}
%  \centering
%  \includegraphics[scale=0.07]{pics/lie-quadric-hyperbolic}
%  \hspace{1cm}
%  \includegraphics[scale=0.07]{pics/moebius-quadric-hyperbolic}
%  \caption{
%    Hyperbolic Laguerre geometry from Lie geomery.
%  }
%\label{fig:hyperbolic_lag}
%\end{figure}
%\begin{figure}
%  \centering
%  \includegraphics[scale=0.07]{pics/lie-quadric-elliptic}
%  \hspace{1cm}
%  \includegraphics[scale=0.07]{pics/moebius-quadric-elliptic}
%  \caption{
%    Elliptic geometry from Lie geometry.
%  }
%\label{fig:spherical_lag}
%\end{figure}
\begin{figure}
  \centering
  \def\svgwidth{0.57\textwidth}
  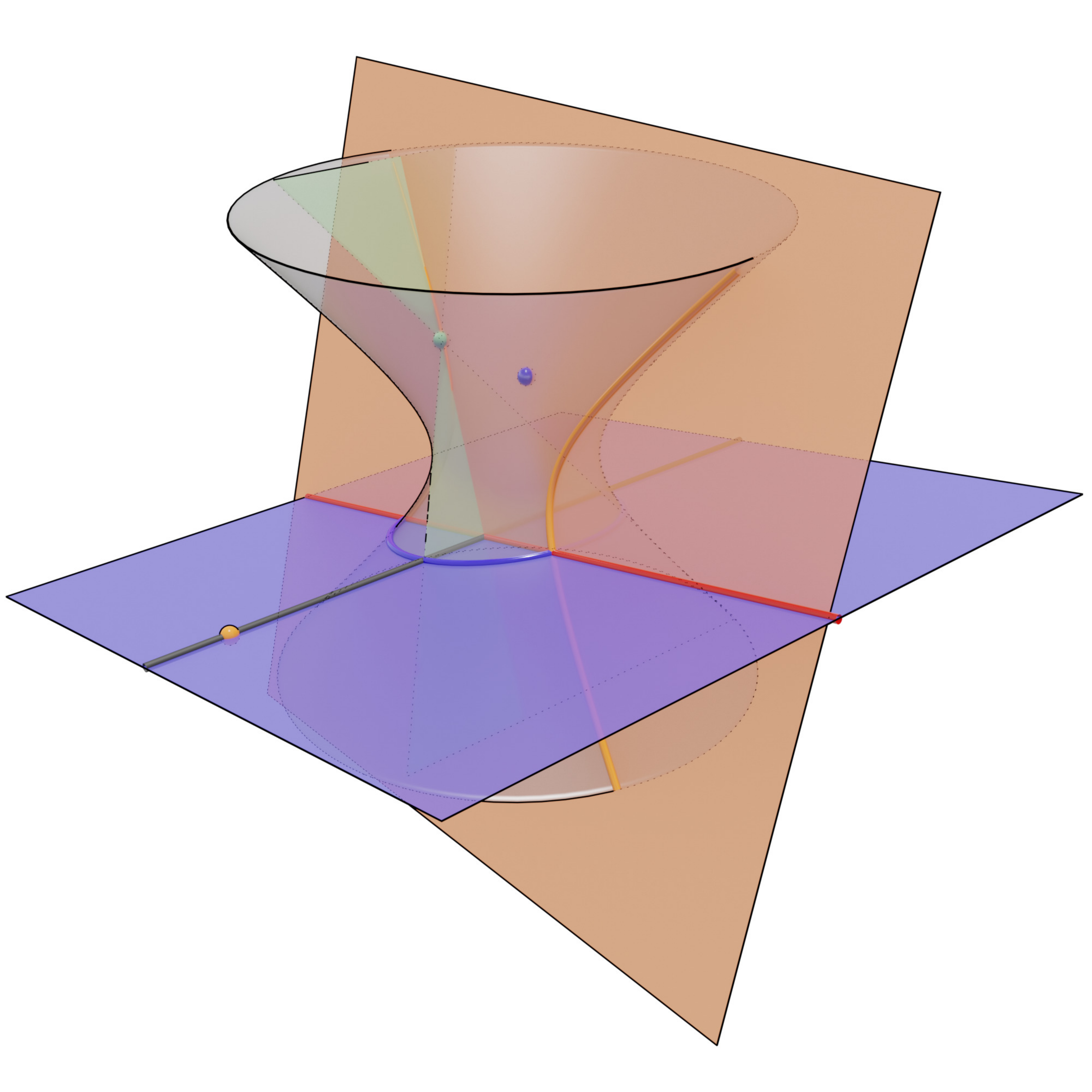
  \caption{
    Laguerre geometry from Lie geometry.
    The choice of a point $\p{q}$ with $\scalarprod{p}{q} = 0$ determines a plane complex, or Laguerre quadric $\lag \subset \lie$.
    This induces Laguerre geometry on a Cayley-Klein space in the base plane $\baseplane$.
    A point $\p{s} \in \lag$ corresponds to an oriented line in that space via polar projection.
  }
\label{fig:hyperbolic_lag}
\end{figure}

We recognized the different Laguerre quadrics by their signature, which depends on the signature of the point $\p{q}$ only,
but is entirely independent of the point $\p{p}$ with $\liesc{p}{q} = 0$.
We yet have to establish the geometric relation to Lie geometry.
\begin{definition}
  Given the two points $\p{p}, \p{q} \in \RP^{n+2}$ defining the point complex and the plane complex respectively,
  we call the set
  \[
    \baseplane \coloneqq \p{p}^\perp \cap \p{q}^\perp \simeq \RP^n.
  \]
  the \emph{base plane}.
\end{definition}
In the restriction to the hyperplane of the point complex $\p{p}^\perp$,
the point $\pi_{\p{p}}(\p{q})$ plays the role of the point $\p{q}$
from the Sections \ref{sec:hyperbolic-laguerre-geometry} and \ref{sec:elliptic-laguerre-geometry}.
Thus, polar projection with respect to this point yields hyperbolic/elliptic geometry in the base plane $\baseplane$.
In the parabolic case, projection with respect to $\p{q}$ should be replaced by stereographic projection,
which recovers Euclidean (similarity) geometry (cf.\ Appendix~\ref{sec:mobius-euclidean}).

On the other hand, in the restriction to the hyperplane of the plane complex $\p{q}^\perp$,
the point $\pi_{\p{q}}(\p{p})$ plays the role of the point $\p{p}$
from the Sections \ref{sec:hyperbolic-laguerre-geometry} and \ref{sec:elliptic-laguerre-geometry}.
Thus, polar projection with respect to this point yields hyperbolic/elliptic geometry in the base plane $\baseplane$,
while in the parabolic case projection with respect to $\p{p}$ leads to dual Euclidean (similarity) geometry
(cf.\ Appendix \ref{sec:euclidean-laguerre-geometry}).

On the level of the transformation group this can be described in the following way.
Consider the Lie transformations $\lietrafos_{\p{p}, \p{q}}$
that fix all ``points'' and ``planes'', i.e.\ the point complex $\mob$ and the plane complex $\lag$,
or equivalently, the two points $\p{p}, \p{q} \in \RP^{n+2}$.
These transformations naturally factor to $\sfrac{\lietrafos_{\p{p}, \p{q}}}{\{\sigma_{\p p}, \sigma_{\p q}\}}$,
where again we set $\sigma_{\p q} = \id$ if $\liesc{q}{q} = 0$.
Their action is well-defined on the quotient space $\sfrac{\mob}{\sigma_{\p q}}$, due to Remark \ref{rem:involutions-commute}.
The quotient space $\sfrac{\mob}{\sigma_{\p q}}$ can be embedded into the base plane $\baseplane$ using the projection $\pi_{\p{q}}$,
which should be replaced by stereographic projection in the Euclidean case (cf.\ Appendix~\ref{sec:euclidean-laguerre-geometry}).
Thus, we may equivalently consider the action of these Lie transformations on the base plane $\baseplane$,
on which they act as lower dimensional projective orthogonal groups again:
\[
  \faktor{\lietrafos_{\p{p}, \p{q}}}{\{\sigma_{\p p}, \sigma_{\p q}\}}
  \simeq \faktor{\PO(n+1, 1)}{\sigma_{\p q}}
  \simeq \left\{
    \begin{aligned}
      &\PO(n,1), &&\text{if}~\liesc{q}{q} > 0\\
      &\PO(n+1), &&\text{if}~\liesc{q}{q} < 0\\
      &\PO(n,0,1), &&\text{if}~\liesc{q}{q} = 0.
    \end{aligned}
  \right.
\]
We recognize $\PO(n,1)$ and $\PO(n+1)$ as the isometry groups of hyperbolic and elliptic space
(cf.\ Sections \ref{sec:hyperbolic-space} and \ref{sec:elliptic-space}),
while $\PO(n,0,1)$ corresponds to the group of dual similarity transformations,
i.e.\ the group of dual transformations $\PO(n,0,1)^*$ corresponds to isometries and scalings of Euclidean space
(cf.\ Appendix~\ref{sec:euclidean-space}).
\begin{remark}
  \label{rem:space-form-models}
  We end up with two models of the space form associated to each Laguerre geometry (see Figure \ref{fig:hyperbolic_lag}).
  One is represented by the point complex $\mob \subset \p{p}^\lieperp \simeq \RP^{n+1}$,
  with opposite points with respect to $\sigma_{\p q}$ identified,
  which we refer to as the \emph{sphere model} (see Figures \ref{fig:ell_cic_nets}, \ref{fig:hyp_cic_nets_ellipse}, and \ref{fig:hyp_cic_nets_hyperbola}, top).
  In this model the oriented hyperspheres that correspond to sections of $\mob$
  with hyperplanes that contain the point $\pi_{\p p}(\p{q})$ are the distinguished ``oriented hyperplanes''

  Another model is obtained by its projection $\pi_{\p q}(\mob)$ onto the base plane $\baseplane \simeq \RP^n$,
  which we refer to as the \emph{projective model} (see Figures \ref{fig:ell_cic_nets}, \ref{fig:hyp_cic_nets_ellipse}, and \ref{fig:hyp_cic_nets_hyperbola}, left).
  In this model the ``oriented hyperplanes'' become (oriented) projective hyperplanes.
\end{remark}
\begin{proposition}\
  \label{prop:laguerre-space-forms}
  \nobreakpar
  \begin{enumerate}
  \item In the non-Euclidean cases of Laguerre geometry, i.e.\ $\liesc{q}{q} \neq 0$,
    the point complex $\mob$ may be identified with hyperbolic/elliptic space respectively,
    after taking the quotient with respect to $\sigma_{\p{q}}$, or equivalently, projection onto the base plane
    \[
      \faktor{\mob}{\sigma_{\p q}}
      \simeq \pi_{\p q}(\mob)
      \subset \baseplane
      \simeq \RP^n.
    \]
    The Lie transformations that fix the point complex and the plane complex act on $\pi_{\p q}(\mob) \subset \baseplane$
    as the corresponding isometry group.
  \item In the case of Euclidean Laguerre geometry, i.e.\ $\liesc{q}{q} = 0$,
    the point complex $\mob$ may be identified with Euclidean space upon stereographic projection.
    The Lie transformations that fix the point complex and the plane complex act on $\baseplane$
    as dual similarity transformations.
  \end{enumerate}
\end{proposition}
\begin{remark}
  \label{rem:lie-quadric-projection}
  In Laguerre geometry the hyperplanar sections correspond to oriented spheres, which, in the non-Euclidean cases,
  can be identified with their polar points.
  In elliptic Laguerre geometry the Lie quadric projects to the ``outside'' of the elliptic Laguerre quadric
  \[
    \pi_{\p{q}}(\lie) = \lagell^+ \cup \lagell
  \]
  which represents all poles of hyperplanar sections \eqref{eq:laguerre-spheres-elliptic}.
  In hyperbolic Laguerre geometry, on the other hand, the Lie quadric projects to the ``inside'' of the hyperbolic Laguerre quadric
  \[
    \pi_{\p{q}}(\lie) = \laghyp^- \cup \lagell,
  \]
  while the poles of hyperplanar sections are the whole space \eqref{eq:laguerre-spheres-hyperbolic}.
  The Lie quadric only projects to the points corresponding to hyperbolic spheres/distance hypersurfaces/horospheres,
  and not to points representing deSitter spheres
  (cf.\ Remark \ref{rem:hyperbolic-Laguerre-spheres} \ref{rem:hyperbolic-Laguerre-spheres-distinction}).
  Vice versa, Laguerre spheres that are deSitter spheres do not possess a (real) lift to the Lie quadric.
\end{remark}

\subsection{Subgeometries of Lie geometry}
\label{sec:classification}

Choosing different signatures for the points $\p{p}$ and $\p{q}$,
i.e.\ different signatures for the point complex and plane complex,
we recover different subgeometries of Lie geometry (see Table \ref{tab:classification}).

Fixing both points in the Lie group induces (a quadruple covering of) the corresponding isometry group.
We call the group obtained by fixing only $\p{p}$ (a double cover of) the corresponding \emph{Möbius group},
and the group obtained by fixing only $\p{q}$ (a double cover of) the corresponding \emph{Laguerre group}.
For each isometry group the corresponding Möbius group describes the transformations that map
points in the space form to points while preserving spheres,
while the Laguerre group describes the transformations that map hyperplanes to hyperplanes
while preserving spheres.
For this to hold, the transformations either have to be considered locally,
or acting on the set of oriented points/oriented hyperplanes respectively
(see Theorem \ref{thm:mobius-transformation-lift},
Remark ~\ref{rem:moebius-transformation-lift}~\ref{rem:conformal-geometry-lift},
Remark \ref{rem:hyperbolic-mobius} \ref{rem:hyperbolic-mobius-transformations},
Remark \ref{rem:elliptic-mobius} \ref{rem:elliptic-mobius-transformations},
Section \ref{sec:hyperbolic-laguerre-transformations},
Section \ref{sec:elliptic-laguerre-transformations}).
\begin{remark}
  Note that certain geometries have the same transformation group.
  In particular, $n$-dimensional Lie geometry has the same transformation group
  as $(n+1)$-dimensional hyperbolic Laguerre geometry.
  Geometrically this is due to the fact that one can identify the oriented hyperspheres of $\mob$
  with the oriented hyperbolic hyperplanes of the inside $\hyp = \mob^-$.
\end{remark}
\begin{table}[H]
  \centering
  \begin{tabular}{c|c|c|c|c}
    \textbf{space form} & \textbf{isometry grp.} & \textbf{Möbius grp.} & \textbf{Laguerre grp.} & \textbf{sign.} $\p{p}$, $\p{q}$\\
    \hline\hline
    elliptic space & $\PO(n+1)$ & $\PO(n+1,1)$ & $\PO(n+1,1)$ & $(-)$ $(-)$\\
    \hline
    hyperbolic space & $\PO(n,1)$ & $\PO(n+1,1)$ & $\PO(n,2)$ & $(-)$ $(+)$\\
    \hline
    deSitter space & $\PO(n,1)$ & $\PO(n,2)$ & $\PO(n+1,1)$ & $(+)$ $(-)$\\
    \hline
    (dual) Euclidean space & $\PO(n,0,1)$ & $\PO(n+1,1)$ & $\PO(n,1,1)$ & $(-)$ $(0)$\\
    \hline
    (dual) Minkowski space & $\PO(n-1,1,1)$ & $\PO(n,2)$ & $\PO(n,1,1)$ & $(+)$ $(0)$
  \end{tabular}
  \caption{
    Isometry group, Möbius group, and Laguerre group for different space forms,
    and the signatures of the points $\p{p}$ and $\p{q}$ defining the corresponding point complex and plane complex in Lie geometry respectively.
    In the degenerate cases of Euclidean and Minkowski geometry, the given ``isometry group''
    is actually the group of similarity transformations represented on the dual space.
  }
  \label{tab:classification}
\end{table}

%%% Local Variables:
%%% TeX-master: "main"
%%% End:

\newpage
\section{Checkerboard incircular nets}
\label{sec:icnets}

In this section, as an application of two-dimensional Lie and Laguerre geometry, we present new research results.
While incircular nets and their Laguerre geometric generalization to \emph{checkerboard incircular nets} have been studied in great detail \cite{B, AB, BST}, we introduce their generalization to Lie geometry, and show that they may be classified in terms of checkerboard incircular nets in hyperbolic/elliptic/Euclidean Laguerre geometry.
We prove incidence theorems of Miquel type, show that all lines of a checkerboard incircular net are tangent to a hypercycle, and give explicit formulas in terms of Jacobi elliptic functions.
This generalizes the results from \cite{BST} and leads to a unified treatment of checkerboard incircular nets in all space forms.

\subsection{Checkerboard incircular nets in Lie geometry}
\begin{figure}
  \centering
  \def\svgwidth{0.4\textwidth}
  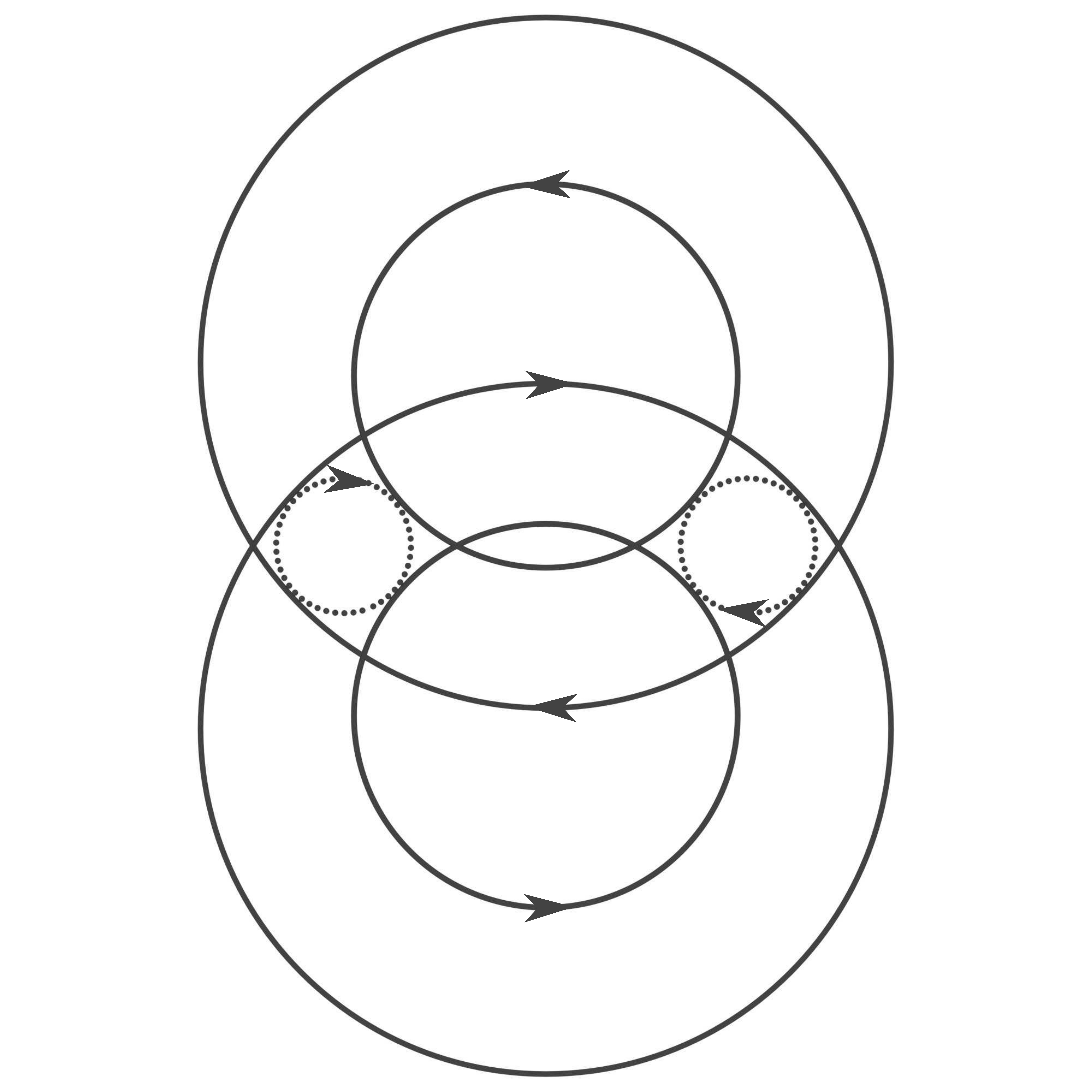
  \hspace{0.08\textwidth}
  \def\svgwidth{0.4\textwidth}
  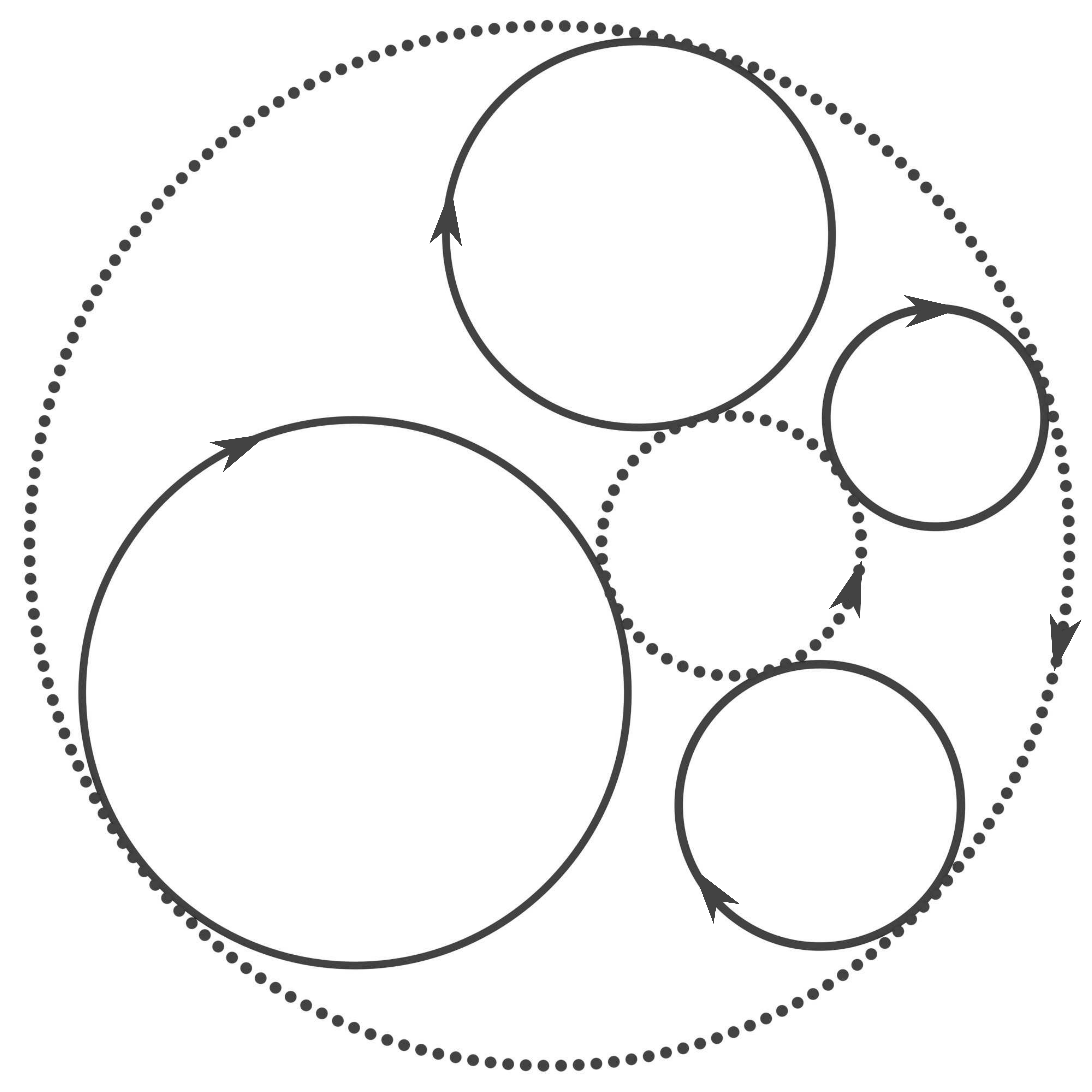
  \caption{
    Lie-circumscribed quadrilaterals.
  }
\label{fig:lie-circumscribed}
\end{figure}

To investigate configurations of oriented circles and their oriented contact on the two-sphere,
we identify oriented circles with points on the Lie quadric $\lie \subset \RP^4$,
which is a quadric of signature $(+++--)$, as described in Section \ref{sec:lie}.
\begin{definition}[Lie quadrilateral]
  A \emph{Lie quadrilateral} is a quadruple of oriented circles,
  called \emph{edge circles}.
\end{definition}
\begin{remark}
  Two edge circles of a Lie quadrilateral do not necessarily intersect.
  Thus, e.g., both quadrilaterals shown in Figure \ref{fig:lie-circumscribed} are admissible
  Lie quadrilaterals.
\end{remark}
\begin{definition}[Lie circumscribed]
  A Lie quadrilateral is called \emph{circumscribed} if the four points on the Lie quadric
  corresponding to its four oriented edge circles are coplanar.
  We call the signature of the plane in which these points lie
  the \emph{signature} of the circumscribed Lie quadrilateral.
\end{definition}
To justify the term ``circumscribed'' consider a plane $\mathbf{U} \subset \RP^4$ of signature $(++-)$.
Then according to Lemma \ref{lemma:orthogonal-subspaces} its polar line
has signature $(+-)$, and thus, ${\mathbf{U}^\lieperp \cap \lie = \{\p{c}_1, \p{c}_2\}}$ consists of exactly two points.
The one parameter family of circles corresponding to the points in $\mathbf{U} \cap \lie$
are the circles in oriented contact with the two circles corresponding to $\p{c}_1$ and $\p{c}_2$.
Therefore, a circumscribed Lie quadrilateral of signature $(++-)$ is in oriented contact with exactly two circles (see Figure \ref{fig:lie-circumscribed}).

To characterize all possible cases of circumscribed Lie quadrilaterals we need to
distinguish all possible signatures of the plane $\mathbf{U}$. 
\begin{proposition}
  \label{prop:circumscribed-quads}
  For a plane $\mathbf{U} \subset \RP^4$ the family of oriented circles
  corresponding to $\mathbf{U} \cap \lie$ is exactly one of the following
  depending on the signature of $\mathbf{U}$ with respect to the Lie quadric~$\lie$.
  \begin{itemize}
  \item $(+++)$ Empty family.
  \item $(++-)$ One parameter family of circles in oriented contact with
    the two oriented circles given by $\mathbf{U} \cap \lie$.
  \item $(+--)$ Circles from the intersection of two hyperbolic circle complexes
    (cf.\ Definition \ref{def:sphere-complexes} and Appendix \ref{sec:sphere-complexes-geometric}).
  \item $(+-0)$ Two contact elements (see Definition \ref{def:contact-element}) with a common circle.
  \item $(++0)$ One circle.
  \item $(+00)$ One contact element.
  \end{itemize}
\end{proposition}
\begin{proof}
  The Lie quadric has signature $(+++--)$.
  Thus, the listed signatures are all possible cases that can occur.
  A plane with signature $(+++)$ does not intersect the Lie quadric.
  The case $(++-)$ was already discussed before the proposition.
  For the case $(+--)$ the polar line has signature $(++)$.
  Thus, we may view $\mathbf{U}$ as the intersection of two hyperbolic circle complexes.
  The cases $(+-0)$, $(++0)$, and $(+00)$ each describes a tangent plane that, in turn,
  intersects the Lie quadric in two isotropic subspaces,
  touches the Lie quadric in exactly one point,
  intersects the Lie quadric in exactly one isotropic subspace.
\end{proof}
\begin{remark}
  \label{rem:generic-quads}
  For a generic circumscribed Lie quadrilateral, i.e., no three of the four points on the Lie quadric are collinear, only the signatures $(++-)$, $(+--)$, and $(+-0)$ can occur.
\end{remark}
\begin{figure}
  \centering
  \def\svgwidth{0.3\textwidth}
  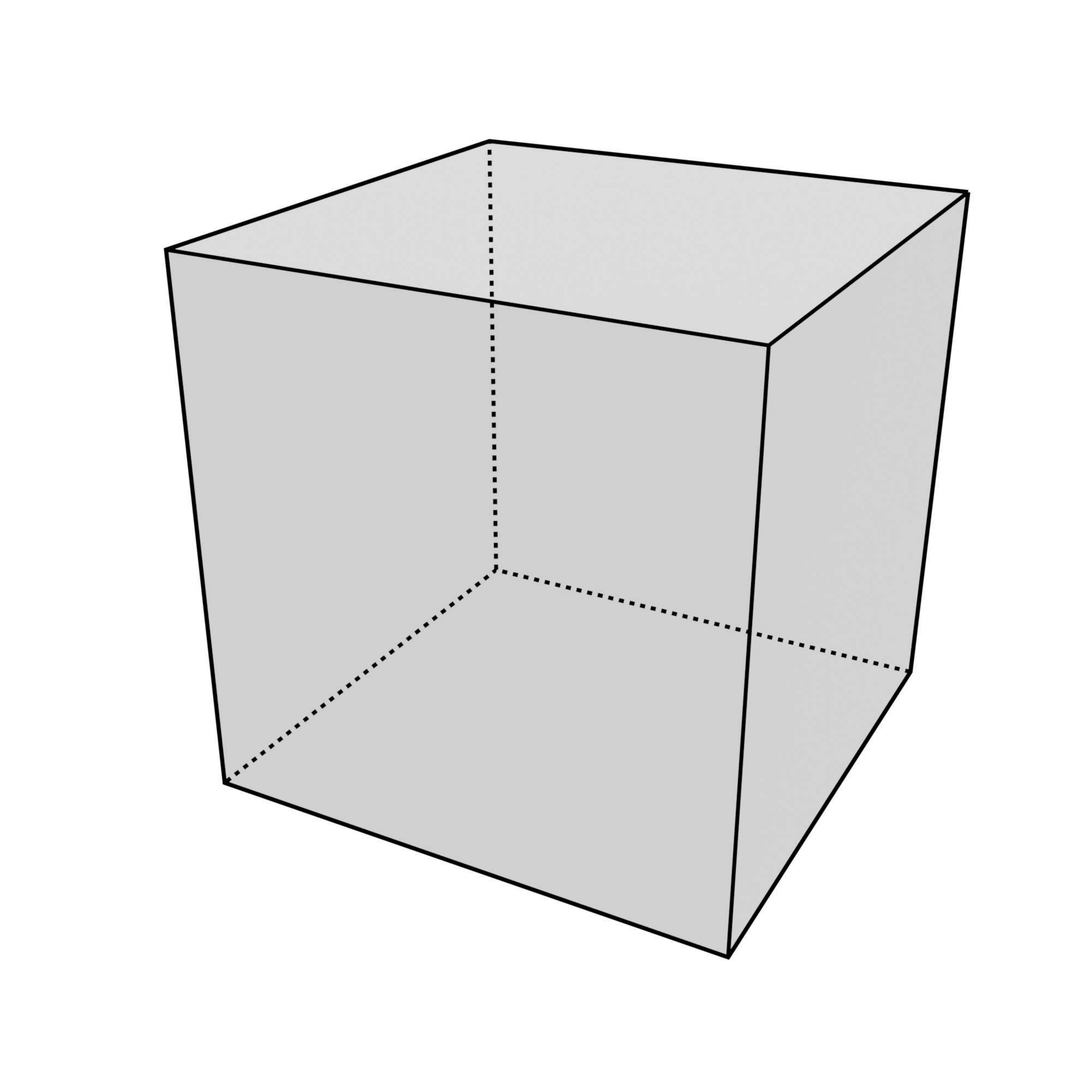
  \hspace{0.08\textwidth}
  \def\svgwidth{0.55\textwidth}
  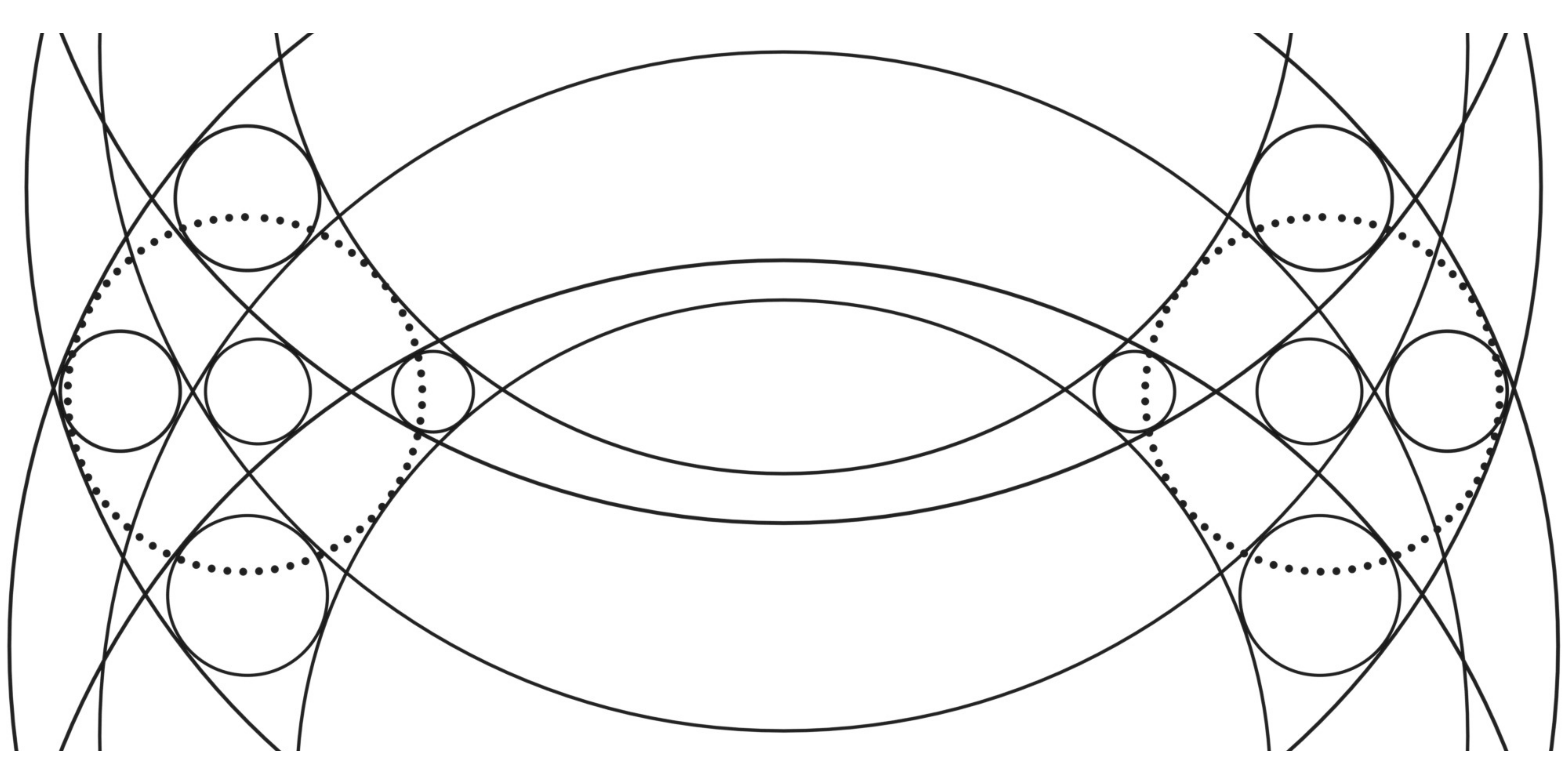
  \caption{
    Lie geometric version of Miquel's theorem.
    \emph{Left:} Combinatorial picture.
    \emph{Right:} Geometric picture.
  }
\label{fig:miquel}
\end{figure}
The definition of Lie circumscribility via planarity in the Lie quadric
immediately implies a Lie geometric version of the classical Miquel's theorem.
To see this, we employ the following statement of projective geometry
about the eight intersection points of three quadrics in space, see, e.g., \cite[Theorem 3.12]{BS}.
\begin{lemma}[Associated points]
  \label{lem:associated-points}
  Given eight distinct points which are the set of intersections of three quadrics in $\RP^3$,
  all quadrics through any seven of those points must pass through the eighth point.
\end{lemma}
\begin{theorem}[Miquel's theorem in Lie geometry]
  \label{thm:miquel-lie}
  \begin{sloppypar}
    Let $\ell_1, \ell_2, \ell_3, \ell_4, m_1, m_2, m_3, m_4$ be eight generic oriented circles on the sphere
    such that the five Lie quadrilaterals $(\ell_1, \ell_2, m_1, m_2)$, $(\ell_1, \ell_2, m_3, m_4)$, $(\ell_3, \ell_4, m_1, m_2)$,
    $(\ell_3, \ell_4, m_3, m_4)$, $(\ell_2, \ell_3, m_2, m_3)$ are circumscribed,
    then so is the Lie quadrilateral $(\ell_1, \ell_4, m_1, m_4)$ (see Figure \ref{fig:miquel}).
  \end{sloppypar}
\end{theorem}
\begin{remark}
  \label{rem:miquel-genericity}
  A sufficient genericity condition for the eight points on the Lie quadric is that no five points are coplanar.
\end{remark}
\begin{proof}
  Consider the eight points on the Lie quadric as the vertices of a combinatorial cube (see Figure \ref{fig:miquel}).
  Coplanarity of the bottom and side faces corresponds to the assumed circumscribility.
  Thus, we have to show that the top face is planar as well.  
  
  As a first step we show that all eight vertices of the cube are contained in a three-dimensional projective subspace.
  Indeed, let $\p V$ be the subspace spanned by $\ell_2, \ell_3, m_1, m_2$.
  Then the assumed circumscribility implies that for instance $\ell_1$ lies in a plane with $\ell_2, m_1, m_2$
  and therefore $\ell_1 \in \p V$. Similarly, $\ell_4, m_3 \in V$, and finally $m_4 \in \p V$.

  A three-dimensional subspace intersects the Lie quadric in a (at most once degenerate) two-dimensional quadric $\tilde{\lie}$.
  Consider the two degenerate quadrics $\mathcal{Q}_1, \mathcal{Q}_2$ consisting of two opposite face planes of the cube, respectively.
  Then, due to the genericity condition, the eight points of the cube are the intersection points of $\tilde{\lie}, \mathcal{Q}_1, \mathcal{Q}_2$.
  Now consider the degenerate quadric $\mathcal{Q}_3$ consisting of the bottom plane of the cube and the plane spanned by $\ell_1, \ell_4, m_1$.
  Then $\mathcal{Q}_3$ contains seven of the eight points, and therefore, according to Lemma \ref{lem:associated-points},
  also the eighth point $m_4$. Since $m_4$ may not lie in the bottom plane,
  we conclude that the quadrilateral $(\ell_1, \ell_4, m_1, m_4)$ is circumscribed.
\end{proof}

We now introduce nets consisting of two families of oriented circles
such that every second Lie quadrilateral (in a checkerboard-manner) is circumscribed.
\begin{definition}[Lie checkerboard incircular nets]
  \label{def:lie-cbic-net}
  Two families $(\ell_i)_{i\in\Z}$, $(m_j)_{j\in\Z}$ of oriented circles on the sphere
  are called a \emph{Lie checkerboard incircular net} if for every $i, j \in \Z$ with even $i+j$
  the Lie quadrilateral $(\ell_i, \ell_{i+1}, m_j, m_{j+1})$ is circumscribed.
\end{definition}
In the following we will always assume generic Lie checkerboard incircular nets in the sense of Remark \ref{rem:miquel-genericity}.
As an immediate consequence of Theorem \ref{thm:miquel-lie} we find that Lie checkerboard incircular nets
have many more circumscribed Lie quadrilaterals than introduced in its definition.
\begin{corollary}
  \label{cor:cbic-all-circles}
  Let $(\ell_i)_{i\in\Z}$, $(m_j)_{j\in\Z}$ be the oriented circles of a Lie checkerboard incircular net.
  Then for every $i, j, k \in \Z$ with even $i+j$ the Lie quadrilateral $(\ell_i, m_j, \ell_{i+2k+1}, m_{j+2k+1})$ is circumscribed.
\end{corollary}
Similar to the argument in the proof of Theorem \ref{thm:miquel-lie} (or as a consequence thereof),
we find that the points on the Lie quadric corresponding to a Lie checkerboard incircular net
can not span the entire space.
\begin{figure}
  \centering
  \def\svgwidth{0.5\textwidth}
  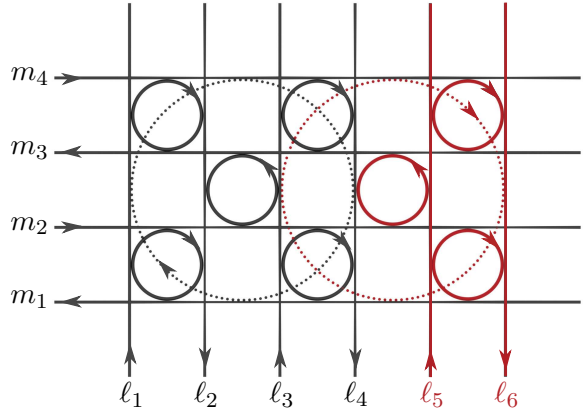
  \caption{
    On the combinatorics of adjacent ``cubes'' of a checkerboard incircular net.
  }
\label{fig:lie-cbic-dimension}
\end{figure}
\begin{theorem}
  \label{thm:lie-cbic-dimension}
  The points on the Lie quadric $\lie \subset \RP^4$ corresponding to
  the oriented circles of a Lie checkerboard incircular net lie in a common hyperplane of $\RP^4$.
\end{theorem}
\begin{proof}
  Consider ``adjacent'' cubes $(\ell_1, \ell_2, \ell_3, \ell_4, m_1, m_2, m_3, m_4)$
  and $(\ell_3, \ell_4, \ell_5, \ell_5, m_1, m_2, m_3, m_4)$ from the Lie checkerboard incircular net
  with vertices on the Lie quadric (see Figure \ref{fig:lie-cbic-dimension}).
  Each of these cubes lies in a three-dimensional subspace of $\RP^4$,
  and they coincide in six of its eight vertices.
  Thus, both cubes, and by induction the whole net, lie in the same three-dimensional subspace.
\end{proof}
As we have seen in Section \ref{sec:lie}, depending on its signature,
a three-dimensional subspace of the Lie quadric induces one of the three types
of Laguerre geometry.
Thus, our study of Lie checkerboard incircular nets may be reduced to the study
of its three Laguerre geometric counterparts as we will see in the next section.

\subsection{Laguerre checkerboard incircular nets}
\begin{figure}
  \centering
  \def\svgwidth{0.18\textwidth}
  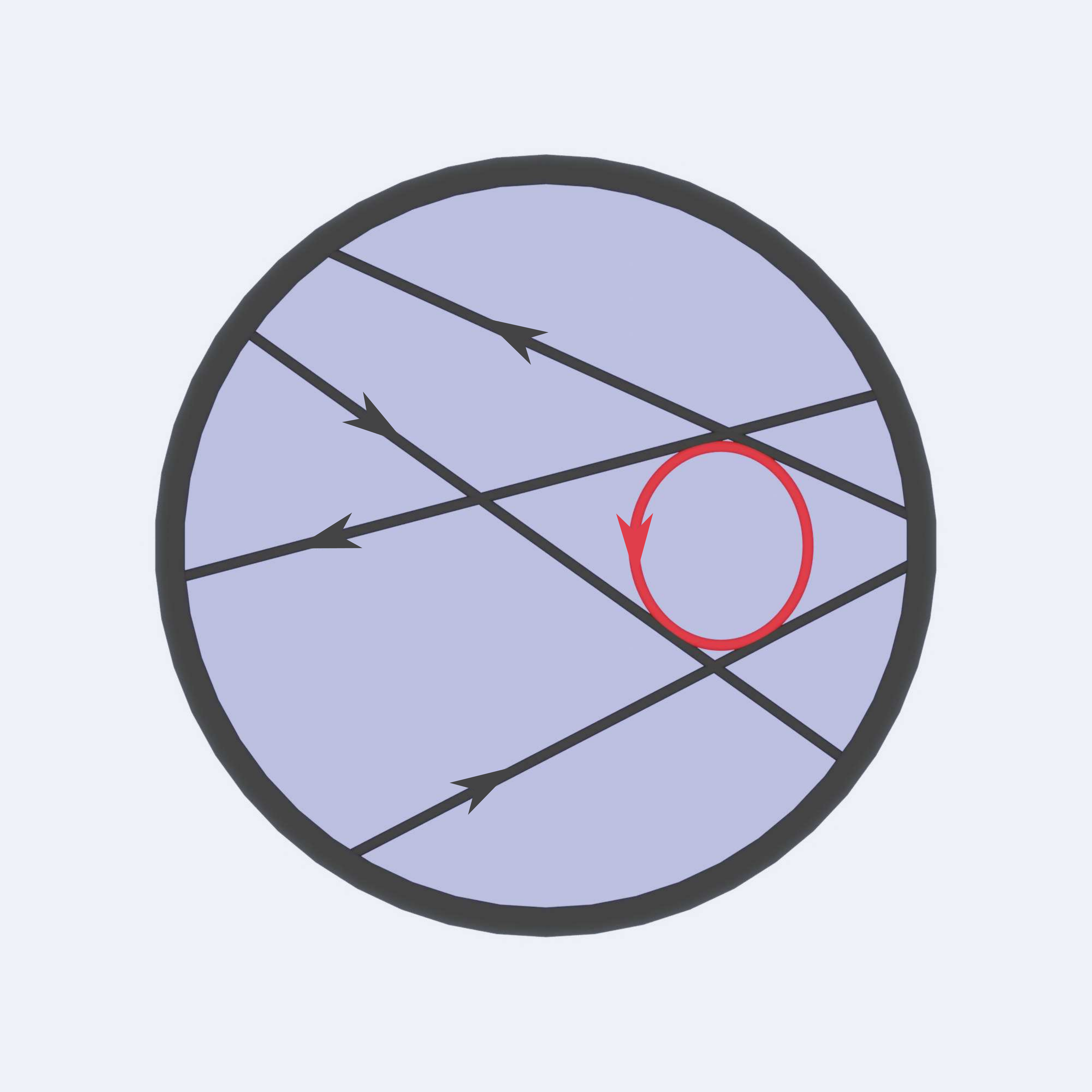
  \def\svgwidth{0.18\textwidth}
  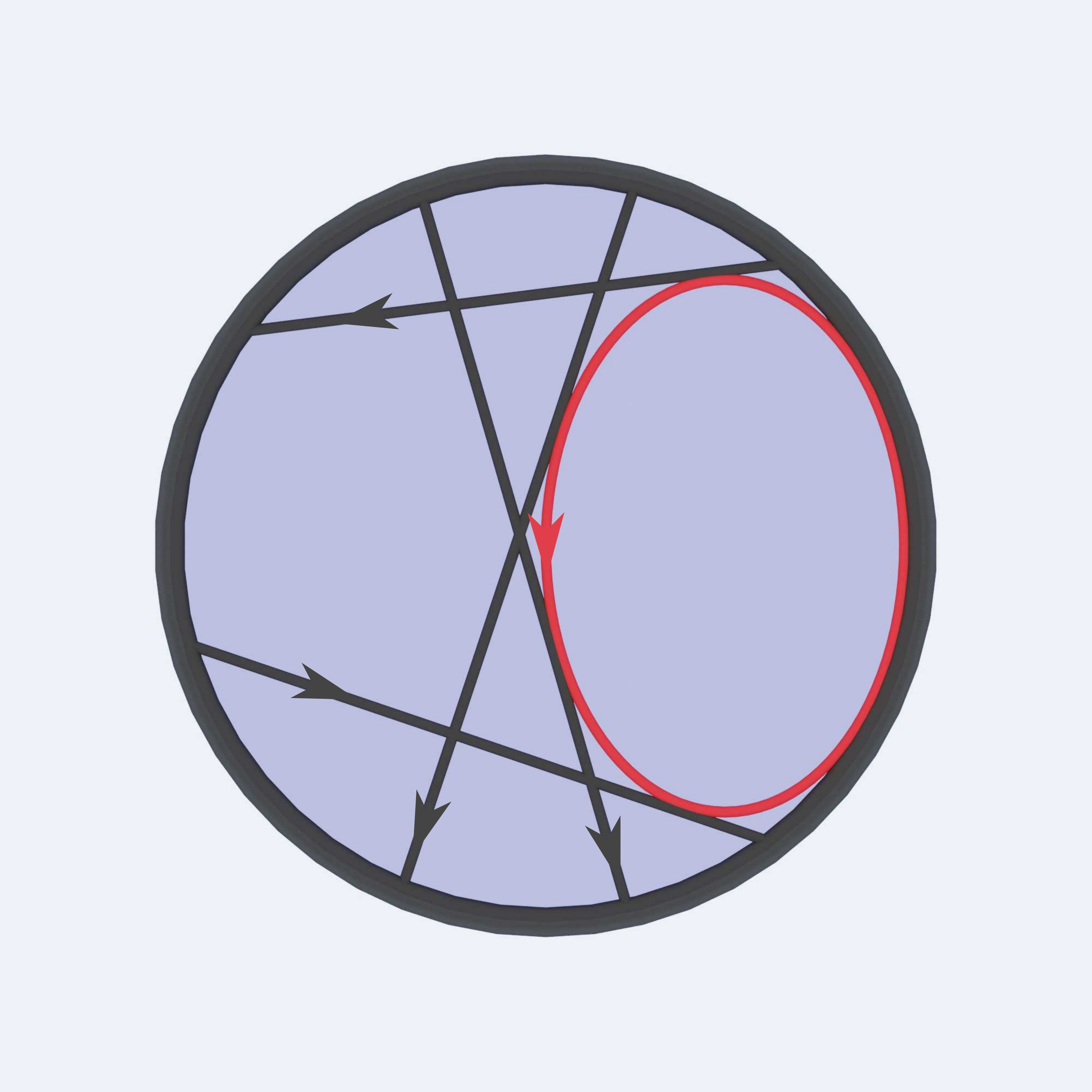
  \def\svgwidth{0.18\textwidth}
  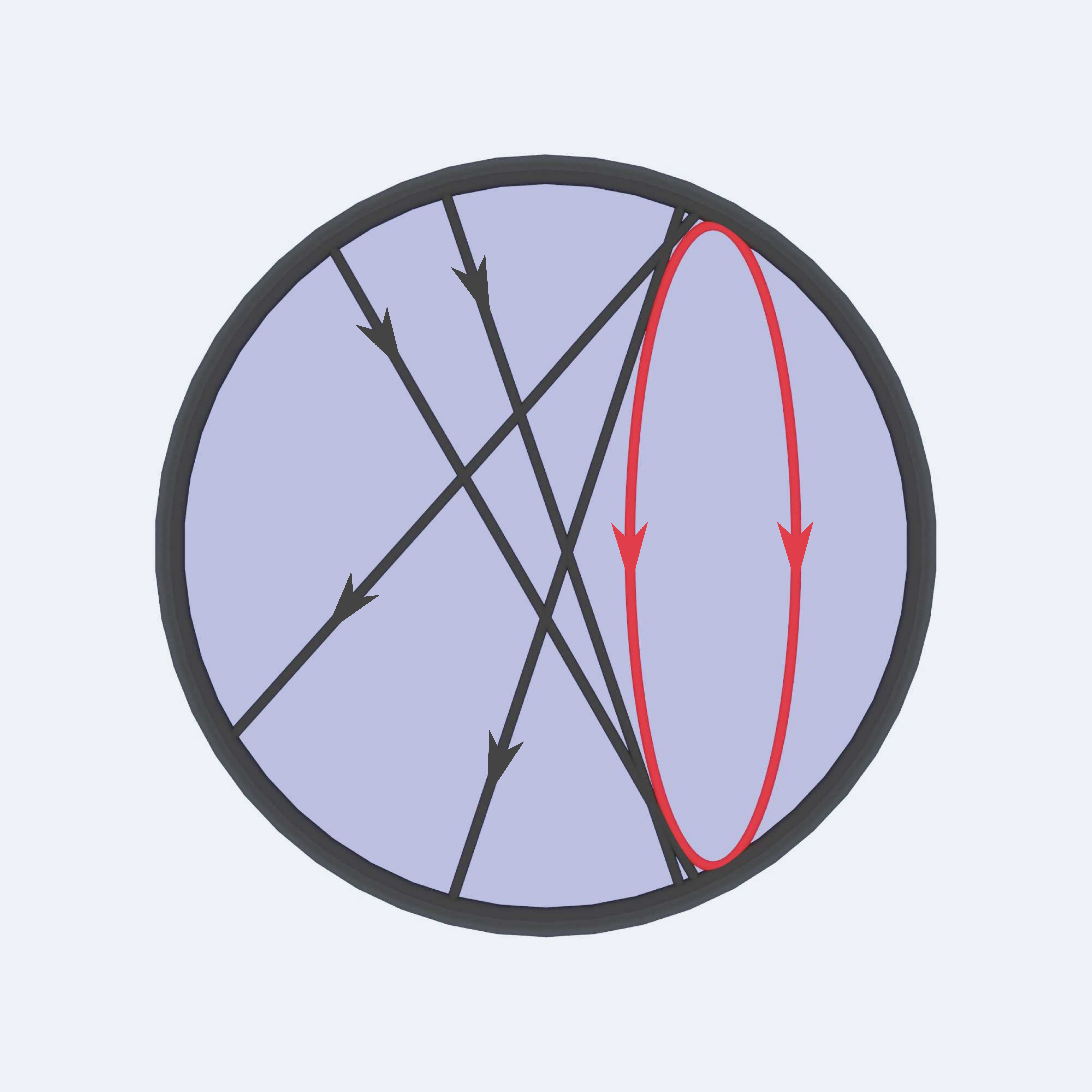
  \def\svgwidth{0.18\textwidth}
  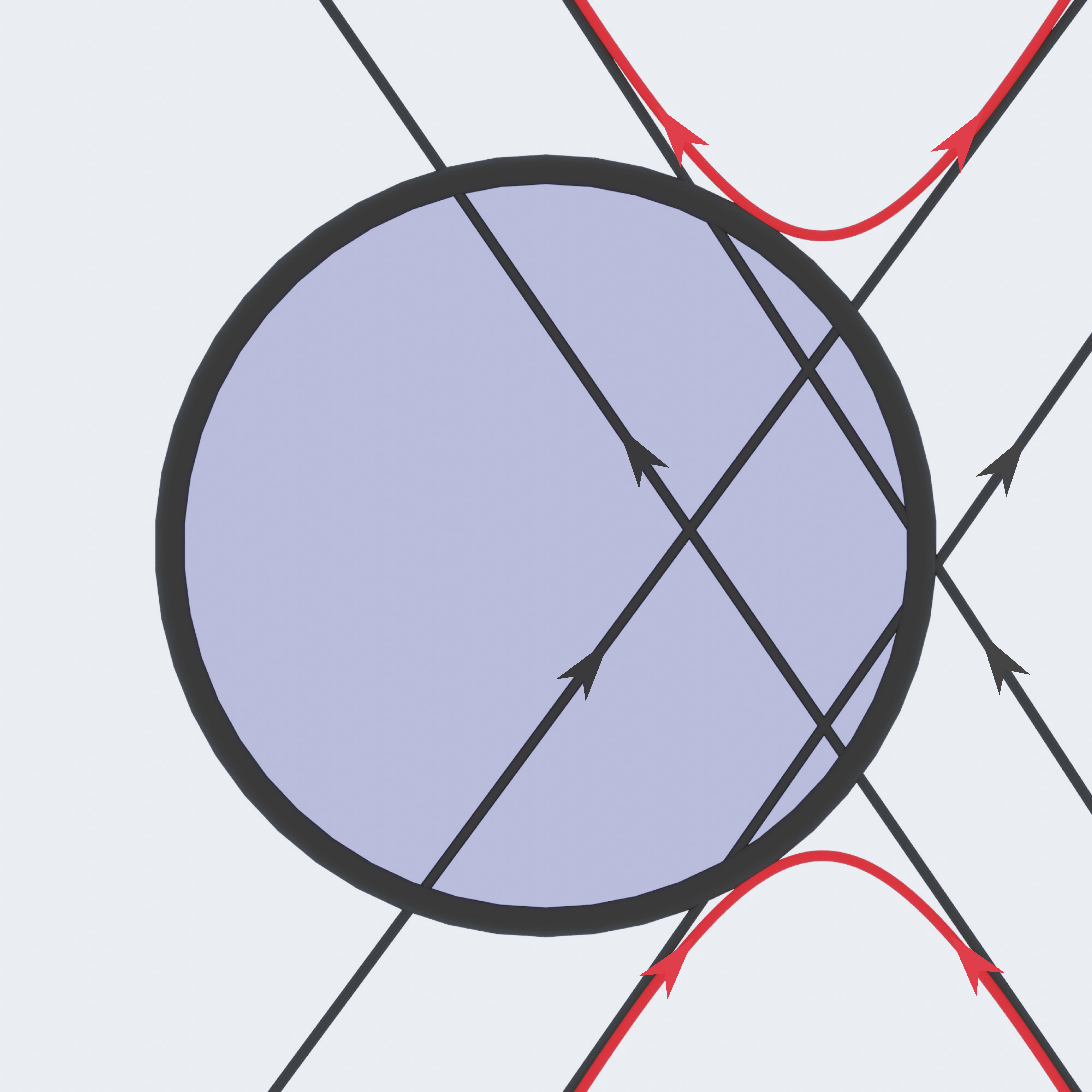
  \def\svgwidth{0.18\textwidth}
  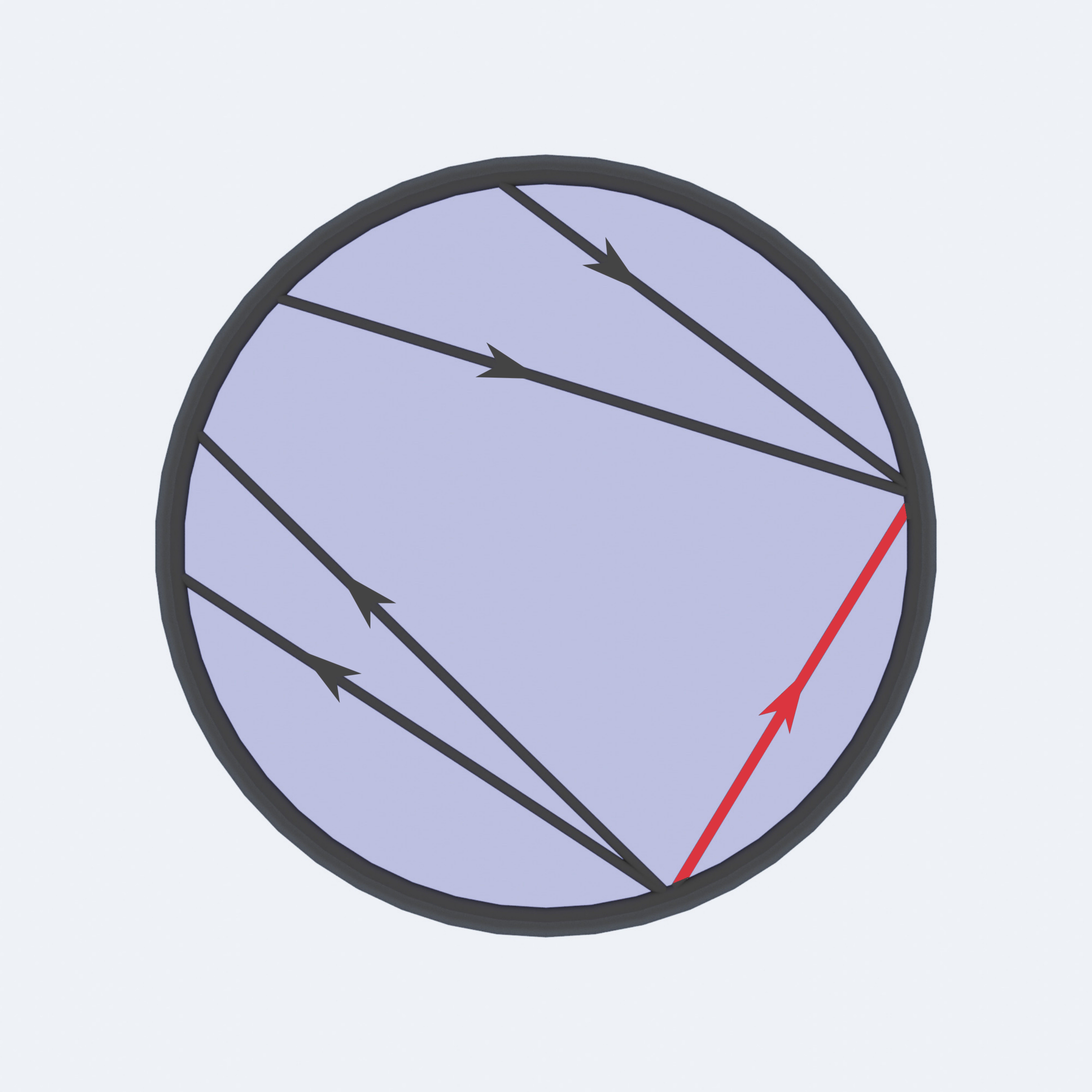
  \caption{
    Inscribed quadrilaterals in the hyperbolic plane.
    The right most case is degenerate and consists of four oriented lines ``touching'' an oriented line
    at its points at infinity.
  }
\label{fig:inscribed-quads}
\end{figure}

Two dimensional hyperbolic/elliptic/Euclidean Laguerre geometry is the geometry of oriented lines
in the hyperbolic/elliptic/Euclidean plane and their oriented contact to oriented circles (Laguerre circles) in the respective space form.
We identify oriented lines with points on, and oriented circles with planar sections of, the corresponding Laguerre quadric $\lag$,
which is a quadric of signature $(++--)$, $(+++-)$, $(++-0)$ respectively
(see Sections \ref{sec:hyperbolic-laguerre-geometry}, \ref{sec:elliptic-laguerre-geometry}, \ref{sec:euclidean-laguerre-geometry}).

Similar to the condition for Lie circumscribility, four oriented lines touch a common oriented circle
if and only if the corresponding points on the Laguerre quadric are coplanar.
On the other hand, all three Laguerre geometries are subgeometries of Lie geometry,
by restricting the Lie quadric to a three-dimensional subspace.
Thus, in this restriction, a Lie circumscribed quadrilateral turns into four lines touching a common oriented circle.
Accordingly one obtains the following Laguerre geometric version of Theorem \ref{thm:miquel-lie}.
\begin{theorem}[Miquel's theorem in Laguerre geometry]
  \label{thm:miquel-laguerre}
  Let $\ell_1, \ell_2, \ell_3, \ell_4, m_1, m_2, m_3, m_4$ be eight generic oriented lines in the hyperbolic/elliptic/Euclidean plane
  such that the five quadrilaterals $(\ell_1, \ell_2, m_1, m_2)$, $(\ell_1, \ell_2, m_3, m_4)$, $(\ell_3, \ell_4, m_1, m_2)$,
  $(\ell_3, \ell_4, m_3, m_4)$, $(\ell_2, \ell_3, m_2, m_3)$ are circumscribed (each touches a common oriented circle),
  then so is the quadrilateral $(\ell_1, \ell_4, m_1, m_4)$ (cf.\ Figure \ref{fig:lie-cbic-dimension}).
\end{theorem}
The Laguerre geometric version of checkerboard incircular nets (see Definition \ref{def:lie-cbic-net}) is given the following definition \cite{AB}.
Examples of checkerboard incircular nets in the elliptic and hyperbolic plane are shown in Figures~\ref{fig:ell_cic_nets}, \ref{fig:hyp_cic_nets_ellipse}, and \ref{fig:hyp_cic_nets_hyperbola} (see also \cite{gallery}).
\begin{definition}[Laguerre checkerboard incircular nets]
  \label{def:cbic-net}
  Two families $(\ell_i)_{i\in\Z}$, $(m_j)_{j\in\Z}$ of oriented lines in
  the hyperbolic/elliptic/Euclidean plane
  are called a \emph{(hyperbolic/elliptic/Euclidean) checkerboard incircular net}
  if for every $i, j \in \Z$ with even $i+j$
  the four lines $\ell_i, \ell_{i+1}, m_j, m_{j+1}$ touch a common
  oriented circle (Laguerre circle).
\end{definition}
\begin{remark}
  \label{rem:laguerre-all-circles}
  From Corollary \ref{cor:cbic-all-circles}, or Theorem \ref{thm:miquel-laguerre}, we find that,
  same as in the Lie geometric case, every quadrilateral $(\ell_i, m_j, \ell_{i+2k+1}, m_{j+2k+1})$, $i, j, k \in \Z$
  of a checkerboard incircular net with even $i+j$ is circumscribed.
\end{remark}
Now we can formulate the following classification result for Lie checkerboard incircular nets.
\begin{theorem}[classification of Lie checkerboard incircular nets]
  \label{cbic-classification}
  Every Lie checkerboard incircular net is given by a Lie transformation of
  a hyperbolic, elliptic, or Euclidean checkerboard incircular net.
\end{theorem}
\begin{proof}
  According to Theorem \ref{thm:lie-cbic-dimension} every Lie checkerboard incircular net
  lies in a three-dimensional subspace of $\RP^4$.
  This subspace can only have one of the signatures $(+++-)$, $(++--)$, $(++-0)$ and thus may
  be identified (after a certain Lie transformation) with a checkerboard incircular net
  in the corresponding Laguerre geometry.
\end{proof}
In the different space forms different types of generic (see Remark \ref{rem:generic-quads}) circumscribed quadrilaterals (see Proposition \ref{prop:circumscribed-quads}) can still occur
(cf.\ Remark \ref{rem:lie-quadric-projection}):
\begin{itemize}
\item hyperbolic Laguerre geometry $(++--)$
  (see Remark \ref{rem:hyperbolic-Laguerre-spheres} \ref{rem:hyperbolic-Laguerre-spheres-distinction}
  and Figure \ref{fig:inscribed-quads})
  \begin{itemize}
  \item $(++-)$:
    Four lines touching a common oriented hyperbolic circle/distance curve/horocircle.
  \item $(+--)$:
    Four lines touching a common deSitter circle.
  \item $(+-0)$:
    Four lines touching a common oriented hyperbolic line at infinity.
  \end{itemize}
\item elliptic Laguerre geometry $(+++-)$:
  \begin{itemize}
  \item $(++-)$:
    Four lines touching a common oriented elliptic circle.
  \end{itemize}
\item Euclidean Laguerre geometry $(++-0)$:
  \begin{itemize}
  \item $(++-)$:
    Four lines touching a common oriented Euclidean circle.
  \item $(+-0)$:
    Four lines from two families of parallel oriented lines.
  \end{itemize}  
\end{itemize}

\begin{figure}
  \centering
  \includegraphics[width=0.62\textwidth]{ell_cic_sphere_1}\\
  \hspace{\fill}
  \includegraphics[width=0.47\textwidth]{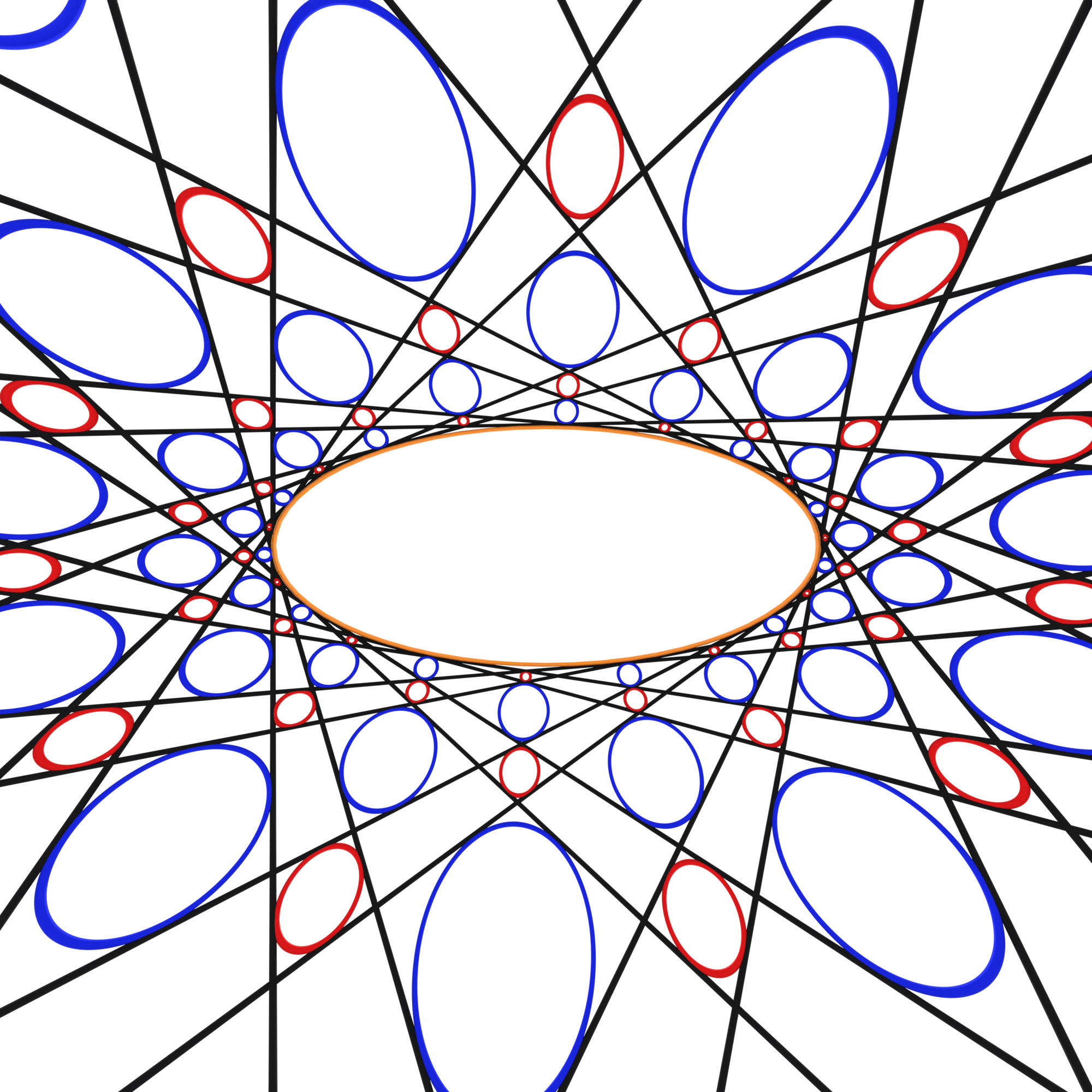}
  \hspace{\fill}
  \includegraphics[width=0.47\textwidth]{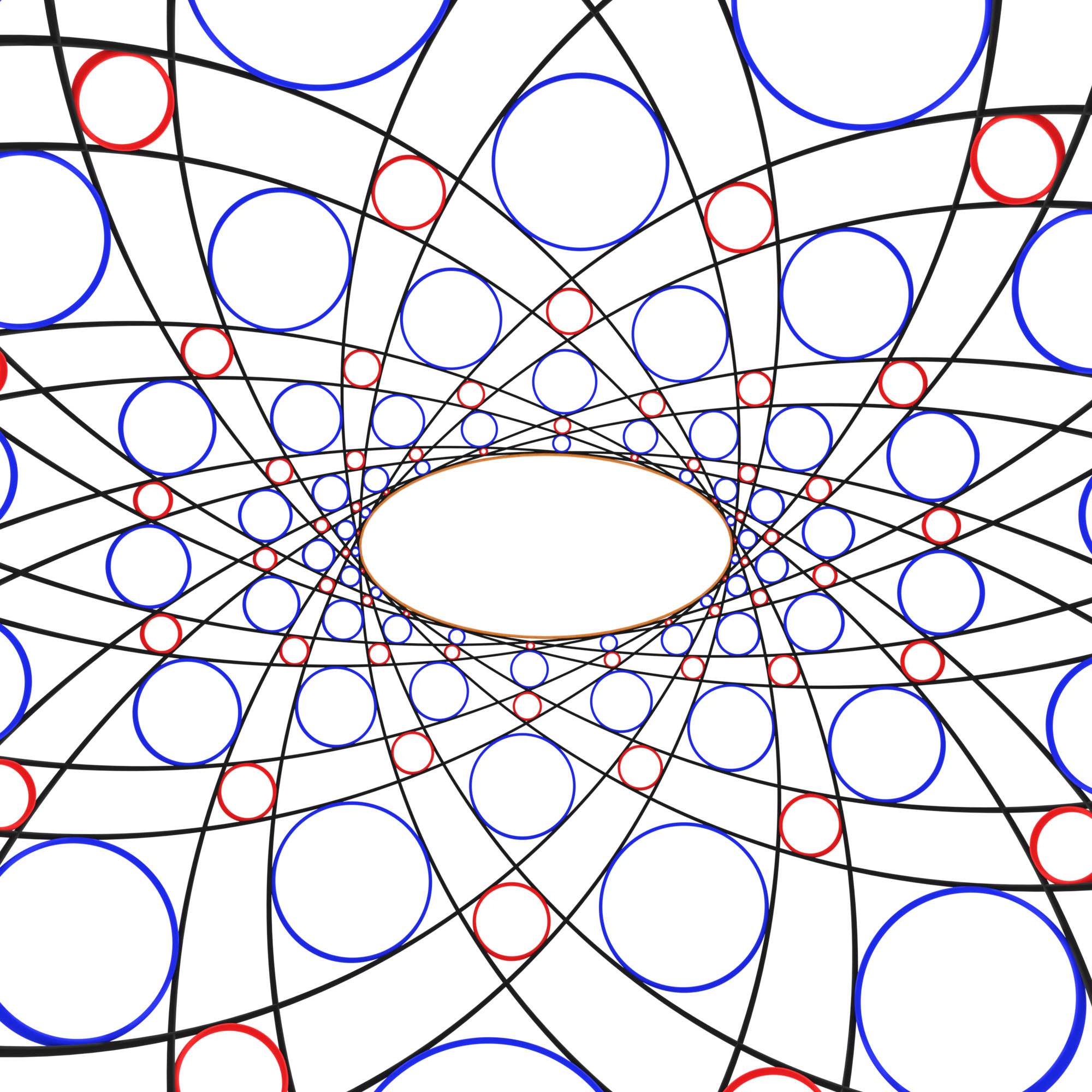}
  \hspace{\fill}
  \caption{
    \emph{Top:} Checkerboard incircular net tangent to an ellipse in the sphere model of the elliptic plane.
    \emph{Bottom-left:} Central projection to the projective model of the elliptic plane.
    \emph{Bottom-right:} Stereographic projection to a conformal model of the elliptic plane.
  }
\label{fig:ell_cic_nets}
\end{figure}
\begin{figure}
  \centering
  \includegraphics[width=0.62\textwidth]{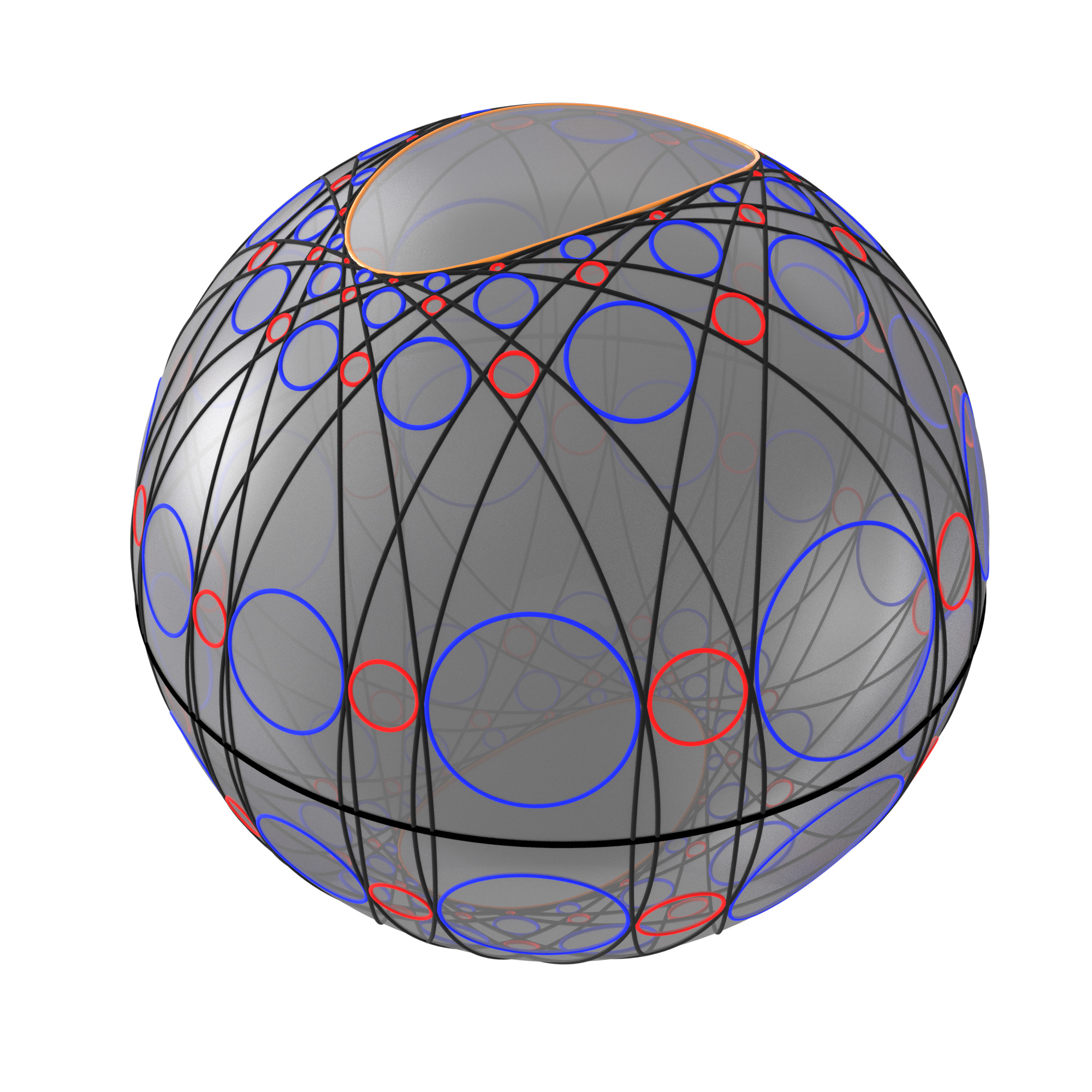}\\
  \hspace{\fill}
  \includegraphics[width=0.47\textwidth]{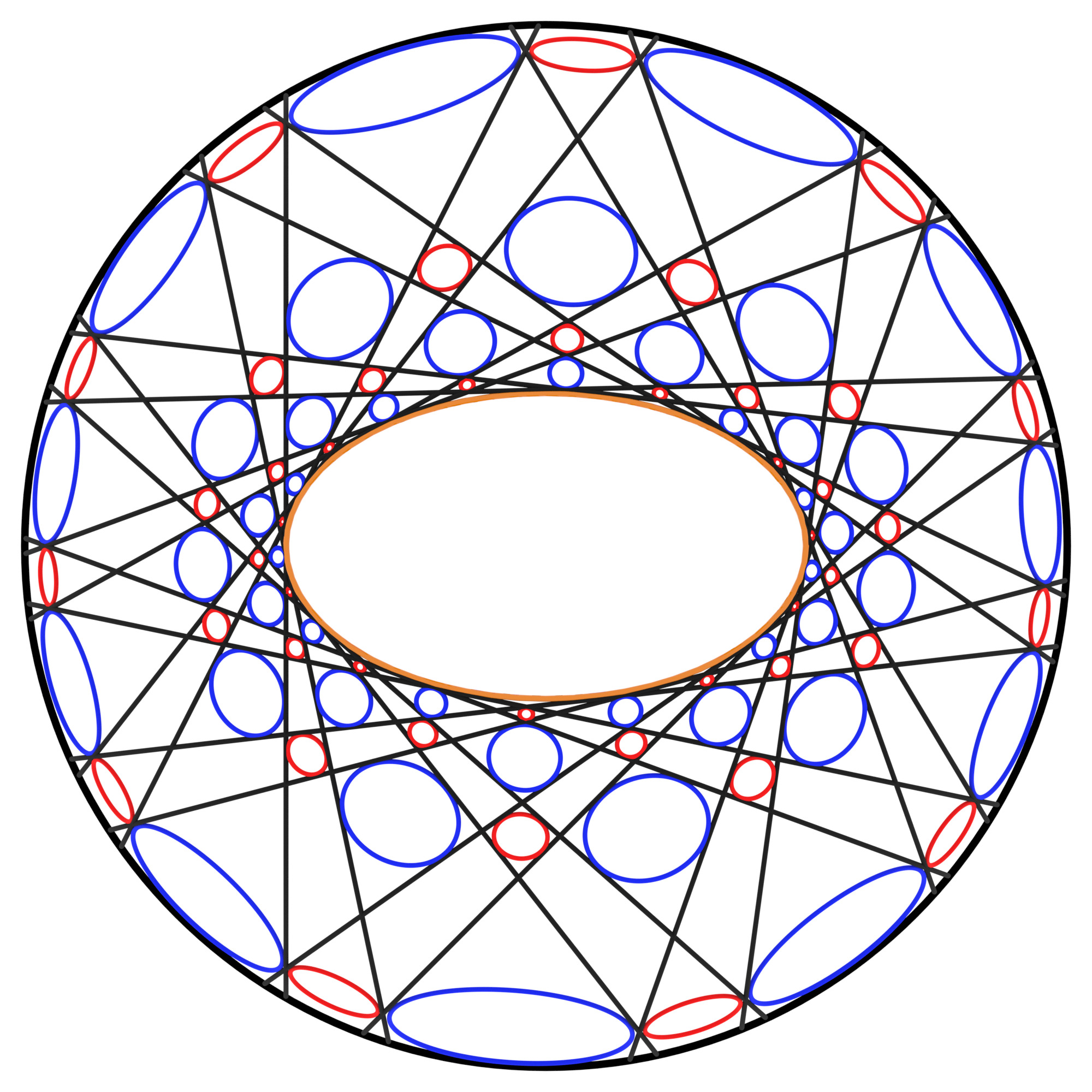}
  \hspace{\fill}
  \includegraphics[width=0.47\textwidth]{hyp_cic_ellipse_stereog_projection_1}
  \hspace{\fill}\\
  \vspace{1cm}
  \includegraphics[width=0.65\textwidth]{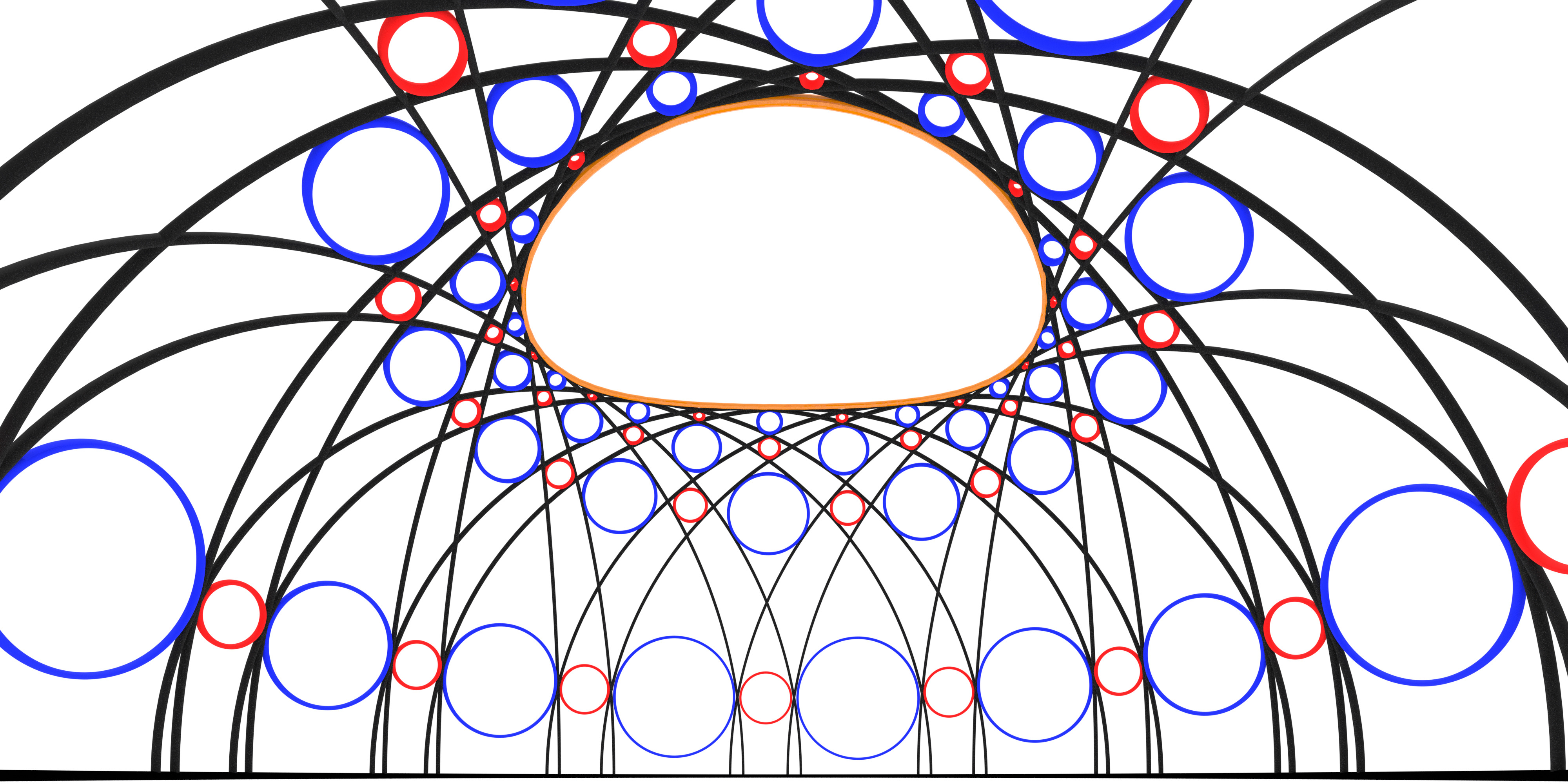}
  \caption{
    \emph{Top:} Checkerboard incircular net tangent to an ellipse in the sphere model of the hyperbolic plane.
    Two copies of the hyperbolic plane are realized as half-spheres.
    \emph{Middle-left:} Orthogonal projection to the Klein-Beltrami disk model.
    \emph{Middle-right:} Stereographic projection to the Poincaré disk model.
    \emph{Bottom:} Stereographic projection to the Poincaré half-plane model.
  }
\label{fig:hyp_cic_nets_ellipse}
\end{figure}
\begin{figure}
  \centering
  \includegraphics[width=0.62\textwidth]{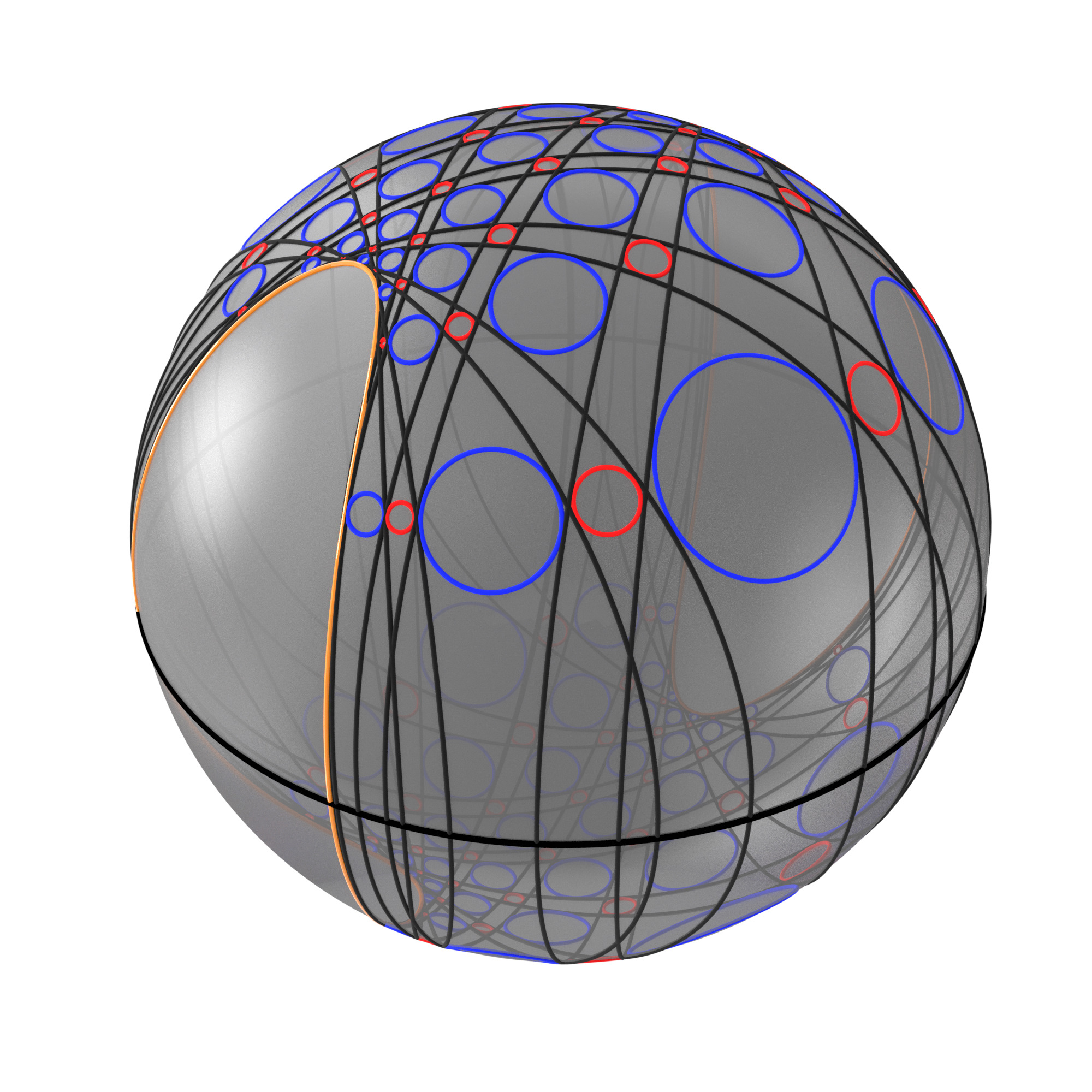}\\
  \hspace{\fill}
  \includegraphics[width=0.47\textwidth]{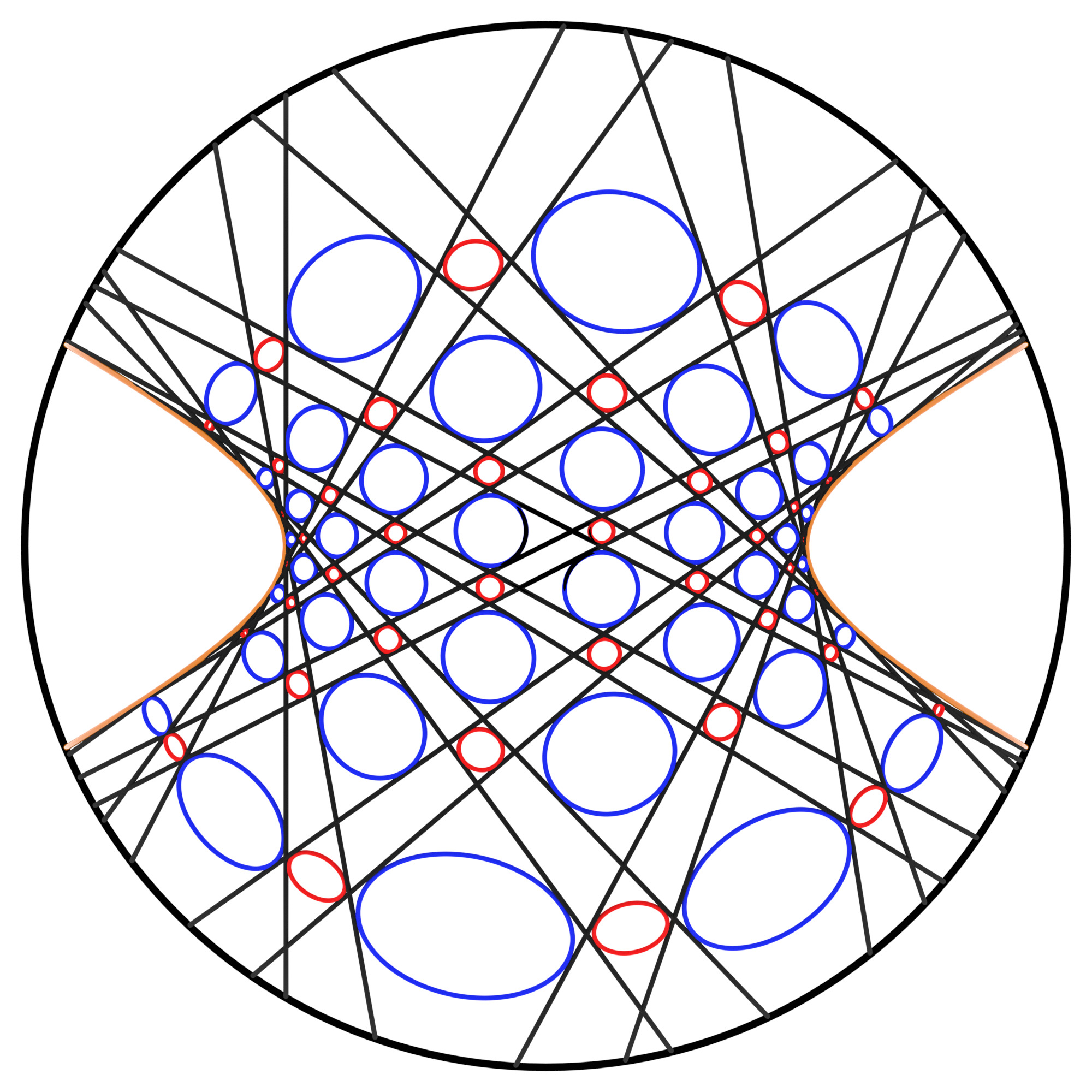}
  \hspace{\fill}
  \includegraphics[width=0.47\textwidth]{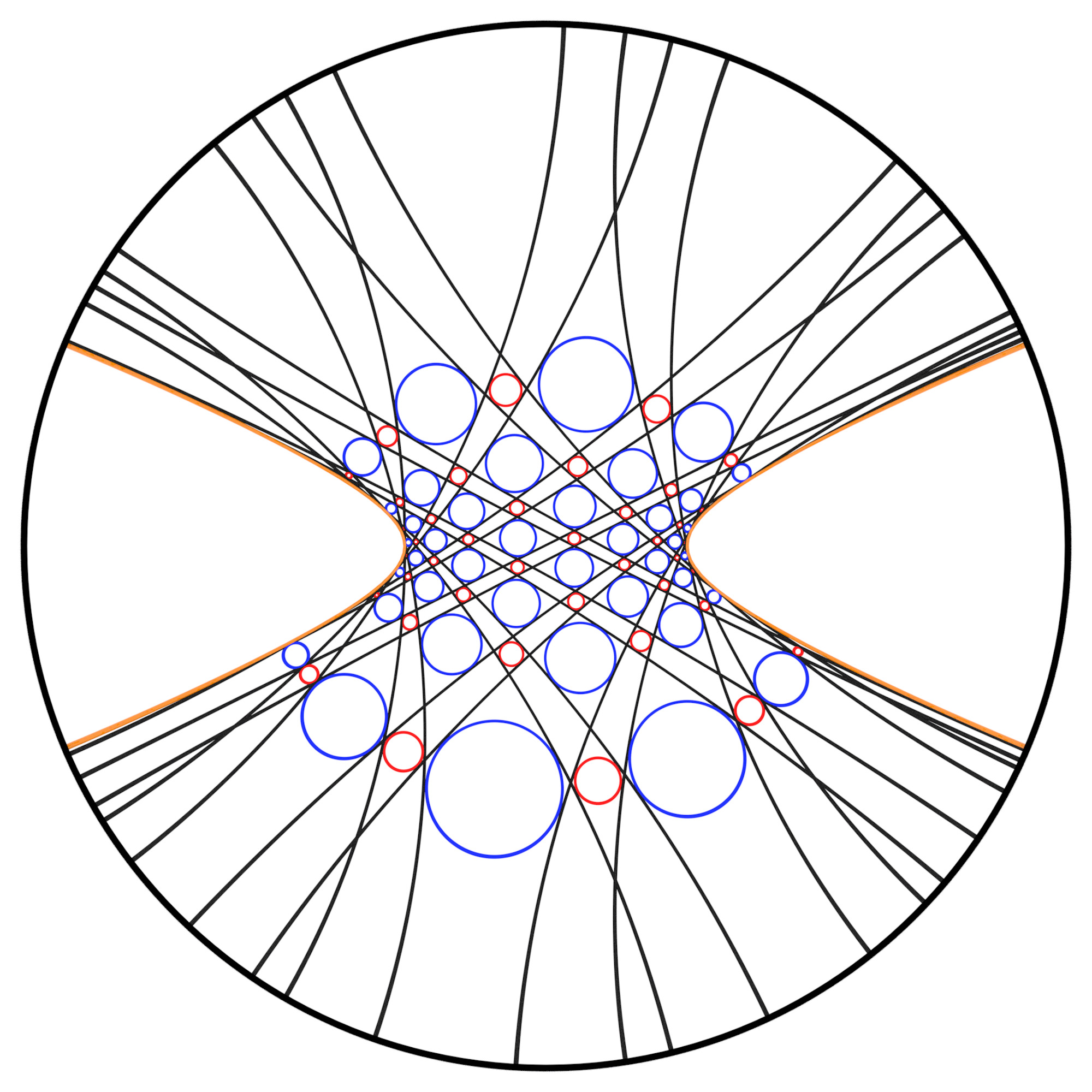}
  \hspace{\fill}\\
  \vspace{1cm}
  \includegraphics[width=0.65\textwidth]{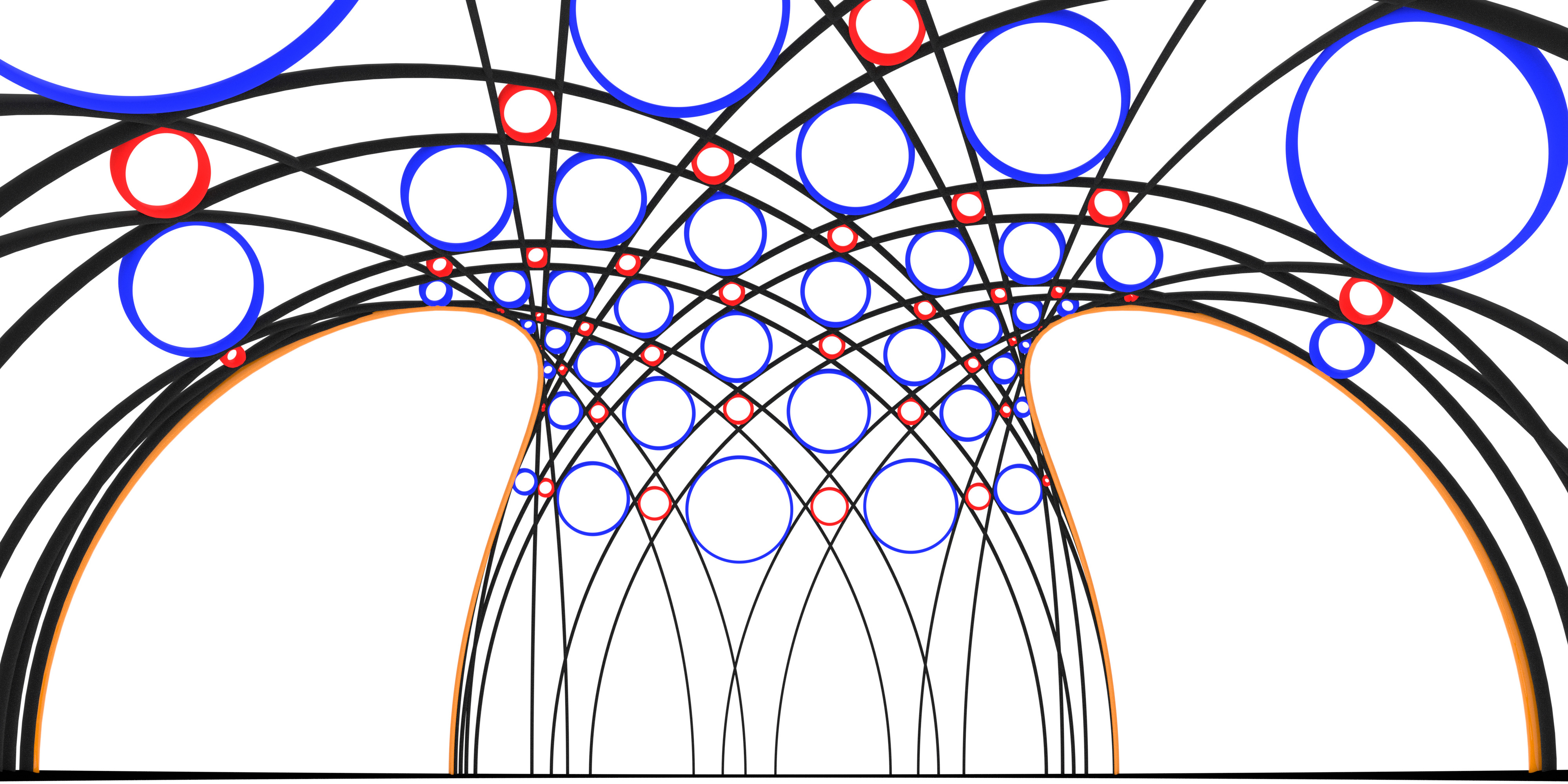}
  \caption{
    \emph{Top:} Checkerboard incircular net tangent to a hyperbola in the sphere model of the hyperbolic plane.
    Two copies of the hyperbolic plane are realized as half-spheres.
    \emph{Middle-left:} Orthogonal projection to the Klein-Beltrami disk model.
    \emph{Middle-right:} Stereographic projection to the Poincaré disk model.
    \emph{Bottom:} Stereographic projection to the Poincaré half-plane model.
  }
\label{fig:hyp_cic_nets_hyperbola}
\end{figure}

\newpage
\subsection{Hypercycles}
\begin{figure}
  \centering
  \def\svgwidth{0.55\textwidth}
  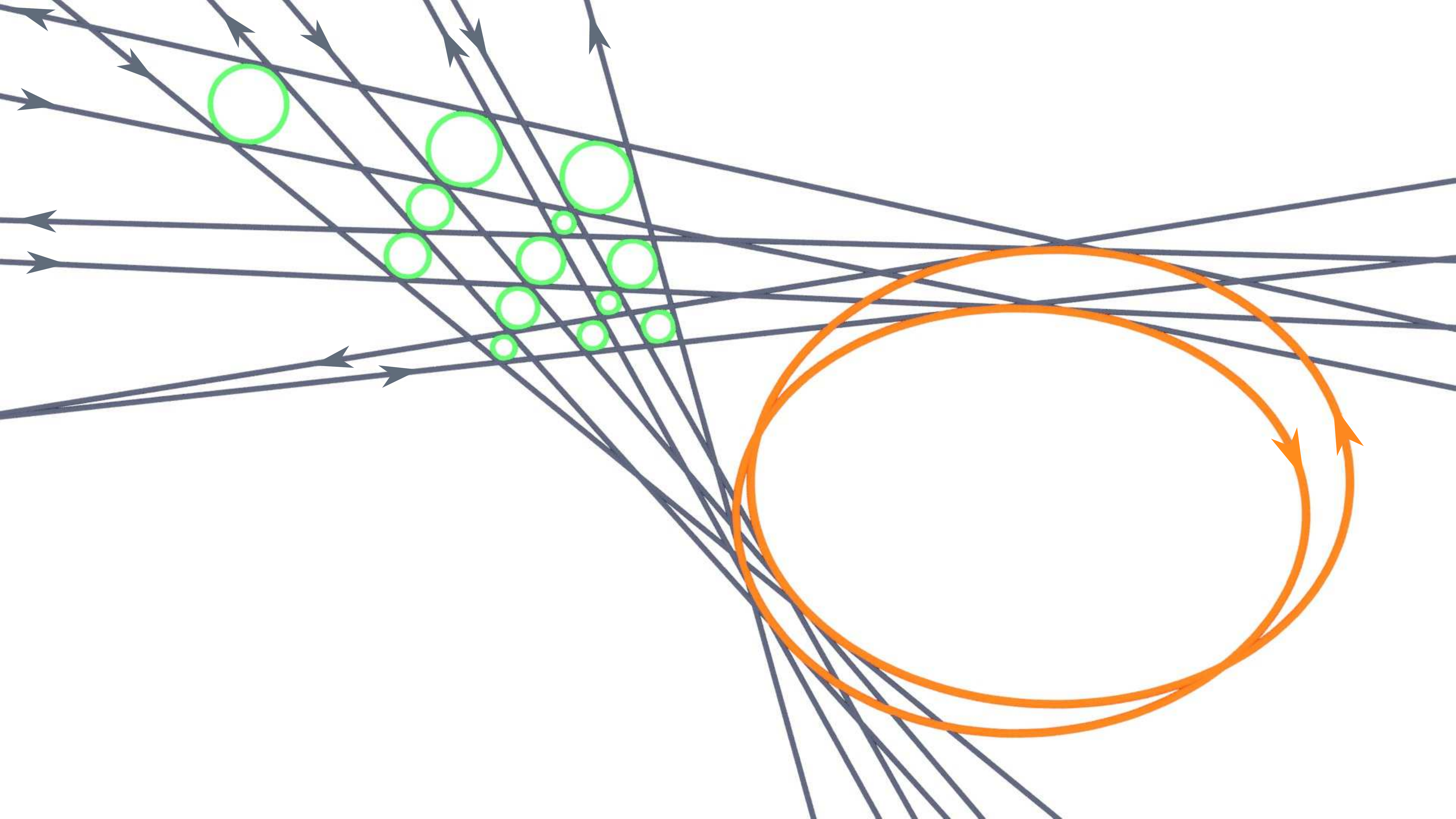
  % \raisebox{-0.1\height}{
  %   \includegraphics[scale=0.16]{pics_final/elliptic_cone}
  % }
  \caption{
    Hypercycle in the Euclidean plane.
  }
\label{fig:hypercycle}
\end{figure}
In Laguerre geometry the oriented lines, and not the points, of a given space form are invariant objects.
Thus, in Laguerre geometry, it is natural to describe an (oriented) curve in the
(hyperbolic/elliptic/Euclidean) plane by its (oriented) tangent lines.
% \footnote{The classical german terminology for this description is ``Klassenkurve'' \cite{K}.}
Conversely, we say that every curve on the Laguerre quadric corresponds to a curve in the plane.
Note that in the case of the hyperbolic plane the envelope of such a ``curve'' might lie partially
(or even entirely) ``outside'' the hyperbolic plane.
We still consider this to be an admissible (non-empty) \emph{Laguerre curve}.
\begin{definition}
  The one-parameter family of oriented lines (in the hyperbolic/elliptic/Euclidean plane)
  corresponding to a curve on the Laguerre quadric $\lag$
  is called a \emph{(hyperbolic/elliptic/Euclidean) Laguerre curve}.
\end{definition}
We have noted that planar sections of the Laguerre quadric correspond to Laguerre circles,
also called (generalized) \emph{cycles} in the two-dimensional case.
Consequently, the next higher order intersections with the Laguerre quadric are called \emph{hypercycles} \cite{Bl1}.
\begin{definition}
  A (hyperbolic/elliptic/Euclidean) Laguerre curve corresponding to the intersection of the Laguerre quadric
  with another quadric is called a \emph{(hyperbolic/elliptic/Euclidean) hypercycle}.
  The corresponding curve on the Laguerre quadric is called \emph{hypercycle base curve}.
\end{definition}
\begin{example}
  In every space form a conic endowed with both orientations, joined together as the two components of a single oriented curve (see Section \ref{sec:conics}) is a hypercycle.
  A more general example is shown in Figure \ref{fig:hypercycle}.
\end{example}
The intersection curve of two quadrics (\emph{base curve}) is contained in all quadrics of the pencil spanned by the two quadrics.
Thus, a hypercycle, through its hypercycle base curve, corresponds not just to one quadric
but the whole pencil of quadrics spanned by it and the Laguerre quadric.

We call a hypercycle \emph{non-degenerate} if its hypercycle base curve contains at least 8 points in general position.
In this case the hypercycle can be uniquely identified with the corresponding pencil of quadrics.
% We call a non-degenerate hypercycle \emph{generic} if the corresponding pencil is generic, in the sense that it contains four (distinct, possibly imaginary) degenerate quadrics.
In the following we assume all hypercycles to be non-degenerate.

The following theorem establishes a relation between a checkerboard incircular net and a hypercycle,
as well as two certain hyperboloids in the pencil of quadrics corresponding to its hypercycle base curve.
In the Euclidean case this was shown in \cite[Theorem 3.4]{BST} as part of an incidence theorem
for checkerboard incircular nets (see Theorem \ref{thm:5x5-incidence}).
% A proof of Lemma \ref{lem:pencils-and-lines} can also be found there.
%
\begin{figure}[h]
    \begin{center}
      \begin{tikzpicture}[line cap=line join=round,>=stealth,x=1.0cm,y=1.0cm, scale=0.7]
        \begin{scriptsize}
          %lines
          \draw (0, -0.8) node {$\ell_1$};
          \draw (1, -0.8) node {$\ell_2$};
          \draw (2, -0.8) node {$\ell_3$};
          \draw (3, -0.8) node {$\ell_4$};
          \draw (4, -0.8) node {$\ell_5$};
          \draw (5, -0.8) node {$\ell_6$};
          \draw (-0.8, 0) node {$m_1$};
          \draw (-0.8, 1) node {$m_2$};
          \draw (-0.8, 2) node {$m_3$};
          \draw (-0.8, 3) node {$m_4$};
          \draw (-0.8, 4) node {$m_5$};
          \draw (-0.8, 5) node {$m_6$};
          % generators
          \draw [color=blue] (0.5, -0.6) node {$L_1$};
          \draw [color=red]  (1.5, -0.6) node {$L_2$};
          \draw [color=blue] (2.5, -0.6) node {$L_3$};
          \draw [color=red]  (3.5, -0.6) node {$L_4$};
          \draw [color=blue] (4.5, -0.6) node {$L_5$};
          \draw [color=blue] (-0.6, 0.5) node {$M_1$};
          \draw [color=red]  (-0.6, 1.5) node {$M_2$};
          \draw [color=blue] (-0.6, 2.5) node {$M_3$};
          \draw [color=red]  (-0.6, 3.5) node {$M_4$};
          \draw [color=blue] (-0.6, 4.5) node {$M_5$};
        \end{scriptsize}
        \clip(-0.5,-0.5) rectangle (6,6);
        % generators
        \draw [line width=0.8pt, color=red,  domain=-0.3:5.3] plot(1.5,\x);
        \draw [line width=0.8pt, color=red,  domain=-0.3:5.3] plot(3.5,\x);
        \draw [line width=0.8pt, color=red,  domain=-0.3:5.3] plot(\x,1.5);
        \draw [line width=0.8pt, color=red,  domain=-0.3:5.3] plot(\x,3.5);
        \draw [line width=0.8pt, color=blue, domain=-0.3:5.3] plot(0.5,\x);
        \draw [line width=0.8pt, color=blue, domain=-0.3:5.3] plot(2.5,\x);
        \draw [line width=0.8pt, color=blue, domain=-0.3:5.3] plot(4.5,\x);
        \draw [line width=0.8pt, color=blue, domain=-0.3:5.3] plot(\x,0.5);
        \draw [line width=0.8pt, color=blue, domain=-0.3:5.3] plot(\x,2.5);
        \draw [line width=0.8pt, color=blue, domain=-0.3:5.3] plot(\x,4.5);
        %circles
        \draw [->, thick] (1.0,0.5) arc (0:-685:0.5);
        \draw [->, thick] (3.0,0.5) arc (0:-685:0.5);
        \draw [->, thick] (5.0,0.5) arc (0:-685:0.5);
        \draw [->, thick] (1.0,2.5) arc (0:-685:0.5);
        \draw [->, thick] (3.0,2.5) arc (0:-685:0.5);
        \draw [->, thick] (5.0,2.5) arc (0:-685:0.5);
        \draw [->, thick] (1.0,4.5) arc (0:-685:0.5);
        \draw [->, thick] (3.0,4.5) arc (0:-685:0.5);
        %% \draw [->, thick, dotted] (5.0,4.5) arc (0:-685:0.5);
        \draw [->, thick] (5.0,4.5) arc (0:-685:0.5);
        \draw [->, thick] (2.0,1.5) arc (0:415:0.5);
        \draw [->, thick] (4.0,1.5) arc (0:415:0.5);
        \draw [->, thick] (2.0,3.5) arc (0:415:0.5);
        \draw [->, thick] (4.0,3.5) arc (0:415:0.5);
        \begin{tiny}
          \def\co{0.2}
          \draw ({0.5+\co},{0.5+\co}) node {$S_1$};
          \draw ({2.5+\co},{0.5+\co}) node {$S_2$};
          \draw ({4.5+\co},{0.5+\co}) node {$S_3$};
          \draw ({1.5+\co},{1.5+\co}) node {$S_4$};
          \draw ({3.5+\co},{1.5+\co}) node {$S_5$};
          \draw ({0.5+\co},{2.5+\co}) node {$S_6$};
          \draw ({2.5+\co},{2.5+\co}) node {$S_7$};
          \draw ({4.5+\co},{2.5+\co}) node {$S_8$};
          \draw ({1.5+\co},{3.5+\co}) node {$S_9$};
          \draw ({3.5+\co},{3.5+\co}) node {$S_{10}$};
          \draw ({0.5+\co},{4.5+\co}) node {$S_{11}$};
          \draw ({2.5+\co},{4.5+\co}) node {$S_{12}$};
          \draw ({4.5+\co},{4.5+\co}) node {$S_{13}$};
        \end{tiny}
        % l lines
        \draw [line width=1pt, domain=-0.5:5.5] plot(0,\x);
        \draw [>-, thick] (0,-0.4) -- (0,-0.3);
        \draw [line width=1pt, domain=-0.5:5.5] plot(1,\x);
        \draw [<-, thick] (1,-0.4) -- (1,-0.3);
        \draw [line width=1pt, domain=-0.5:5.5] plot(2,\x);
        \draw [>-, thick] (2,-0.4) -- (2,-0.3);
        \draw [line width=1pt, domain=-0.5:5.5] plot(3,\x);
        \draw [<-, thick] (3,-0.4) -- (3,-0.3);
        \draw [line width=1pt, domain=-0.5:5.5] plot(4,\x);
        \draw [>-, thick] (4,-0.4) -- (4,-0.3);
        \draw [line width=1pt, domain=-0.5:5.5] plot(5,\x);
        \draw [<-, thick] (5,-0.4) -- (5,-0.3);
        %m lines
        \draw [line width=1pt, domain=-0.5:5.5] plot(\x,0);
        \draw [<-, thick] (-0.4,0) -- (-0.3,0);
        \draw [line width=1pt, domain=-0.5:5.5] plot(\x,1);
        \draw [>-, thick] (-0.4,1) -- (-0.3,1);
        \draw [line width=1pt, domain=-0.5:5.5] plot(\x,2);
        \draw [<-, thick] (-0.4,2) -- (-0.3,2);
        \draw [line width=1pt, domain=-0.5:5.5] plot(\x,3);
        \draw [>-, thick] (-0.4,3) -- (-0.3,3);
        \draw [line width=1pt, domain=-0.5:5.5] plot(\x,4);
        \draw [<-, thick] (-0.4,4) -- (-0.3,4);
        \draw [line width=1pt, domain=-0.5:5.5] plot(\x,5);
        \draw [>-, thick] (-0.4,5) -- (-0.3,5);
      \end{tikzpicture}
      \def\svgwidth{0.32\textwidth}
      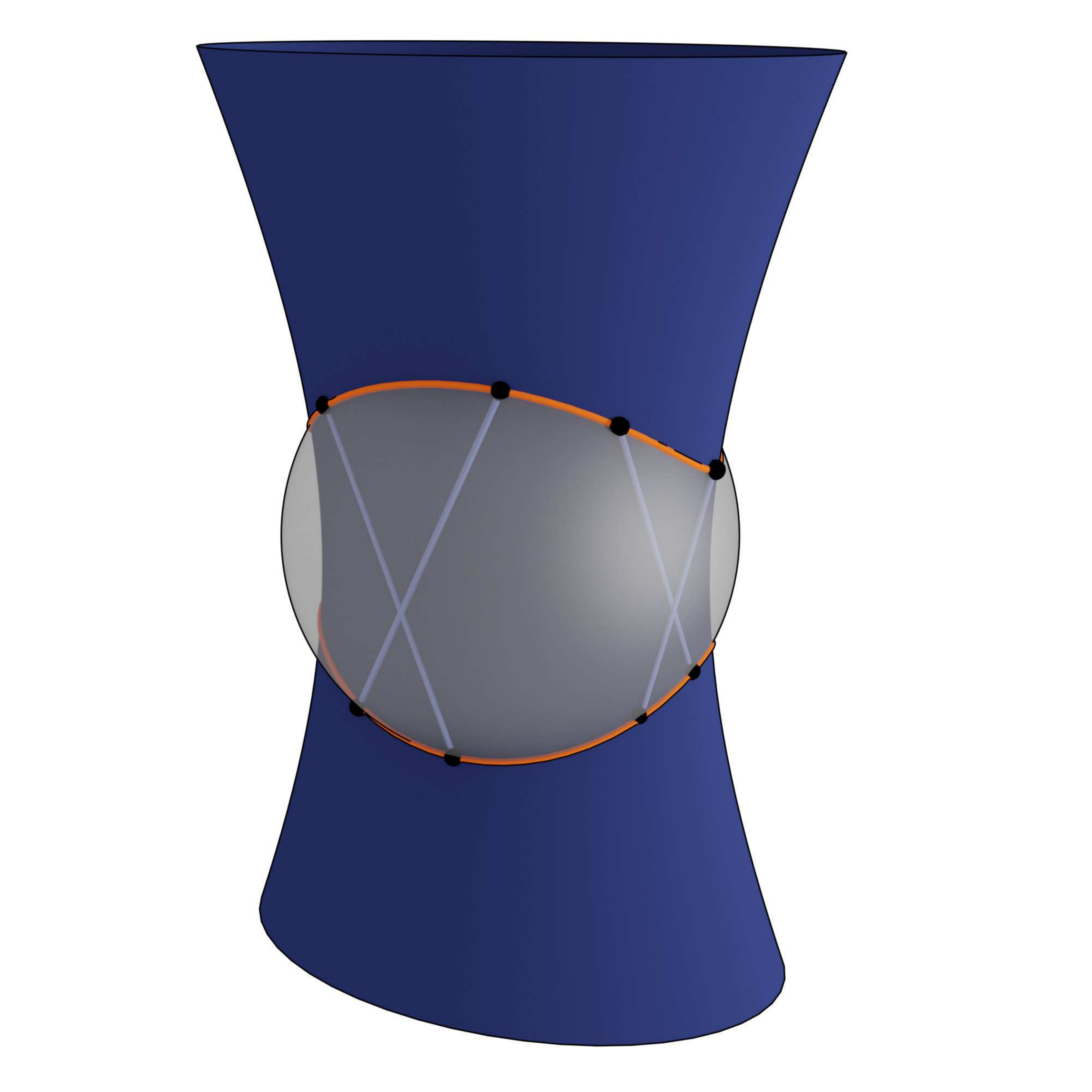
      \def\svgwidth{0.32\textwidth}
      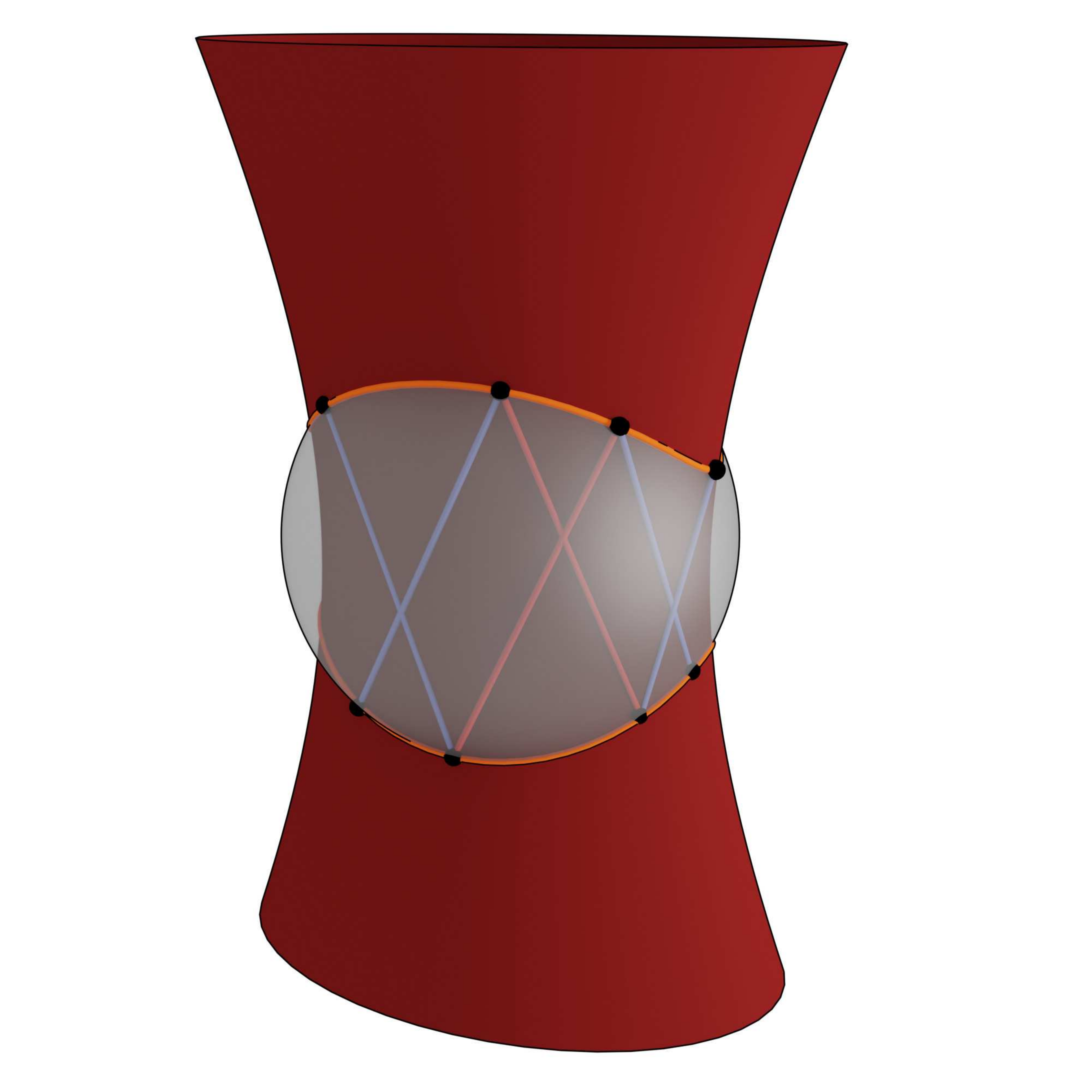
    \end{center}
    \caption{
      \emph{Left:} Combinatorial picture of the lines of a checkerboard incircular net.
      \emph{Middle/Right:} The two hyperboloids in the pencil of quadrics through the hypercycle base curve
      associated with a checkerboard incircular net in the elliptic plane.
  }
\label{fig:associated-hyperboloids}
\end{figure}
\begin{theorem}
  \label{thm:touching-hypercycle}
  The lines of a (hyperbolic/elliptic/Euclidean) checkerboard incircular net are in oriented contact with a common hypercycle (see Figure \ref{fig:hypercycle}).

  Moreover, the corresponding pencil of quadrics, which contains the hypercycle base curve,
  contains two unique hyperboloids $\cbichyp, \widetilde{\cbichyp}$ distinguished in the following way (see Figure \ref{fig:associated-hyperboloids}).
  Let $(\ell_i)_{i\in\Z}$, $(m_j)_{j\in\Z}$ be the points on the Laguerre quadric $\lag \subset \RP^3$
  corresponding to the oriented lines of the checkerboard incircular net.
  Consider the lines
  \[
    L_i \coloneqq \ell_i \wedge \ell_{i+1},\qquad
    M_i \coloneqq m_i \wedge m_{i+1}.
  \]
  Then, all lines $L_{2k}$, $M_{2l}$ lie on a common hyperboloid $\cbichyp \subset \RP^3$,
  and similarly, all lines $L_{2k+1}$, $M_{2l+1}$ lie on a common hyperboloid $\widetilde{\cbichyp} \subset \RP^3$.  
\end{theorem}
\begin{proof}
  Due to the inscribability property of checkerboard incircular nets every line $L_{2k}$ intersects every line $M_{2l}$, and vice versa.
  Thus, all lines $L_{2k}$, $M_{2l}$ generically lie on a common hyperboloid $\cbichyp$.
  Similarly, all lines $L_{2k+1}$, $M_{2l+1}$ lie on a common hyperboloid $\widetilde{\cbichyp}$.
  We now show that both hyperboloids $\cbichyp$, $\widetilde{\cbichyp}$ intersect the Laguerre quadric $\lag$
  in the same curve, that is, they belong to the same pencil of quadrics.
  Indeed, according to Lemma \ref{lem:pencils-and-lines},
  for each line $L_{2k+1}$, there exists a unique quadric in the pencil spanned by $\lag$ and $\cbichyp$ containing $L_{2k+1}$.
  Same for each line $M_{2l+1}$.
  Since the lines $L_{2k+1}$ and $M_{2l+1}$ pairwise intersect, again according to Lemma \ref{lem:pencils-and-lines},
  the corresponding quadrics coincide with each other and eventually with $\widetilde{\cbichyp}$.
  Thus, all points $\ell_i$, $m_j$ lie on the intersection $\lag \cap \cbichyp = \lag \cap \widetilde{\cbichyp}$.
\end{proof}
\begin{lemma}
  \label{lem:pencils-and-lines}
  Let $\p{x}_1,\p{x}_2$ be two points which belong to all members of a pencil of quadrics $\mathcal{Q}_\lambda$.
  Then, there exists a unique quadric $\mathcal{Q}_{\lambda_{12}}$ from the pencil which contains the whole line $L_{12}=\p{x}_1 \wedge \p{x}_2$.
  
  If the line $L_{34}= \p{x}_3 \wedge \p{x}_4$ associated with another pair of base points $\p{x}_3,\p{x}_4$ intersects the line $L_{12}$ then the two quadrics  $\mathcal{Q}_{\lambda_{12}}$ and $\mathcal{Q}_{\lambda_{34}}$ coincide.
\end{lemma}
\begin{proof}
  Let $q_1,q_2$ be two quadratic forms generating the pencil with the quadratic form $q_\lambda=q_1+\lambda q_2$.
  The points $\p{x}_1$, $\p{x}_2$ belong to all quadrics of the pencil if and only if
  \[
    q_1(p_1)=q_1(p_2)=q_2(p_1)=q_2(p_2)=0.
  \]
  The line $L_{12}=\p{x}_1\wedge \p{x}_2$ belongs to the quadric determined by $q_{\lambda_{12}}$ if and only if $q_{\lambda_{12}}(p_1,p_2)=0$
  so that
  \[
    t_{12}=-\frac{q_1(p_1,p_2)}{q_2(p_1,p_2)}.
  \]
  Vanishing of the denominator is the case when the line lies on the quadric determined by $q_2$.
  
  Moreover, if the line $L_{34}=\p{x}_3\wedge \p{x}_4$ passing through another pair of common points $\p{x}_3,\p{x}_4$ intersects the line $L_{12}$
  then the point of intersection and $\p{x}_3,\p{x}_4$ belong to the quadric $\mathcal{Q}_{\lambda_{12}}$.
  Accordingly, the line $L_{34}$ is contained in $\mathcal{Q}_{\lambda_{12}}$ so that $\mathcal{Q}_{\lambda_{12}}=\mathcal{Q}_{\lambda_{34}}$.
\end{proof}
The oriented circles of a checkerboard incircular net correspond to the planes spanned by pairs of lines $L_{2k}, M_{2l}$ or $L_{2k+1}, M_{2l+1}$,
i.e.\ they correspond to tangent planes of the two hyperboloids $\cbichyp, \widetilde{\cbichyp}$, respectively.
We identify each circle with its polar point with respect to the Laguerre quadric $\lag$,
or in the Euclidean case with a point in the cyclographic model (cf.\ Appendix~\ref{sec:euclidean-laguerre-geometry}).
\begin{corollary}\
  \label{cor:cbic-dual-pencil}
  \nobreakpar
  \begin{enumerate}
  \item The polar points corresponding to the oriented circles of a hyperbolic/elliptic checkerboard incircular net
    lie on two quadrics, the polar pencil of which contains the Laguerre quadric (polar with respect to the Laguerre quadric).
  \item The points in the cyclographic model corresponding to the oriented circles of a Euclidean checkerboard incircular net
    lie on two quadrics, the dual pencil of which contains the absolute quadric (i.e.\ the two quadrics are Minkowski confocal quadrics).
  \end{enumerate}
\end{corollary}
\begin{proof}\
  \nobreakpar
  \begin{enumerate}
  \item
    Under polarization in the Laguerre quadric $\lag$ the tangent planes of $\cbichyp$ become points on the polar quadric $\cbichyp^\lagperp$.
    Similarly, the tangent planes of $\widetilde{\cbichyp}$ become points on the polar quadric $\widetilde{\cbichyp}^\lagperp$.
    Since $\cbichyp$, $\widetilde{\cbichyp}$, and $\lag$ are contained in a common pencil of quadrics,
    their polar images $\cbichyp^\lagperp$, $\widetilde{\cbichyp}^\lagperp$, and $\lag^\lagperp \cong \lag$
    are contained in the polar pencil of quadrics.
  \item For the Euclidean case a similar argument holds by dualization to the cyclographic model.
  \end{enumerate}
\end{proof}
We conclude this section on hypercycles by stating an incidence result concerning eight lines touching a hypercycle,
which is similar to Theorem \ref{thm:miquel-laguerre}.
\begin{figure}
  \centering
  \def\svgwidth{0.55\textwidth}
  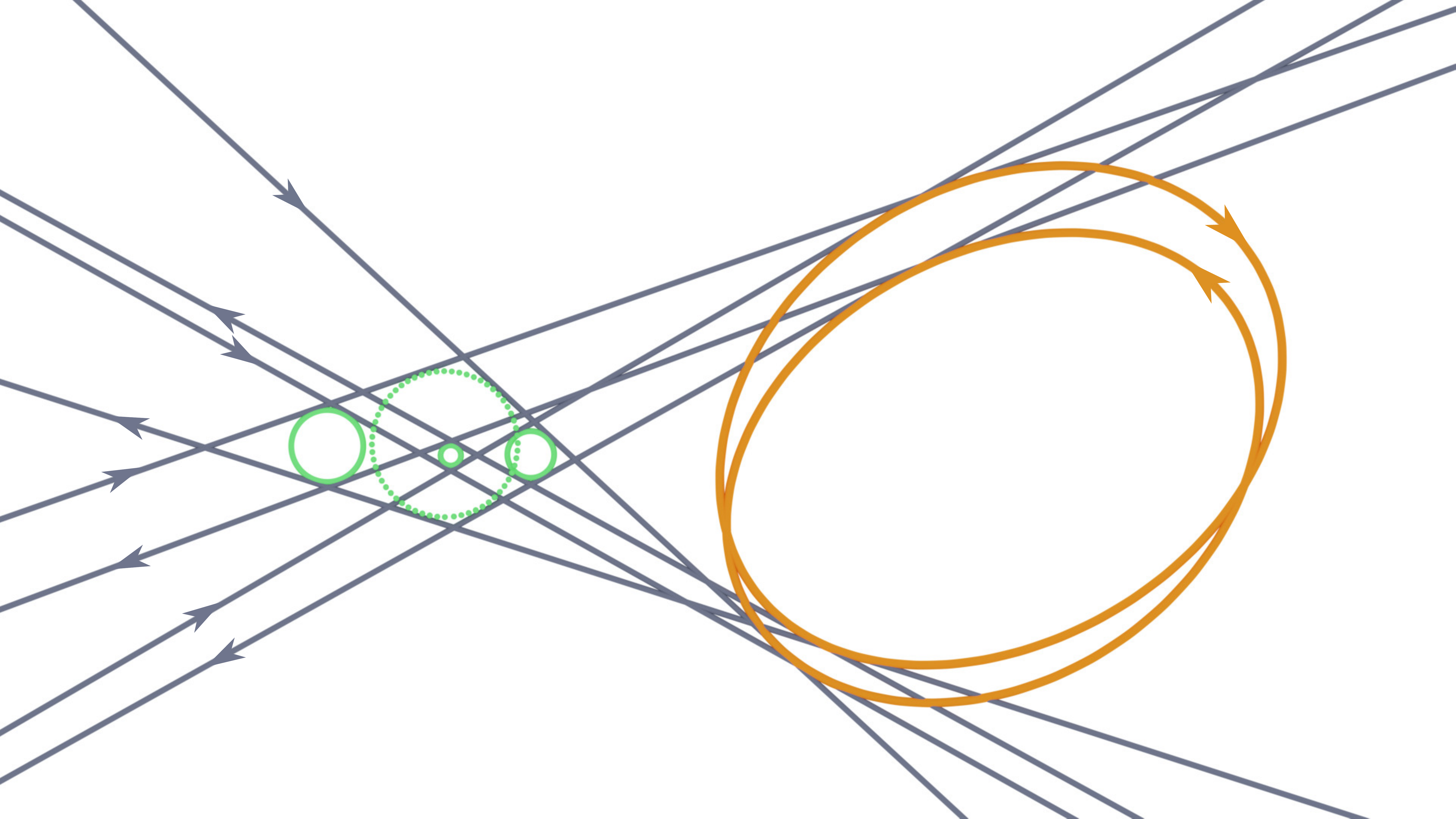
  \caption{
    Incidence theorem for eight lines touching a hypercycle.
  }
\label{fig:laguerre-subdivision}
\end{figure}
\begin{theorem}
  \label{thm:laguerre-subdivision}
  Let $\ell_1, \ell_2, \ell_3, \ell_4, m_1, m_2, m_3, m_4$ be eight generic lines touching a hypercycle.
  If the three quadrilaterals $(\ell_1, \ell_2, m_1, m_2)$, $(\ell_2, \ell_3, m_2, m_3)$, $(\ell_3, \ell_4, m_3, m_4)$ are circumscribed,
  then so is the quadrilateral $(\ell_1, \ell_4, m_1, m_4)$ (see Figure \ref{fig:laguerre-subdivision}).
\end{theorem}
\begin{proof}
  We identify the eight oriented lines with its corresponding points on the Laguerre quadric.
  The hypercycle base curve is the intersection of two quadrics.
  Define the degenerate quadric given by the two planes through $\ell_1, \ell_2, m_1, m_2$ and $\ell_3, \ell_4, m_3, m_4$ respectively.
  Then the given eight points on the Laguerre quadric are the intersection of those three quadrics.
  According to Lemma \ref{lem:associated-points} every quadric through seven of those points must pass through the eighth.
  Consider the degenerate quadric given by the two planes through $\ell_2, \ell_3, m_2, m_3$ and $\ell_1, \ell_4, m_1$ respectively.
  Then this quadric must also pass through $m_4$.
  Since no five points may lie in a plane we can conclude that
  $\ell_1, \ell_4, m_1, m_4$ lie in a common plane, and thus, that the corresponding quadrilateral is circumscribed.
\end{proof}

\subsection{Conics and incircular nets}
\label{sec:conics}
Towards the parametrization of checkerboard incircular nets
it turns out to be useful to consider certain normal forms of hypercycles, one of which are conics.
In \cite{BST} it is demonstrated that in the Euclidean case a generic hypercycle can be mapped to a conic by a Laguerre transformation
if and only if the corresponding pencil of quadrics is diagonalizable.
In the non-Euclidean cases diagonalizable hypercycles are still a subset of hypercycles that can be mapped to conics.
% On the other hand, conics are closely related to the special case of ``ordinary'' incircular nets.
% In the projective model of the hyperbolic/elliptic/Euclidean plane (Cayley-Klein spaces, see Section \ref{sec:Cayley-Klein-metric}) (metric) conics are given by projective conics.
%
\begin{definition}[Conics in spaceforms]
  In the projective model of the hyperbolic/elliptic/Euclidean plane embedded into $\RP^2$
  a \emph{(hyperbolic/elliptic/Euclidean) conic} is a projective conic in $\RP^2$.
\end{definition}
\begin{remark}\
  \nobreakpar
  \begin{enumerate}
  \item
    From this projective definition of conics one recovers the familiar metric properties of conics in the different space forms, see, e.g., \cite{Cha, Sto, I}.
  \item
    In hyperbolic geometry a conic might lie ``outside'' the hyperbolic plane and be considered a ``deSitter conic''.
    These cases are still relevant in our Laguerre geometric considerations as long as they possess hyperbolic tangent lines.
    % \label{rem:deSitter-conics}
  \end{enumerate}
\end{remark}
Recall that in Laguerre geometry reflection in the special point $\p{p}$ corresponds to orientation reversion (see Section \ref{sec:laguerre}).
We use this point in the following way to characterize conics in the set of hypercycles.
\begin{lemma}
  \label{lem:hypercycle-conic}
  A hypercycle in the hyperbolic/elliptic/Euclidean plane is a conic (doubly covered with opposite orientation)
  if and only if its hypercycle base curve is given by the intersection of the Laguerre quadric with a cone with vertex $\p{p}$.
\end{lemma}
\begin{proof}
  In hyperbolic and elliptic Laguerre geometry $\p{p}$ is the polar point of the base plane of the projective model of the corresponding space form.
  The polar of a cone with vertex $\p{p}$ is therefore a conic contained in this base plane.
  Thus, the tangent planes to the hypercycle base curve are the planes tangent to a conic,
  if and only if $\p{p}$ is the vertex of a cone intersecting the Laguerre quadric in the hypercycle base curve.
  In that case corresponding oriented lines envelop the conic (twice with opposite orientation).
  
  In Euclidean Laguerre geometry a similar argument holds upon dualization and considering the cyclographic model.
\end{proof}
\begin{remark}
  A generic hypercycle for which the corresponding pencil of quadrics is in diagonal form is a conic.
  Vice versa, in elliptic and Euclidean geometry a generic conic (excluding the non-generic case of parabolas)
  can be brought into diagonal form by an isometry (a Laguerre transformation fixing the point $\p{p}$).
  In hyperbolic geometry there also exist non-diagonalizable generic conics (semihyperbolas, cf.\ \cite{I}).
  Thus, by considering conics up to Laguerre transformations, we are restricting the class
  of hypercycles to (a subclass of) diagonalizable hypercycles.
  \label{rem:diagonalizable-hypercycles}
\end{remark}
%
% \begin{remark}
%   This also shows that conics are the only hypercycles that are invariant under orientation reversion.
%   In this case the hypercycle consists of two components that coincide up to their orientation.    
% \end{remark}
%
% \begin{remark}
%   The image of the set of (doubly covered) conics under all Laguerre transformations in a space form
%   is open in the set of all hypercycles in that space form.
% \end{remark}
%
We now give the definition for incircular nets \cite{B, AB}.
Examples of incircular nets in the elliptic and hyperbolic plane are shown in Figures~\ref{fig:ell_ic_nets}, \ref{fig:hyp_ic_nets_ellipse}, and \ref{fig:hyp_ic_nets_hyperbola} (see also \cite{gallery}).
\begin{definition}[Incircular nets]
  Two families $(\ell_k)_{k\in\Z}$, $(m_l)_{l\in\Z}$ of (non-oriented) lines in the hyperbolic/elliptic/Euclidean plane
  are called a \emph{(hyperbolic/elliptic/Euclidean) incircular net} (IC-net)
  if for every $k, l \in \Z$ the four lines $\ell_k, \ell_{k+1}, m_l, m_{l+1}$ touch
  a common circle (non-oriented Laguerre circle) $S_{kl}$.
\end{definition}
\begin{remark}
  While checkerboard incircular nets are instances of the corresponding (hyperbolic/elliptic/ Euclidean) Laguerre geometry,
  incircular nets are a notion of the corresponding metric geometry, i.e. only invariant under isometries (Laguerre transformations that fix $\p{p}$).
\end{remark}
In the limit of a checkerboard incircular net $(\ell_i)_{i\in\Z}$, $(m_j)_{j\in\Z}$
in which all incircles of the quadrilaterals $\ell_{2k}, \ell_{2k+1}, m_{2l}, m_{2l+1}$ collapse to a point,
the pairs of lines $\ell_{2k}, \ell_{2k+1}$ as well as the pairs of lines $m_{2l}, m_{2l+1}$ coincide respectively up to their orientation.
Such a pair of oriented lines may be regarded as a non-oriented line.
The points on the Laguerre quadric corresponding to two lines that agree up to their orientation are connected by a line
that goes through the point $\p{p}$.
Considering the associated hyperboloids of a checkerboard incircular net from Theorem~\ref{thm:touching-hypercycle}
we find that the generator lines $L_{2k}$, $M_{2l}$ all go through the point $\p{p}$ and the hyperboloid $\cbichyp$
becomes a cone with vertex at $\p{p}$.
In this limit a checkerboard incircular net becomes an ``ordinary'' incircular net.
\begin{remark}
  An incircular net obtained from a checkerboard incircular net as the special case described above possesses the following additional \emph{regularity property}:
  The line through the centers of $S_{kl}$, $S_{k+1,l+1}$ and the line through the centers of $S_{k+1,l}$, $S_{k,l+1}$ are the distinct angle bisectors of the lines $\ell_{k+1}$ and $m_{l+1}$.
  Thus, it is convenient to append this property to the definition of incircular nets.
\end{remark}
By Lemma~\ref{lem:hypercycle-conic} incircular nets are now characterized as special checkerboard incircular nets
in terms of its associated hyperboloids (see Theorem \ref{thm:touching-hypercycle}).
\begin{theorem}
  \label{thm:icnets}
  A checkerboard incircular net is an incircular net,
  if one of its two associated hyperboloids is a cone with vertex at $\p{p}$.
\end{theorem}
Together with Theorem \ref{thm:touching-hypercycle} and Lemma \ref{lem:hypercycle-conic} we obtain that for incircular nets the tangent hypercycle is a conic \cite{B}, \cite{AB}.
\begin{corollary}
  All lines of a (hyperbolic/elliptic/Euclidean) incircular net touch a common conic.
\end{corollary}
\begin{remark}
  By the classical Graves-Chasles theorem incircular nets are closely related to configurations of confocal conics (see \cite{B} for the Euclidean case, \cite{AB} for the Euclidean and hyperbolic case, and \cite{I} for a treatment in all space forms).
  A relation to discrete confocal conics is given in \cite{BSST16, BSST18}.
\end{remark}

% For the purpose of parametrization it turns out to be convenient to introduce the class of confocal checkerboard incircular nets,
% which lies in between the classes of checkerboard incircular nets and incircular nets.
% %
% \begin{definition}
%   A checkerboard incircular net is called \emph{confocal checkerboard incircular net}
%   if all its oriented lines touch a common conic.
% \end{definition}
%
% \textcolor{red}{
%   Every checkerboard incircular net for which the tangent hypercycle can be mapped to a conic by a Laguerre transformation
%   is itself mapped to a confocal checkerboard incircular net.
%   By Lemma \ref{lem:hypercycle-conic} \ref{lem:hypercycle-conic-diagonal}, this is possible if the corresponding pencil of quadrics is diagonalizable by a Laguerre transformation (cf.\ \cite{BST} for the Euclidean case).
%   We give explicit parametrizations for confocal checkerboard incircular nets in the following.
% }

\begin{figure}[h]
  \centering
  \includegraphics[width=0.62\textwidth]{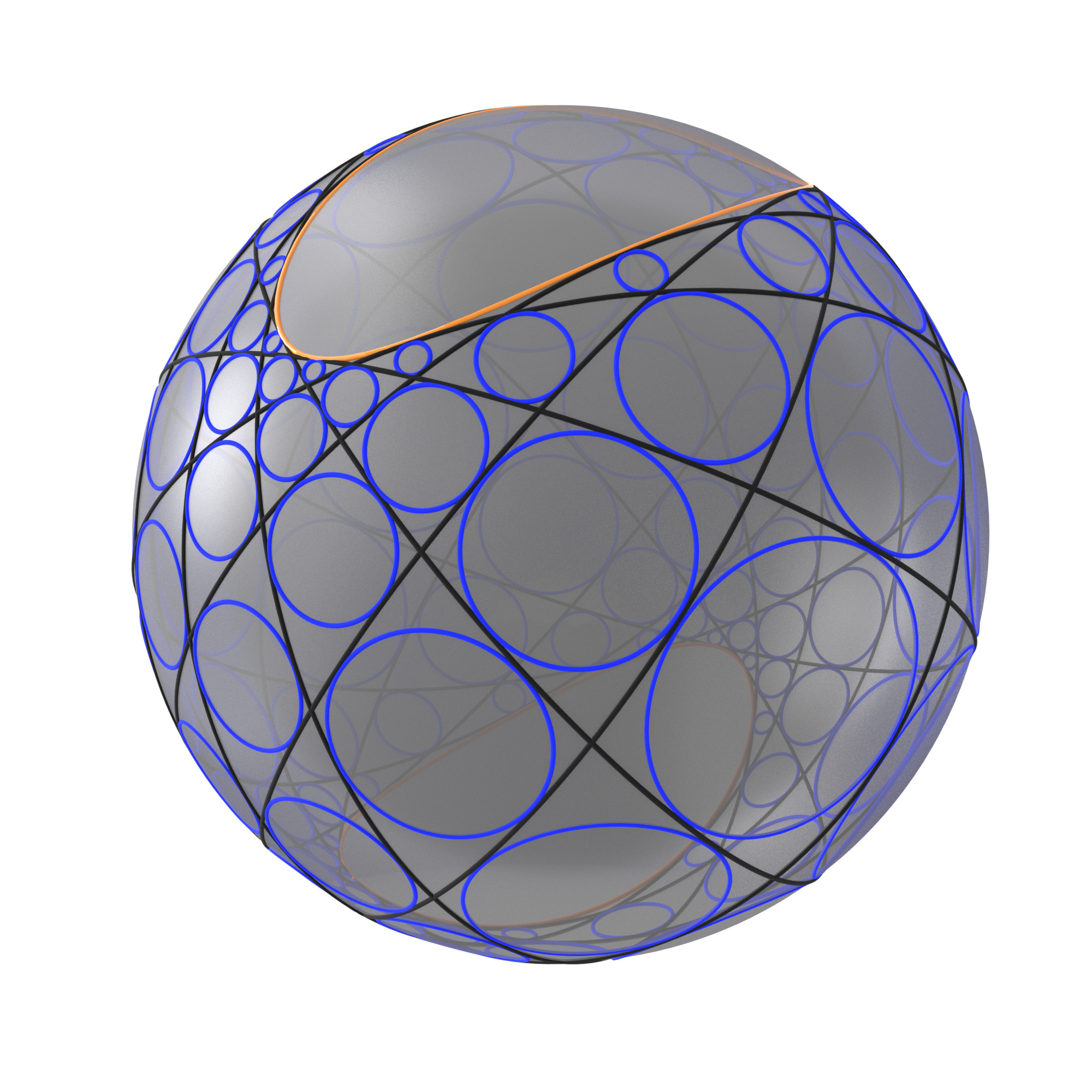}\\
  \hspace{\fill}
  \includegraphics[width=0.47\textwidth]{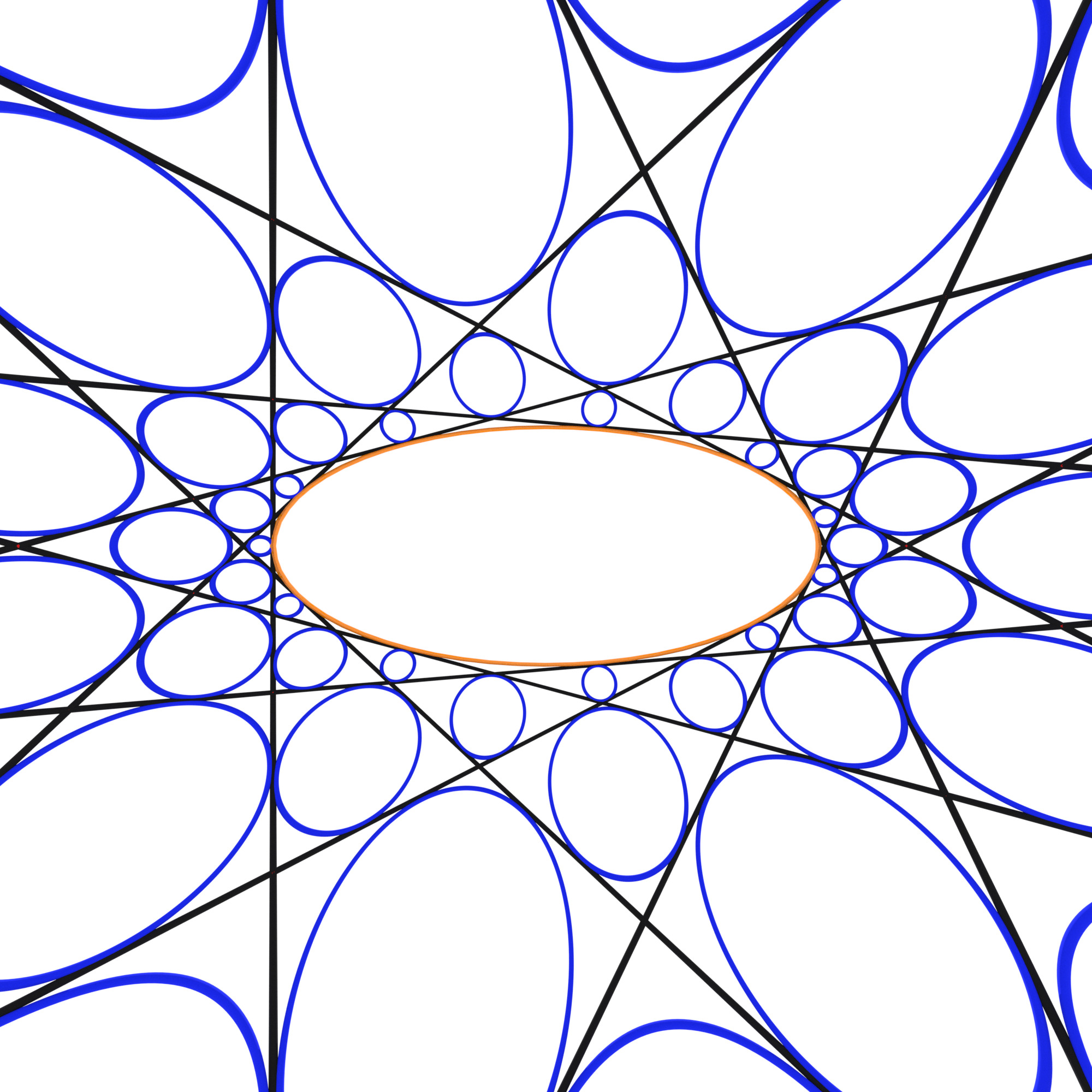}
  \hspace{\fill}
  \includegraphics[width=0.47\textwidth]{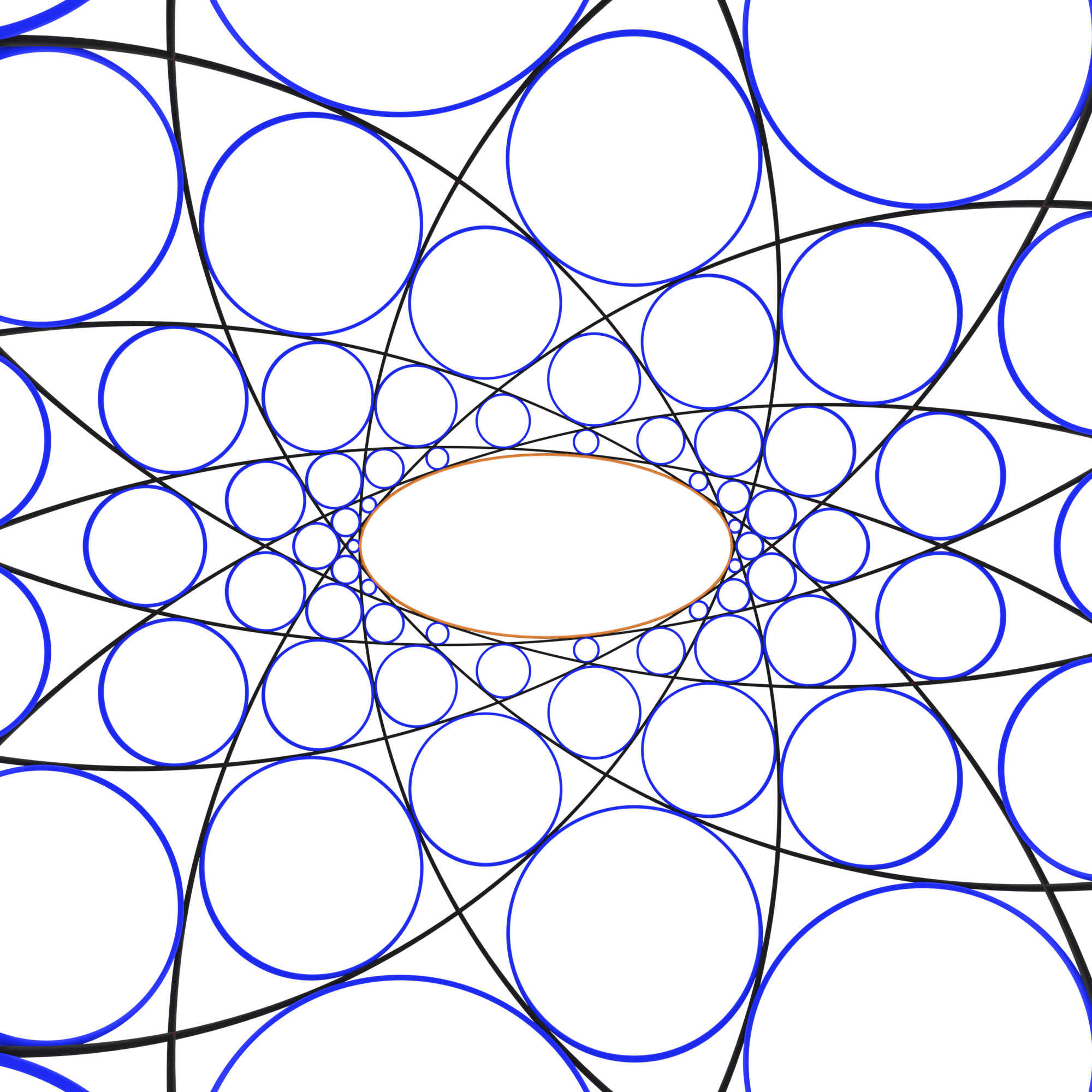}
  \hspace{\fill}
  \caption{
    \emph{Top:} Incircular net tangent to an ellipse in the sphere model of the elliptic plane.
    \emph{Bottom-left:} Central projection to the projective model of the elliptic plane.
    \emph{Bottom-right:} Stereographic projection to a conformal model of the elliptic plane.
  }
\label{fig:ell_ic_nets}
\end{figure}
\begin{figure}[h]
  \centering
  \includegraphics[width=0.62\textwidth]{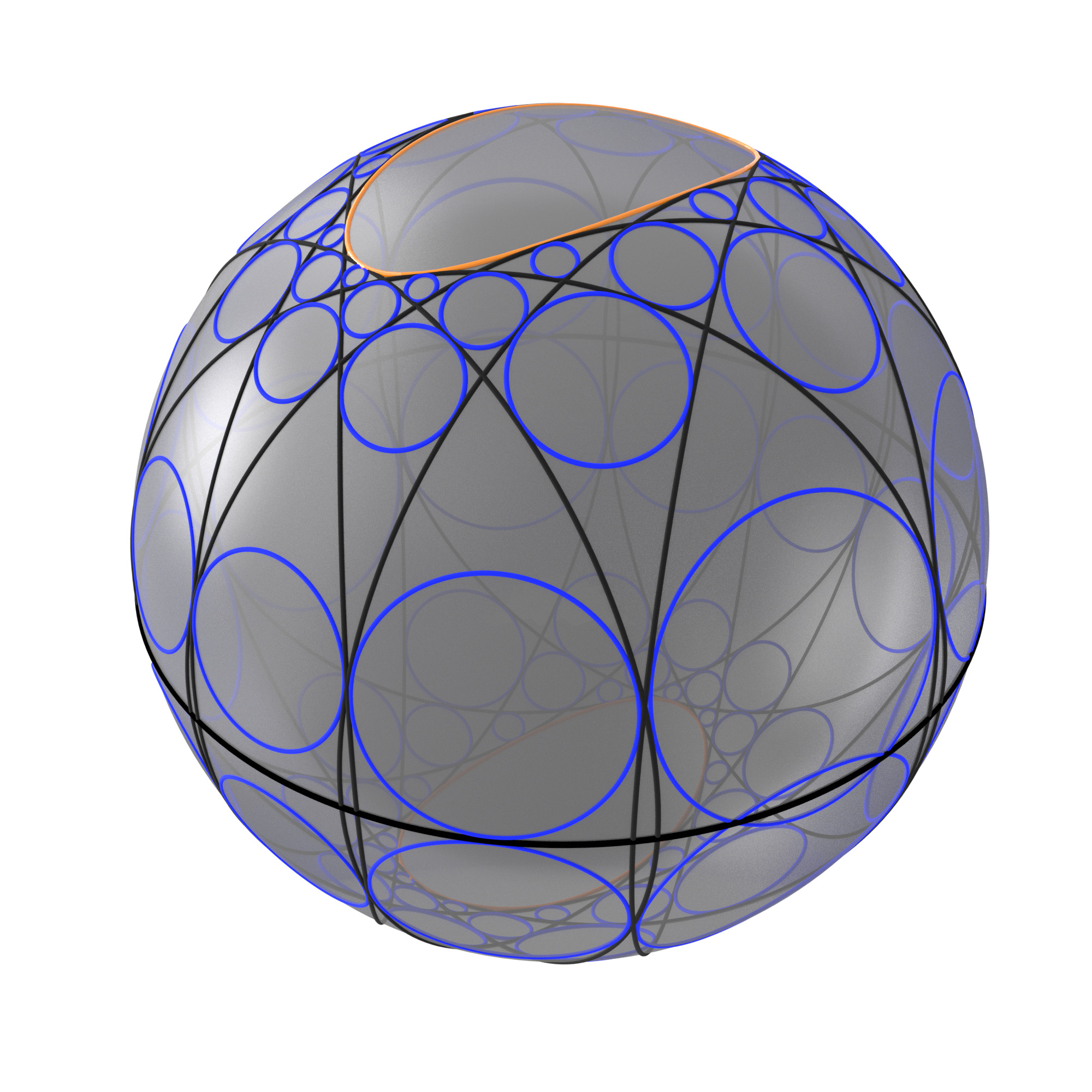}\\
  \hspace{\fill}
  \includegraphics[width=0.47\textwidth]{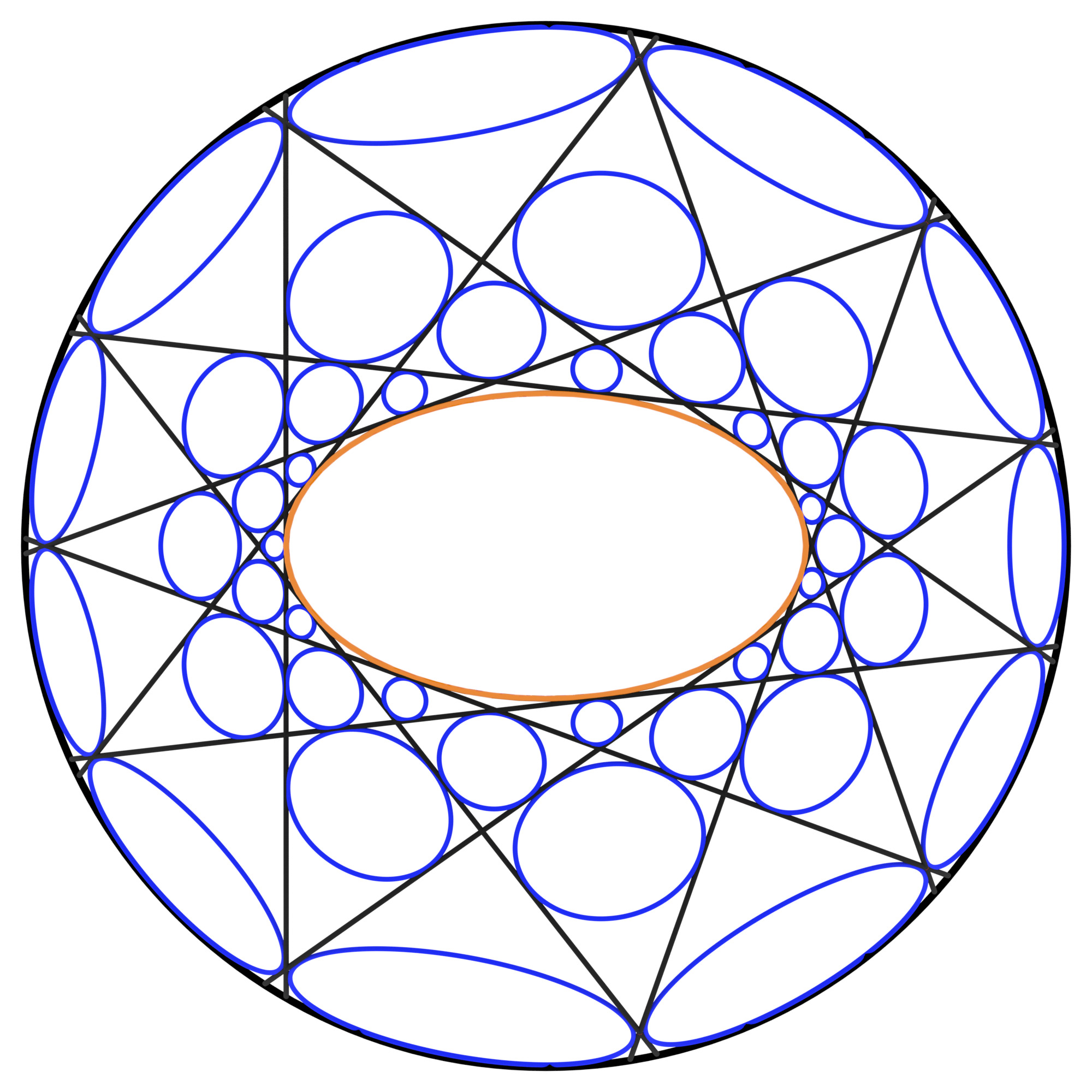}
  \hspace{\fill}
  \includegraphics[width=0.47\textwidth]{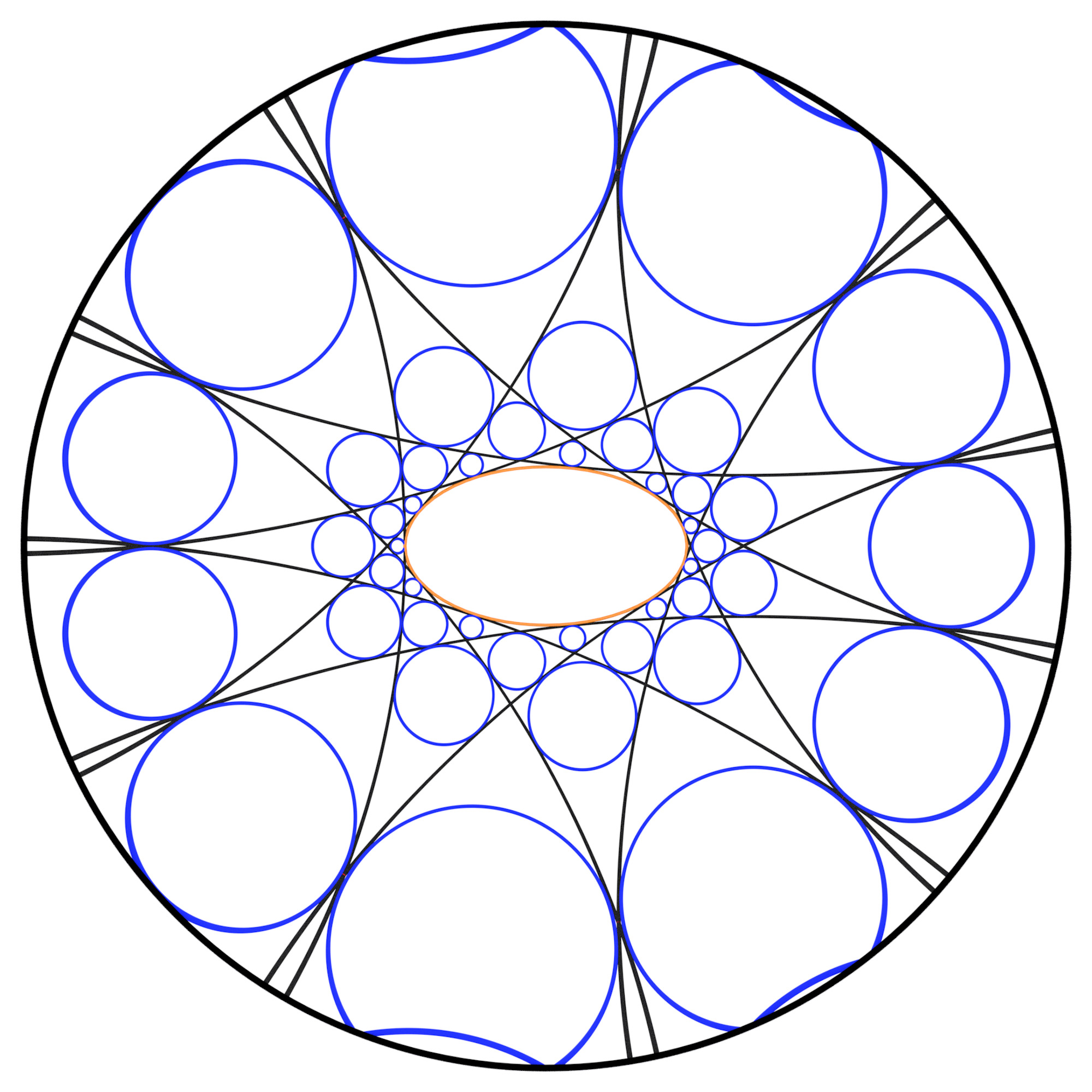}
  \hspace{\fill}\\
  \vspace{1cm}
  \includegraphics[width=0.65\textwidth]{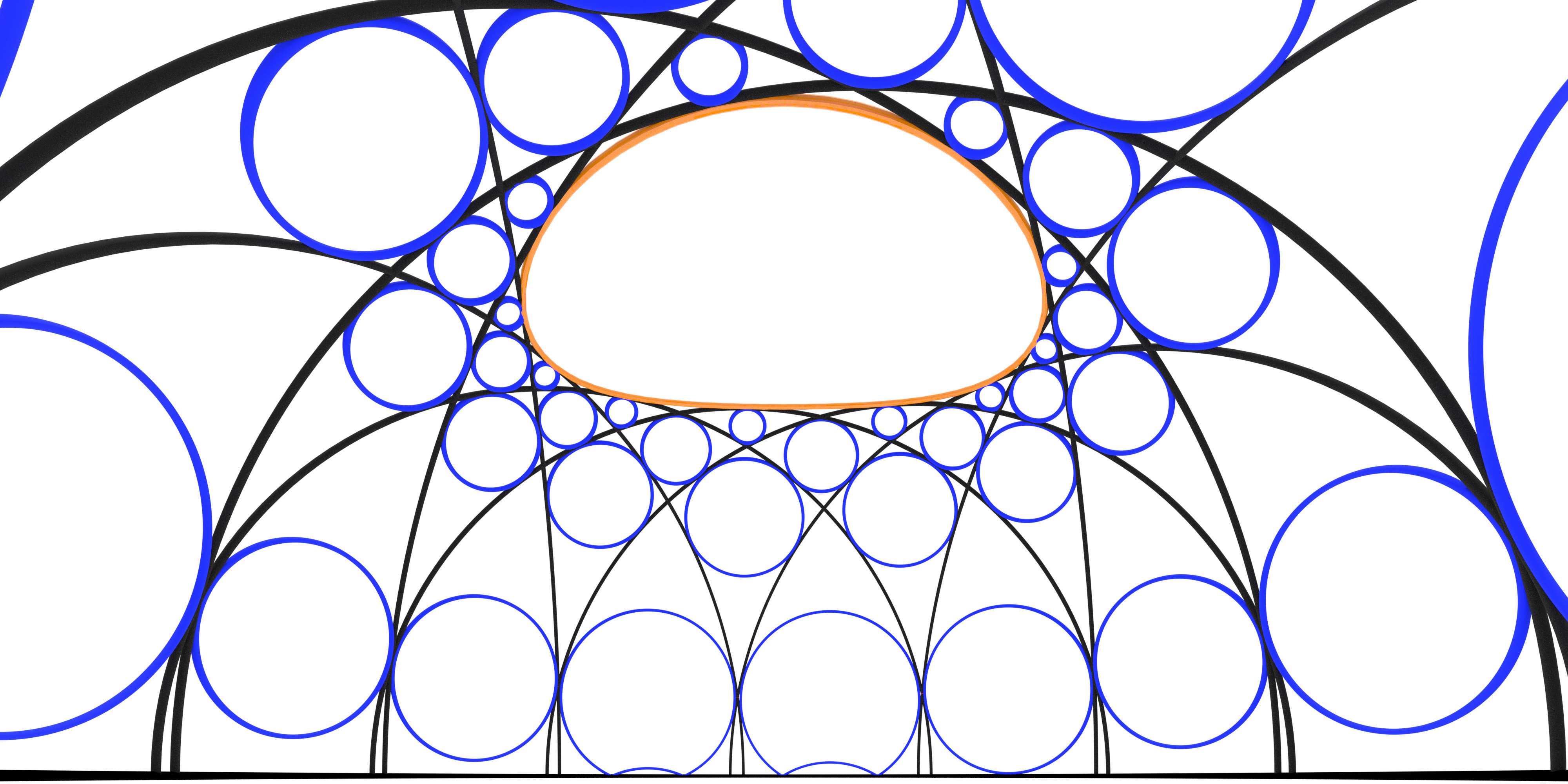}
  \caption{
    \emph{Top:} Incircular net tangent to an ellipse in the sphere model of the hyperbolic plane.
    Two copies of the hyperbolic plane are realized as half-spheres.
    \emph{Middle-left:} Orthogonal projection to the Klein-Beltrami disk model.
    \emph{Middle-right:} Stereographic projection to the Poincaré disk model.
    \emph{Bottom:} Stereographic projection to the Poincaré half-plane model.
  }
\label{fig:hyp_ic_nets_ellipse}
\end{figure}
\begin{figure}[h]
  \centering
  \includegraphics[width=0.62\textwidth]{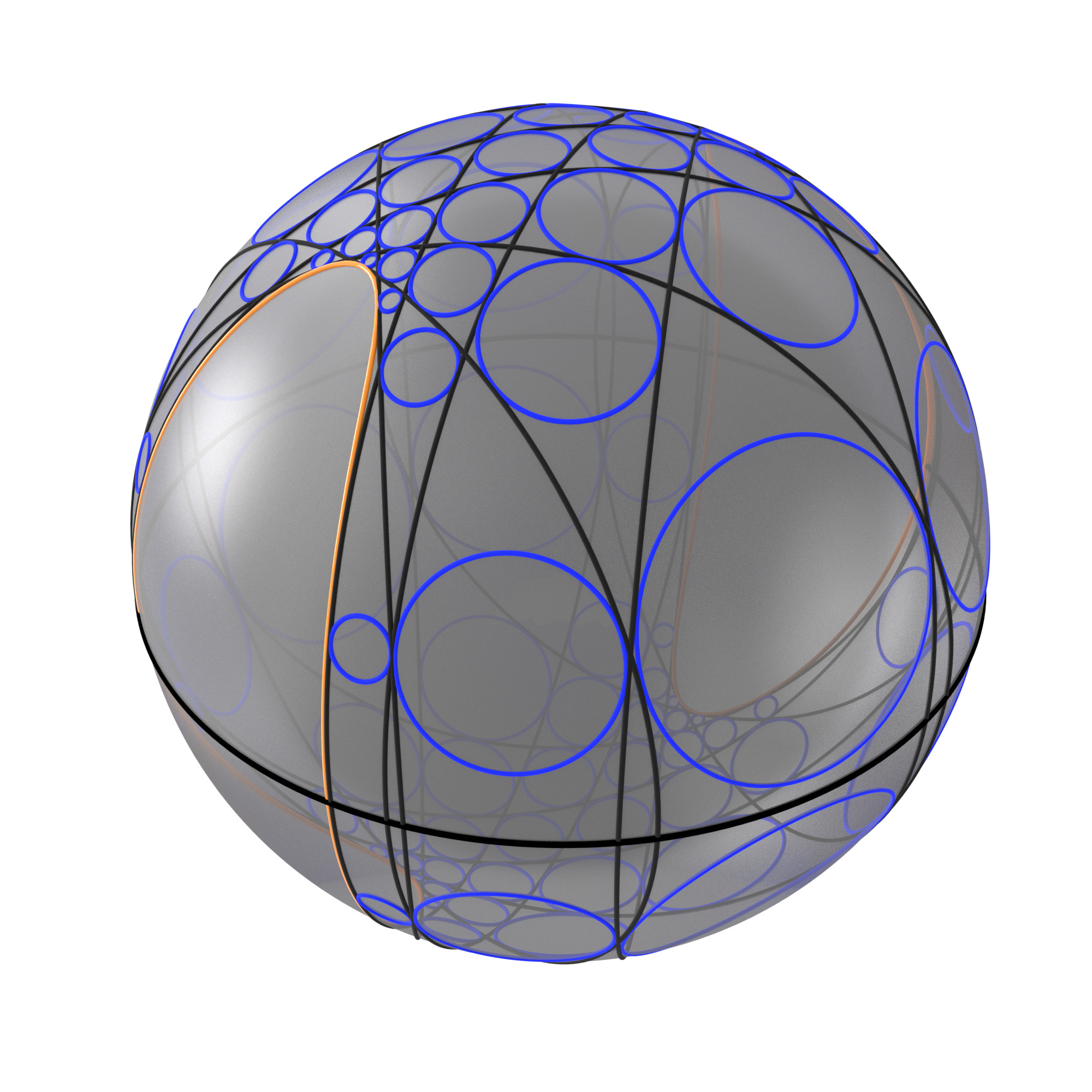}\\
  \hspace{\fill}
  \includegraphics[width=0.47\textwidth]{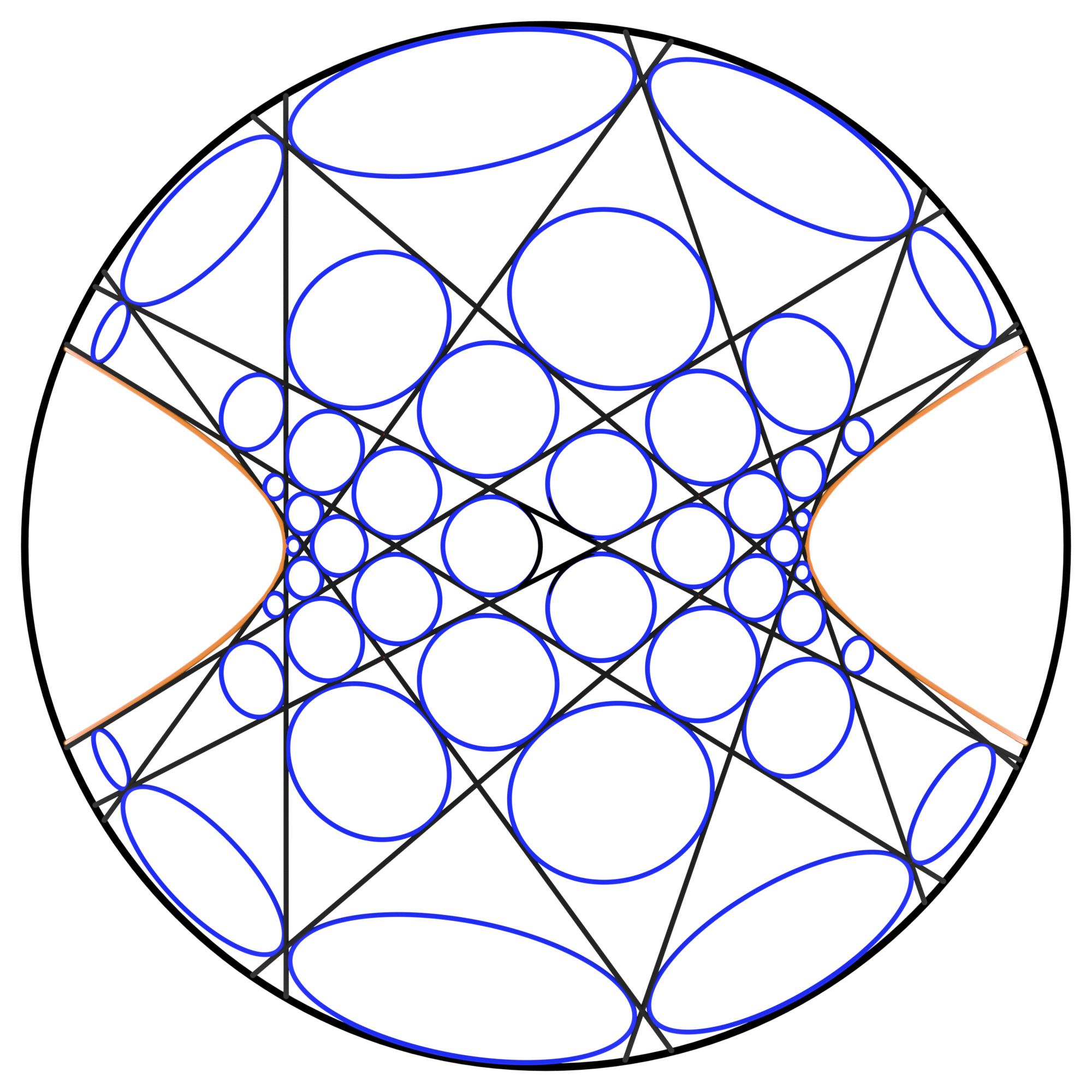}
  \hspace{\fill}
  \includegraphics[width=0.47\textwidth]{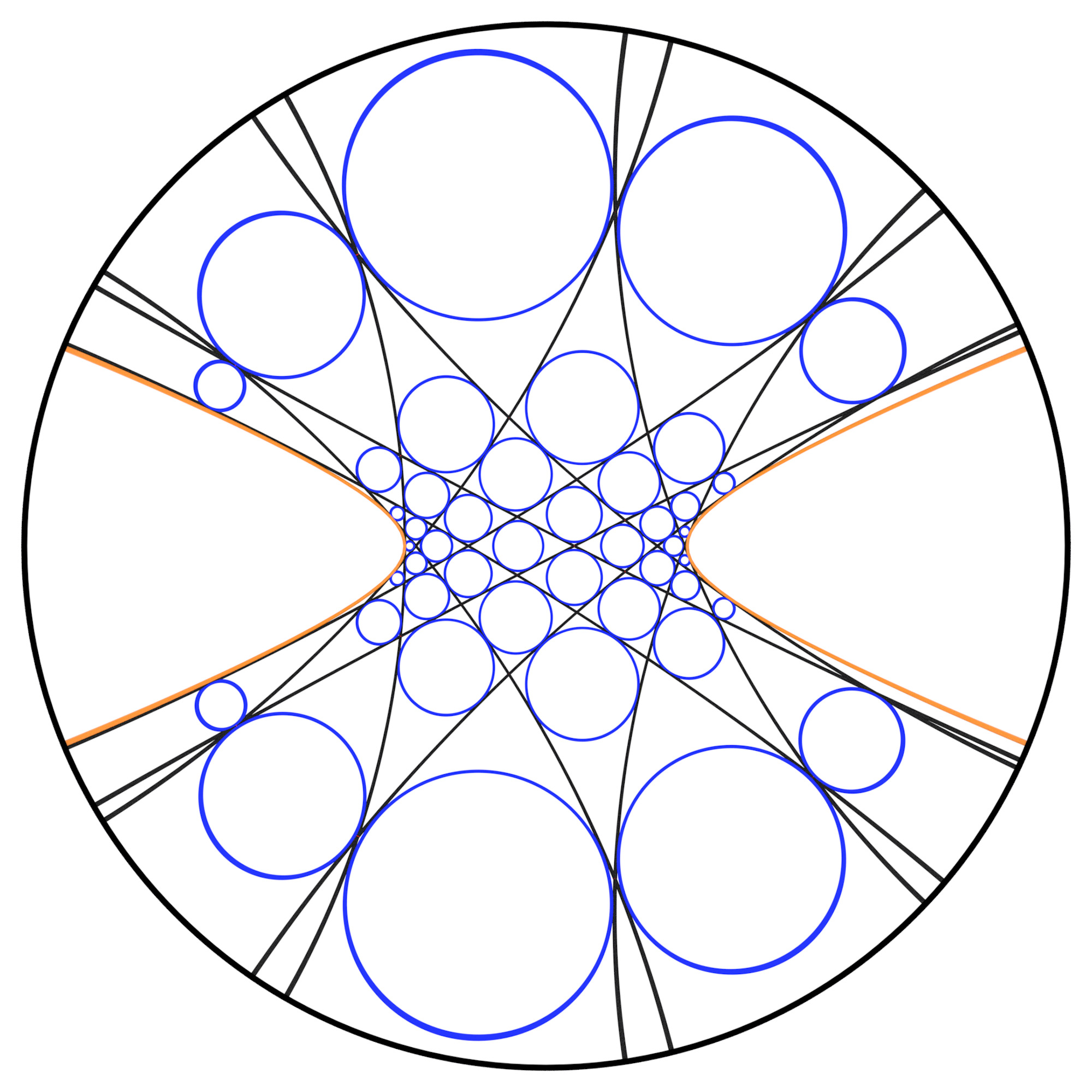}
  \hspace{\fill}\\
  \vspace{1cm}
  \includegraphics[width=0.65\textwidth]{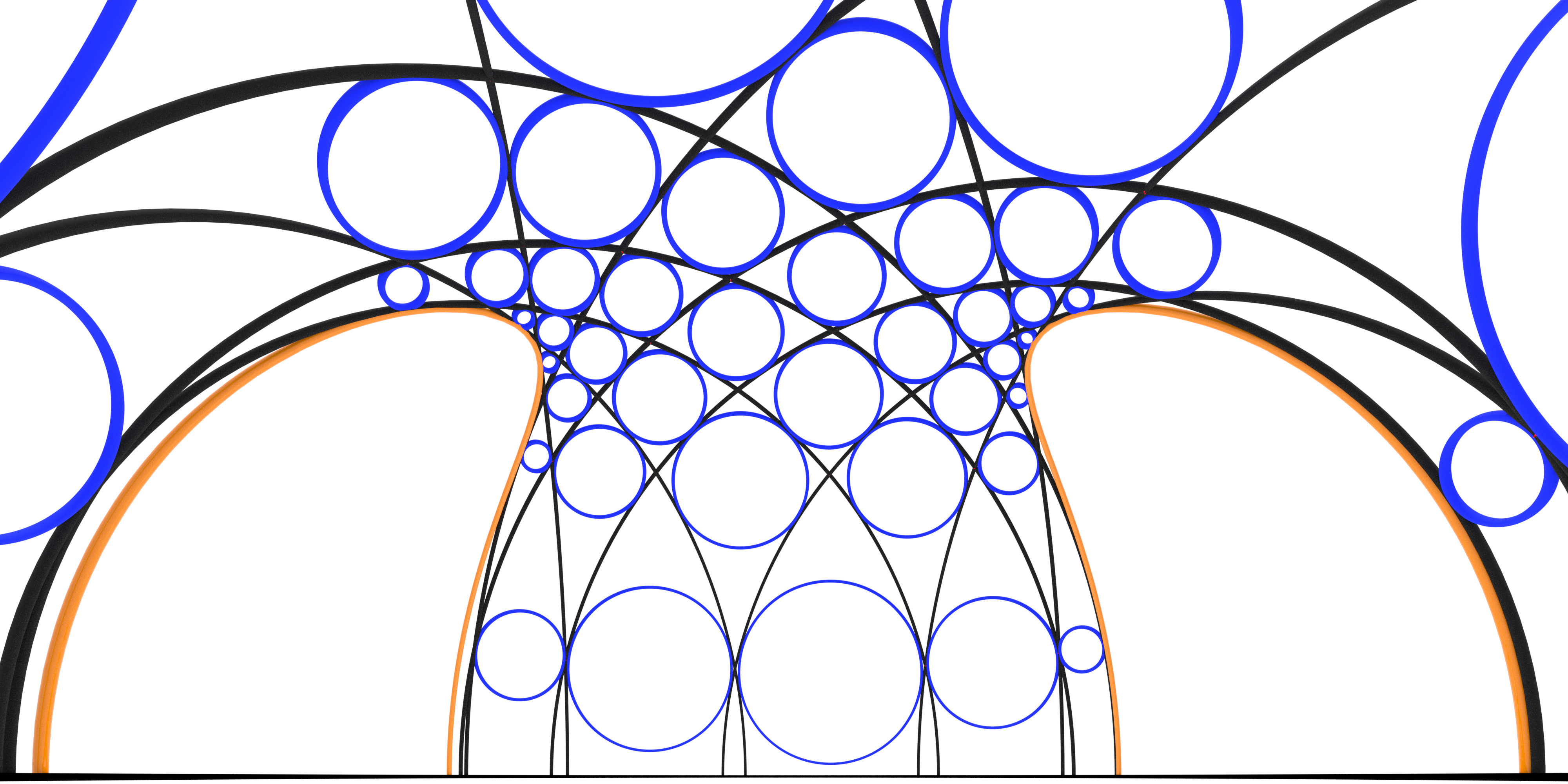}
  \caption{
    \emph{Top:} Incircular net tangent to a hyperbola in the sphere model of the hyperbolic plane.
    Two copies of the hyperbolic plane are realized as half-spheres.
    \emph{Middle-left:} Orthogonal projection to the Klein-Beltrami disk model.
    \emph{Middle-right:} Stereographic projection to the Poincaré disk model.
    \emph{Bottom:} Stereographic projection to the Poincaré half-plane model.
  }
\label{fig:hyp_ic_nets_hyperbola}
\end{figure}

\clearpage
\subsection{Construction and parametrization of checkerboard incircular nets}
The elementary construction of a checkerboard incircular net from a small patch
(line by line, while ensuring the incircle constraint)
is guaranteed to work due to the following incidence theorem (see Figure \ref{fig:associated-hyperboloids}, left) \cite{AB, BST}.
This construction has 12 real degrees of freedom.
\begin{theorem}
  \label{thm:5x5-incidence}
  Let $\ell_1, \ldots, \ell_6$, $m_1,\ldots,m_6$ be 12 oriented lines in the hyperbolic/elliptic/Euclidean plane
  which are in oriented contact with 12 oriented circles $S_1, \ldots, S_{12}$, in a checkerboard manner,
  as shown in Figure \ref{fig:associated-hyperboloids}, left.
  In particular, the lines $\ell_1, \ell_2, m_1, m_2$ are in oriented contact with the circle $S_1$,
  the lines $\ell_3, \ell_4, m_1, m_2$ are in oriented contact with the circle $S_2$ etc.
  Then, the 13th checkerboard quadrilateral also has an inscribed circle,
  i.e., the lines $\ell_5, \ell_6, m_5, m_6$ have a common circle $S_{13}$ in oriented contact.
\end{theorem}
\begin{remark}
  This incidence theorem holds in all three Laguerre geometries with literally the same proof as given in \cite{BST} for the Euclidean case.
\end{remark}
Though possible in principle, the elementary construction from, e.g., 6 lines as initial data,
which only describes the local behavior,
is not stable, and thus impractical for the construction of large checkerboard incircular nets.
Yet, by Theorem \ref{thm:touching-hypercycle}, we find that a checkerboard incircular net can equivalently be prescribed by
\begin{itemize}
\item choosing a hypercycle (8 degrees of freedom),
\item choosing two hyperboloids $\mathcal{Q}, \tilde{\mathcal{Q}}$ in the pencil of quadrics
  corresponding to the hypercycle base curve (2 degrees of freedom),
\item and choosing two initial lines tangent to the hypercycle, one from each of the $m$- and $\ell$-family (2 degrees of freedom).
\end{itemize}
Then further lines of, say, the $\ell$-family are obtained by alternately going along a chosen family of rulings of $\mathcal{Q}$ and $\tilde{\mathcal{Q}}$ from one point of the base curve to the next (see Figure \ref{fig:associated-hyperboloids}, middle/right).
Similarly for the $m$-family of lines, while using the respective other families of rulings of the two hyperboloids.
The intersection of two rulings from the two different families of the same hyperboloid implies the coplanarity
of the four intersection points with the base curve, which, in turn, corresponds to the existence of an incircle.
We demonstrate for certain classes of checkerboard incircular nets
how the parametrization of the hypercycle base curve in terms of Jacobi elliptic functions leads to explicit formulas for the net,
in which the free parameters determine the global behavior.
They can be further constraint to obtain periodic and ``embedded'' solutions.
\begin{remark}
  Note the resemblance to a ``confocal billiards'' type construction and a Poncelet porism type statement in the periodic case.
\end{remark}

In the following we derive explicit formulas for checkerboard incircular nets tangent to certain types of diagonalizable conics (see Remark \ref{rem:diagonalizable-hypercycles}).
We treat the hyperbolic/elliptic/Euclidean cases simultaneously by considering the standard bilinear form
of signature $(++ \varepsilon -)$ in $\R^4$, i.e.,
\[
  \scalarprod{x}{y} = x_1y_1 + x_2y_2 + \varepsilon x_3y_3 - x_4y_4
\]
for $x, y \in \R^4$,
which defines the corresponding Laguerre quadric $\lag \in \RP^3$.
The hyperbolic case is given by $\varepsilon = -1$, the elliptic case by $\varepsilon = 1$, and the Euclidean case by $\varepsilon = 0$
(see~Sections~\ref{sec:elliptic-laguerre-geometry},~\ref{sec:hyperbolic-laguerre-geometry},~\ref{sec:euclidean-laguerre-geometry}).
By Lemma \ref{lem:hypercycle-conic}, a hypercycle that corresponds to a conic is given by the intersection curve
of $\lag$ with a cone with vertex $\p{p} = [0,0,0,1]$.
We consider cones in diagonal form and cover checkerboard incircular nets tangent to
\begin{itemize}
\item ellipses in all space forms,
\item hyperbolas in the Euclidean plane, and convex hyperbolas in the hyperbolic plane,
\end{itemize}
excluding concave hyperbolas, deSitter hyperbolas and (the non-diagonalizable) semihyperbolas in the hyperbolic plane (cf.\ Remark \ref{rem:diagonalizable-hypercycles}),
as well as all further non-diagonalizable hypercycles.

\subsubsection{Parametrization of checkerboard incircular nets tangent to an ellipse}
\label{sec:ellipse}
Consider a cone $\mathcal{C}$ given by
\begin{equation}
  \alpha^2 x_1^2 + \beta^2 x_2^2 - x_3^2 = 0.
  \label{eq:ellipse-cone}
\end{equation}
with
\begin{equation}
  \label{eq:ellipse-cone-coefficients}
  \alpha > \beta > 0,\qquad
  1 + \varepsilon \alpha^2, 1 + \varepsilon \beta^2 > 0.
\end{equation}
It intersects the Laguerre quadric $\lag$ given by
\begin{equation}
  x_1^2 + x_2^2 + \varepsilon x_3^2 - x_4^2 = 0
  \label{eq:laguerre-quadric}
\end{equation}
in the hypercycle base curve $\lag \cap \mathcal{C}$ (see Figure \ref{fig:hypercycle-base-curves-ellipse}).
\begin{figure}
  \centering
  \begin{overpic}[width=0.32\textwidth]{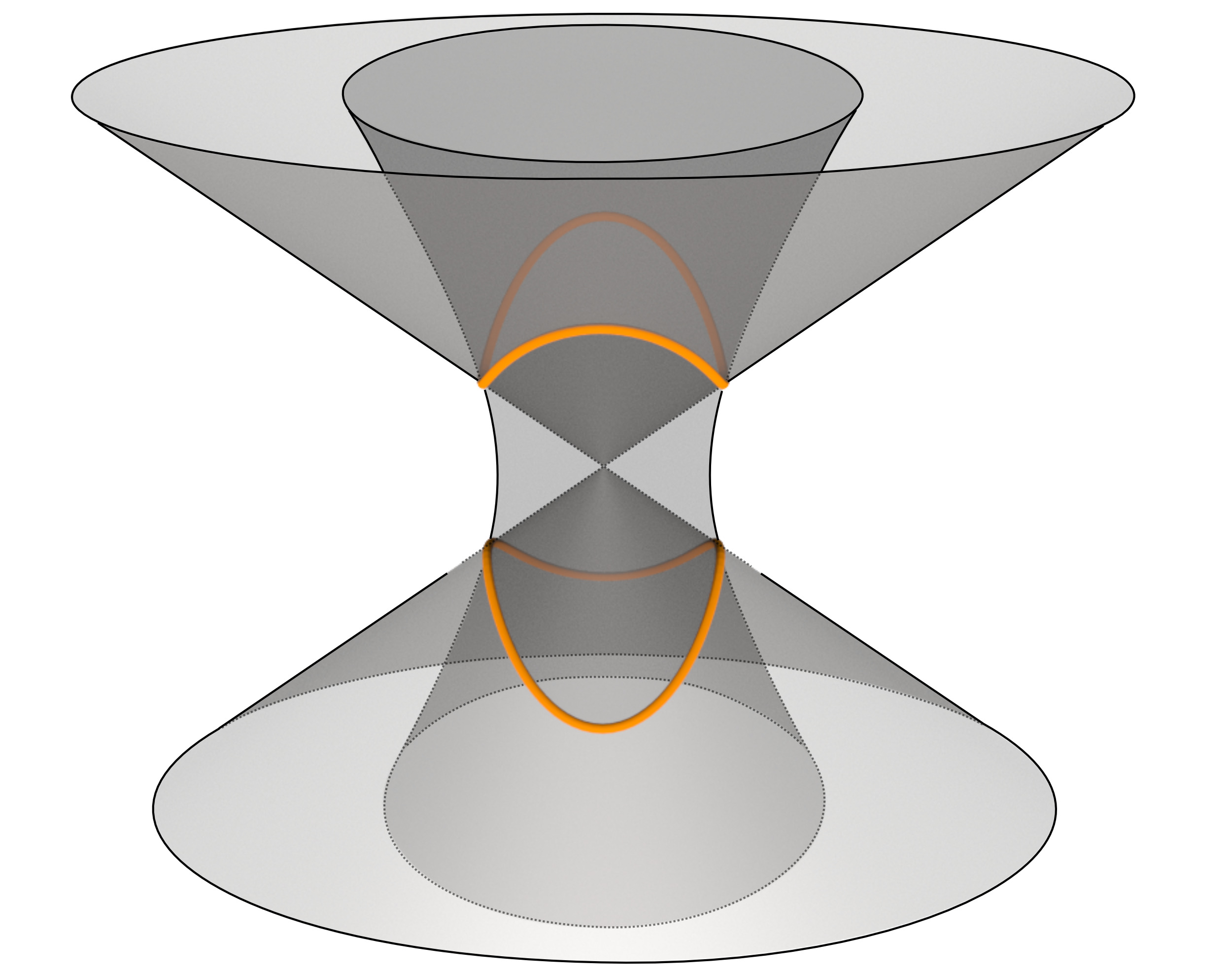}
    \put(60,38){$\laghyp$}
    \put(82,77){$\mathcal{C}$}
  \end{overpic}
    \begin{overpic}[width=0.32\textwidth]{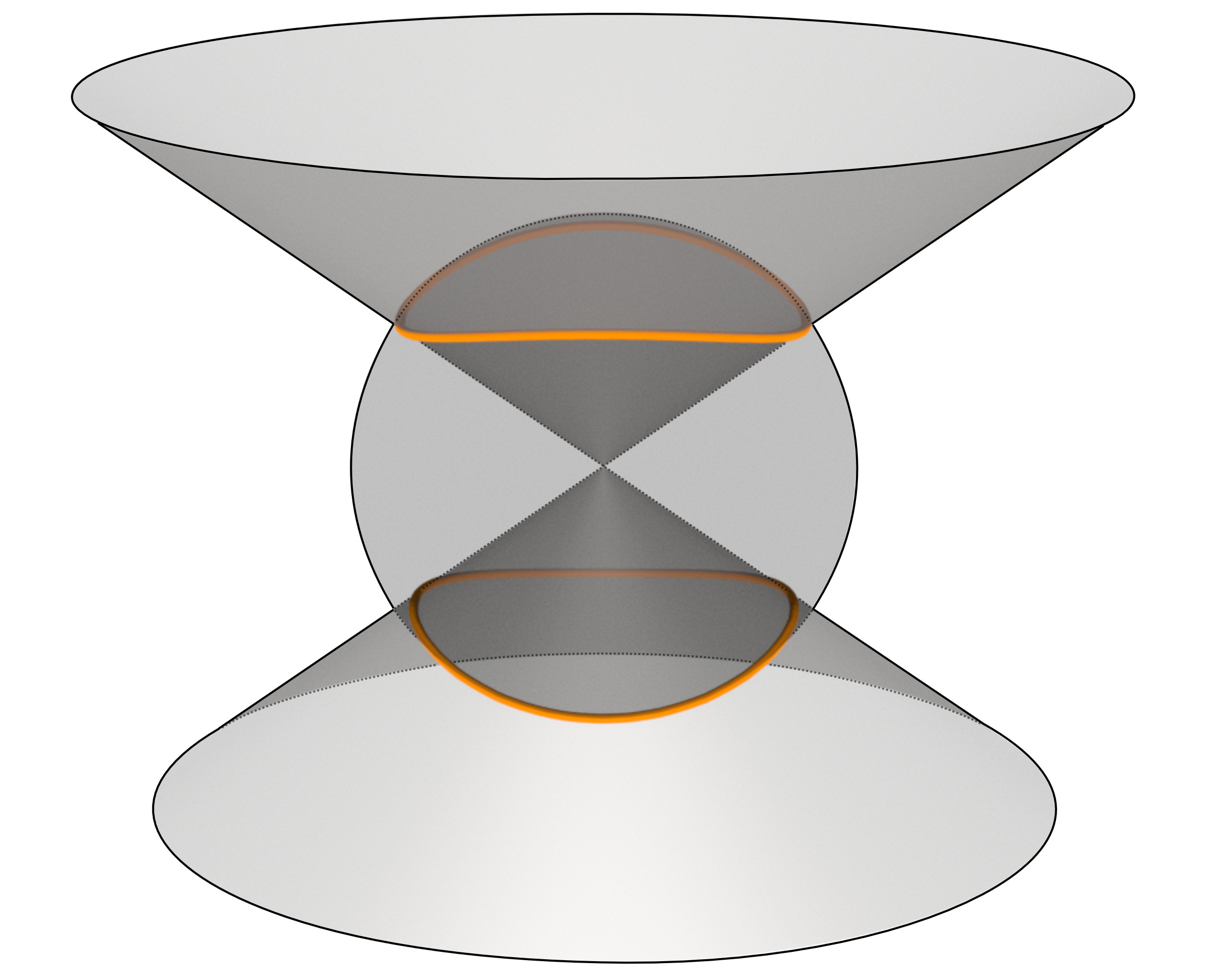}
    \put(71,38){$\lagell$}
    \put(82,77){$\mathcal{C}$}
  \end{overpic}
    \begin{overpic}[width=0.32\textwidth]{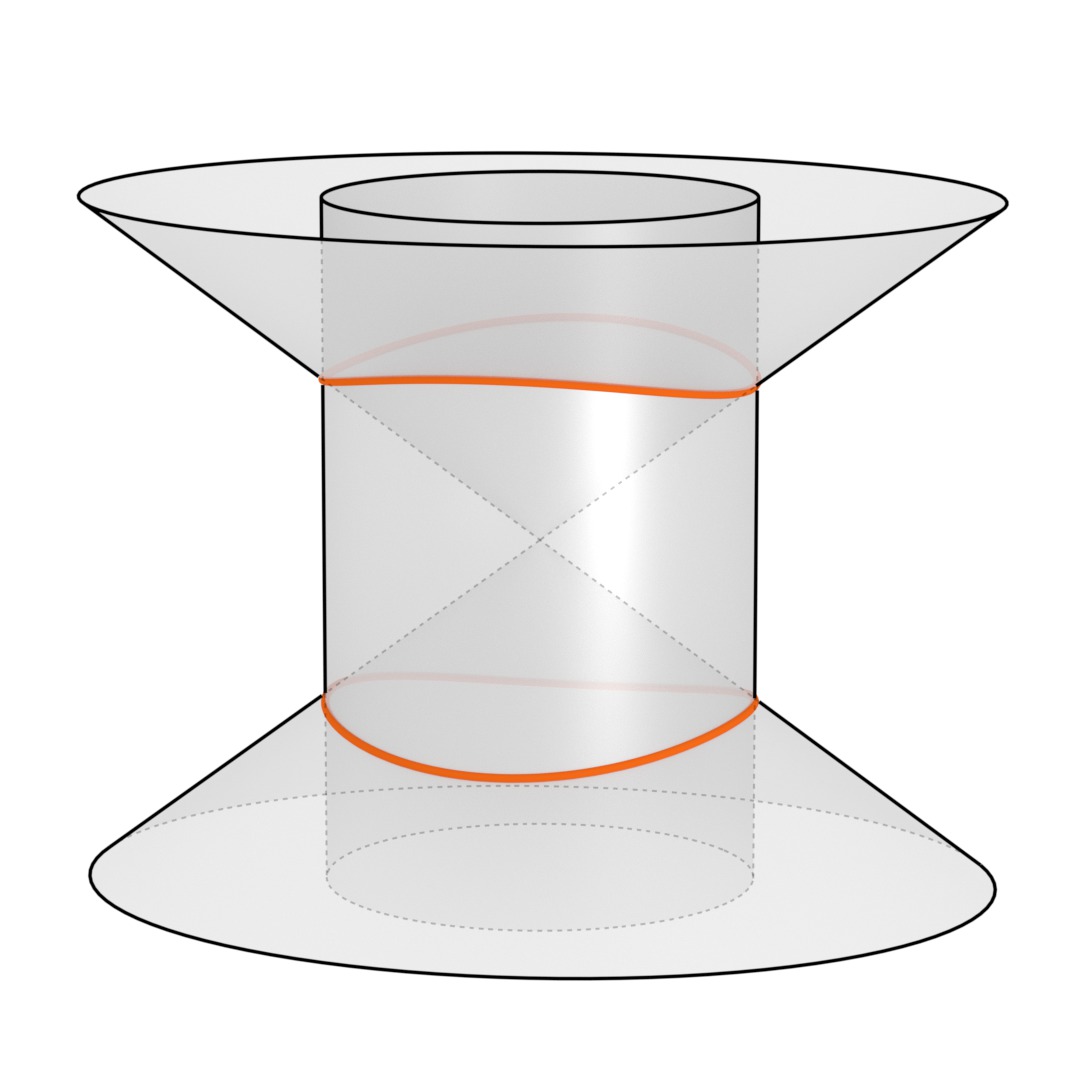}
    \put(63,38){$\lageucl$}
    \put(82,77){$\mathcal{C}$}
  \end{overpic}
  \caption{
    Hypercycle base curve $\lag \cap \mathcal{C}$ for an ellipse
    in hyperbolic (\emph{left}), elliptic (\emph{middle}), and Euclidean (\emph{right}) Laguerre geometry.
  }
\label{fig:hypercycle-base-curves-ellipse}
\end{figure}
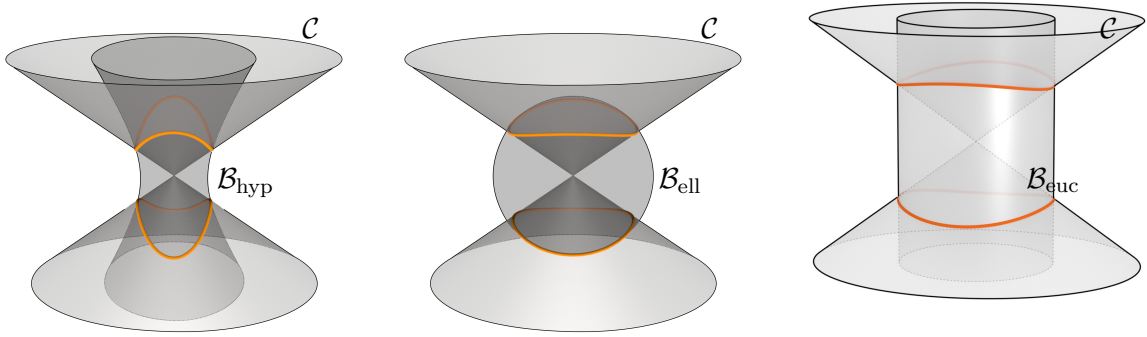
\begin{proposition}
  \label{prop:tangent-ellipse}
  The hypercycle base curve $\lag \cap \mathcal{C}$ corresponds to the (oriented) tangent lines
  of an ellipse given in homogeneous coordinates of the hyperbolic/elliptic/Euclidean plane by
  \begin{equation}
    \frac{x_1^2}{\alpha^2} + \frac{x_2^2}{\beta^2} - x_3^2 = 0.
    \label{eq:ellipse}
  \end{equation}
\end{proposition}
\begin{proof}
  The hyperbolic and elliptic planes are naturally embedded into $\p{p}^\perp$.
  The projection of the intersection curve onto $\p{p}^\perp$ is a conic with equation \eqref{eq:ellipse-cone}.
  Its polar conic is given by \eqref{eq:ellipse}.
  For the Euclidean case, see \cite{BST}.
\end{proof}
\begin{proposition}
  The hypercycle base curve $\lag \cap \mathcal{C}$ consists of two components
  which are parametrized in terms of Jacobi elliptic functions by
  \begin{equation}
    \p{v}_{\pm}(u) =
    \left[
      \frac{1}{\sqrt{1 + \varepsilon\alpha^2}}\, \jac{cn}(u,k),~
      \frac{1}{\sqrt{1 + \varepsilon\beta^2}}\, \jac{sn}(u,k),~
      \frac{\alpha}{\sqrt{1 + \varepsilon\alpha^2}}\, \jac{dn}(u,k),~
      \pm 1
    \right],
    \label{eq:ellipse-parametrization}
  \end{equation}
  for $u \in \R$, where the modulus $k$ is given by
  \begin{equation}
    k^2 = 1 - \frac{\beta^2(1 + \varepsilon\alpha^2)}{\alpha^2(1 + \varepsilon\beta^2)}.
    \label{eq:ellipse-parametrization-modulus}
  \end{equation}
  Alternatively,
  \begin{equation}
    \p{v}_{\pm}(\hat{u}) =
    \left[
      \frac{1}{\alpha}\, \jac{cn}(\hat{u},\hat{k}),~
      \frac{1}{\beta}\, \jac{sn}(\hat{u},\hat{k}),~
      1,~
      \pm \frac{\sqrt{1 + \varepsilon\alpha^2}}{\alpha}\, \jac{dn}(\hat{u},\hat{k})
    \right],
    \label{eq:ellipse-parametrization2}
  \end{equation}
  for $\hat{u} \in \R$, where the modulus $\hat{k}$ is given by
  \begin{equation}
    \hat{k}^2 = 1 - \frac{\alpha^2(1 + \varepsilon\beta^2)}{\beta^2(1 + \varepsilon\alpha^2)}.
    \label{eq:ellipse-parametrization2-modulus}
  \end{equation}
\end{proposition}
\begin{proof}
  Using the elementary identities \cite{NIST, WW}
  \[
    \jac{cn}^2 + \jac{sn}^2 = 1,\qquad
    \jac{dn}^2 + k^2\jac{sn}^2 = 1,
  \]
  one easily checks that, e.g., the parametrization \eqref{eq:ellipse-parametrization2} with \eqref{eq:ellipse-parametrization2-modulus}
  satisfies the two equations \eqref{eq:ellipse-cone} and \eqref{eq:laguerre-quadric}.
  The two parametrizations are related by the real Jacobi transformations
  \[
    \jac{cd}(u,k) = \jac{cn}(\hat{u},\hat{k}),\quad
    \jac{sd}(u,k) = \frac{1}{\sqrt{1 - k^2}}\jac{sn}(\hat{u},\hat{k}),\quad
    \jac{nd}(u,k) = \jac{dn}(\hat{u},\hat{k}).
  \]
  where
  \[
    \hat{u} = \sqrt{1 - k^2}\, u,\qquad
    \hat{k}^2 = \frac{k^2}{k^2 - 1}, \qquad
    \jac{cd}=\frac{\jac{cn}}{\jac{dn}},\ \jac{sd}=\frac{\jac{sn}}{\jac{dn}},\ \jac{nd}=\frac{1}{\jac{dn}}.
  \]
\end{proof}
\begin{remark}\
  \label{rem:base-curve-conventions}
  \nobreakpar
  \begin{enumerate}
  \item
    From \eqref{eq:ellipse-cone-coefficients} we find $0<k^2<1$, or equivalently $\hat{k}^2 < 0$,
    and thus $\p{v}_{\pm}$ attains real values for $u, \hat{u} \in \R$.
  \item
    Over the complex numbers the intersection curve $\lag \cap \mathcal{C}$ is connected
    and constitutes an embedding of an elliptic curve, i.e., a torus.
    The two real components are related by
    \begin{equation}
      \label{eq:elliptic-curve-components}
      \p{v}_\pm(u) = \p{v}_\mp(2i\KK'(k) - u),
    \end{equation}
    where $\KK(k)$ and $\KK'(k) = \KK(\sqrt{1-k^2})$ are the quarter periods of the Jacobi elliptic functions.
  \item
    The signs in the parametrizations of the two compontents are chosen
    such that points on the different components with the same argument $u$ represent the same line with opposite orientation
    \[
      \p{v}_\pm(u) = \sigma_{\p{p}}\left( \p{v}_\mp(u) \right).
    \]
  \item
    The hypercycle base curves treated in Section \ref{sec:ellipse} and \ref{sec:hyperbola} are all projectively
    equivalent for different values of $\varepsilon$.
    Thus, their parametrizations may all be obtained from, e.g., \eqref{eq:ellipse-parametrization} with $\varepsilon=0$ 
    by reinterpreting another quadric of the pencil as the Laguerre quadric and applying a suitable projective transformation.
    \label{rem:base-curve-conventions-signs}
  % \item
  %   For $\varepsilon = 0$ the parametrization \label{eq:ellipse-parametrization} with \label{eq:ellipse-parametrization-modulus}
  %   coincides with the parametrization given in \cite{BST} for Euclidean checkerboard incircular nets tangent to an ellipse
  %   (up to the convention of signs, cf. Remark \ref{rem:base-curve-conventions} \ref{rem:base-curve-conventions-signs}).
  \end{enumerate}
\end{remark}
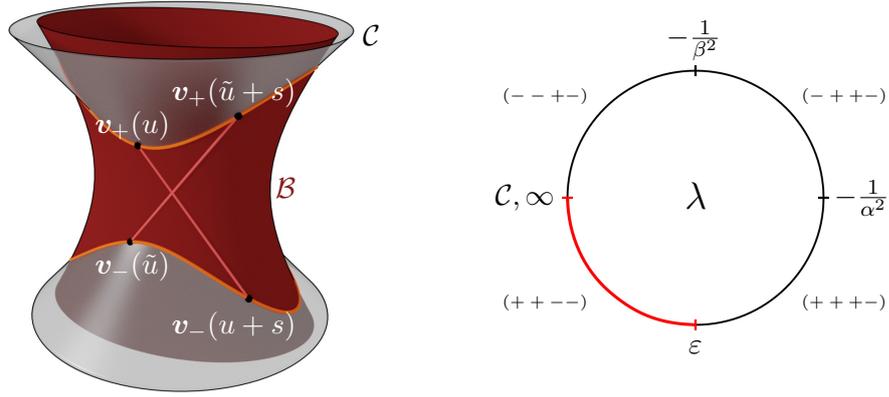
\begin{figure}[h]
  \centering
  \begin{overpic}[width=0.45\textwidth]{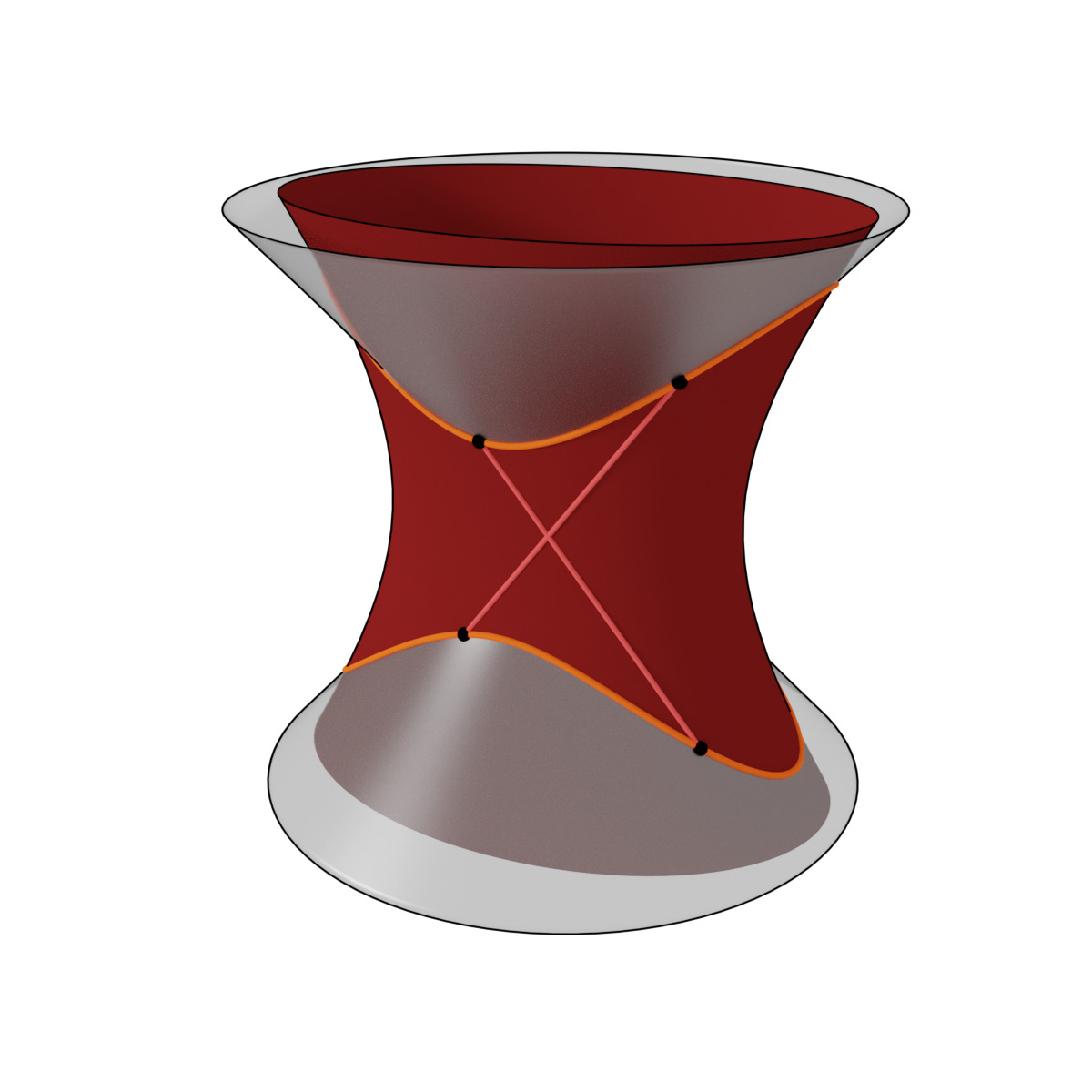}
    \definecolor{darkred}{HTML}{831b1a}
    \put(69,50){$\color{darkred}\lag$}
    \put(85,78){$\mathcal{C}$}
    \put(36,62){\color{white}$\p{v}_+(u)$}
    \put(50,68){\color{white}$\p{v}_+(\tilde u + s)$}
    \put(36,36){\color{white}$\p{v}_-(\tilde u)$}
    \put(50,25){\color{white}$\p{v}_-(u + s)$}
  \end{overpic}\qquad
  \raisebox{1.2cm}{
    \begin{overpic}[width=0.3\textwidth]{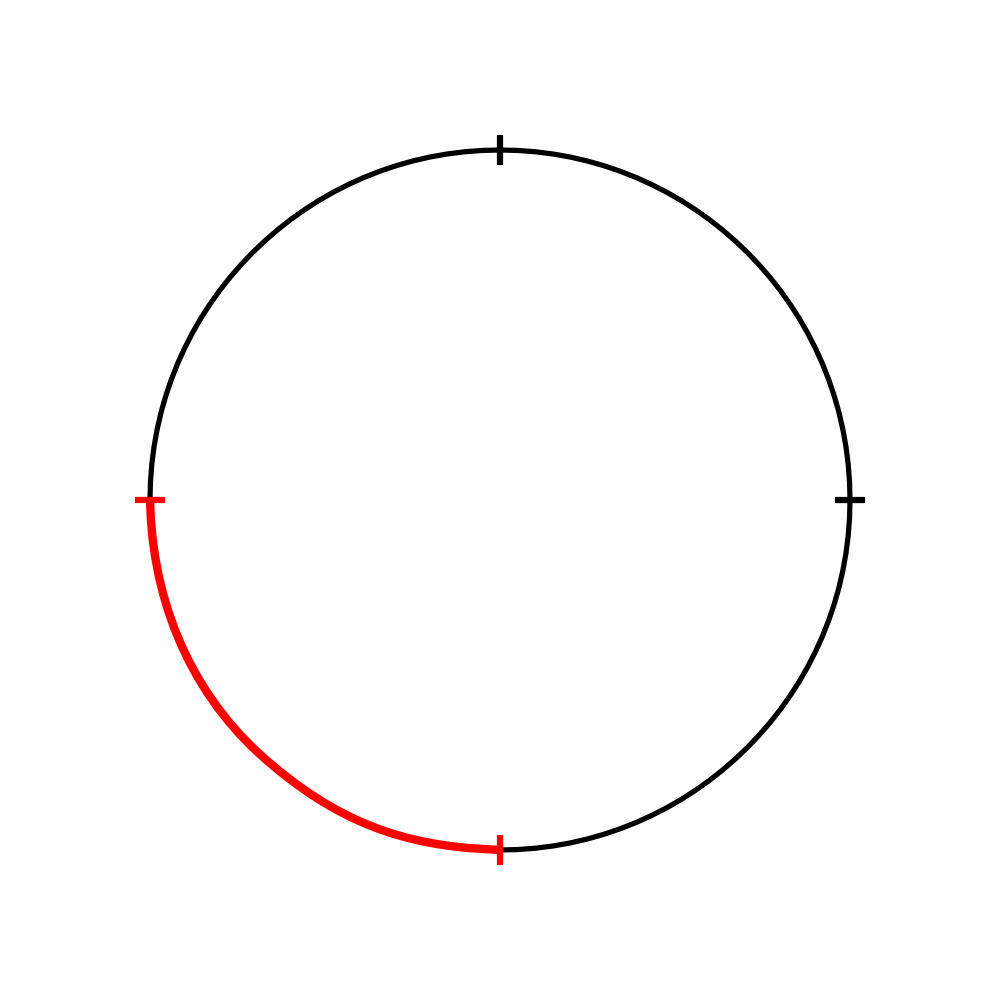}
      \put(47,47){\Large$\lambda$}
      \put(-5,48){$\mathcal{C}, \infty$}
      \put(88,48){$-\frac{1}{\alpha^2}$}
      \put(42,92){$-\frac{1}{\beta^2}$}
      \put(48,7){$\varepsilon$}
      \put(79,77){\tiny$(-++-)$}
      \put(-3,77){\tiny$(--+-)$}
      \put(-3,20){\tiny$(++--)$}
      \put(79,20){\tiny$(+++-)$}
    \end{overpic}
  }
  \caption{
    \emph{Left:}
    Four coplanar points on a hypercycle base curve $\lag \cap \mathcal{C}$.
    The two lines are rulings from a common hyperboloid in the pencil corresponding to $\lag \cap \mathcal{C}$.
    \emph{Right:}
    The parameter $\lambda$ for the pencil $\lag \wedge \mathcal{C}$ as given by \eqref{eq:pencil-ellipse}.
    The four values $-\frac{1}{\beta^2}, -\frac{1}{\alpha^2}, \varepsilon, \infty$ correspond to the degenerate quadrics in the pencil.
    In between, the signature of the quadrics from the pencil are given.
    The function \eqref{eq:lambda-ellipse} takes values in $[\varepsilon, \infty]$ and corresponds to hyperboloids whose rulings intersect both components of the base curve.
  }
  \label{fig:planarity-and-pencil-ellipse}
\end{figure}
This parametrization features the following remarkable property which is related to the addition on elliptic curves (cf.\ \cite{husemoeller}).
\begin{proposition}\
  \label{prop:parametrization-coplanarity}
  \nobreakpar
  \begin{enumerate}
  \item
    \label{prop:parametrization-coplanarity1}
    Let $u, \tilde{u}, s \in \R$.
    Then the four points $\p{v}_+(u), \p{v}_-(u+s), \p{v}_-(\tilde{u}), \p{v}_+(\tilde{u} + s)$ are coplanar (see Figure \ref{fig:planarity-and-pencil-ellipse}, \textit{left}).
  \item
    Let $s \in \R$.
    Then the lines $\p{v}_+(u)\wedge\p{v}_-(u+s)$ with $u \in \R$
    constitute one family of rulings of a common hyperboloid in the pencil $\lag \wedge \mathcal{C}$
    \begin{equation}
      \label{eq:pencil-ellipse}
      (1 + \lambda\alpha^2)x_1^2 + (1 + \lambda\beta^2)x_2^2 + (\varepsilon - \lambda)x_3^2 - x_4^2 = 0
    \end{equation}
    given by
    \begin{equation}
      \label{eq:lambda-ellipse}
      \lambda(s) = \frac{1}{\beta^2}\jac{cs}^2(\tfrac{s}{2},k) + \varepsilon\jac{ns}^2(\tfrac{s}{2}, k),\qquad
      \text{where}~ \jac{cs} = \frac{\jac{cn}}{\jac{sn}},\ \jac{ns} = \frac{1}{\jac{sn}}.
    \end{equation}
    The second family of rulings is given by the lines $\p{v}_+(u)\wedge\p{v}_-(u-s)$ with $u \in \R$.
  \end{enumerate}
\end{proposition}
\begin{proof}\
  \nobreakpar
  \begin{enumerate}
  \item
    By \eqref{eq:elliptic-curve-components} we obtain
    \[
      \begin{aligned}
        &\det\left(v_+(u), v_-(u+s), v_-(\tilde{u}), v_+(\tilde{u} + s)\right)\\
        &\qquad= \det\left(v_+(u), v_+(2i\KK'(k)-u-s), v_+(-2i\KK'(k)-\tilde{u}), v_+(\tilde{u} + s)\right)        
      \end{aligned}
    \]
    which is zero due to the following addition theorem for Jacobi elliptic functions \cite{NIST}:
    \begin{equation}
      \label{eq:Jacobi-identity}
      \def\arraystretch{1.}
      \begin{vmatrix}
        \jac{cn}(z_1,k) & \jac{sn}(z_1,k) &\jac{dn}(z_1,k) & 1\\
        \jac{cn}(z_2,k) & \jac{sn}(z_2,k) &\jac{dn}(z_2,k) & 1\\
        \jac{cn}(z_3,k) & \jac{sn}(z_3,k) &\jac{dn}(z_3,k) & 1\\
        \jac{cn}(z_4,k) & \jac{sn}(z_4,k) &\jac{dn}(z_4,k) & 1                     
      \end{vmatrix} = 0,\quad z_1+z_2+z_3+z_4 = 0.
    \end{equation}
  \item
    By Lemma \ref{lem:pencils-and-lines}, there exists a unique hyperboloid $\mathcal{Q}$ in the pencil $\lag \wedge \mathcal{C}$
    containing the line $\p{v}_+(u)\wedge\p{v}_-(u+s)$.
    If we denote by $\scalarprod{\cdot}{\cdot}_{\mathcal{C}}$ the symmetric bilinear form
    corresponding to the quadratic form \eqref{eq:ellipse-cone} of the cone $\mathcal{C}$,
    the parameter $\lambda$ corresponding to $\mathcal{Q}$ is given by
    \[
      \scalarprod{v_+(u)}{v_-(u+s)} + \lambda \scalarprod{v_+(u)}{v_-(u+s)}_{\mathcal{C}} = 0,
    \]
    which is equivalent to \eqref{eq:linear-combination-addition-theorems} with
    \[
      \rho_{\jac{cn}} = \frac{1+\lambda\alpha^2}{1+\varepsilon\alpha^2},\quad
      \rho_{\jac{sn}} = \frac{1+\lambda\beta^2}{1+\varepsilon\beta^2},\quad
      \rho_{\jac{dn}} = \frac{(\varepsilon-\lambda)\alpha^2}{1+\varepsilon\alpha^2},\quad
      \rho_{1} = 1.
    \]
    Thus, by Lemma \ref{lem:linear-combination-addition-theorems}, one obtains
    \[
      \begin{aligned}
        (1 + \lambda\alpha^2) \jac{cn}(s)
        +  (\varepsilon - \lambda) \alpha^2 \jac{dn}(s)
        + 1 + \varepsilon\alpha^2 &= 0,\\
        (1 + \lambda\beta^2) \jac{cn}(s)
        +  (\varepsilon - \lambda) \beta^2
        + (1 + \varepsilon\beta^2) \jac{dn}(s) &= 0.
      \end{aligned}
    \]
    These two equations for $\lambda$ are equivalent and give
    \[
      \lambda
      = \frac{1}{\beta^2}\frac{\jac{cn}(s) + \jac{dn}(s)}{1 - \jac{cn}(s)}
      + \varepsilon \frac{1 + \jac{dn}(s)}{1 - \jac{cn}(s)}
      = \frac{1}{\beta^2}\jac{cs}^2(\tfrac{s}{2})
      + \varepsilon\jac{ns}^2(\tfrac{s}{2}).
    \]
  \end{enumerate}
  Note that $s$ and $-s$ correspond to the same hyperboloid since $\lambda(s) = \lambda(-s)$.
  Given the line $\p{L} = \p{v}_+(u) \wedge \p{v}_-(u+s)$, all lines $\p{v}_+(\tilde{u})\wedge\p{v}_-(\tilde{u}-s)$ with $\tilde{u} \in \R$
  intersect the line $\p{L}$ by \ref{prop:parametrization-coplanarity1},
  and thus belong to the respective other (and therefore the same) family of rulings of $\mathcal{Q}$.
\end{proof}
\begin{remark}\
  \nobreakpar
  \begin{enumerate}
  \item
    The alternative expression for $\lambda$ in terms of $\hat{s} = \sqrt{1-k^2}s$ and $\hat{k}^2 = \frac{k^2}{k^2-1}$ is given by
    \[
      \lambda(\hat{s}) = \frac{1}{\alpha^2}\jac{cs}^2(\tfrac{\hat{s}}{2},\hat{k}) + \varepsilon\jac{ns}^2(\tfrac{\hat{s}}{2}, \hat{k}).
    \]
  \item
    The four degenerate quadrics from the pencil \eqref{eq:pencil-ellipse} are given by
    the values $\lambda = -\frac{1}{\beta^2}, -\frac{1}{\alpha^2}, \varepsilon, \infty$,
    where $\lambda = \infty$ corresponds to the cone $\mathcal{C}$ (see Figure \ref{fig:planarity-and-pencil-ellipse}, right).
    By construction, the hyperboloids obtained by \eqref{eq:lambda-ellipse} have rulings connecting the two components of the base curve,
    which corresponds to the fact that
    \[
      \varepsilon \leq \lambda(s) \leq \infty
    \]
    for $s \in \R$ with $\lambda(0) = \infty$ and $\lambda(2\KK(k)) = \varepsilon$.
  \end{enumerate}
\end{remark}
\begin{lemma}
  \label{lem:linear-combination-addition-theorems}
  Let $s \in \R$ and $\rho_{\jac{cn}}, \rho_{\jac{sn}}, \rho_{\jac{dn}}, \rho_{1} \in \R$ such that
  \begin{equation}
    \label{eq:linear-combination-addition-theorems}
    \rho_{\jac{cn}}\jac{cn}(u)\jac{cn}(u+s)
    + \rho_{\jac{sn}}\jac{sn}(u)\jac{sn}(u+s)
    + \rho_{\jac{dn}}\jac{dn}(u)\jac{dn}(u+s)
    + \rho_1
    = 0
  \end{equation}
  for all $u \in \R$.
  Then
  \begin{equation}
    \label{eq:linear-combination-addition-theorems-coefficients}
    \begin{aligned}
      \rho_{\jac{cn}}\jac{cn}(s) + \rho_{\jac{dn}}\jac{dn}(s) + \rho_{1} &= 0,\\
      \rho_{\jac{sn}}\jac{cn}(s) + \rho_{\jac{dn}}(1-k^2) + \rho_{1}\jac{dn}(s) &= 0.
    \end{aligned}
  \end{equation}
\end{lemma}
\begin{proof}
  Applying the addition theorems for Jacobi elliptic functions \cite{NIST} to \eqref{eq:linear-combination-addition-theorems},
  the resulting equation can be written as a sum of three independent functions, say, $\jac{sn}^2(u)$, $\jac{sn}(u)\jac{cn}(u)\jac{dn}(u)$, and 1.
  The vanishing of the coefficients of these three functions leads to three linear equations in $\rho_{\jac{cn}}, \rho_{\jac{sn}}, \rho_{\jac{dn}}, \rho_{1}$, which constitute a rank 2 linear system equivalent to \eqref{eq:linear-combination-addition-theorems-coefficients}.
\end{proof}
This allows to parametrize a checkerboard incircular net
tangent to a given ellipse in the following way (see Figures~\ref{fig:periodic-cbic-net-elliptic-ellipse},~\ref{fig:periodic-cbic-net-hyperbolic-ellipse}).
\begin{theorem}
  \label{thm:cbic-formulas-ellipse}
  Let $\varepsilon \in \{-1, 0, 1\}$.
  Then for $\alpha > \beta > 0$ with $1 + \varepsilon\alpha^2, 1 + \varepsilon\beta^2 > 0$,
  $s, \tilde{s} \in \R$, and $u_0^{\ell}, u_0^{m} \in \R$
  the two families of lines $\left(\ell_i\right)_{i\in\Z}$ and $\left(m_j\right)_{j\in\Z}$ given by
  \[
    \begin{aligned}
      \ell_{2k} &= \p{v}_+(u_0^{\ell} + k(s + \tilde{s}))\\
      \ell_{2k+1} &= \p{v}_-(u_0^{\ell} + k(s + \tilde{s}) + s)\\
      m_{2l} &= \p{v}_-(u_0^{m} + l(s + \tilde{s}))\\
      m_{2l+1} &= \p{v}_+(u_0^{m} + l(s + \tilde{s}) + s)
    \end{aligned}
  \]
  constitute a hyperbolic/elliptic/Euclidean checkerboard incircular net (according to the value of~$\varepsilon$)
  tangent to the ellipse
  \[
    \frac{x_1^2}{\alpha^2} + \frac{x_2^2}{\beta^2} - x_3^2 = 0.
  \]
\end{theorem}
\begin{proof}
  The existence of incircles follows from Proposition \ref{prop:parametrization-coplanarity},
  while tangency to the given ellipse follows from Proposition \ref{prop:tangent-ellipse}.
\end{proof}
The choice of
\begin{itemize}
\item $\alpha$ and $\beta$ determines the ellipse,
\item $s$ and $\tilde{s}$ determines two hyperboloids in the pencil of quadrics,
  and further allows to distinguish the two families of rulings on each of them,
\item $u_0^{\mathrm{v}}$, $u_0^{\mathrm{h}}$ determines one initial line tangent to the ellipse
  in each of the two families of lines.
\end{itemize}

Note that for $s = 0$ the corresponding hyperboloid degenerates to the cone $\mathcal{C}$.
In this case,
\[
  \ell_{2k} = \ell_{2k+1},\qquad
  m_{2l} = m_{2l+1},
\]
and thus, $\left(\ell_i\right)_{i\in\Z}$ and $\left(m_j\right)_{j\in\Z}$ constitutes an ``ordinary'' incircular net.

Periodicity can be achieved by setting
\[
  s + \tilde{s} = \frac{4\KK(k)}{N}.
\]
To achieve ``embeddedness'' of the checkerboard incircular nets one has to additionally demand
that the two different families of lines agree (up to their orientation), e.g.,\
\[
  \ell_{i} = \sigma_{\p{p}}(m_i),
\]
which is obtained by setting
\[
  u_0^{\ell} = u_0^{m}.
\]
%\note[JT]{It might be more convinient to choose something like $u_0^{\ell} = u_0^{m} + 2\KK$ or $u_0^{\ell} = u_0^{m} + \KK$ to have the ``correct'' circles being embedded. Check while implementing.}

\begin{figure}[H]
  \centering
  \hspace{\fill}
  \includegraphics[width=0.38\textwidth]{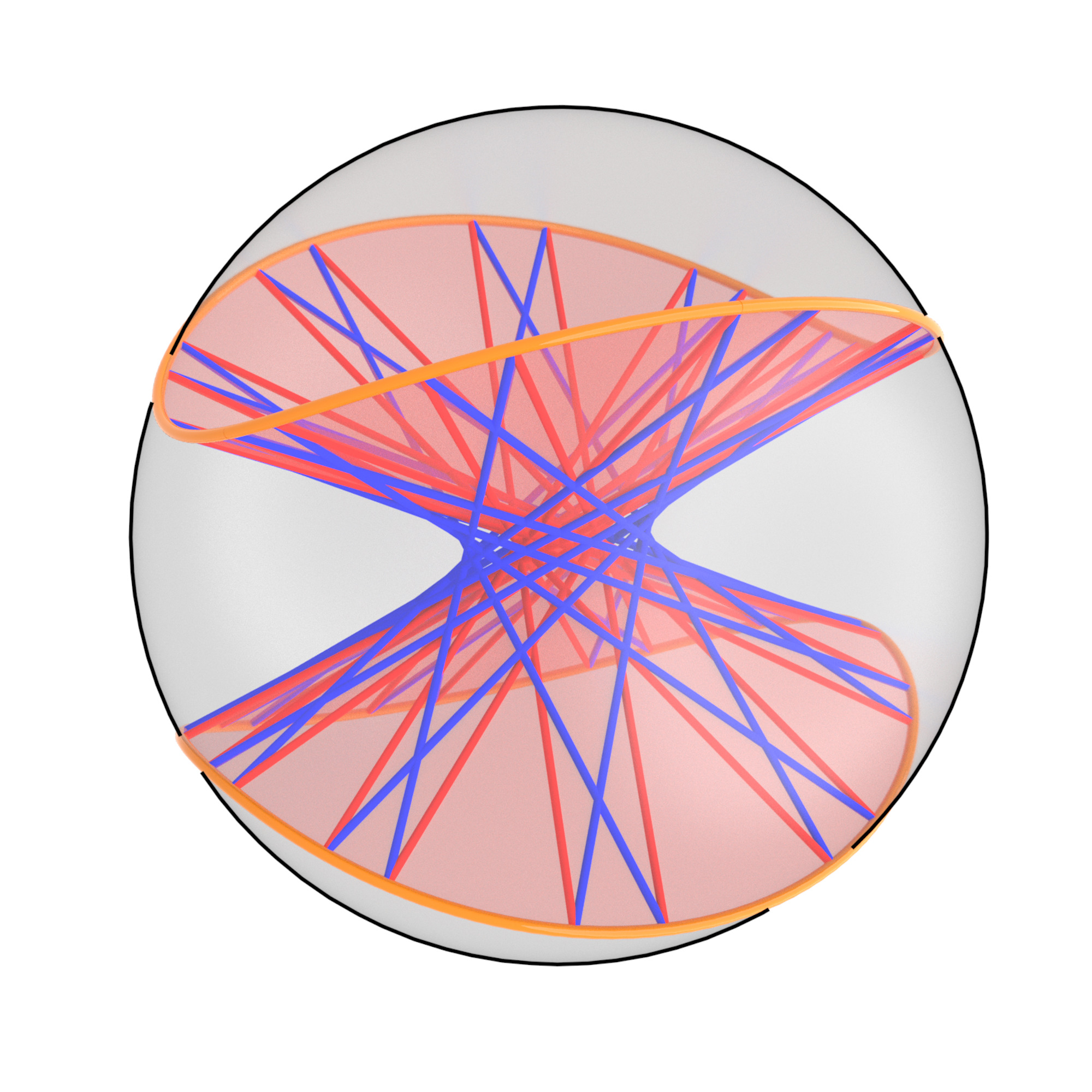}
  \hspace{\fill}
  \includegraphics[width=0.38\textwidth]{ell_cic_sphere_1}
  \hspace{\fill}
  \newline
  \mbox{}
  \hspace{\fill}
  \includegraphics[width=0.38\textwidth]{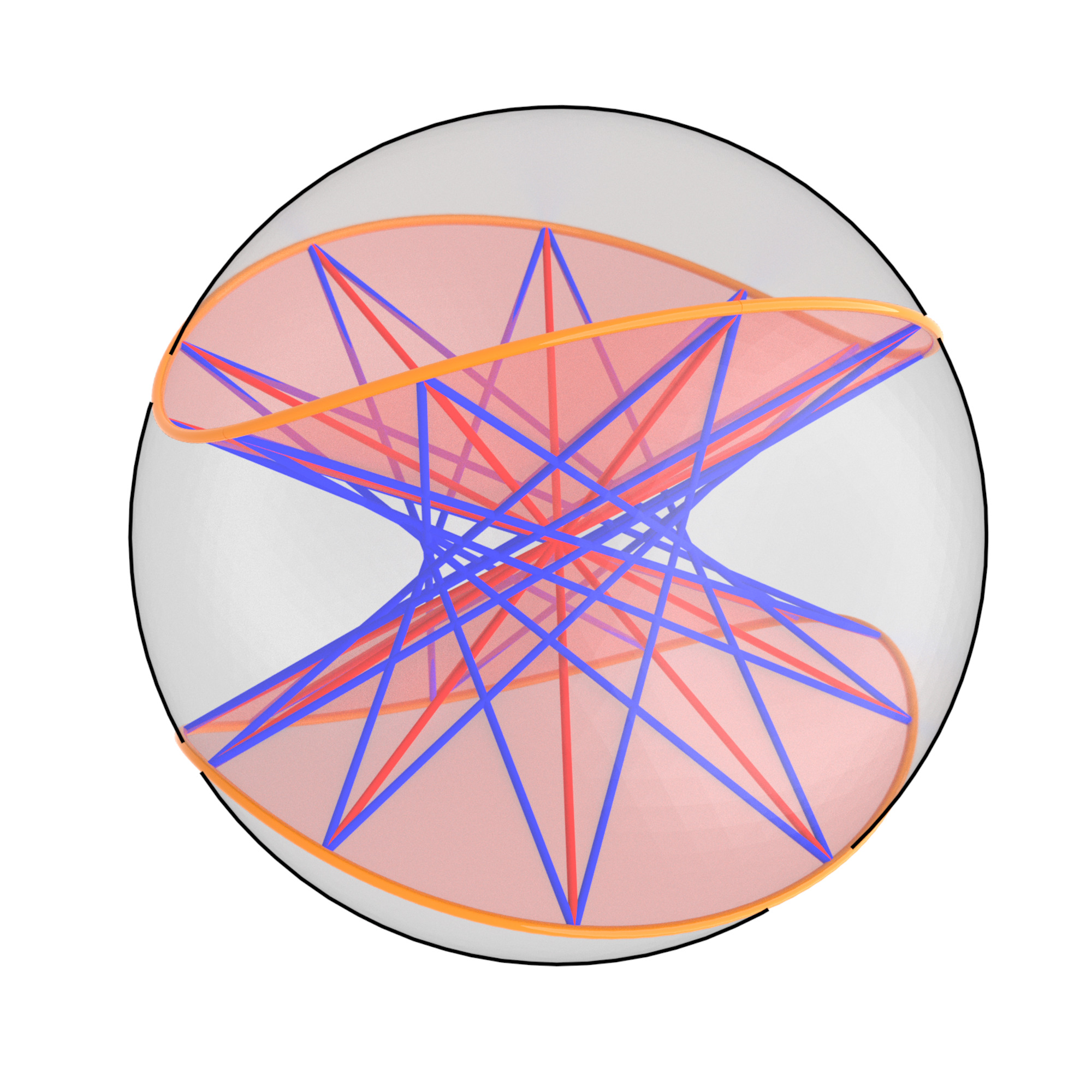}
  \hspace{\fill}
  \includegraphics[width=0.38\textwidth]{ell_ic_sphere_1}
  \hspace{\fill}
  \caption{
    Periodic checkerboard incircular net in elliptic geometry ($\varepsilon=1$)
    tangent to an ellipse with $\alpha = 0.9$, $\beta = 0.4$, $N = 11$, and $s=0.23$ (\emph{top}) / $s = 0$ (\emph{bottom}).
    Represented on the Laguerre quadric (\emph{left}) and on the sphere model of elliptic geometry (\emph{right}).
  }
\label{fig:periodic-cbic-net-elliptic-ellipse}
\end{figure}
\begin{figure}[H]
  \centering
  \hspace{\fill}
  \includegraphics[width=0.38\textwidth]{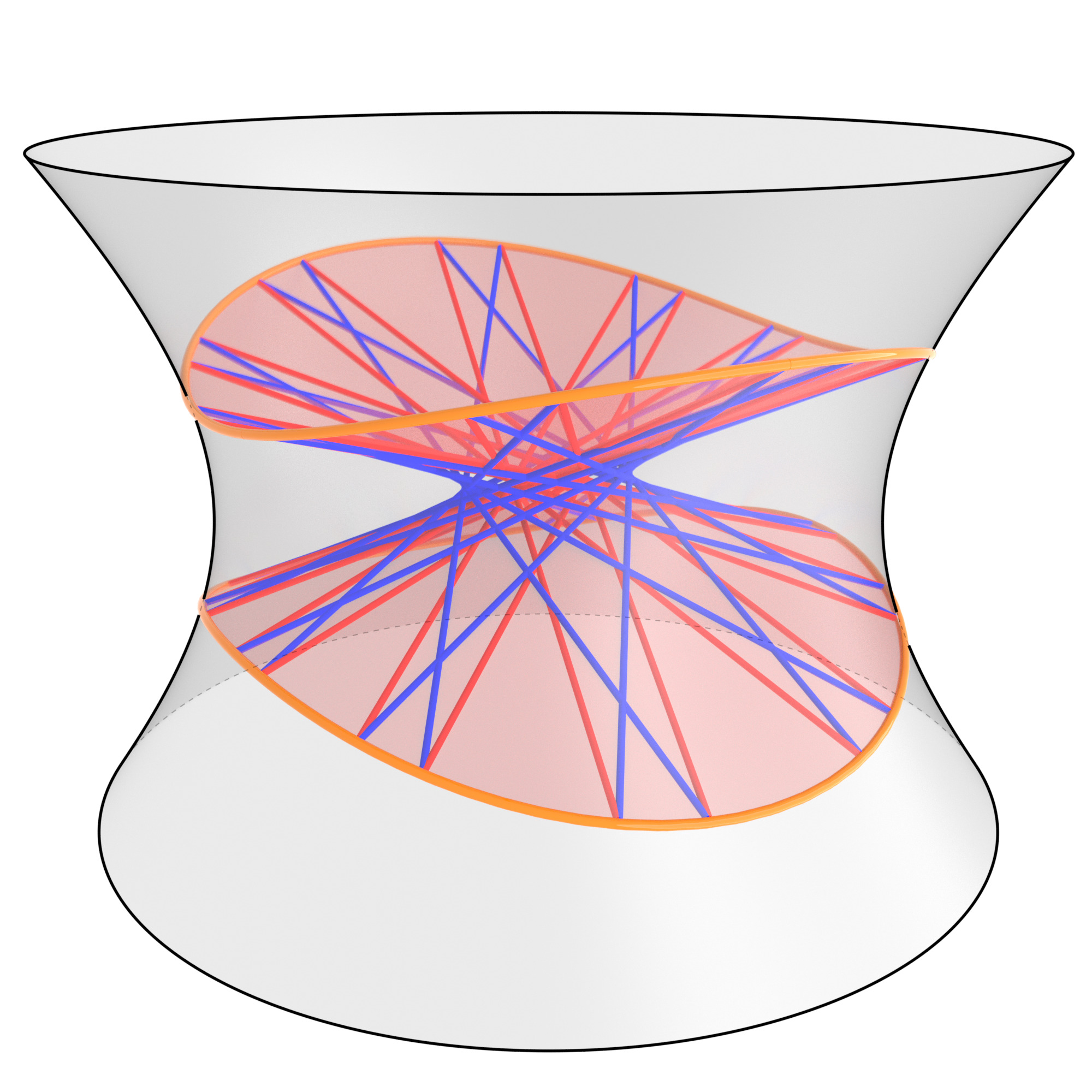}
  \hspace{\fill}
  \raisebox{0.01\textwidth}{
    \includegraphics[width=0.36\textwidth]{hyp_cic_ellipse_stereog_projection_1}
  }
  \hspace{\fill}
  \newline
  \mbox{}
  \hspace{\fill}
  \includegraphics[width=0.38\textwidth]{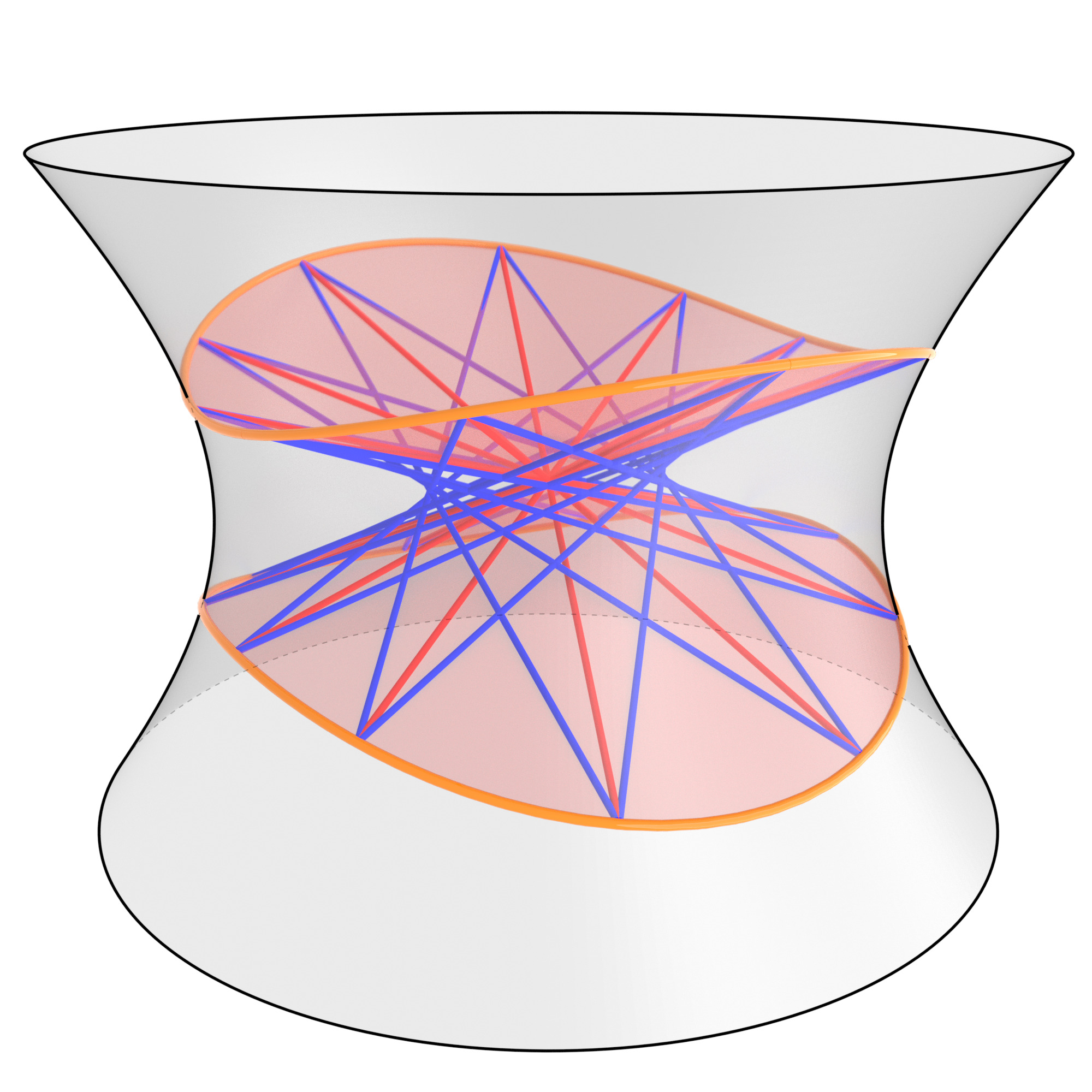}
  \hspace{\fill}
  \raisebox{0.01\textwidth}{
    \includegraphics[width=0.36\textwidth]{hyp_ic_ellipse_stereog_projection_1}
  }
  \hspace{\fill}
  \caption{
    Periodic checkerboard incircular net in hyperbolic geometry ($\varepsilon=-1$)
    tangent to an ellipse with $\alpha = 0.5$, $\beta = 0.3$, $N = 11$, and $s=0.23$ (\emph{top}) / $s = 0$ (\emph{bottom}).
    Represented on the Laguerre quadric (\emph{left}) and on the Poincaré disk model of hyperbolic geometry (\emph{right}).
  }
\label{fig:periodic-cbic-net-hyperbolic-ellipse}
\end{figure}

\subsubsection{Parametrization of checkerboard incircular nets tangent to a hyperbola}
\label{sec:hyperbola}
Consider a cone $\mathcal{C}$ given by
\begin{equation}
  \alpha^2 x_1^2 - \beta^2 x_2^2 - x_3^2 = 0
  \label{eq:hyperbola-cone}
\end{equation}
with
\begin{equation}
  \label{eq:hyperbola-cone-coefficients}
  \alpha, \beta > 0,\qquad
  1 + \varepsilon \alpha^2, 1 - \varepsilon \beta^2 > 0.
\end{equation}
It intersects the Laguerre quadric $\lag$ given by
\begin{equation}
  x_1^2 + x_2^2 + \varepsilon x_3^2 - x_4^2 = 0
  \label{eq:laguerre-quadric-alt}
\end{equation}
in the hypercycle base curve $\lag \cap \mathcal{C}$ (see Figure \ref{fig:hypercycle-base-curves-hyperbola}).
\begin{remark}
  In the elliptic plane all generic conics are ellipses.
  Correspondingly, for $\varepsilon = 1$ the case \eqref{eq:hyperbola-cone} is equivalent to \eqref{eq:ellipse-cone}.
\end{remark}
\begin{figure}
  \centering
  \begin{overpic}[width=0.32\textwidth]{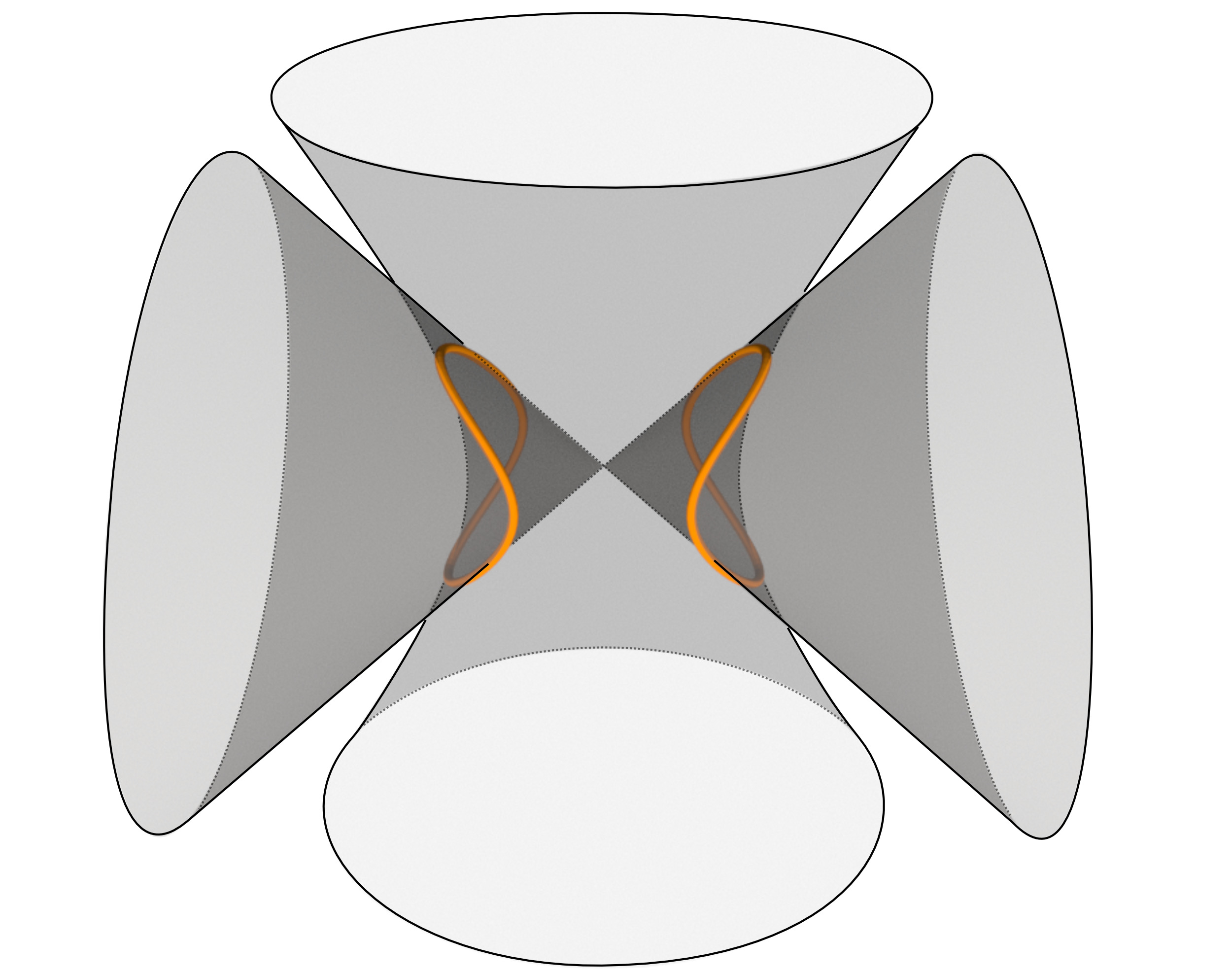}
    \put(60,80){$\laghyp$}
    \put(91,21){$\mathcal{C}$}
  \end{overpic}
  \hspace{1.5cm}
  \begin{overpic}[width=0.32\textwidth]{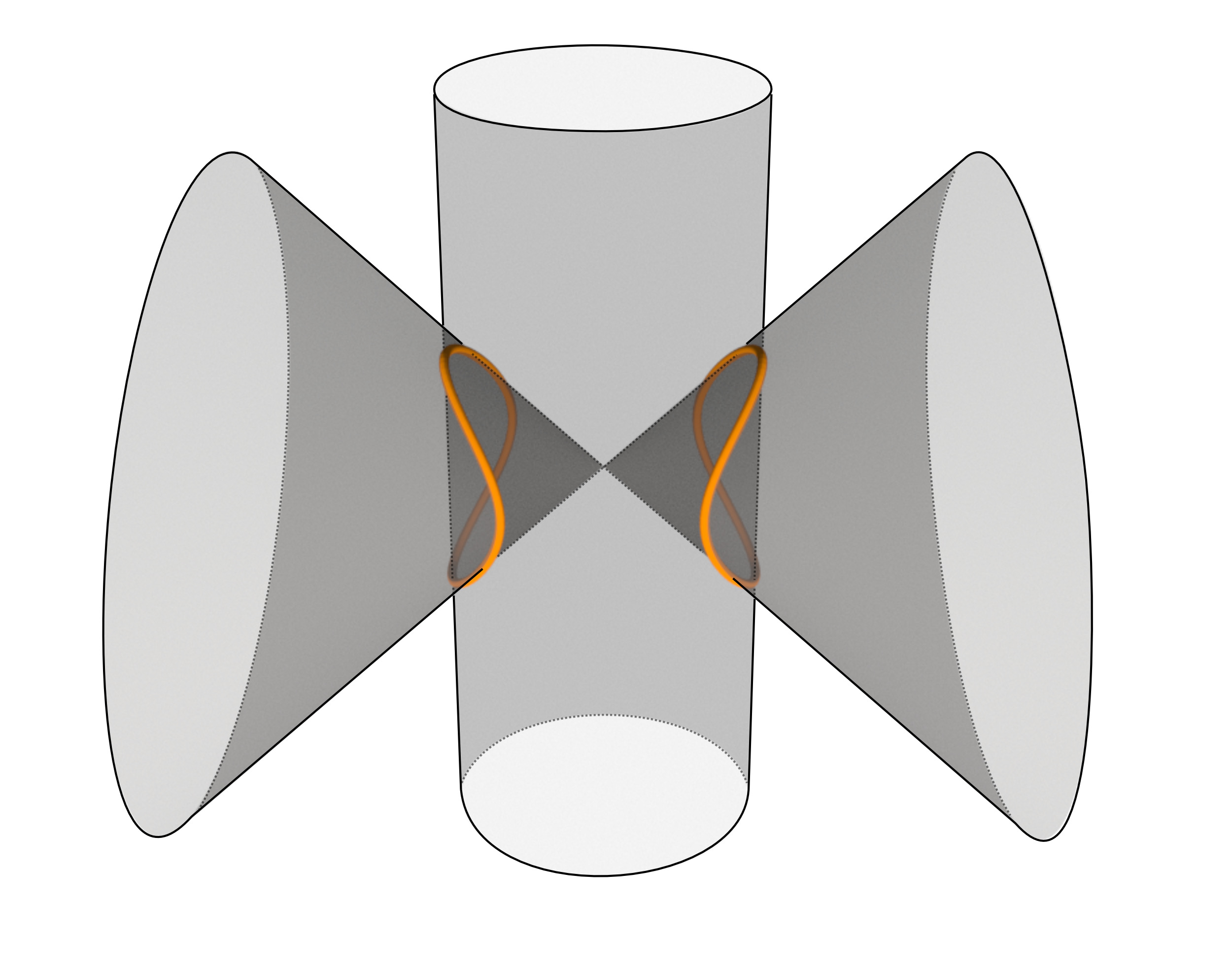}
    \put(58,76){$\lageucl$}
    \put(91,21){$\mathcal{C}$}
  \end{overpic}
  \caption{
    Hypercycle base curve $\lag \cap \mathcal{C}$ for a hyperbola
    in hyperbolic (\emph{left}), and Euclidean (\emph{right}) Laguerre geometry.
  }
\label{fig:hypercycle-base-curves-hyperbola}
\end{figure}
\begin{proposition}
  The hypercycle base curve $\lag \cap \mathcal{C}$ corresponds to the (oriented) tangent lines of a hyperbola
  given in homogeneous coordinates of the hyperbolic/Euclidean plane by
  \begin{equation}
    \frac{x_1^2}{\alpha^2} - \frac{x_2^2}{\beta^2} - x_3^2 = 0.
    \label{eq:hyperbola}
  \end{equation}
\end{proposition}
\begin{proposition}
  The intersection curve $\lag \cap \mathcal{C}$ consists of two components
  which are parametrized in terms of Jacobi elliptic functions by
  \begin{equation}
    \p{v}_{\pm}(u) =
    \left[
      \frac{1}{\sqrt{1 + \varepsilon\alpha^2}}\, \jac{dn}(u,k),~
      \frac{\alpha}{\sqrt{\alpha^2 + \beta^2}}\, \jac{sn}(u,k),~
      \frac{\alpha}{\sqrt{1 + \varepsilon\alpha^2}}\, \jac{cn}(u,k),~
      \pm 1
    \right],
  \end{equation}
  for $u \in \R$, where the modulus is given by
  \[
    k^2 = \frac{\alpha^2(1 - \varepsilon \beta^2)}{\alpha^2 + \beta^2}.
  \]
  Alternatively
  \begin{equation}
    \label{eq:hyperbola-parametrization2}
    \p{v}_{\pm}(\hat{u}) =
    \left[
      \frac{1}{\alpha}\, \jac{nc}(\hat{u},\hat{k}),~
      \frac{1}{\beta}\, \jac{sc}(\hat{u},\hat{k}),~
      1,~
      \pm \frac{\sqrt{1 + \varepsilon\alpha^2}}{\alpha}\, \jac{dc}(\hat{u},\hat{k})
    \right],
  \end{equation}
  for $\hat{u} \in \R$, where the modulus is given by
  \begin{equation}
    \label{eq:hyperbola-parametrization2-modulus}
    \hat{k}^2 = - \frac{\alpha^2(1 - \varepsilon\beta^2)}{\beta^2(1 + \varepsilon\alpha^2)}.
  \end{equation}
\end{proposition}
%
% \begin{proof}
%   Using the elementary identities \cite{NIST}
%   \[
%     \jac{nc}^2 - \jac{sc}^2 = 1,\qquad
%     \jac{nc}^2 - k^2\jac{sc}^2 = \jac{dc}^2
%   \]
%   one easily checks that, e.g., the parametrization \eqref{eq:hyperbola-parametrization2} with \eqref{eq:hyperbola-parametrization2-modulus}
%   satisfies the two equations \eqref{eq:hyperbola-cone} and \eqref{eq:laguerre-quadric-alt}.
%   The two parametrizations are related by the real Jacobi transformations
%   \[
%     \jac{dc}(u,k) = \jac{nc}(\hat{u},\hat{k}),\quad
%     \jac{sc}(u,k) = \frac{1}{\sqrt{1 - k^2}}\jac{sc}(\hat{u},\hat{k}),\quad
%     \jac{nc}(u,k) = \jac{dc}(\hat{u},\hat{k}).
%   \]
%   where
%   \[
%     \hat{u} = \sqrt{1 - k^2}\, u,\qquad
%     \hat{k}^2 = \frac{k^2}{k^2 - 1}.
%   \]
% \end{proof}
%
\begin{remark}\
  \nobreakpar
  \begin{enumerate}
  \item
    The two parametrizations are related by the real Jacobi transformations
    \[
      \jac{dc}(u,k) = \jac{nc}(\hat{u},\hat{k}),\quad
      \jac{sc}(u,k) = \frac{1}{\sqrt{1 - k^2}}\jac{sc}(\hat{u},\hat{k}),\quad
      \jac{nc}(u,k) = \jac{dc}(\hat{u},\hat{k}).
    \]
    where
    \[
      \hat{u} = \sqrt{1 - k^2}\, u,\qquad
      \hat{k}^2 = \frac{k^2}{k^2 - 1},\qquad
      \jac{dc} = \frac{\jac{dn}}{\jac{cn}},\ \jac{sc} = \frac{\jac{sn}}{\jac{cn}},\ \jac{nc} = \frac{\jac{1}}{\jac{cn}}.
    \]
  \item
    All points from Remark \ref{rem:base-curve-conventions} also apply to this parametrization.
  \end{enumerate}
\end{remark}
\begin{figure}[h]
  \centering
  \begin{overpic}[width=0.48\textwidth]{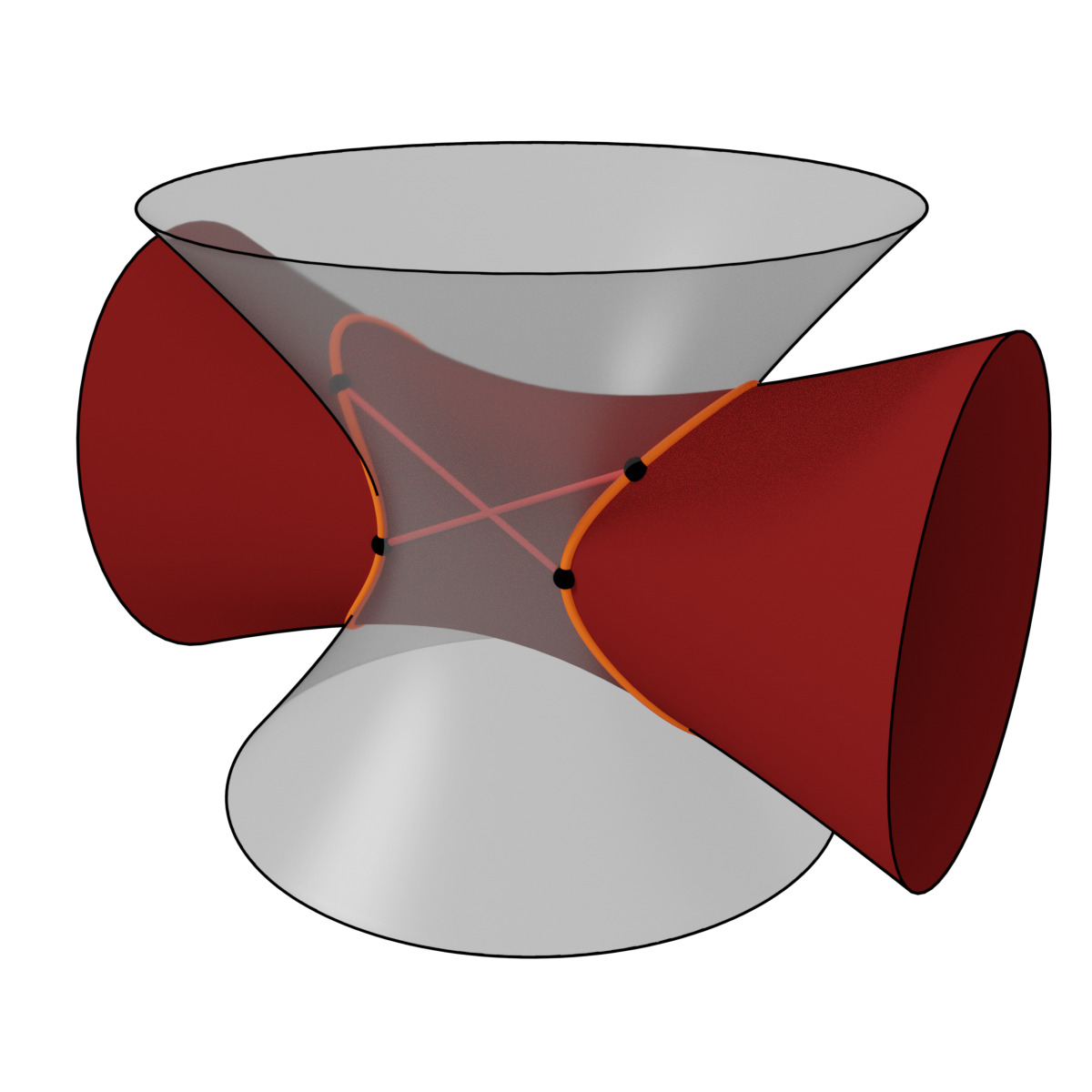}
    \definecolor{darkred}{HTML}{831b1a}
    \put(90,20){$\color{darkred}\lag$}
    \put(87,80){$\mathcal{C}$}
    \put(15,64){\color{white}$\p{v}_+(u)$}
    \put(11,49){\color{white}$\p{v}_+(\tilde u + s)$}
    \put(61,56){\color{white}$\p{v}_-(\tilde u)$}
    \put(55,46){\color{white}$\p{v}_-(u + s)$}
  \end{overpic}\qquad
  \raisebox{0.8cm}{
    \begin{overpic}[width=0.3\textwidth]{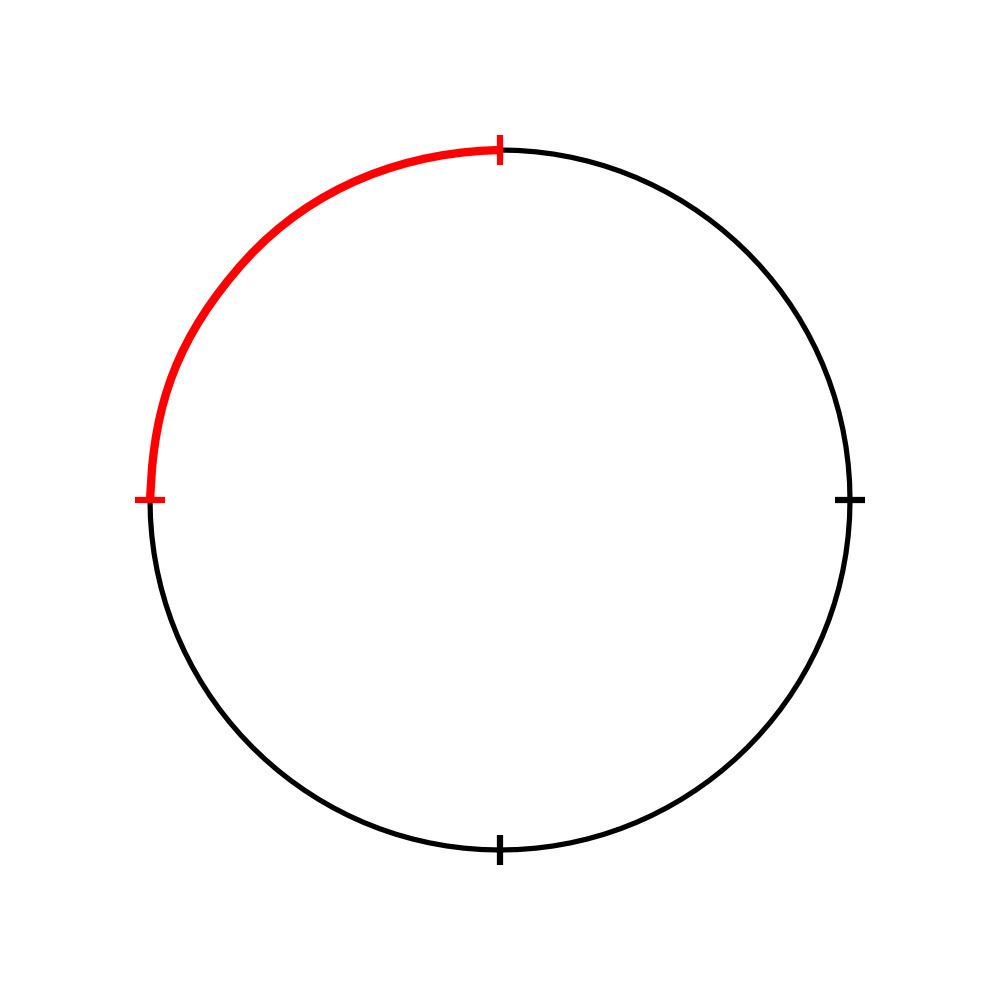}
      \put(47,47){\Large$\lambda$}
      \put(-5,48){$\mathcal{C}, \infty$}
      \put(89,48){$\varepsilon$}
      \put(42,92){$-\frac{1}{\alpha^2}$}
      \put(47,5){$\frac{1}{\beta^2}$}
      \put(79,77){\tiny$(+++-)$}
      \put(-3,77){\tiny$(-++-)$}
      \put(-3,20){\tiny$(+---)$}
      \put(79,20){\tiny$(++--)$}
    \end{overpic}
  }
  \caption{
    \emph{Left:}
    Four coplanar points on a hypercycle base curve $\lag \cap \mathcal{C}$.
    The two lines are rulings from a common hyperboloid in the pencil corresponding to $\lag \cap \mathcal{C}$.
    \emph{Right:}
    The parameter $\lambda$ for the pencil $\lag \wedge \mathcal{C}$ as given by \eqref{eq:pencil-hyperbola}.
    The four values $-\frac{1}{\alpha^2}, \varepsilon, \frac{1}{\beta^2}, \infty$ correspond to the degenerate quadrics in the pencil.
    In between, the signature of the quadrics from the pencil are given.
    The function \eqref{eq:lambda-hyperbola} takes values in $[-\infty, -\tfrac{1}{\alpha^2}]$ and corresponds to hyperboloids whose rulings intersect both components of the base curve.
  }
  \label{fig:planarity-and-pencil-hyperbola}
\end{figure}
\begin{proposition}\
  \nobreakpar
  \begin{enumerate}
  \item
    Let $u, \tilde{u}, s \in \R$.
    Then the four points $\p{v}_+(u), \p{v}_-(u+s), \p{v}_-(\tilde{u}), \p{v}_+(\tilde{u} + s)$ are coplanar (see Figure \ref{fig:planarity-and-pencil-hyperbola}, \textit{left}).
  \item
    Let $s \in \R$.
    Then the lines $\p{v}_+(u)\wedge\p{v}_-(u+s)$ with $u \in \R$
    constitute one family of rulings of a common hyperboloid in the pencil $\lag \wedge \mathcal{C}$
    \begin{equation}
      \label{eq:pencil-hyperbola}
      (1 + \lambda\alpha^2)x_1^2 + (1 - \lambda\beta^2)x_2^2 + (\varepsilon - \lambda)x_3^2 - x_4^2 = 0
    \end{equation}
    given by
    \begin{equation}
      \label{eq:lambda-hyperbola}
      \lambda(s) = -\frac{1}{\beta^2}\jac{cs}^2(\tfrac{s}{2},k) - \frac{1}{\alpha^2}\jac{ns}^2(\tfrac{s}{2}, k),\qquad
      \text{where}~ \jac{cs} = \frac{\jac{cn}}{\jac{sn}},\ \jac{ns} = \frac{\jac{1}}{\jac{sn}}.
    \end{equation}
    The second family of rulings is given by the lines $\p{v}_+(u)\wedge\p{v}_-(u-s)$ with $u \in \R$.
  \end{enumerate}
\end{proposition}
\begin{remark}\
  \nobreakpar
  \begin{enumerate}
  \item
    The alternative expression for $\lambda$ in terms of $\hat{s} = \sqrt{1-k^2}s$ and $\hat{k}^2 = \frac{k^2}{k^2-1}$ is given by
    \[
      \lambda(\hat{s}) = - \varepsilon \jac{cs}^2(\tfrac{\hat{s}}{2},\hat{k}) - \frac{1}{\alpha^2}\jac{ns}^2(\tfrac{\hat{s}}{2}, \hat{k}).
    \]
  \item
    The four degenerate quadrics from the pencil \eqref{eq:pencil-hyperbola} are given by
    the values $\lambda = -\frac{1}{\alpha^2}, \varepsilon, \frac{1}{\beta^2}, \infty$,
    where $\lambda = \infty$ corresponds to the cone $\mathcal{C}$ (see Figure \ref{fig:planarity-and-pencil-hyperbola}, right).
    By construction, the hyperboloids obtained by \eqref{eq:lambda-hyperbola} have rulings connecting the two components of the base curve,
    which corresponds to the fact that
    \[
      -\infty \leq \lambda(s) \leq -\frac{1}{\alpha^2}
    \]
    for $s \in \R$ with $\lambda(0) = \infty$ and $\lambda(2\KK(k)) = -\frac{1}{\alpha^2}$.
  \end{enumerate}
\end{remark}
This allows to parametrize a checkerboard incircular net
tangent to a given hyperbola in the following way (see Figure~\ref{fig:periodic-cbic-net-hyperbolic-hyperbola}).
\begin{theorem}
  \label{thm:cbic-formulas-hyperbola}
  Let $\varepsilon \in \{-1, 0\}$.
  Then for $\alpha, \beta > 0$ with $1 + \varepsilon\alpha^2, 1 - \varepsilon\beta^2 > 0$,
  $s, \tilde{s} \in \R$, and $u_0^{\ell}, u_0^{m} \in \R$
  the two families of lines $\left(\ell_i\right)_{i\in\Z}$ and $\left(m_j\right)_{j\in\Z}$ given by
  \[
    \begin{aligned}
      \ell_{2k} &= \p{v}_+(u_0^{\ell} + k(s + \tilde{s}))\\
      \ell_{2k+1} &= \p{v}_-(u_0^{\ell} + k(s + \tilde{s}) + s)\\
      m_{2l} &= \p{v}_-(u_0^{m} + l(s + \tilde{s}))\\
      m_{2l+1} &= \p{v}_+(u_0^{m} + l(s + \tilde{s}) + s)
    \end{aligned}
  \]
  constitute a hyperbolic/Euclidean checkerboard incircular net (according to the value of~$\varepsilon$)
  tangent to the hyperbola
  \[
    \frac{x_1^2}{\alpha^2} - \frac{x_2^2}{\beta^2} - x_3^2 = 0.
  \]
\end{theorem}
\begin{remark}
  Periodicity and ``embeddedness'' are achieved as described in Section \ref{sec:ellipse}.
\end{remark}
\begin{figure}[H]
  \centering
  \hspace{\fill}
  \includegraphics[width=0.38\textwidth]{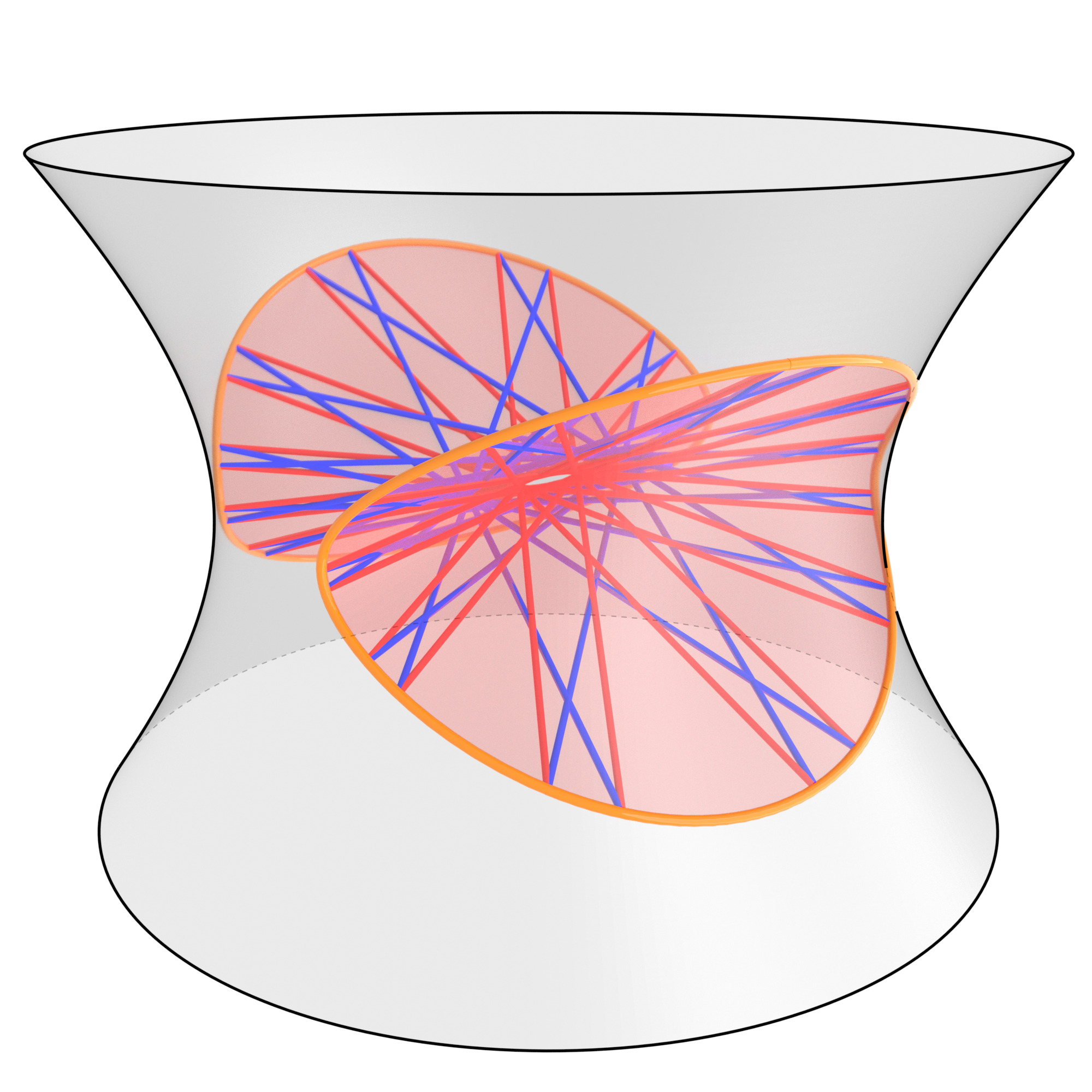}
  \hspace{\fill}
  \raisebox{0.01\textwidth}{
    \includegraphics[width=0.36\textwidth]{hyp_cic_hyperbola_stereog_projection_1}
  }
  \hspace{\fill}
  \newline
  \mbox{}
  \hspace{\fill}
  \includegraphics[width=0.38\textwidth]{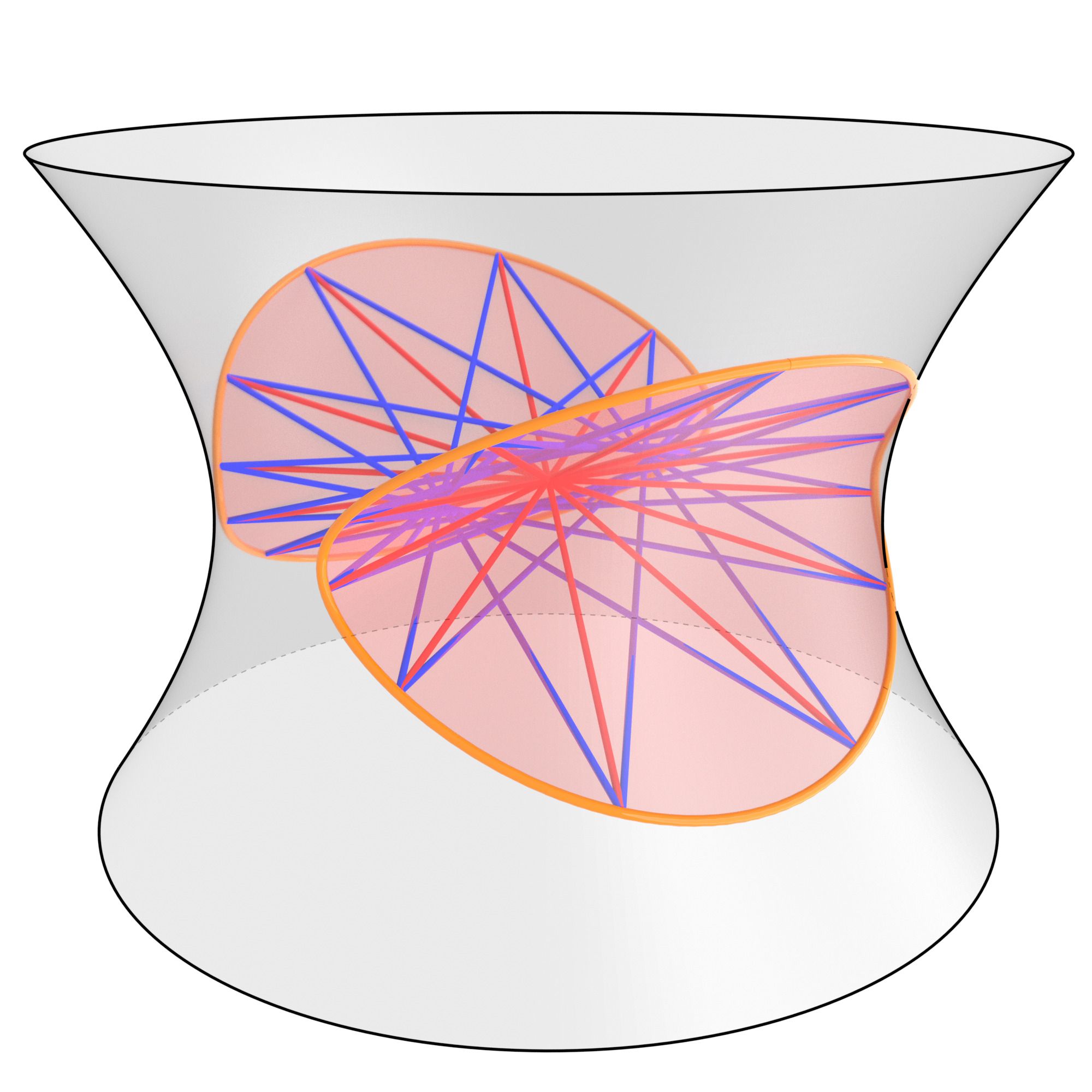}
  \hspace{\fill}
  \raisebox{0.01\textwidth}{
    \includegraphics[width=0.36\textwidth]{hyp_ic_hyperbola_stereog_projection_1}
  }
  \hspace{\fill}
  \caption{
    Periodic checkerboard incircular net in hyperbolic geometry ($\varepsilon=-1$)
    tangent to a hyperbola with $\alpha = 0.5$, $\beta = 0.25$, $N = 11$, and $s=0.23$ (\emph{top}) / $s = 0$ (\emph{bottom}).
    Represented on the Laguerre quadric (\emph{left}) and on the Poincaré disk model of hyperbolic geometry (\emph{right}).
  }
\label{fig:periodic-cbic-net-hyperbolic-hyperbola}
\end{figure}

\subsection{Checkerboard incircular nets as octahedral grids}
\label{sec:octahedral-grids}
According to Remark \ref{rem:laguerre-all-circles} a checkerboard incircular net
possesses more incircles than immediate from its definition.
The full symmetry of such a net is revealed when considering all these circles
and dividing its two families of lines $\left(\ell_i\right)_{i\in\Z}$, $\left(m_j\right)_{j\in\Z}$ into four families
\[
    \nu^{(1)}_{k_1} = \ell_{2k_1},\quad
    \nu^{(2)}_{k_2} = \ell_{-2k_2+1},\quad
    \nu^{(3)}_{k_3} = m_{2k_3},\quad
    \nu^{(4)}_{k_4} = m_{-2k_4+1},
\]
for $k_1, k_2, k_3, k_4 \in \Z$.
\begin{remark}
  The decomposition into four families of lines also seems natural after considering the formulas for checkerboard incircular nets given in Theorems \ref{thm:cbic-formulas-ellipse} and \ref{thm:cbic-formulas-hyperbola}, and more basically the identity \eqref{eq:Jacobi-identity}.
\end{remark}
\begin{proposition}
  For a checkerboard incircular net each quadrilateral $(\nu^{(1)}_{k_1}, \nu^{(2)}_{k_2}, \nu^{(3)}_{k_3}, \nu^{(4)}_{k_4})$ with
  \begin{equation}
    \label{eq:A3-root-system}
    k_1 + k_2 + k_3 + k_4 = 0,\qquad
    k_1, k_2, k_3, k_4 \in \Z
  \end{equation}
  is circumscribed.
\end{proposition}
\begin{proof}
  This is a reformulation of the statement given in Remark \ref{rem:laguerre-all-circles},
  which describes the whole collection of incircles of a checkerboard incircular net.
\end{proof}
From \eqref{eq:A3-root-system} we find that the collection of incircles of a checkerboard incircular net
is naturally assigned to the points of an $A_3$ root-system (vertices of a \emph{tetrahedral-octahedral honeycomb lattice}, see Figure \ref{fig:octahedral-grid}, left), where
\[
  A_3 = \set{(k_1, k_2, k_3, k_4) \in \Z^4}{k_1 + k_2 + k_3 + k_4 = 0}.
\]
This correspondence can also be made geometric. To this end we identifying the four families of lines $(\nu^{(i)}_{k_i})_{k_i\in\Z}$, $i = 1,2,3,4$
with its corresponding points on the Laguerre quadric $\lag$ and denote its polar planes by $P^{(i)}_{k_i} = (\nu^{(i)}_{k_i})^\perp$
(or, in the Euclidean case, its dual planes by $P^{(i)}_{k_i} = (\nu^{(i)}_{k_i})^*$ in the cyclographic model).
\begin{figure}
  \centering
  \raisebox{.25\height}{
    \includegraphics[width=0.45\textwidth]{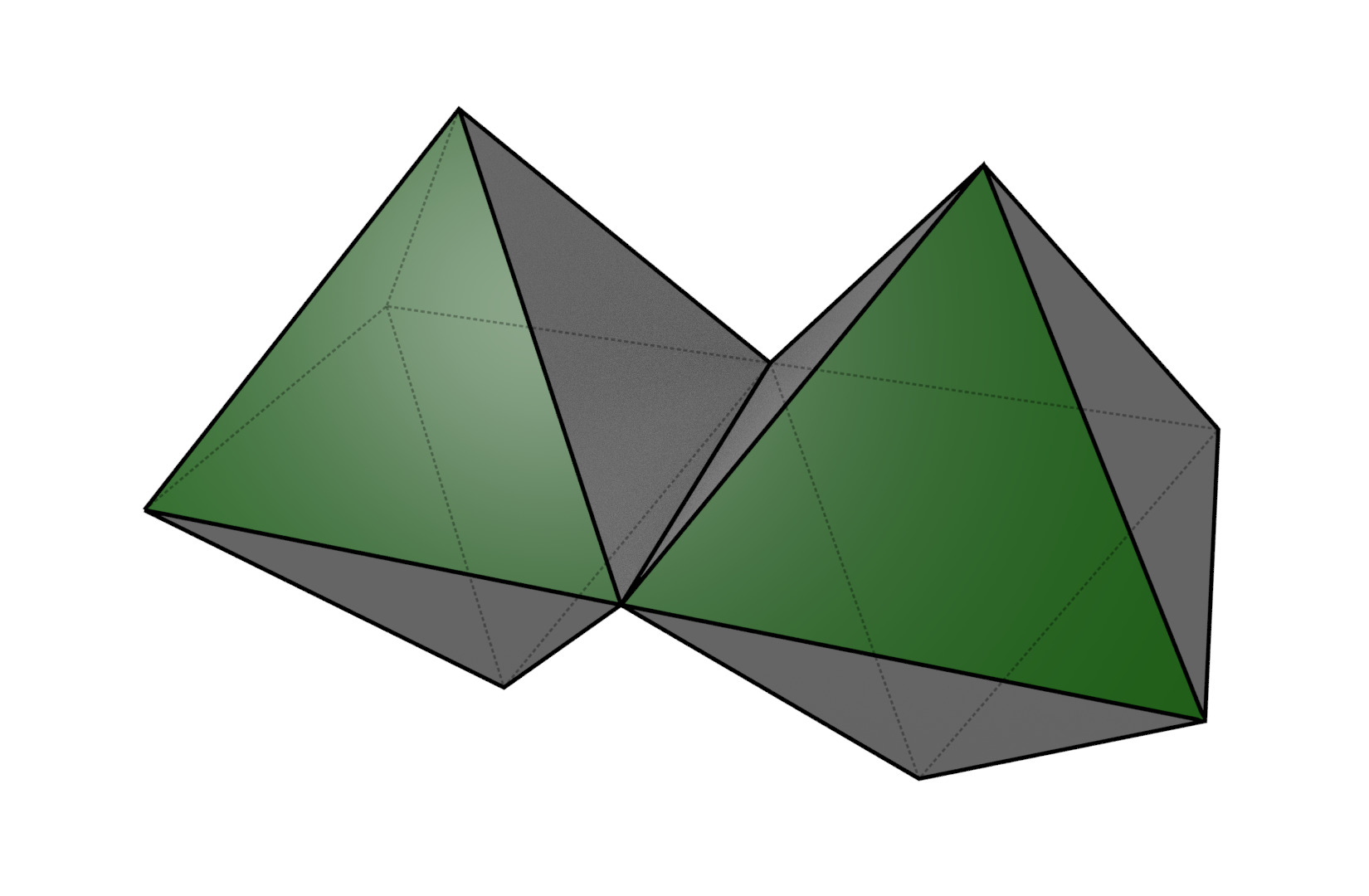}
  }
  \includegraphics[width=0.49\textwidth]{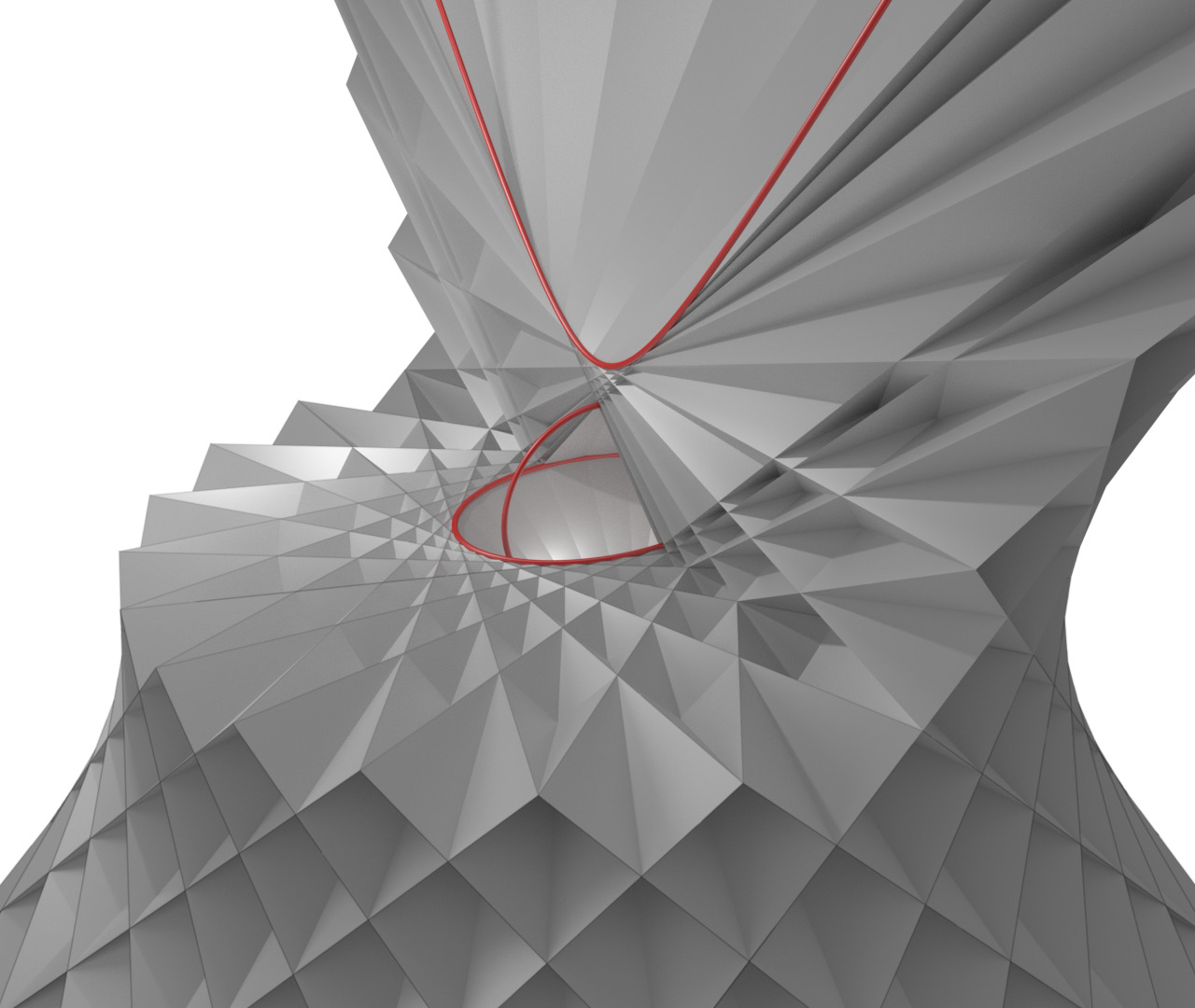}
  \caption{
    \emph{Left:} Two adjacent octahedra from an octahedral grid of planes.
    These correspond to the geometric configuration shown in Figure \ref{fig:lie-cbic-dimension}.
    \emph{Right:} Octahedral grid of planes corresponding to an Euclidean incircular net in the cyclographic model. All planes are tangent to the red conics, which are the degenerate dual quadrics in a dual pencil.
  }
\label{fig:octahedral-grid}
\end{figure}
\begin{proposition}
  \label{prop:octahedral-grid}
  The four families of planes $(P^{(i)}_{k_i})_{k_i\in\Z}$, $i = 1,2,3,4$ corresponding to a checkerboard incircular net
  constitute an \emph{octahedral grid of planes}, i.e.,\ for each
  \[
    k_1 + k_2 + k_3 + k_4 = 0,\qquad
    k_1, k_2, k_3, k_4 \in \Z
  \]
  the four planes $P^{(1)}_{k_1}$, $P^{(2)}_{k_2}$, $P^{(3)}_{k_3}$, $P^{(4)}_{k_4}$ intersect in a point (see Figure \ref{fig:octahedral-grid}, right, and cf.\ \cite{ABST}).
\end{proposition}
\begin{remark}
  Generally, octahedral grids of planes have the property that all its planes are tangent to all quadrics of a dual pencil, or equivalently, to a certain developable surface (cf.\ \cite{Blacht, S}).
  In the case of checkerboard incircular nets this property is polar (or dual) to the property,
  that all the points $\nu^{(i)}_{k_i}$ lie on the hypercycle base curve.
  For the Euclidean case of incircular nets this fact was already employed by Böhm in \cite{B}.
\end{remark}

Denote the intersection points of the octahedral grid of planes by
\[
  c_a = P^{(1)}_{k_1} \cap P^{(2)}_{k_2} \cap P^{(3)}_{k_3} \cap P^{(4)}_{k_4},\qquad
  a = (k_1, k_2, k_3, k_4) \in A_3.
\]
By polarity (or duality) the points $c_a$ correspond to the incircles of the checkerboard incircular net.
We may now extend the statement from Corollary \ref{cor:cbic-dual-pencil} to all ``diagonal surfaces'' of $A_3$.
\begin{proposition}
  \label{prop:circles-on-dual-pencil}
  For an octahedral grid of planes corresponding to a checkerboard incircular net,
  the points of intersection $c_{a_1}, c_{a_2}, c_{a_3}, c_{a_4}$ with
  \[
    a_1 + a_2 + a_3 + a_4 = 0,\qquad
    a_1, a_2, a_3, a_4 \in A_3
  \]
  lie on a quadric from the dual pencil of quadrics which is polar (or dual in the Euclidean case) to the pencil of quadrics corresponding to the hypercycle base curve.
\end{proposition}
\begin{remark}
  In the case of an ``ordinary'' incircular net this implies that the intersection points of its lines lie on conics which are confocal with the touching conic \cite{B}, \cite{AB}.
\end{remark}

%%% Local Variables:
%%% mode: latex
%%% TeX-master: "main"
%%% End:

\newpage

\appendix
\section{Euclidean cases}
\label{sec:euclidean}

The cases of Euclidean geometry and Euclidean Laguerre geometry, which we have excluded from our general discussion, are induced by degenerate quadrics, see, e.g., \cite{K, Bl2, Gie}.
For a degenerate quadric $\quadric \subset \RP^n$, polarity (see Section \ref{sec:polarity}) does no longer define a bijection between the set of points and the set of hyperplanes.
Instead one can apply the concept of \emph{duality}.

\subsection{Duality}
The $n$-dimensional \emph{dual real projective space} is given by
\[
  (\RP^n)^* \coloneqq \P\left((\R^{n+1})^*\right),
\]
where $\left(\R^{n+1}\right)^*$ is the space of linear functionals on $\R^{n+1}$.
We identify $(\RP^n)^{**} = \RP^n$ in the canonical way,
and obtain a bijection between projective subspaces $\p{U} = \P(U) \subset \RP^n$
and their \emph{dual projective subspaces}
\[
  \p{U}^* \coloneqq \set{\p{y} \in (\RP^n)^*}{y(x) = 0 ~\text{for all}~ x \in U},
\]
satisfying
\[
  \dim \p{U} + \dim \p{U}^* = n-1.
\]
Every projective transformation $f:\RP^n\rightarrow\RP^n \in \PGL(n+1)$
induces a \emph{dual projective transformation} $f^*:(\RP^n)^* \rightarrow (\RP^n)^* \in \PGL(n+1)^*$ such that
\[
  f(\p{U})^* = f^*(\p{U}^*)
\]
for every projective subspace $\p{U} \subset \RP^n$.
Introduce a basis on $\R^{n+1}$, say the conical basis, and its dual basis on $(\R^{n+1})^*$.
Then, if $F \in \R^{(n+1)\times(n+1)}$ is a matrix representing the transformation $f = [F]$,
a matrix $F^* \in \R^{(n+1)\times(n+1)}$ representing the dual transformation $f^* = [F^*]$ is given by
\begin{equation}
  \label{eq:dual-matrix}
  F^* \coloneqq \invtranspose{F}.
\end{equation}

For a quadric $\quadric \subset \RP^n$ its \emph{dual quadric} $\quadric^* \subset (\RP^n)^*$
may be defined as the set of points dual to the tangent hyperplanes of $\quadric$.
\begin{example}\
  \label{ex:dual-quadrics}
  \nobreakpar
  \begin{enumerate}
  \item
    For a non-degenerate quadric $\quadric \subset \RP^n$ of signature $(r,s)$ its dual quadric $\quadric^* \subset (\RP^n)^*$ is non-degenerate with the same signature.
  \item
    \label{ex:dual-cone}
    For a cone $\quadric \subset \RP^n$ of signature $(r,s,1)$ with vertex $\p{v} \in \quadric$,
    its dual quadric $\quadric^* \subset (\RP^n)^*$ consists of the set of points on a lower dimensional quadric of signature $(r,s)$
    contained in the hyperplane $\p{v}^* \subset (\RP^n)^*$.
  \end{enumerate}
\end{example}

\subsection{Euclidean geometry}
\label{sec:euclidean-space}
Let $\scalarprod{\cdot}{\cdot}$ be the standard degenerate bilinear form of signature $(n,0,1)$, i.e.
\[
  \scalarprod{x}{y} \coloneqq x_1y_1 + \ldots + x_ny_n
\]
for $x, y \in \R^{n+1}$.
The corresponding quadric $\euclb$ is an imaginary cone (cf.\ Example \ref{ex:quadrics} \ref{ex:quadrics-cone}).
Its real part consisting only of one point, the vertex of the cone:
\[
  \p{m}_\infty \in \RP^n, \qquad m_\infty \coloneqq e_{n+1} = (0, \ldots, 0, 1).
\]
While the set $\euclb^- = \varnothing$ is empty, the set
\[
  \eucl^* \coloneqq \euclb^+ = \RP^n \setminus \{\p{m}_\infty\}
\]
consists of the whole projective space but one point,
which we identify with the $n$-dimensional \emph{dual Euclidean space}, i.e., the space of Euclidean hyperplanes.

While in the projective models of hyperbolic/elliptic geometry,
we were able to identify certain points with hyperplanes in the same projective space by polarity,
this is not possible in the projective model of Euclidean geometry due to the degeneracy of the absolute quadric $\euclb$.
Instead, by duality, every point $\p{m} \in \eucl^*$ corresponds to a hyperplane $\p{m}^* \subset \eucl$ in
\[
  \eucl \coloneqq \left(\RP^n\right)^* \setminus (\p{m}_\infty)^* \simeq \R^n,
\]
which we identify with the $n$-dimensional \emph{Euclidean space}.
The hyperplane $(\p{m}_\infty)^*$ is called the \emph{hyperplane at infinity}.

For two points $\p{k}, \p{m} \in \eucl^*$ one always has $0 \leq \ck{\euclb}{\p{k}}{\p{m}} \leq 1$,
and the Euclidean angle $\alpha$, or equivalently its conjugate angle $\pi - \alpha$,
between the two hyperplanes $\p{k}^*, \p{m}^* \subset \eucl$ is given by
\[
  \ck{\euclb}{\p{k}}{\p{m}} = \cos^2 \alpha(\p{k}^*, \p{m}^*).
\]
The two planes are \emph{parallel} if the line $\p{k} \wedge \p{m}$ contains the point $\p{m}_\infty$.

The dual quadric $\euclb^*$ of the absolute cone can be identified with an imaginary quadric in the hyperplane at infinity $(\p{m}_\infty)^*$ of signature $(n,0)$ (cf.\ Example~\ref{ex:dual-quadrics}~\ref{ex:dual-cone}).
Since this does not induce a bilinear form on $(\RP^n)^*$, the Cayley-Klein distance is not well-defined on $\eucl$.
Yet the Euclidean distance may still be recovered in this setup,
e.g., as the limit of the Cayley-Klein distance of hyperbolic/elliptic space \cite{K, G}.
One may avoid these difficulties by treating Euclidean geometry
as a subgeometry of Möbius geometry (see Section \ref{sec:mobius-euclidean}).

We employ the following normalization for the dual Euclidean space
\[
  (\E^n)^* \coloneqq \set{m\in\R^{n+1}}{\scalarprod{m}{m}=1}
  = \set{(\widehat{m},-d)\in\R^{n+1}}{\widehat{m}\in\R^n,~d\in\R,~\dotprod{\widehat{m}}{\widehat{m}} = 1},
\]
where $\dotprod{\widehat{m}}{\widehat{m}}$ denotes the standard scalar product on $\R^n$.
Upon the (non-canonical) identification $(\R^{n+1})^* \simeq \R^{n+1}$, by identifying the canonical basis of $(\R^{n+1})^*$ with the dual basis of the canonical basis of $\R^{n+1}$, we introduce the following normalization for the Euclidean space.
\[
  \E^n \coloneqq \set{x \in (\R^{n+1})^*}{x(m_\infty) = 1}
  \simeq \set{(\widehat{x}, 1)\in\R^{n+1}}{\widehat{x}\in\R^n}.
\]
Then $\P(\E^n) = \eucl$ is an embedding and $\P((\E^n)^*) = \eucl^*$ a double cover.
The double cover may be used to encode the orientation of the corresponding Euclidean plane.

In this normalization the Euclidean distance of two points $\p{x}, \p{y} \in \eucl$, $x, y \in \E^n$ is given by
\[
  \abs{x - y} = d(\p{x}, \p{y}).
\]
The Euclidean hyperplane corresponding to a point $\p{m} \in \eucl^*$, $m = (\widehat{m},-d) \in (\E^n)^*$ is given by
\[
  \set{\p{x} \in \eucl}{\scalarprod{m}{x} = 0}
  = \P\left( \set{(\widehat{x},1)\in\E^n}{\dotprod{\widehat{m}}{\widehat{x}}} = d \right),
\]
while the formula for the angle between two Euclidean planes $\p{k}\in \eucl^*$, $k = (\widehat{k}, -c) \in (\E^n)^*$ and $\p{m} \in \eucl^*$, $m = (\widehat{m},-d) \in (\E^n)^*$ becomes
\[
  \scalarprod{k}{m} = \dotprod{\widehat{k}}{\widehat{m}} = \cos\alpha(\p{k}^*, \p{m}^*),
\]
where the intersection angle and its conjugate angle can be distinguished now.
Finally, the signed distance of a point $\p{x} \in \eucl$, $x = (\widehat{x},1) \in \E^n$ and a plane $\p{m} \in \eucl^*$, $(\widehat{m}, -d) \in (\E^n)^*$ is given by
\[
  \scalarprod{m}{x} = \dotprod{\widehat{m}}{\widehat{x}} - d = d(\p{x}, \p{m}^*)
\]

The transformation group induced by the absolute quadric $\euclb$ on the dual Euclidean space $\eucl^*$ is given by $\PO(n,0,1)$.
Its elements are of the form
\[
  [A] =
  \left[
    \begin{array}{c|c}
      \widehat{A} & 0\\
      \hline
      \transpose{\widehat{a}} & \varepsilon
    \end{array}
  \right] \in \PO(n,0,1),
\]
where $\widehat{A} \in O(n), \widehat{a} \in \R^n, \varepsilon \neq 0$.
Thus, its dual transformations, see \eqref{eq:dual-matrix}, are given by
\[
  [\invtranspose{A}] =
  \left[
    \begin{array}{c|c}
      \widehat{A} & -\widehat{A}\widehat{a}\\
      \hline
      0 & \varepsilon^{-1}
    \end{array}
  \right] \in \PO(n,0,1)^*.
\]
They act on $\eucl$ as the group of \emph{similarity transformations}, i.e., Euclidean motions and scalings.

\subsection{Euclidean geometry from Möbius geometry}
\label{sec:mobius-euclidean}
\begin{figure}
  \centering
  \begin{overpic}[width=0.7\textwidth]{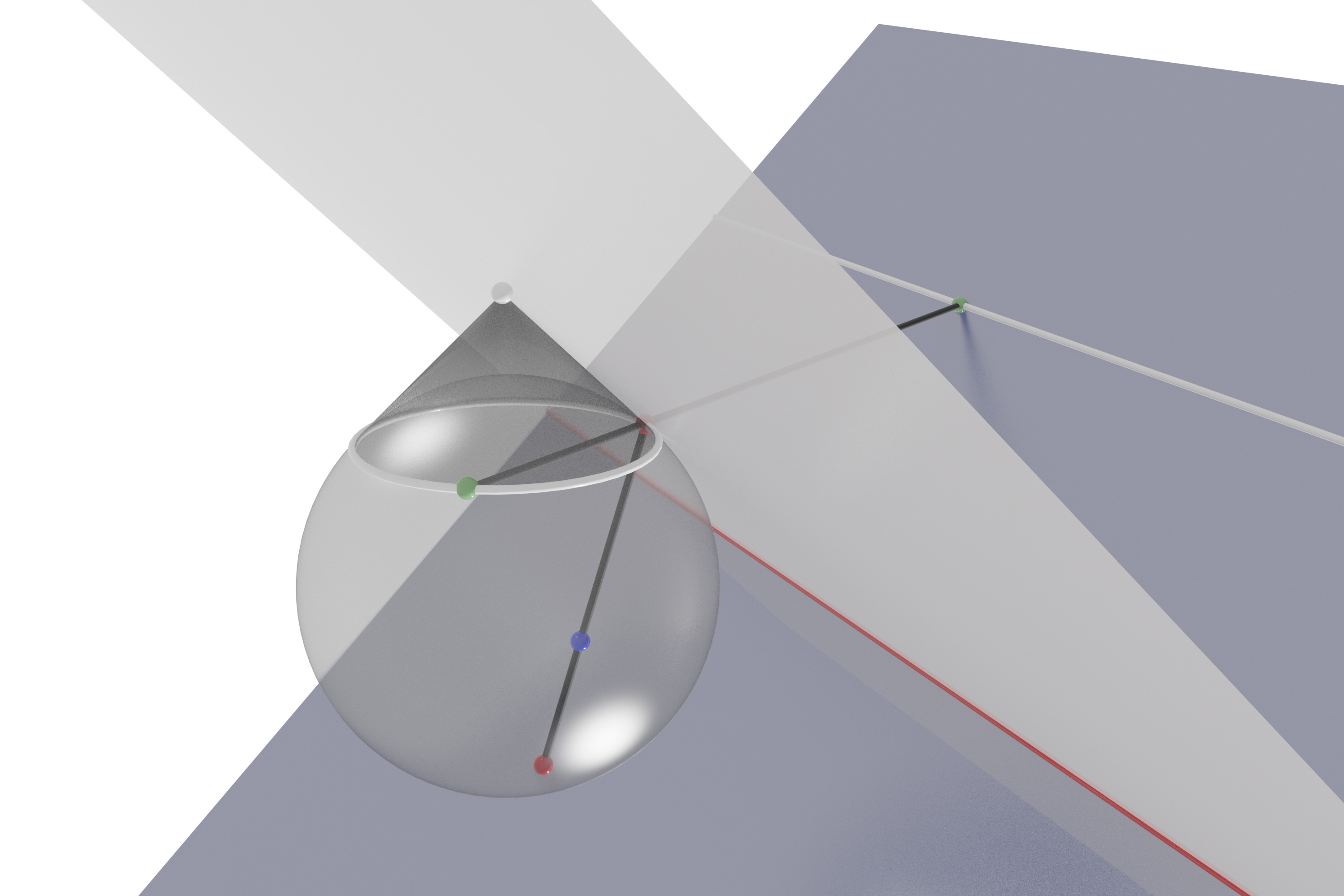}
    \put(19,18){$\mathcal{S}$}
    \put(65,59){$\baseplane = \p{b}^\perp$}
    \put(13,63){$\p{q}^\perp$}
    \put(47,37){$\p{q} = [e_\infty]$}
    \put(37,46.5){$\p{k}$}
    \put(40,18){$\p{b}$}
    \put(38,5){$[e_0]$}
    \put(25,26.5){$\sigma_{\p{q},\p{b}}(\p{x})$}
    \put(72,45){$\p{x}$}
    \put(70,11){$\ell_\infty$}
  \end{overpic}
  \caption{
    Stereographic projection from $\baseplane$ to $\mob$ through the point $\p{q}$.
    Every point $\p{k} \in \p{q}^\perp$ corresponds to a hyperplane in $\baseplane$.
  }
\label{fig:stereographic-projection}
\end{figure}
In Section \ref{sec:projection} we have excluded the choice of a point $\p{q} \in \quadric$ on the quadric,
since the projection $\pi_{\p{q}}$ (see Definition \ref{def:involution-projection}) to the polar hyperplane $\p{q}^\perp$ is not well-defined in that case.
Yet most of the constructions described still apply if we project to any other hyperplane instead.
We show this in the example of recovering Euclidean (similarity) geometry from Möbius geometry.

Thus, let $\scalarprod{\cdot}{\cdot}$ be the standard non-degenerate bilinear form of signature $(n+1,1)$, i.e.
\[
  \scalarprod{x}{y} \coloneqq x_1y_1 + \ldots + x_{n+1}y_{n+1} - x_{n+2}y_{n+2}
\]
for $x, y \in \R^{n+2}$, and denote by $\mob \subset \RP^{n+1}$ the corresponding \emph{Möbius quadric}.
Let $\p{q} \in \mob$ be a point on the Möbius quadric, w.l.o.g.,
\[
  \p{q} \coloneqq [e_{\infty}], \qquad
  e_{\infty} \coloneqq \tfrac{1}{2} \left( e_{n+1}+e_{n+2} \right) = (0, \ldots, 0, \tfrac{1}{2}, \tfrac{1}{2}).
\]
While $\p{q}^\perp$ is the tangent plane of $\mob$ at $\p{q}$,
we choose a different plane $\baseplane$ for the projection, w.l.o.g.,
\[
  \baseplane \coloneqq \p{b}^\perp, \qquad
  b \coloneqq e_{n+1} = (0, \ldots, 0, 1, 0),
\]
and consider the central projection from $\baseplane$ to $\mob$ through the point $\p{q}$,
which is also called \emph{stereographic projection} (see Figure \ref{fig:stereographic-projection}).
To this end, denote by $[e_0]$ the intersection point of the line $\p{q} \wedge \p{b}$ with $\mob$, where
\[
  e_0 \coloneqq \tfrac{1}{2} \left( e_{n+2}-e_{n+1} \right) = (0, \ldots, 0, -\tfrac{1}{2}, \tfrac{1}{2}).
\]
Then we have
\[
  \liesc{e_0}{e_0} = \liesc{e_\infty}{e_\infty} = 0, \quad
  \liesc{e_0}{e_\infty} = -\tfrac{1}{2},
\]
and $\liesc{e_0}{e_i} = \liesc{e_\infty}{e_i} = 0$ for $i = 1, \ldots, n$,
and the vectors $e_1, \ldots, e_n, e_0, e_{\infty}$ constitute a basis of $\R^{n+1,1}$.
\begin{proposition}
  Let $\ell_\infty \coloneqq \baseplane \cap \p{q}^\perp$.
  The stereographic projection from $\baseplane \setminus \ell_\infty$ to $\mob \setminus \p{q}$ through the point $\p{q}$ is given by the map
  \[
    \sigma_{\p{q}, \p{b}} : \p{x} = [\tilde{x} + e_0 - e_\infty] \mapsto [\tilde{x} + e_0 + \abs{\tilde{x}}^2e_\infty],
  \]
  where $\tilde{x} \in \Span\{e_1, \ldots, e_n\}$.
\end{proposition}
\begin{proof}
  First note that a point in $\p{x} \in \baseplane\setminus\ell_\infty$ may be normalized to $x = \tilde{x} + e_0 - e_\infty$
  The (second) intersection point of the line $\p{q} \wedge \p{x}$ with $\quadric$ is then given by
  \[
    -2\scalarprod{x}{e_\infty}x + \scalarprod{x}{x}e_\infty
    = x + (\abs{\tilde{x}}^2 - 1)e_\infty
    = \tilde{x} + e_0 + \abs{\tilde{x}}^2e_\infty.
  \]
\end{proof}
Now the Euclidean metric on $\baseplane$ may be recovered from the bilinear form corresponding to $\mob$
by observing that
\[
  \scalarprod{x}{y}
  = \scalarprod{\tilde{x} + e_0 + \abs{\tilde{y}}^2 e_\infty}{\tilde{y} + e_0 + \abs{\tilde{y}}^2 e_\infty}
  = -\frac{1}{2}\abs{\tilde{x} - \tilde{y}}.
\]
\begin{remark}
  To obtain the Euclidean metric in a projectively well-defined way one can start by considering the quantity
  \[
    \frac{\scalarprod{x}{y}}{\scalarprod{e_\infty}{x}\scalarprod{e_\infty}{y}},
  \]
  similar to Definition \ref{def:sphere-complex-distance}.
  Though not being invariant under different choices of homogeneous coordinate vectors for the point $\p{q} = [e_\infty]$,
  the quotient of two such expressions is.
  This fits the fact that it is not actually Euclidean geometry that we are recovering
  but similarity geometry.
\end{remark}

The restriction of the Möbius quadric $\mob$ to the tangent hyperplane $\p{q}^\perp$ yields a quadric of signature $(n,0,1)$.
Thus, we can identify the tangent hyperplane with the dual Euclidean space (see Section \ref{sec:euclidean-space}).
Indeed, by polarity in the Möbius quadric $\mob$, every point $\p{k} \in \p{q}^\perp$
corresponds to a hyperplanar section of $\mob$ containing the point $\p{q}$, i.e., an $\mob$-sphere through $\p{q}$,
which is, in turn, mapped to a hyperplane of $\baseplane$ by stereographic projection.
The Cayley-Klein distance of two points in the tangent hyperplane yields the Euclidean angle
between the two corresponding hyperplanes of $\baseplane$.
The group of Möbius transformations fixing the point $\p{q}$ induces the group of dual similarity transformations on $\baseplane$
\[
  \mobtrafos_{\p{q}} = \PO(n+1,1)_{\p{q}} \simeq \PO(n,0,1).
\]

\subsection{Euclidean Laguerre geometry}
\label{sec:euclidean-laguerre-geometry}
\begin{figure}
  \centering
  \begin{overpic}[width=0.46\textwidth]{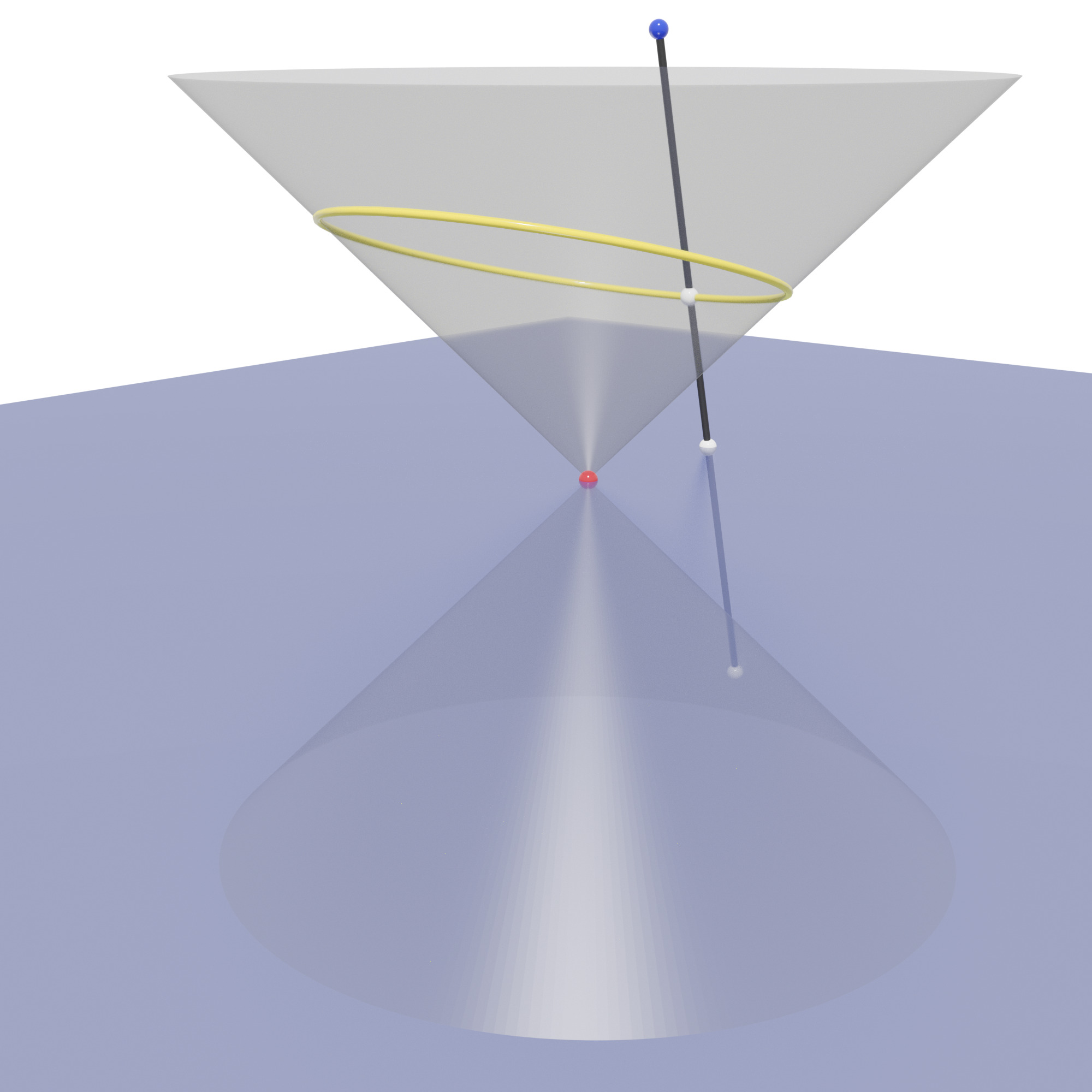}
    \put(0,80){$\RP^n$}
    \put(3,55){$\eucl^*$}
    \put(20,95){$\lageucl$}
    \put(59,100){$\p{p}$}
    \put(66,72){$\p{x}$}
    \put(67,58){$\pi_{\p{p}}(\p{x})$}
    \put(70,37){$\sigma_{\p{p}}(\p{x})$}
    \put(45,53){$\p{m}_\infty$}
    \put(36,83){$G(c)$}
  \end{overpic}
  \hspace{0.2cm}
  \begin{overpic}[width=0.46\textwidth]{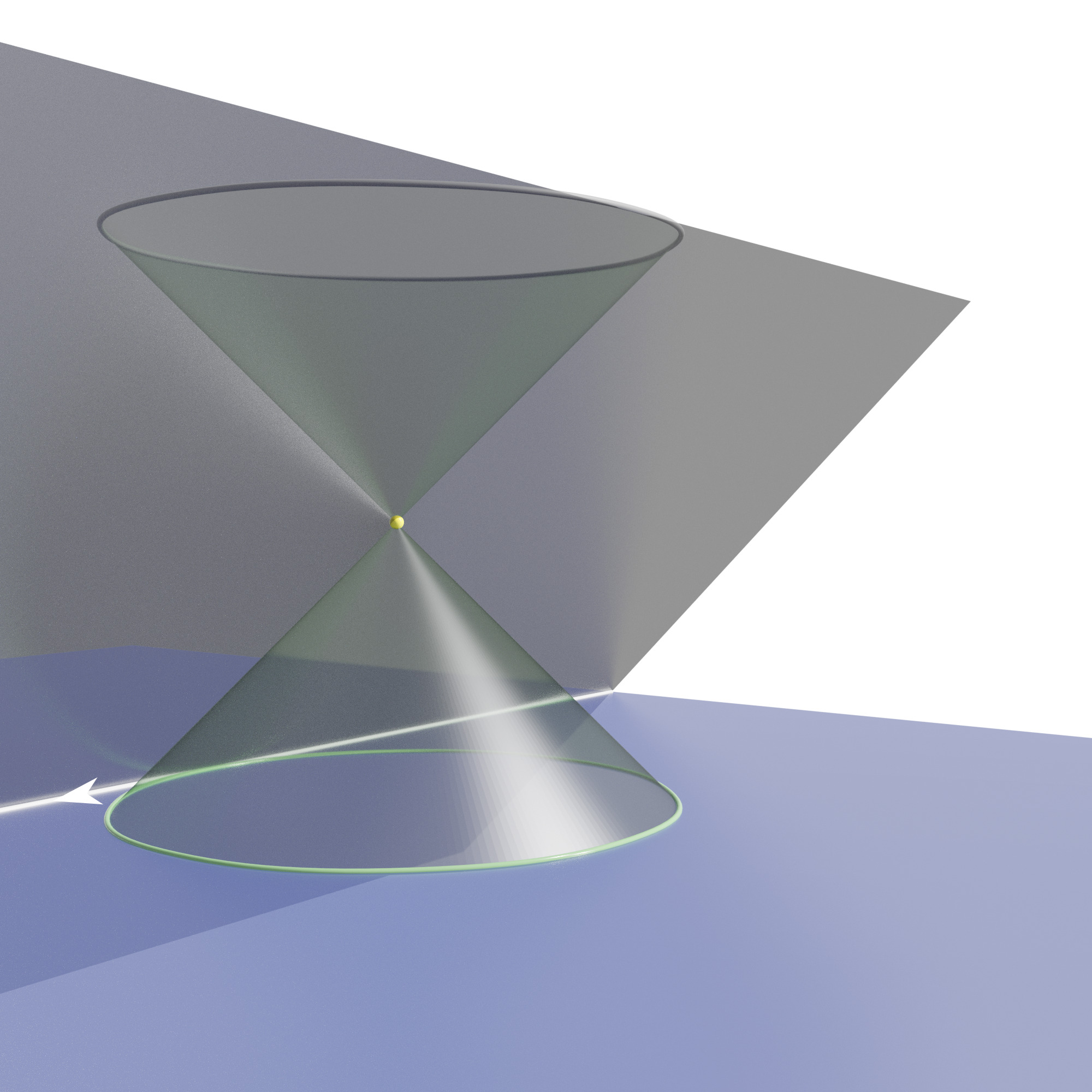}
    \put(85,90){$(\RP^n)^*$}
    \put(93,26){$\eucl$}
    \put(58,82){$\lageucl^*$}
    \put(0,20){$\ell$}
    \put(40,15){$c$}
    \put(44.5,57){$G(c)^*$}
    \put(78,68){$\p{x}^*$}
  \end{overpic}
  \caption{
    Euclidean Laguerre geometry.
    The corresponding Laguerre quadric $\lageucl$ is a cone (``Blaschke cylinder'') with its vertex corresponding to the point $\p{m}_\infty$ that represents the line at infinity in the dual Euclidean plane $\eucl^*$.
    Under dualization the Laguerre quadric becomes a conic $\lageucl^*$ in the cyclographic model of Laguerre geometry.
    A point $\p{x} \in \lageucl$ represents an oriented line $\ell$ in the Euclidean plane $\eucl$.
    By dualization the point becomes a plane $\p{x}^*$ that touches the conic $\lageucl^*$ and intersects $\eucl$ in the line $\ell$.
    A planar section $G(c)$ of $\lageucl$ represents an oriented circle $c$.
    By dualization it becomes a cone $G(c)^*$ that contains the conic $\lageucl^*$ and intersects $\eucl$ in the circle $c$.
  }
\label{fig:laguerre-euclidean}
\end{figure}
In the spirit of Sections \ref{sec:projection} and \ref{sec:laguerre} the absolute quadric $\euclb \subset \RP^n$
of the dual Euclidean (similarity) space with signature $(n,0,1)$ can be lifted to a quadric
$\lageucl \subset \RP^{n+1}$ of signature $(n,1,1)$, which we call the \emph{Euclidean Laguerre quadric},
or classically the \emph{Blaschke cylinder}.
The group of \emph{Euclidean Laguerre transformations} is given by
\[
  \lageucltrafos = \PO(n,1,1).
\]
For a point $\p{p}$ with $\scalarprod{p}{p} < 1$, w.l.o.g.,
\[
  \p{p} \coloneqq [0, \cdots,0,1,0]
\]
the involution $\sigma_{\p{p}}$ and projection $\pi_{\p{p}}$ (see Definition \ref{def:involution-projection})
are still well-defined, and the quotient
\[
  \faktor{\left(\lageucltrafos\right)_{\p{p}}}{\sigma_{\p{p}}} \simeq \PO(n,0,1)
\]
recovers the group of dual Euclidean (similarity) transformations.

The projection $\pi_{\p{p}}$ restricted to $\lageucl$ realizes a double cover of the dual Euclidean space $\euclb^+~=~\eucl^*$,
which may be interpreted as carrying the information of the orientation of the corresponding hyperplanes in $\eucl$.
The involution $\sigma_{\p{p}}$ plays again the role of orientation reversion (see Figure \ref{fig:laguerre-euclidean}, left).

The hyperplanar sections of $\lageucl \subset \RP^{n+1}$ correspond to (the tangent hyperplanes) of a Euclidean sphere in $\eucl$.
Yet due to the degeneracy of $\lageucl$ they cannot be identified with (polar) points in the same space.
Instead they can be identified with points in the dual space $(\RP^{n+1})^*$,
which is classically called the \emph{cyclographic model} of Laguerre geometry (see Figure \ref{fig:laguerre-euclidean}, right).
The dual quadric $\lageucl^*$ is given by a lower dimensional quadric of signature $(n,1)$ contained in the hyperplane $\p{m}_{\infty}^*$.
Thus, the cyclographic model is isomorphic to $(n+1)$-dimensional \emph{Minkowski space}.

\subsection{Lie geometry in Euclidean space}
A Euclidean model of Lie geometry is obtained by stereographic projection of the point complex $\mob \subset \lie$ (cf.\ Section~\ref{sec:lie}).

We write the bilinear form corresponding to the Lie quadric as
\[
  \liesc{x}{y} \coloneqq \dotprod{\project{x}}{\project{y}} - x_{n+2} y_{n+2} - x_{n+3} y_{n+3}
  ~=~ \sum_{i=1}^{n+1}x_i y_i - x_{n+2} y_{n+2} - x_{n+3} y_{n+3}
\]
for $x, y \in \R^{n+3}$, where
\[
  \project{\cdot} : \R^{n+3} \rightarrow \R^{n+1},\quad
  (x_1, \ldots, x_{n+3}) \mapsto (x_1, \ldots, x_{n+1}).
\]
The point complex $\mob$ is projectively equivalent to $\S^n$.
Indeed,
$\p{p}^{\lieperp} = \set{\p{x} \in \RP^{n+2}}{x_{n+3}=0} \simeq \RP^{n+1}$,
and for a point $\p{x} = [\project{x}, 1, 0] \in \p{p}^\lieperp$ we find that in affine coordinates ($x_{n+2} = 1$)
\[
  \liesc{x}{x} = 0 ~\Leftrightarrow~ \dotprod{\project{x}}{\project{x}} = 1.
\]
Thus, we obtain the identification
\[
  \mob
  = \set{\p{x} \in \p{p}^\lieperp}{ \liesc{x}{x} = 0}
  \simeq \set{\project{x} \in \R^{n+1}}{ \dotprod{\project{x}}{\project{x}} = 1}
  = \S^n.
\]
We embed the sphere $\S^n$ into the light cone
\[
  \L^{n+1,2} = \set{x\in\R^{n+3}}{\scalarprod{x}{x} = 0}
\]
in the following way
\begin{equation}
  \label{eq:embedding-sphere}
  \embedS : \S^n\hookrightarrow\L^{n+1,2}, \quad
  \project{x} \mapsto~ \project{x} + e_{n+2} + 0\cdot e_{n+3}.
\end{equation}
Then we have $\mob = \P(\embedS(\S^n))$,
where $\P$ acts one-to-one on the image of $\embedS$.

Denote
\[
  e_{\infty} := \tfrac{1}{2} \left( e_{n+2}+e_{n+1} \right), \quad
  e_0 := \tfrac{1}{2} \left( e_{n+2}-e_{n+1} \right),
\]
which are homogeneous coordinate vectors for the \emph{north pole} and \emph{south pole} of $\mob \simeq \S^n$ respectively.
They satisfy
\[
  \liesc{e_0}{e_0} = \liesc{e_\infty}{e_\infty} = 0, \quad
  \liesc{e_0}{e_\infty} = -\tfrac{1}{2},
\]
and $\liesc{e_0}{e_i} = \liesc{e_\infty}{e_i} = 0$ for $i = 1, \ldots, n, n+3$.
The vectors $e_1, \ldots, e_n, e_0, e_{\infty}, e_{n+3}$ constitute a basis of $\R^{n+1,2}$.
We define an embedding of $\R^n$ into the light cone $\L^{n+1,2}$ by the map
\begin{equation*}
  \embedR : \R^n\hookrightarrow\L^{n+1,2}, \quad
  \tilde{x} \mapsto~ \tilde{x} + e_0 + \abs{\tilde{x}}^2 e_\infty + 0\cdot e_{n+3}
\end{equation*}
and recognize that upon renormalizing the $(n+2)$-nd coordinate to $1$ this is nothing but stereographic projection from $\R^n$
onto the sphere $\S^n$, i.e.\
\[
  (\embedS)^{-1} \circ \embedR : \R^n \rightarrow \S^n, \quad
  \tilde{x} \mapsto \left( \frac{2\tilde{x}}{\abs{\tilde{x}}^2 + 1}, \frac{1 - \abs{\tilde{x}}^2}{1 + \abs{\tilde{x}}^2} \right),
\]
and $\mob = \ P(\embedS(\S^n)) = \P(\embedR(\R^n)) \cup \{ [e_\infty]\}$.
Every point $\p{s} \in \lie$ with $s_0 \neq 0$ can be represented by
\begin{equation*}
    s = \tilde{s} + e_0 + (\abs{\tilde{s}}^2 - r^2)e_\infty + r e_{n+3}
\end{equation*}
with $\tilde{s} \in \R^n$ and $r \in \R$.
Then for $x = \tilde{x} + e_0 + \abs{\tilde{x}}^2 e_\infty$ we find
\[
  \liesc{s}{x} = 0 \quad\Leftrightarrow\quad
  \abs{\tilde{s} - \tilde{x}}^2 = r^2.
\]
Thus, we may identify the point $\p{s}$ with the oriented Euclidean hypersphere of $\R^n$ with center $\tilde{s}$ and signed radius $r \in \R$.
Analogously a point $\p{n} \in \lie$ with $n_0 = 0$  may be represented by
\[
  n = \tilde{n} + 0 \cdot e_0 + 2d e_\infty + e_{n+3}
\]
and identified with the oriented hyperplane of $\R^n$ with normal $\tilde{n} \in \S^{n-1}$ and signed distance $d \in \R$ to the origin.
\begin{proposition}
  \label{prop:geo_interpretation_lie_euclidean}
  Under the aforementioned identification two oriented hyperspheres/hyperplanes of Euclidean space
  are in oriented contact if and only if the corresponding points on the Lie quadric are Lie orthogonal.
\end{proposition}
\begin{proof}
  For, e.g.,\ two oriented hyperspheres of $\R^n$ represented by homogeneous coordinate vectors
  $s_i = \tilde{s_i} + e_0 + (\abs{\tilde{s_i}}^2 - r_i^2)e_\infty + r_i e_{n+3}$, $i = 1,2$
  we find
  \[
    \liesc{s_1}{s_2} = 0
    ~\Leftrightarrow~
    \abs{\tilde{s}_1 - \tilde{s}_2}^2 = (r_1 - r_2)^2.
  \]
\end{proof}
The condition $n_0 = 0$, which characterizes the oriented hyperplanes among all oriented hyperspheres,
is equivalent to $\liesc{n}{e_\infty} = 0$.
Thus, we can interpret oriented hyperplanes as oriented hyperspheres containing the point $\p{q} \coloneqq [e_\infty]$.
Similar to the point complex (see Definition \ref{def:point-complex}), we may introduce the \emph{Euclidean plane complex}
(cf.\ Definition \ref{def:plane-complex})
\begin{equation}
  \label{eq:euclidean-plane-complex}
  \lie \cap \p{q}^\lieperp \simeq \lageucl
\end{equation}
representing all oriented hyperplanes of $\R^n$.
The Euclidean plane complex is a parabolic sphere complex (see Definition \ref{def:sphere-complexes}).
Its signature is given by $(n,1,1)$, and we recover Euclidean Laguerre geometry (cf. Section \ref{sec:euclidean-laguerre-geometry})
by considering the action on $\p{q}^\perp$ of all Lie transformations that fix the point $\p{q}$
\[
  \lietrafos_{\p{q}} \simeq \lageucltrafos.
\]

\begin{table}
  \centering
  \begin{tabular}{c|c}
    \textbf{Euclidean geometry} & \textbf{Lie geometry}\\
    \hline\hline
    \emph{point}\;$\tilde{x}\in\R^n$
              &$\begin{aligned} & [\tilde{x} + e_0 + \abs{x}^2 e_\infty + 0\cdot e_{n+3}] \\
                &= \left[\tilde{x}, \tfrac{1-\abs{\tilde{x}}^2}{2}, \tfrac{1+\abs{\tilde{x}}^2}{2}, 0\right] \in \lie\end{aligned}$\\
    \hline
    \makecell{\emph{oriented hypersphere}\\ with center $\tilde{s}\in\R^n$ and signed radius $r \in \R$}
              &$\begin{aligned} & [\tilde{s} + e_0 + (\abs{\tilde{s}}^2 - r^2)e_\infty + r e_{n+3}] \\
                &= \left[\tilde{s}, \tfrac{1 - \abs{\tilde{s}}^2 + r^2}{2}, \tfrac{1 + \abs{\tilde{s}}^2 - r^2}{2}, r\right] \in \lie \end{aligned}$\\
    \hline
    \makecell{\emph{oriented hyperplane}\;$\scalarprod{\tilde{n}}{\tilde{x}}=d$,\\ with normal $\tilde{n} \in \S^{n-1}$ and signed distance $d \in \R$}
              &$\begin{aligned} & [\tilde{n} + 0 \cdot e_0 + 2d e_\infty + e_{n+3}] \\
                &= \left[\tilde{n}, -2d, 2d, 1\right] \in \lie \end{aligned}$
  \end{tabular}
  \caption{Correspondence between the geometric objects of Lie geometry in Euclidean space and points on the Lie quadric.}
\end{table}

%%% Local Variables:
%%% mode: latex
%%% TeX-master: "main"
%%% End:

\newpage
\section{Generalized signed inversive distance}
\label{sec:invariant}

While two points $\p{x}, \p{y} \in \quadric$ on a quadric $\quadric \subset \RP^{n+1}$ with $\scalarprod{x}{y} \neq 0$ possess no projective invariant, the additional choice of a fixed point $\p{q} \in \RP^{n+1} \setminus \quadric$ allows for the definition of such an invariant. It is closely related to the Cayley-Klein distance under the projection from the point $\p{q}$.

A special case is given by a signed version of the classical \emph{inversive distance} introduced Coxeter \cite{Cox}, which generalizes the intersection angle of spheres. It can be used for a geometric description of sphere complexes in Lie geometry.

\subsection{Invariant on a quadric induced by a point}
\begin{definition}
  \label{def:sphere-complex-distance}
  Let $\p{q} \in \RP^{n+1} \setminus \quadric$.
  Then we call
  \[
    \inv{\quadric}{\p{q}}{\p{x}}{\p{y}} \coloneqq 1 - \frac{\scalarprod{x}{y} \scalarprod{q}{q}}{\scalarprod{x}{q} \scalarprod{y}{q}}.
  \]
  the \emph{$\p{q}$-distance} for any two points $\p{x}, \p{y} \in \quadric$.
\end{definition}
\begin{remark}
  Although we are interested in the $\p{q}$-distance of points on the quadric for now,
  it can be extended to all of $\RP^{n+1} \setminus \p{q}^\perp$.
  Then the relation between the $\p{q}$-distance and the Cayley-Klein distance induced by $\quadric$ is given by
  \[
    \ck{\quadric}{\p{x}}{\p{y}}
    = \frac{(1-\inv{\quadric}{\p{q}}{\p{x}}{\p{y}})^2}{(1-\inv{\quadric}{\p{q}}{\p{x}}{\p{x}})(1-\inv{\quadric}{\p{q}}{\p{y}}{\p{y}})}
  \]
  for $\p{x}, \p{y} \in \RP^{n+1}\setminus(\quadric \cup \p{q}^\perp)$.
\end{remark}
The $\p{q}$-distance is projectively well-defined,
in the sense that it does not depend on the choice of homogeneous coordinate vectors for the points $\p{q}$, $\p{x}$, and $\p{y}$,
and it is invariant under the action of the group $\PO(r,s,t)_{\p{q}}$:
\begin{proposition}
  Let $\p{q} \in \RP^{n+1} \setminus \quadric$.
  Then the $\p{q}$-distance is invariant under all projective transformations that preserve the quadric $\quadric$
  and fix the point $\p{q}$, i.e.\
  \[
    \inv{\quadric}{\p{q}}{f(\p{x})}{f(\p{y})} = \inv{\quadric}{\p{q}}{\p{x}}{\p{x}}
  \]
  for $f \in \PO(r,s,t)_{\p{q}}$ and $\p{x}, \p{y} \in \quadric$.
\end{proposition}
Applying the involution $\sigma_{\p{q}}$ to only one of the arguments of the $\p{q}$-distance
results in a change of sign.
\begin{proposition}
  Let $\p{q} \in \RP^{n+1} \setminus \quadric$.
  Then the $\p{q}$-distance satisfies
  \[
    \inv{\quadric}{\p{q}}{\sigma_{\p{q}}(\p{x})}{\p{y}} =
    \inv{\quadric}{\p{q}}{\p{x}}{\sigma_{\p{q}}(\p{y})} =
    - \inv{\quadric}{\p{q}}{\p{x}}{\p{y}}.
  \]
  for all $\p{x}, \p{y} \in \quadric$.
\end{proposition}
\begin{proof}
  Using Definitions \ref{def:sphere-complex-distance} and \ref{def:involution-projection} we obtain
  \[
    \begin{aligned}
      \inv{\quadric}{\p{q}}{\sigma_{\p{q}}(\p{x})}{\p{y}}
      &= 1 - \frac{\scalarprod{\sigma_q(x)}{y}\scalarprod{q}{q}}{\scalarprod{\sigma_q(x)}{q}\scalarprod{y}{q}}
      = 1 - \frac{\scalarprod{x}{y}\scalarprod{q}{q} - 2\scalarprod{x}{q}\scalarprod{y}{q}}{-\scalarprod{x}{q}\scalarprod{y}{q}}\\
      &= \frac{\scalarprod{x}{y}\scalarprod{q}{q}}{\scalarprod{x}{q}\scalarprod{y}{q}} - 1
      = - \inv{\quadric}{\p{q}}{\p{x}}{\p{y}}.
    \end{aligned}
  \]
\end{proof}
Thus, we find that the square of the $\p{q}$-distance is well-defined on the quotient $\sfrac{\quadric}{\sigma_{\p{q}}}$,
which, according to Proposition \ref{prop:involution-projection}, can be identified with its projection $\pi_{\p{q}}(\quadric)$ to the plane $\p{q}^\perp$.
In this projection the square of the $\p{q}$-distance becomes the Cayley-Klein distance induced by $\secquadric = \quadric \cap \p{q}^\perp$ (see Proposition \ref{prop:Cayley-Klein-distance-lift})
\[
  \inv{\quadric}{\p{q}}{\p{x}}{\p{y}}^2 = \ck{\secquadric}{\pi_{\p{q}}(\p{x})}{\pi_{\p{q}}(\p{y})}.
\]
Hypersurfaces of $\quadric$ of constant $\p{q}$-distance to a point on $\quadric$
are hyperplanar sections of $\quadric$, i.e.\ the $\quadric$-spheres (see Definition \ref{def:Q-sphere}).
\begin{proposition}
  \label{prop:q-spheres-planar-sections}
  The hypersurface in $\quadric$ of constant $\p{q}$-distance $\nu \in \R$ to a point $\p{\tilde{x}} \in \quadric$
  is given by the intersection with the polar hyperplane of the point $\p{x} \in \RP^{n+1}$,
  \[
    x \coloneqq \scalarprod{q}{q}\tilde{x} + (\nu - 1)\scalarprod{\tilde{x}}{q}q,
  \]
  i.e.\
  \[
    \set{\p{y} \in \quadric}{\inv{\quadric}{\p{q}}{\p{\tilde{x}}}{\p{y}} = \nu} ~=~ \p{x}^\perp \cap \quadric.
  \]
\end{proposition}
\begin{proof}
  The equation
  \[
    \inv{\quadric}{\p{q}}{\p{\tilde{x}}}{\p{y}} = 1 - \frac{\scalarprod{\tilde{x}}{y}\scalarprod{q}{q}}{\scalarprod{\tilde{x}}{q}\scalarprod{y}{q}} = \nu
  \]
  is equivalent to
  \[
    \scalarprod{x}{y} = \scalarprod{q}{q}\scalarprod{\tilde{x}}{y} + (\nu - 1)\scalarprod{\tilde{x}}{q}\scalarprod{q}{y} = 0.
  \]
\end{proof}
But are all hyperplanar sections of $\quadric$ such hypersurfaces (cf.\ Proposition \ref{prop:Q-sphere-projection})?
Following Proposition \ref{prop:q-spheres-planar-sections} the potential centers of a given planar section $\p{x}^\perp \cap \quadric$
are given by the points of intersection of the line $\p{q} \wedge \p{x}$ with the quadric $\quadric$.
Yet such lines do not always intersect the quadric in real points.
\begin{proposition}
  \label{prop:q-spheres-planar-sections2}
  Denote by
  \[
    \Delta_q(x) \coloneqq \scalarprod{x}{q}^2 - \scalarprod{x}{x}\scalarprod{q}{q} = - \scalarprod{q}{q}\scalarprod{x}{x}_q
  \]
  the quadratic form of the cone of contact $\cone{\quadric}{\p{q}}$.
  Let $\p{x} \in \RP^{n+1}$ such that $\p{x}^\perp \cap \quadric \neq \varnothing$.
  \begin{itemize}
  \item If $\Delta_q(x) > 0$, then the line $\p{q} \wedge \p{x}$ intersects the quadric $\quadric$ in two (real) points, and
    \[
      \p{x}^\perp \cap \quadric = \set{\p{y} \in \quadric}{\inv{\quadric}{\p{q}}{\p{x}_\pm}{\p{y}} = \nu_\pm}
    \]
    with
    \[
      x_{\pm} = \scalarprod{q}{q}x + \left( -\scalarprod{x}{q} \pm \sqrt{\Delta} \right) q,\qquad
      \nu_\pm \coloneqq \pm \frac{\scalarprod{x}{q}}{\sqrt{\Delta}}.
    \]
  \item If $\Delta_q(x) < 0$, then the line $\p{q} \wedge \p{x}$ intersects the quadric $\quadric$ in two complex conjugate points, and
    \[
      \p{x}^\perp \cap \quadric = \set{\p{y} \in \quadric}{\inv{\quadric}{\p{q}}{\p{x}_\pm}{\p{y}} = \nu_\pm}
    \]
    with
    \[
      x_{\pm} = \scalarprod{q}{q}x + \left( -\scalarprod{x}{q} \pm i\sqrt{-\Delta} \right) q,\qquad
      \nu_\pm \coloneqq \pm \frac{\scalarprod{x}{q}}{i\sqrt{-\Delta}}.
    \]
  % \item If $\Delta_q(x) = 0$, then the line $\p{q} \wedge \p{x}$ is tangent to $\quadric$, and
  %   \[
  %     \note[JT]{add something sensible here}
  %   \]
  %   with
  %   \[
  %     \p{\tilde{x}} = \pi_{\p{q}}(\p{x}) = \left[ \scalarprod{q}{q}x - \scalarprod{q}{x} q \right].
  %   \]
  \end{itemize}
\end{proposition}
\begin{proof}
  The first equality for the quadratic form of the cone of contact follows from Lemma \ref{lem:cone-of-contact},
  while the second equality immediately follows from substituting $x = \alpha q + \pi_q(x)$.
  
  In the case $\Delta_q(x) \neq 0$ the form of the intersection points $\p{x}_\pm$ follows from Lemma \ref{lem:quadric-line-intersection}.
  Substituting into the $\p{q}$-distance gives, e.g., in the case $\Delta_q(x) > 0$
  \[
    \inv{\quadric}{\p{q}}{\p{x}_\pm}{\p{y}}
    = 1 - \frac{(-\scalarprod{x}{q} \pm \sqrt{\Delta})\scalarprod{q}{q}}{\scalarprod{x_\pm}{q}}
    = \pm \frac{\scalarprod{x}{q}}{\sqrt{\Delta}},
  \]
  where we used $\scalarprod{x}{q} = \pm \sqrt{\Delta} \scalarprod{q}{q}$.
\end{proof}
\begin{remark}
  The $\p{q}$-distance of two points $\p{\tilde{x}}, \p{\tilde{y}} \in \RP^{n+1}$ with $\p{x}^\perp \cap \quadric \neq \varnothing$, which represent two $\p{q}$-spheres with centers $\p{x}, \p{y} \in \quadric$ and $\p{q}$-radii $\nu_1, \nu_2$ is given by
  \[
    \inv{\quadric}{\p{q}}{\p{\tilde{x}}}{\p{\tilde{y}}} = \frac{\inv{\quadric}{\p{q}}{\p{x}}{\p{y}}}{\nu_1 \nu_2}.
  \]
  Note that the change of the representing center and radius, e.g.\ $\p{x} \rightarrow \sigma_{\p{q}}(\p{x})$, $\nu_1 \rightarrow -\nu_1$,
  leaves the resulting quantity invariant.
\end{remark}

\subsection{Signed inversive distance}
\label{sec:signed-inversive-distance}
We first give a Euclidean definition for the signed inversive distance.
\begin{definition}
  The \emph{signed inversive distance} of two oriented hyperspheres in $\R^n$ 
  with centers $\tilde{s}_1, \tilde{s}_2 \in \R^n$ and signed radii $r_1, r_2 \in \R$ is given by
  \[
    I \coloneqq \frac{r_1^2 + r_2^2 - \abs{\tilde{s}_1 - \tilde{s}_2}^2}{2 r_1 r_2}.
  \]
  In particular, if the two spheres intersect, it is the cosine of their intersection angle,
  by the cosine law for Euclidean triangles.
\end{definition}
\begin{remark}
  This classical invariant is usually given in its unsigned version,
  which was introduced by Coxeter \cite{Cox} as a Möbius invariant.
\end{remark}
\begin{proposition}
  \label{prop:angle}
  The signed inversive distance $I$ satisfies
  \begin{itemize}
  \item $I \in (-1,1)$ $\Leftrightarrow$ the two oriented hyperspheres intersect.
    In this case $I = \cos\varphi$ where $\varphi \in [0,\pi]$ is the angle between the two oriented hyperspheres.
  \item $I = 1$ $\Leftrightarrow$ the two oriented hyperspheres touch with matching orientation (Lie incidence).
  \item $I = -1$ $\Leftrightarrow$ the two oriented hyperspheres touch with opposite orientation.
  \item $I \in (\infty, -1) \cup (1,\infty)$ $\Leftrightarrow$ the two oriented hyperspheres are disjoint.
  \end{itemize}
\end{proposition}
The signed inversive distance is nothing but the $\p{p}$-distance (see Definition \ref{def:sphere-complex-distance})
associated with the point complex $\mob \subset \lie$ in Lie geometry (see Definition \ref{def:point-complex}),
where
\[
  \p{p} = [0, \cdots, 0, 1] \in \RP^{n+2}.
\]
\begin{proposition}
  For two oriented hyperspheres represented by
  \[
    \p{s}_i = [\tilde{s}_i + e_0 + (\abs{\tilde{s}_i}^2 - r_i^2)e_\infty + r_ie_{n+3}], \qquad i=1,2
  \]
  with Euclidean centers $\tilde{s}_1, \tilde{s}_2 \in \R^n$ and signed radii $r_1, r_2 \neq 0$ the $\p{p}$-distance
  associated with the point complex $\mob$ is equal to the signed inversive distance, i.e.
  \[
    \inv{\lie}{\p{p}}{\p{s}_1}{\p{s}_2} = \frac{r_1^2 + r_2^2 - \abs{\tilde{s}_1 - \tilde{s}_2}^2}{2 r_1 r_2}.
  \]
\end{proposition}
\begin{proof}
  With the given representation of the hyperspheres we find
  \[
    \inv{\lie}{\p{p}}{\p{s}_1}{\p{s}_2}
    = 1 - \frac{\liesc{s_1}{s_2}\liesc{p}{p}}{\liesc{s_1}{p}\liesc{s_2}{p}}
    = 1 + \frac{(r_1^2 + r_2^2 - 2r_1r_2)- \abs{\tilde{s}_1 - \tilde{s}_2}^2}{2r_1r_2}
    = \frac{r_1^2 + r_2^2 - \abs{\tilde{s}_1 - \tilde{s}_2}^2}{2r_1r_2}.
  \]
\end{proof}
\begin{remark}
  Since we have expressed the signed inversive distance in terms of the $\p{p}$-distance
  it follows that it is similarly well-defined for two oriented hyperspheres of $\S^n$.
  Furthermore, the signed inversive distance is invariant under all Lie transformations that preserve the point complex $\mob$, i.e.\ all Möbius transformations.
  In particular, the intersection angle of spheres is a Möbius invariant.
  As follows from Proposition \ref{prop:Cayley-Klein-distance-lift} the Cayley-Klein distance of Möbius geometry,
  i.e.\ the Cayley-Klein distance induced by $\mob$ onto $\p{p}^\lieperp$ is the squared inversive distance.
\end{remark}

% \begin{proposition}
%   The Cayley-Klein distance of two points $[s_1], [s_2] \in \mob^+$ is equal to their inversive distance.
% \end{proposition}

% \note[JT]{Remark on cone of contact to $\mob$ given by $K(\p{s_1},\p{s_2}) = 1$.}
% Instead of the sphere complex distance one might want to use
% \[
%   \frac{\liesc{s_1}{s_2} \liesc{p}{p}}{\liesc{s_1}{p} \liesc{s_2}{p}} = 1 - I = 2\sin^2\frac{\varphi}{2},
% \]
% where $\varphi$ is the intersection angle if the spheres intersect.

\subsection{Geometric interpretation for sphere complexes}
\label{sec:sphere-complexes-geometric}
We now use the inversive distance to give a geometric interpretation for most sphere complexes in Lie geometry (see Definition \ref{def:sphere-complexes}).
Let again
\[
  \p{p} = [0, \cdots, 0, 1] \in \RP^{n+2},
\]
which distinguishes the point complex $\mob = \lie \cap \p{p}^\perp$.
\begin{proposition}
  \label{prop:sphere-complexes}
  Let $\p{q} \in \RP^{n+2}$, $\p{q} \neq \p{p}$ such that the line $\p{p} \wedge \p{q}$
  through $\p{p}$ and $\p{q}$ intersects the Lie quadric in two points,
  i.e.\ $\p{p} \wedge \p{q}$ has signature $(+-)$.
  Denote by
  \[
    \{ \p{q}_+, \p{q}_- \} \coloneqq (\p{p} \wedge \p{q}) \cap \lie
  \]
  the two intersection points of this line with the Lie quadric
  (the oriented hyperspheres corresponding to $\p{q}_+$ and $\p{q}_-$ only differ in their orientation).

  Then the sphere complex corresponding to the point $\p{q}$
  is given by the set of oriented hyperspheres that have some fixed constant inversive distance $I_{\lie, \p{p}}$
  to the oriented hypersphere corresponding to $\p{q}_+$,
  or equivalently, fixed constant inversive distance $-I_{\lie, \p{p}}$
  to the oriented hypersphere corresponding to $\p{q}_-$.
  
  In particular, in this case the sphere complex is
  \begin{itemize}
  \item
    \emph{elliptic} if $I_{\lie, \p{p}} \in (-1, 1)$,
  \item
    \emph{hyperbolic} if $I_{\lie, \p{p}} \in (-\infty, -1) \cup (1, \infty)$,
  \item
    \emph{parabolic} if $I_{\lie, \p{p}} \in \{ -1, 1 \}$.
  \end{itemize}
\end{proposition}
\begin{proof}
  The two points $\p{q}_{\pm}$ may be represented by
  \[
    q_{\pm} = \tilde{q} + e_0 + \left(\abs{\tilde{q}}^2 - R^2\right)e_{\infty} \pm R e_{n+3},
  \]
  with some $R \neq 0$, where we assumed that the $e_0$-component of $q$ does not vanish.
  The case with $\liesc{q}{e_\infty} = 0$, which corresponds to $\p{q}_{\pm}$ being planes, may be treated analogously.
  
  Now the point $\p{q}$ may be represented by
  \[
    q = \tilde{q} + e_0 + \left(\abs{\tilde{q}}^2 - R^2\right)e_{\infty} + \kappa e_{n+3}
  \]
  with some $\kappa \in \R$.
  For any point $\p{s} \in \lie$ represented by
  \[
    s = \tilde{q} + e_0 + \left(\abs{\tilde{q}}^2 - r^2\right)e_{\infty} +  r e_{n+3},
  \]
  we find that the condition to lie on the sphere complex is given by
  \[
    \liesc{q}{s} = 0
    \quad\Leftrightarrow\quad
    \liesc{q}{s}_p = r \kappa.
  \]
  Thus, the signed inversive distance of $\p{q}_+$ and $\p{s}$ is given by
  \[
    I_{\p{p}}(\p{q}_+, \p{s})
    = 1 - \frac{\liesc{q_+}{s}\liesc{p}{p}}{\liesc{q_+}{p}\liesc{s}{p}}
    = \frac{\liesc{s}{q}_p}{rR}
    = \frac{\kappa}{R}.
  \]
  The change $\p{q}_+ \rightarrow \p{q}_-$ is equivalent to $R \rightarrow -R$
  which leads to $I \rightarrow -I$.

  The distinction of the three types of sphere complexes in terms of the value of the inversive distance
  is obtained by observing that
  \[
        \begin{aligned}
          \liesc{q}{q} &> 0, &&\text{if } \kappa^2 < R^2,\\
          \liesc{q}{q} &< 0, &&\text{if } \kappa^2 > R^2,\\
          \liesc{q}{q} &= 0, &&\text{if } \kappa^2 = R^2.
        \end{aligned}
  \]
\end{proof}
\begin{remark}
  \label{rem:elliptic-sphere-complexes-angle}
  For an elliptic sphere complex the line $\p{p}\wedge\p{q}$ always has signature $(+-)$.
  Furthermore, in this case we have $I_{\p{p}} \in (-1, 1)$.
  Thus, according to Proposition \ref{prop:angle},
  any elliptic sphere complex is given by all oriented hyperspheres with constant angle
  to some fixed oriented hypersphere.

  For hyperbolic sphere complexes the line $\p{p}\wedge\p{q}$ can have signature $(+-)$, $(--)$, or $(-0)$.
  The first case is captured by Proposition \ref{prop:sphere-complexes}.
  An example with signature $(--)$ is given by $\p{q} = [0, \sin R, \cos R]$,
  which describes all oriented hyperspheres of $\S^n$ with spherical radius $R$.
  An example with signature $(-0)$ is given by $\p{q} = [-2R e_\infty + e_{n+3}]$,
  which describes all oriented hyperspheres of $\R^n$ with (Euclidean) radius $R$.
  Note that the point complex $\mob$ itself is also a hyperbolic sphere complex.

  Parabolic sphere complexes are captured by Proposition \ref{prop:sphere-complexes}
  if and only if $\p{q} \not\in \mob$.
  Note that the (Euclidean) plane complex \eqref{eq:euclidean-plane-complex} is parabolic.
\end{remark}

%%% Local Variables:
%%% mode: latex
%%% TeX-master: "main"
%%% End:

%\newpage
%\input{oldlaguerre}

\newpage
% !TeX root  = main.tex

\end{document}